\documentclass[10pt]{amsart}
\usepackage{amssymb}
\usepackage{amsmath}
\usepackage{verbatim}
\usepackage{enumerate}
\usepackage{graphicx}
\usepackage{amscd}
\usepackage{extarrows}
\usepackage{color}
\usepackage{xcolor}
\usepackage{hyperref}

\newcommand{\bZ}{\mathbb{Z}}

\newcommand{\bC}{\mathbb{C}}
\newcommand{\bR}{\mathbb{R}}

\newcommand{\mr}{\mathrm}
\newcommand{\ra}{\rightarrow}
\newcommand{\xra}{\xrightarrow}
\newcommand{\dd}{\partial}
\newcommand{\be}{\begin{equation}}
\newcommand{\ee}{\end{equation}}
\newcommand{\bt}{\bullet}
\newcommand{\tr}{\mathrm{tr}\;}
\newcommand{\hra}{\hookrightarrow}

\newcommand{\proj}{\twoheadrightarrow}
\newcommand{\mleft}{\left(\begin{array}}
\newcommand{\mright}{\end{array}\right)}
\newcommand{\Det}{\mathrm{Det}}
\newcommand{\bT}{\mathbb{T}}
\newcommand{\Dens}{\mathrm{Dens}}

\newcommand{\rk}{\mathrm{rk}\;}
\newcommand{\mc}{\mathcal}
\newcommand{\E}{\mathcal{E}}
\newcommand{\F}{\mathcal{F}}
\newcommand{\EL}{\mathcal{EL}}
\newcommand{\LL}{\mathcal{L}}
\newcommand{\ii}{\mathbf{i}}
\newcommand{\B}{\mathcal{B}}
\newcommand{\HH}{\mathcal{H}}
\newcommand{\pp}{\mathbf{p}}
\newcommand{\K}{\mathbf{K}}
\newcommand{\din}{\mathrm{in}}
\newcommand{\dout}{\mathrm{out}}
\newcommand{\til}{\widetilde}
\newcommand{\Sym}{\mathrm{Sym}\,}
\newcommand{\Azm}{A_\mathrm{res}}
\newcommand{\Bzm}{B_\mathrm{res}}
\newcommand{\Afluct}{A_\mathrm{fluct}}
\newcommand{\Bfluct}{B_\mathrm{fluct}}
\newcommand{\fluct}{\mathrm{fluct}}
\newcommand{\Aout}{A_\mathrm{out}}
\newcommand{\Bin}{B_\mathrm{in}}
\newcommand{\red}{\mathrm{r}}
\newcommand{\can}{\mathrm{can}}
\newcommand{\wavy}[1]{\stackrel{#1}{\rightsquigarrow}}
\newcommand{\D}{\mathcal{D}}
\newcommand{\rel}{\mathrm{rel}}
\newcommand{\zm}{\mathrm{res}}
\newcommand{\gzm}{\mathrm{gres}}
\newcommand{\gl}{\mathrm{g}}
\newcommand{\comp}{\mathrm{c}}
\newcommand{\bulk}{\mathrm{bulk}}
\newcommand{\cob}[1]{\xLongrightarrow{#1}}
\newcommand{\Ppol}{\mathcal{P}}
\newcommand{\Qpol}{\mathcal{Q}}
\newcommand{\Fun}{\mathrm{Fun}}
\newcommand{\ddelta}{\mathfrak{D}}
\newcommand{\zeromodes}{residual fields}
\newcommand{\KK}{\mathbb{K}}
\newcommand{\g}{\mathfrak{g}}
\newcommand{\sA}{\mathsf{A}}
\newcommand{\sB}{\mathsf{B}}
\newcommand{\sF}{\mathsf{F}}
\newcommand{\sfi}{\mathsf{i}}
\newcommand{\sfp}{\mathsf{p}}
\newcommand{\sfK}{\mathsf{K}}
\newcommand{\wed}{\ast}
\newcommand{\DDelta}{\mathbf{\Delta}}
\newcommand{\HDens}{\mathrm{Dens}^{\frac12}}
\newcommand{\lan}{\left\langle}
\newcommand{\ran}{\right\rangle}

\newcommand{\ola}{\overleftarrow}
\newcommand{\ora}{\overrightarrow}
\newcommand{\Ad}{\mathrm{Ad}}
\newcommand{\sd}{\mathsf{d}}
\newcommand{\EE}{\mathcal{E}}
\newcommand{\R}{\mathcal{R}}
\newcommand{\mf}{\mathfrak}

\newcommand{\nn}{n}
\newcommand{\n}{n}
\newcommand{\nnn}{k}

\newtheorem{theorem}{Theorem}[section]
\newtheorem{Proposition}[theorem]{Proposition}
\newtheorem{lemma}[theorem]{Lemma}

\newtheorem{corollary}[theorem]{Corollary}
\newtheorem{example}[theorem]{Example}
\newtheorem{remark}[theorem]{Remark}

\newtheorem{definition}[theorem]{Definition}

\newtheorem{assumption}[theorem]{Assumption}

\begin{document}
\title{
A cellular topological field theory
}

\author{Alberto S. Cattaneo}
\address{Institut f\"ur Mathematik,
Universit\"at Z\"urich,
Winterthurerstrasse 190,
CH-8057, Z\"urich, Switzerland}

\email{
cattaneo @math.uzh.ch
}

\author{Pavel Mnev}
\address{
University of Notre Dame, Notre Dame, Indiana 46556, USA
}
\address{
St. Petersburg Department of V. A. Steklov Institute of Mathematics of the Russian Academy of Sciences, Fontanka 27, St. Petersburg, 191023 Russia}
\email{
pmnev @nd.edu
}

\author{Nicolai Reshetikhin}
\address{Department of Mathematics,
University of California, Berkeley
California 94305,
USA}
\address{
ITMO University. Saint Petersburg 197101, Russia}
\address{
KdV Institute for Mathematics, University of Amsterdam, 1098 XH Amsterdam, The
Netherlands
}
\email{reshetik @math.berkeley.edu}

\thanks{This research was (partly) supported by the NCCR SwissMAP, funded by the Swiss National Science Foundation, and by the COST Action MP1405 QSPACE, supported by COST (European Cooperation in Science and Technology). P. M. acknowledges partial support of RFBR Grant No. 17-01-00283a}

\date{\today}

\maketitle

\begin{abstract}
We present a construction of cellular $BF$ theory (in both abelian and non-abelian variants) on cobordisms equipped with cellular decompositions. Partition functions of this theory are invariant under subdivisions,  satisfy  a version of the quantum master equation,
and satisfy Atiyah-Segal-type gluing formula with respect to composition of cobordisms.
\end{abstract}

\tableofcontents

\section{Introduction}

In this paper we present a combinatorial model of a topological field theory on  cobordisms endowed with a cellular decomposition and a  local system $E$, where the fields are modelled on cellular cochains. The model is compatible with composition (concatentation) of cobordisms.  In the limit of a dense cellular decomposition (with mesh going to zero), our combinatorial model converges, in an appropriate sense (for details, see Section \ref{sss: cont limit}), to the topological $BF$ theory in the Batalin--Vilkovisky (BV) formalism. Cellular cochains in this context 
arise as a combinatorial replacement of differential forms -- the fields of the continuum model.

Quantization of this model is given by well-defined finite-dimensional integrals (which replace in this context the functional integral of quantum field theory). The model  is formulated in the Batalin--Vilkovisky formalism (or rather its extension, ``BV-BFV formalism'' \cite{CMR,CMRpert}, for manifolds with boundary, which is compatible with gluing/cutting).\footnote{We will give a short, working-knowledge introduction to the BV and the BV-BFV formalisms in this paper, but the reader is referred to the literature, especially \cite{CMRpert}, for more details.}
 The construction of quantization depends on a choice of retraction of cellular cochains onto cohomology (in particular, a choice of chain homotopy); this retraction represents the data of gauge-fixing in this context. The space of choices is contractible.

The result of quantization  is a cocycle (the partition function) in a certain cochain complex, constructed as a tensor product of a complex associated to the boundary (the space of states) and a complex associated to the ``bulk'' -- the cobordism itself (half-densities on the space of ``residual fields'' modelled on cellular cohomology). The partition function satisfies a gluing rule (a variant of Atiyah-Segal gluing axiom of quantum field theory)
with respect to concatenation of cobordisms and, when considered modulo coboundaries, is independent of the cellular decomposition of the cobordism. 
The cocycle property of the partition function is a variant of the Batalin-Vilkovisky quantum master equation modified for the presence of boundary. Changing the choices involved in quantization changes the partition function by a coboundary. 

The model presented in this paper is, on one side, an explicit example of the BV-BFV framework  for quantization of gauge theories in a way compatible with cutting-pasting of the spacetime manifolds, developed by the authors in \cite{CMR,CMRpert} (a short survey of the BV-BFV programme can be found in \cite{CMRsurvey}). On the other side, it is a development of the work \cite{SimpBF, DiscrBF} and provides a replacement for the path integral in a topological field theory by a coherent (w.r.t.\ aggregations) system of cellular models, in such a way that each of them can be used to calculate the partition function as a finite-dimensional integral (exactly, i.e., without having to pass to a limit of dense refinement).\footnote{
Besides casting the model into the BV-BFV setting, with cobordisms and Segal-like gluing,
some of the important advancements 
over \cite{SimpBF,DiscrBF} are the following: 
general regular CW complexes are allowed (as opposed to simplicial and cubic complexes); the new construction of the cellular action which is intrinsically finite-dimensional and in particular does not use regularized infinite-dimensional super-traces; a systematic, intrinsically finite-dimensional, treatment of the behavior w.r.t. moves of CW complexes -- elementary collapses and cellular aggregations; understanding the constant part of the partition function (leading to the contribution of the Reidemeister torsion and the $\bmod\,16$ phase); incorporating the twist by a  nontrivial local system.
} 

We present both the abelian and the non-abelian versions of the model. In the abelian version, when defined on a closed manifold endowed with an acyclic local system, 
the partition function  is the Reidemeister torsion. For a non-acyclic local system, one gets the Reidemeister torsion (which is, in this case, not a number, but an element of the determinant line of cellular cohomology, defined modulo sign), up to a factor depending on Betti numbers and containing a $\bmod\; 16$ complex phase.  

In the non-abelian case, the model depends on the choice of a unimodular Lie algebra $\g$ of coefficients. The action of the model is constructed in terms of local unimodular $L_\infty$ algebras defined on $\g$-valued cochains on closures of individual cells.\footnote{
In our presentation of this result (Theorem \ref{thm: cellBF}), the local unimodular $L_\infty$ algebras are packaged into generating functions -- the local building blocks $\bar{S}_e$ for the cellular action.
To be precise, the sum of building blocks $\bar{S}_{e'}$ over all cells $e'$ belonging to the closure of the given cell $e$ is the generating function for the operations (structure constants) of the local unimodular $L_\infty$ algebra assigned to $e$, in the sense of 
Section \ref{sec: uL_infty} and (\ref{BF_infty ansatz for S_cell}).
} On $0$-cells these local unimodular $L_\infty$ algebras coincide with $\g$;   on higher-dimensional cells 
they are constructed by induction in skeleta. Each step of this induction is an inductive construction in its own right where one starts with the algebra for the boundary of the cell $\dd e$ and extends it by a piece corresponding to the component of the cellular differential mapping $\dd e$ to $e$ and then continues to add higher operations to correct for the error in the coherence relations of the algebra.  This is an inductive  construction by obstruction theory which has a solution which is explicit once certain choice of local chain homotopies is made. Moreover, the space of choices involved is contractible and two different cellular actions are ``homotopic'' in the appropriate sense (i.e., related by a canonical BV transformation). Operations of the local unimodular $L_\infty$ algebras for cells are expressed in terms of nested commutators, traces in $\g$ and certain interesting structure constants which can be made rational (with a good choice of local chain homotopies). For example, for $1$-cells these constants are expressed in terms of Bernoulli numbers. 

The non-abelian partition function for a closed manifold with cellular decomposition 
is expressed in terms of the Reidemeister torsion, the $\bmod\; 16$ phase, and a sum of Feynman diagrams.
The latter 
 encode  the data of the induced unimodular $L_\infty$ algebra structure on the cohomology of the manifold. The classical $L_\infty$ part of this algebra contains the Massey brackets (also known as Massey-Lie brackets). on cohomology and is, in case of a simply connected manifold, a complete invariant of the rational homotopy type of the manifold. Also, this $L_\infty$ algebra yields 
a deformation-theoretic description of the formal neighborhood of the (possibly, singular) point corresponding to the local system 
$E$ on the moduli space of 
local systems (in non-abelian case $E$ is interpreted as a choice of background flat bundle around which the theory is perturbatively quantized). The quantum part of the partition function (corresponding to the ``unimodular'' or ``quantum'' operations of the algebraic structure  on cohomology) is related to the behavior of the Reidemeister torsion in the neighborhood of $E$ on the moduli space of local systems. In the case of a cobordism, we have a version of this structure relative to the boundary and compatible with concatenation of cobordisms. The space of states associated to a boundary component in the non-abelian case is the same as in the abelian case as a graded vector space but with a more complicated differential. The cohomology of this differential is the Chevalley-Eilenberg cohomology of the 
$L_\infty$ algebra structure on the cohomology of the boundary and thus
 is an invariant of the rational homotopy type of the boundary.
 
The non-abelian actions assigned to CW complexes, when considered modulo canonical BV transformations, are compatible with local moves of CW complexes -- cellular aggregations (inverses of subdivisions) and Whitehead's elementary collapses (which, together with their inverses, elementary expansions, generate the simple-homotopy equivalence of CW complexes). Both moves -- an aggregation and a collapse are represented by 
a fiber BV integral (a.k.a.\  BV pushforward) along the corresponding fibration of fields on a bigger complex over fields on a smaller complex.  Expansions and  collapses are the more fundamental moves (aggregations can be decomposed as expansions and collapses) but generally do not preserve the property of CW complexes to correspond to manifolds. In fact, we consider two versions of the non-abelian theory: 
\begin{enumerate}[I]
\item The ``canonical" version -- Section \ref{sec: non-ab BF I}. Here the fields are a cochain and a chain of the same CW complex $X$ which is not required to be a manifold. The cellular actions are compatible with elementary collapses (Lemma \ref{lemma: collapse}) and the partition function, defined via BV pushforward to cohomology,  is  a simple-homotopy invariant (Proposition \ref{prop: non-ab BF I simple homotopy invariance}). Here one has a version of Mayer-Vietoris gluing formula for cellular actions (see (\ref{simpBF (vi)}) of Theorem \ref{thm: simpBF}), which  is not of Segal type, since one of the fields has ``wrong''  (covariant) functoriality.
\item The version on cobordisms (of a fixed dimension) -- Section \ref{sec: non-ab BF cobordism}, with Segal-type gluing formula w.r.t. concatenation of cobordisms. Here the fields are a pair of cochains of $X$ and the dual complex $X^\vee$, and $X$ is required to be a cellular decomposition of a cobordism. In this picture one does not have elementary expansions and collapses on the nose, but one has cellular aggregations, and one can prove the compatibility of the theory w.r.t.\ aggregations by temporarily passing to the canonical version and presenting the aggregation via expansions and collapses (Proposition \ref{prop: aggregations}, (\ref{prop 8.15 (iii)}) of Theorem \ref{prop 8.15}).
\end{enumerate}
One can regard the passage from more dense to more sparse cellular decompositions via BV pushforwards as a version of Wilson's renormalization group flow, passing from a higher energy effective theory to a lower energy effective theory.

It is important to note two (related) features that set the cellular model apart from continuum field theories in the BV-BFV formalism and could be regarded as artifacts of discretization:
\begin{itemize}
\item The polarization of the space of phase spaces (a.k.a.\ spaces of boundary fields) assigned to the boundaries is built into the theory on a cobordism already at the classical level, 
via the convention for the Poincar\'e dual of the cellular decomposition (and thus is built into the definition of the space of classical fields (\ref{cell BF fields on a cobordism})).\footnote{In this paper we use the convention that the polarization is linked to the designation of boundaries as in/out. Thus we link out-boundaries to ``$A$-polarization'' and in-boundaries with ``$B$-polarization''. This convention is entirely optional. On the other hand, the link between polarization (\ref{polarization}) and the notion of the dual CW complex (Section \ref{sec: Lefschetz for cell decomp, cobordism}) is essential for the construction. } This is different from the usual situation \cite{CMRpert} where one chooses the polarization at a later stage, as a datum necessary for quantization.
\item The BV 2-form on the space of fields is degenerate in presence of the boundary, i.e., it is a ($-1$-shifted) pre-symplectic structure, rather than a symplectic one. However, once restricted to the subspace of fields satisfying an admissible boundary condition (as determined by polarization of boundary phase spaces), the BV 2-form becomes non-degenerate. This property hinges on the link between the convention for the Poincar\'e dual complex and the polarization stressed above. As a consequence, BV integrals make sense fiberwise, in the family  over the space $\mc{B}_\dd$ parameterizing the admissible boundary conditions.
\end{itemize}

Throughout the paper we use the language of perturbative integrals (i.e., of stationary phase asymptotic formula for oscillating integrals) and we use the ``Planck constant'' $\hbar$ as the conventional bookkeeping formal parameter 
controlling the frequency of oscillations. However, one can always choose $\hbar$ to be a finite real number instead: by the virtue of the model at hand, we do not encounter series in $\hbar$ of zero convergence radius, as would be usual for stationary phase asymptotics (in fact, in $BF$ theory one does not encounter Feynman diagrams with more than one loop, so the typical power series in $\hbar$ we see in the paper truncate at the order $\mc{O}(\hbar^1)$).

\subsection{Main results}

\subsubsection{Abelian cellular $BF$ theory on cobordisms (Sections \ref{sec: closed}--\ref{sec: quantum cell ab BF on mfd with bdry}) } \label{sec: intro main results abelian}
\begin{enumerate}[I.]
\item \textbf{Classical abelian $BF$ theory on a cobordism endowed with cellular decomposition.} 
For $M$ an $\nn$-dimensional cobordism  between closed $(\nn-1)$-manifolds $M_\din, M_\dout$, endowed with a cellular decomposition $X$ and a coefficient local system (flat bundle) $E$ of rank $m$ with holonomies of determinant $\pm 1$ ($E$ plays the role of an external parameter -- the twist of the model), we construct (Section \ref{sec: classical cell ab BF on mfd with bdry}; the case of $M$ closed is considered as a warm-up in Section \ref{sec: classical closed}) a field theory in BV-BFV formalism with the following data.
\begin{itemize}
\item The space of fields assigned to the cobordism $(M,X)$ is 
$\F=C^\bt(X,E)[1]\oplus C^\bt(X^\vee,E^*)[\nn-2]$. Here $X^\vee$ is the dual cellular decomposition to $X$ (see Section \ref{sec: reminder on Poincare duality} for details on the cellular dual  for a cellular decomposition of a cobordism).
\item The BV action is $S=\lan B,d_EA \ran + \lan B,A \ran_\din $ where $(A,B)\in \F$ is the cellular field.
\item The BV 2-form on fields is induced from chain level Poincar\'e duality.
\item The cohomological vector field $Q$ is the sum of lifts of cellular coboundary operators on CW complexes $X$ and $X^\vee$, twisted by the local system, to vector fields on the space of fields.
\item The boundary of $(M,X)$ gets assigned the space of boundary fields (or the ``phase space'') $\F_\dd=C^\bt(X_\dd,E)[1]\oplus C^\bt(X_\dd^\vee,E^*)[\nn-2]$ (naturally split into in-boundary-fields and out-boundary fields) which carries:
\begin{itemize}
\item A degree zero symplectic structure (induced from chain level Poincar\'e duality on the boundary), with a preferred primitive $1$-form $\alpha_\dd=\lan B,\delta A \ran_\dout - \lan \delta B, A \ran_\din$ (i.e., it distinguishes between in- and out-boundaries). 
\item A natural projection of bulk fields onto boundary fields (pullback by the geometric inclusion of the boundary) $\pi: \F\ra \F_\dd$. 
\item The cohomological vector field $Q_\dd$ on $\F_\dd$ which is constructed analogously to the bulk -- as a lift of the cellular coboundary operator (on the boundary of the cobordism). This vector field has a degree $1$ Hamiltonian $S_\dd=\lan B, dA \ran_\dd$.
\end{itemize}
\end{itemize}
We prove that this set of data satisfies the structural relations of a classical BV-BFV theory \cite{CMR} -- Proposition \ref{lemma: i_Q=dS+alpha_bdry}. Concatenation of cobordisms here maps to fiber product of the corresponding BV-BFV packages -- Section \ref{sec: class gluing}.
\item \textbf{Quantization on a closed manifold.} In Section \ref{sec: quantization closed} we construct the finite-dimensional ``functional'' integral quantization of the abelian cellular theory on a closed $\nn$-manifold. The partition function $Z$ of the theory is defined as a BV pushforward (fiber BV integral) of the exponential of the cellular action from $\F$ to residual fields modelled on cohomology, $\F^\zm=H^\bt(M,E)[1]\oplus H^\bt(M,E^*)[\nn-2]$. 

Gauge-fixing data for the BV pushforward -- the splitting of fields into residual fields plus the complement and a choice of a Lagrangian subspace in the complement -- is inferred from a choice of ``induction data'' or ``retraction'' (see Section \ref{sec: HPT}) of cellular cochains onto cohomology (i.e. a choice of cellular representatives of cohomology classes, a choice of a projection onto cohomology and a chain homotopy between the identity and projection to cohomology -- the latter plays the role of the propagator in the theory, cf. Remark \ref{rem: propagator, closed case}).

We compute the partition function (Proposition \ref{prop: Z closed})\footnote{Partition functions are defined up to sign for the purposes of this paper, so that we don't need to keep track of orientations on the spaces of fields and gauge-fixing Lagrangians.} to be
$$Z= \xi_\hbar^{H^\bt}\cdot \tau(M,E) \quad \in \mathbb{C}\otimes \mr{Det}\, H^\bt(M,E)/\{\pm 1\}$$
where: 
\begin{itemize}
\item $\tau(M,E)$ is the Reidemeister torsion,
\item $\xi_\hbar^{H^\bt}$ is a complex coefficient depeneding on the Betti numbers of $M$ (twisted by the local system $E$): $\xi_\hbar^{H^\bt}=\prod_{k=0}^\nn (\xi_\hbar^k)^{\dim H^k(M,E)}$ with 
$$\xi_\hbar^k = (2\pi \hbar)^{-\frac14+ \frac12 k (-1)^{k-1}} (e^{-\frac{\pi i}{2}}\hbar)^{\frac14+ \frac12 k (-1)^{k-1}} $$
\end{itemize}
In particular, $Z$ depends only on the topology of $M$ and not on the cellular decomposition $X$. Note that $Z$ contains, via the factor $\xi_\hbar^{H^\bt}$, a $\bmod\,16$ complex phase. 

The mechanism that leads to the factor $\xi_\hbar^{H^\bt}$ (discussed in detail in Section \ref{sec: normalization}) is that, in order to have a partition function independent of $X$, we need to scale the reference half-density on the space of fields (playing the role of the  ``functional integral measure'' in the context of cellular theory) in a particular way -- it differs from the standard cellular half-density by a product over cells $e$ of $X$ of local factors $(\xi_\hbar^{\dim e})^m$ depending only on dimensions of cells. This a baby version of renormalization in the cellular theory and it leads to a partition function containing the factor $\xi_\hbar^{H^\bt}$.
\item \textbf{Quantization on a cobordism.} In Section \ref{sec: quantum cell ab BF on mfd with bdry}, we construct the quantum BV-BFV theory on a cobordism $M$ endowed with cellular decomposition $X$ by quantizing the classical cellular theory via BV pushforward to residual fields, in a family over $\mc{B}_\dd=C^\bt(X_\dout,E)[1]\oplus C^\bt(X^\vee_\din,E^*)[\nn-2] = \B_\dout\oplus \B_\din$ -- the base of a Lagrangian fibration $p:\F_\dd\ra \mc{B}_\dd$ of the boundary phase space determining the quantization of the boundary. 

The resulting quantum theory is the following assignment:
\begin{itemize}
\item To the out-boundary, the theory assigns the space $\HH_\dout^{(A)}$ of half-densities on $\B_\dout$ which is a cochain complex with the differential (the ``quantum BFV operator'', arising as the geometric quantization of the Hamiltonian for the boundary cohomological vector field) $\widehat{S}_\dout=-i\hbar \lan d_E A_\dout , \frac{\dd}{\dd A_\dout} \ran$. Likewise to the in-boundary, the theory assigns the space $\HH_\din^{(B)}$ of half-densities on $\B_\din$ with the differential $\widehat{S}_\din=-i\hbar \lan d_E B_\din , \frac{\dd}{\dd B_\din} \ran$.\footnote{Superscripts pertain to the polarization $p: \F_\dd\ra\B_\dd$ (field $A$ fixed on the out-boundary and field $B$ fixed on the in-boundary) used to quantize the in/out-boundary.}
\item To the bulk (the cobordism itself), the theory assigns: 
\begin{itemize}
\item  The space of residual fields built out of cohomology relative to in/out boundary, $\F^\zm=H^\bt(M,M_\dout;E)[1]\oplus H^\bt(M,M_\din;E^*)[\nn-2]$.
\item The partition function 
$$Z=(\mu^\hbar_{\B_\dd})^{\frac12}\cdot \xi_\hbar^{H^\bt}\cdot \tau(M,M_\dout;E)\cdot e^{\frac{i}{\hbar} (\lan B_\zm,A_\dout \ran_\dout +\lan B_\din, A_\zm \ran_\din-\lan B_\din , \mathbf{K} A_\dout \ran_\din)}$$
(see Proposition \ref{prop: Z}). The partiton function is an element of the space of states for the boundary tensored with half-densities of residual fields. Here the coefficient $\xi_\hbar^{H^\bt}\in \mathbb{C}$ is as in closed case, but defined using Betti numbers for cohomology relative to the out-boundary; $(\mu^\hbar_{\B_\dd})^{\frac12}$ is an appropriately normalized half-density on $\B_\dd$; $\K$ is a the chain homotopy part of the retraction of cochains of $X$ relative to the out-boundary to the cohomology relative to the out-boundary (which, as in the closed case, plays the role of gauge-fixing data).
\end{itemize}
\end{itemize}

This data satisfies the following properties.
\begin{enumerate}[(i)]
\item\label{intro (i)} \textbf{Modified quantum master equation} ((\ref{prop: Z (ii)}) of Proposition \ref{prop: Z}): the partition function satisfies the quantum master equation modified by a boundary term:
$$\left(\frac{i}{\hbar}\widehat{S}_\dd-i\hbar \Delta_\zm\right)Z=0$$
\item\label{intro (ii)} \textbf{Dependence on gauge-fixing choices} ((\ref{prop: Z (iii)}) of Proposition \ref{prop: Z}): change of the gauge-fixing data (the retraction of relative cochains onto cohomology) induces a change of partition function of the form
$$Z \mapsto Z+ \left(\frac{i}{\hbar}\widehat{S}_\dd-i\hbar \Delta_\zm\right)(\cdots)$$
\item\label{intro (iii)} \textbf{Gluing property} (Proposition \ref{prop: gluing}): partition function on a concatentation of two cobordisms can be calculated from the partition function on the two constituent cobordisms by first pairing the states in the gluing interface, and then evaluating the BV pushforward to the residual fields for the glued cobordism.
\item\label{intro (iv)} \textbf{``Topological property''}: the partition function function considered modulo $\left(\frac{i}{\hbar}\widehat{S}_\dd-i\hbar \Delta_\zm\right)$-exact terms is independent of changes of the cellular decomposition $X$ of the cobordism $M$, assuming that the cellular decomposition of the boundary is kept fixed.
\end{enumerate}
Here the first three properties 
are the axioms of a quantum BV-BFV theory \cite{CMRpert,CMRsurvey}, and the last one is a manifestation of the quantum field theory being topological. 

The ``topological property'' 
can be improved by passing to the cohomology of the space of states (Section \ref{sec: reduced space of states}): this cohomology (the ``reduced space of states'') is independent of $X_\dd$ and the corresponding reduced partition function satisfies the BV-BFV axioms above and is completely independent of the cellular decomposition $X$ (i.e., one does not have to fix the decomposition of the boundary).

\end{enumerate}

\subsubsection{``Canonical'' non-abelian 
$BF$ theory on CW complexes (Section \ref{sec: non-ab BF I})}
\begin{enumerate}[I.]
\item \textbf{Non-abelian cellular action: existence/uniqueness result.} Theorem \ref{thm: cellBF}: For $X$ a finite regular CW complex and $\g$ a unimodular Lie algebra, we prove, in a constructive way, that there exists 
a BV action $S_X$ on the space of fields modelled on cellular cochains and chains $\sF_X=C^\bt(X)\otimes\g[1]\oplus C_\bt(X)\otimes\g^*[-2]$, satisfying the following properties:
\begin{itemize}
\item $S_X$ satisfies the Batalin-Vilkovisky quantum master equation $\frac12\{S_X,S_X\}-i\hbar\Delta S_X=0$ or, equivalently, $\Delta e^{\frac{i}{\hbar}S_X}=0$. Here $\{,\}$ and $\Delta$ are the odd-Poisson bracket and the BV Laplacian on functions on $\sF_X$ induced from canonical pairing of cochains with chains.
\item The action $S_X$ has the form $S_X=\sum_{e\subset X}\bar{S}_e$; i.e., $S_X$
 is given as the sum over all cells $e$ of $X$ of certain local building blocks $\bar{S}_e$  depending only on the fields restricted to the closure $\bar{e}$ of the cell $e$. The local building blocks satisfy the following ansatz:
 \begin{multline*} 
\bar{S}_{e}
=\sum_{n=1}^\infty\sum_{\Gamma_0}\sum_{e_1,\ldots,e_n\subset \bar{e}}\frac{1}{|\mr{Aut}(\Gamma_0)|} C^e_{\Gamma_0,e_1,\ldots,e_n}
\left\langle \sB_{e},\mr{Jacobi}_{\,\Gamma_0}(\sA_{e_1},\ldots,\sA_{e_n})\right\rangle_\g-\\
-i\hbar \sum_{n= 2}^\infty \sum_{\Gamma_1} 
\sum_{e_1,\ldots,e_n\subset \bar{e}}\frac{1}{|\mr{Aut}(\Gamma_1)|} C^e_{\Gamma_1,e_1,\ldots,e_n}
\mr{Jacobi}_{\,\Gamma_1}(\sA_{e_1},\ldots,\sA_{e_n})
\end{multline*}
Here $\sA_{e},\sB_e$ are the cochain and chain field (valued in $\g$ and $\g^*$, respectively), evaluated on the cell $e$. In the sum above, $\Gamma_0$ runs over binary rooted trees with $n$ leaves, which we decorate with the $n$-tuple of faces (of arbitrary codimension) $e_1,\ldots,e_n\subset \bar{e}$; likewise $\Gamma_1$ runs over oriented connected graphs with one cycle with $n$ incoming leaves and all internal vertices having incoming/outgoing valency $(2,1)$. $\mr{Jacobi}_\Gamma(\cdots)$ is, for $\Gamma=\Gamma_0$ a binary rooted tree, a nested commutator of elements of $\g$, as prescribed by the tree combinatorics. For $\Gamma=\Gamma_1$ a 1-loop graph, it is the trace of an endomorphism of $\g$ given as a nested commutator with one of the slots kept as the input of the endomorphism and other slots populated by fields $\sA_{e_i}$. $C_{\Gamma,e_1,\ldots,e_n}^e$ are some structure constants (i.e. the theorem is that they can be constructed in such a way that the quantum master equation holds for $S_X$).
\item We have two ``initial conditions'':\footnote{The role of these two conditions is to exclude trivial solutions to quantum master equation, e.g. $S_X=0$, and also to have uniqueness up to homotopy 
for solutions satisfying the stated properties.} 
\begin{itemize}
\item $S_X$ is given as the ``abelian (canonical) action'' $\lan \sB_X, d \sA_X \ran$ plus higher order corrections in fields.
\item For $e$ a $0$-cell, the building block encodes the data of the Lie algebra structure on the space of coefficients $\g$: $\bar{S}_e=\lan \sB_e ,\frac12 [\sA_e,\sA_e]\ran$.
\end{itemize}
\end{itemize}

This existence theorem is supplemented by a uniqueness up to homotopy statement (i.e. up to canonical transformations of solutions of the quantum master equation;\footnote{
One says that two solutions of the quantum master equation $S_0$ and $S_1$ are related by a canonical BV transformation (or ``homotopy'') if they can be connected by a family of solutions $S_t$ such that $\frac{d}{dt} S_t=\{S_t,R_t\}-i\hbar R_t$ with $R_t$ a degree $-1$  ``generator.'' This definition implies that $\frac{d}{dt}e^{\frac{i}{\hbar}S_t}=\Delta \left( e^{\frac{i}{\hbar}S_t} R_t\right)$ and hence $\Delta$-closed exponentials $e^{\frac{i}{\hbar}S_1}$ and $e^{\frac{i}{\hbar}S_0}$ differ by a $\Delta$-exact term.
}
 in this case, the generator of the canonical transformation turns out to satisfy the same ansatz as $S_X$ above) -- Lemma \ref{lemma: cellBF well-definedness}.

The local building blocks $\bar{S}_e$ can be chosen universally, uniformly for all CW complexes $X$, so that they 
depend only on the combinatorics of the closure of the cell $e$ and not on the rest of the combinatorial data of $X$ (Remark \ref{rem: standard building blocks}).

Structure constants $C_{\Gamma,e_1,\ldots,e_n}^e$ occurring in the local building blocks can be chosen to be rational by making a good choice in the construction of the Theorem \ref{thm: cellBF}.

\item \textbf{Compatibility with local moves of CW complexes.} Cellular actions, when considered up to canonical BV transformations, are compatible with Whitehead elementary collapses of CW compexes and with cellular aggregations. More precisely, if $X,Y$ are CW complexes and $X$ is an elementary collapse of $Y$, the BV pushforward of $S_Y$ from cellular fields on $Y$ to cellular fields on $X$ is a canonical transformation of $S_X$ (Lemma \ref{lemma: collapse}). Likewise, if $X$ is a cellular aggregation of $Y$ (i.e., $Y$ is a subdivision of $X$), the same holds (Proposition \ref{prop: aggregations}). 

As a corollary of compatibility of cellular actions with elementary collapses, the partition function, defined as the BV pushforward to cohomology (more precisely, to $\sF^\zm=H^\bt(X)\otimes\g[1]\oplus H_\bt(X)\otimes \g^*[-2]$), which is a $\Delta$-cocycle as a consequence of Theorem \ref{thm: cellBF},  is invariant under simple-homotopy equivalence of CW complexes if considered modulo $\Delta$-coboundaries (Proposition \ref{prop: non-ab BF I simple homotopy invariance}). 
\end{enumerate}

\subsubsection{Non-abelian cellular $BF$ theory in BV-BFV setting (Section \ref{sec: non-ab BF cobordism})}


We combine the results of Sections \ref{sec: classical cell ab BF on mfd with bdry}-\ref{sec: quantum cell ab BF on mfd with bdry} and Section \ref{sec: non-ab BF I} to construct cellular non-abelian $BF$ theory on $\nn$-cobordisms 
in BV-BFV formalism. 
\begin{enumerate}[I.]
\item \textbf{Classical non-abelian theory on a cobordism (BV-BFV setting).} We fix a unimodular Lie algebra $\g$ corresponding to a Lie group $G$. For $M$ an $\nn$-cobordism with a cellular decomposition $X$ and a $G$-local system $E$, we construct the space of fields $\F$, space of boundary fields $\F_\dd$,  the BV 2-form on $\F$, the symplectic form on $\F_\dd$ together with the primitive $\alpha_\dd$ exactly as in the abelian case. 
The action $S$, cohomological vector field $Q$ and their boundary counterparts $S_\dd$, $Q_\dd$ are constructed in terms of the local building blocks $\bar{S}_e$ (or, equivalently, local unimodular $L_\infty$ algebras) assigned by the construction of Theorem \ref{thm: cellBF} to cells $e$ of $X$, see  (\ref{S cell non-ab}), (\ref{Q non-ab cob}), (\ref{S bdry non-ab}). Moreover, in this setting  the pullback map of bulk fields to the boundary is deformed to a nontrivial $L_\infty$ morphism (\ref{Pi^*}). This set of data satisfies the axioms of a classical BV-BFV theory (Proposition \ref{prop: non-ab class BV-BFV}).
\item \textbf{Quantization.} The quantization (Section \ref{sec: non-ab quantization cob}) is constructed along the same lines as in the abelian case -- as a geometric (canonical) quantization on the boundary and a finite-dimensional BV pushforward in the bulk. The resulting spaces of states assigned to the in/out-boundaries are same as in the abelian case as graded vector spaces but carry nontrivial differentials (\ref{Shat out}), (\ref{Shat in}) deforming the differentials arising in the abelian case. Residual fields on a cobordism are same as in the abelian case, while the partition function is more involved -- we develop the corresponding Feynman diagram expansion in Proposition \ref{prop: Z non-ab}. As in the abelian theory, this set of data satisfies the properties (i)--(iv) of Section \ref{sec: intro main results abelian} (modified quantum master equation, exact dependence on gauge-fixing choices, Segal's gluing property, independence on cellular decomposition) -- Theorem \ref{prop 8.15}.
\end{enumerate}

\subsection{Open questions/What is not in this paper}
\begin{enumerate}
\item Construction of more general cellular AKSZ theories: our construction of cellular non-abelian $BF$ action in Theorem \ref{thm: cellBF} develops the theory from its value on $0$-cells, by iterative extension to higher-dimensional cells. It would be very interesting to repeat the construction starting from the target data of a more general AKSZ theory assigned to a $0$-cell. It would be particularly interesting to construct cellular versions of Chern-Simons theory and $BF+B^3$ theory in dimension $3$. Chern-Simons theory has the added complication that one has to incorporate in the construction of the BV 2-form the Poincar\'e duality on a single cellular decomposition, without using the dual one.
\item Comparison of the cellular non-abelian $BF$ theory constructed here with non-perturbative answers in terms of the representation theory data of the structure  group  $G$: comparison with zero area limit of Yang-Mills theory in dimension $2$ and comparison with Ponzano-Regge state sum model (defined in terms of $6j$ symbols) in dimension $3$. Cellular $BF+B^3$ theory should be compared with Turaev-Viro state sum model (based on $q6j$ symbols for the quantum group corresponding to $G$).
\item In this paper we use, for the construction of quantization, special polarizations of phase spaces assigned to the boundary components of an $\nn$-manifold -- the ``$A$-polarization'' and ``$B$-polarization''. It would be interesting to consider more general polarizations and construct the corresponding version of Hitchin's connection (mimicking the situation in Chern-Simons theory), controlling the dependence of the quantum theory on an infinitesimal change of the polarization.
\item $Q$-exact renormalization flow along the poset of CW complexes, arising from the fact that the ``standard'' cellular action is sent by a BV pushforward along a cellular aggregation to an action on the aggregated complex which differs from the standard one by a canonical transformation (see \cite{PM_2005draft} for an example of an explicit computation). Keeping track of these canonical transformations should 
lead 
to the picture of a ``$Q$-exact'' Wilsonian RG flow along cellular aggregations (and to the related notion of the combinatorial $Q$-exact stress-energy tensor). This picture is expected to be related to Igusa-Klein's higher Reidemeister torsion.
\item Observables supported on CW subcomplexes, possibly meeting the boundary.
\item 
Gluing and cutting
with corners of codimension $\geq 2$, or the version for the (fully) extended cobordism category, in the sense of Baez-Dolan-Lurie.
\item Partition functions in this paper are constructed up to sign, so as not to deal with orientations of spaces of fields and gauge-fixing Lagrangians. It would be interesting to construct a sign-refined version of the theory.

\end{enumerate}

\subsection{
Plan of the paper
}

In Sections \ref{sec: reminder on Poincare duality}, \ref{sec: reminder on cell local systems}
we recall and set up the conventions and notations for chain-level Poincar\'e duality for cellular decompositions of manifolds with boundary (Section \ref{sec: reminder on Poincare duality}) and local systems in this setting (Section \ref{sec: reminder on Poincare duality}). This sets the stage for the construction of the space of fields of the cellular model.

In Section \ref{sec: HPT} we recall the homological perturbation theory which later plays the crucial role for defining the gauge-fixing for the quantization.

In Section \ref{sec: closed} we construct the abelian cellular theory on a closed manifold endowed with a cellular decomposition. We first set up the classical theory (Section \ref{sec: classical closed}) and then construct the quantization (\ref{sec: quantization closed}).

In Section \ref{sec: classical cell ab BF on mfd with bdry} we construct the extension of the abelian cellular theory to cobordisms, in the BV-BFV setting, on the classical level. 

The quantization of the abelian model on cobordisms is constructed in Section \ref{sec: quantum cell ab BF on mfd with bdry}. In particular, we prove the gluing property of the partition functions in Section \ref{sec: quantum gluing}.

In Section \ref{sec: non-ab BF I} we construct the ``canonical version'' (i.e. with covariant $B$-field) of the non-abelian $BF$ theory on arbitrary regular CW complexes in BV formalism and establish the invariance of the theory (up to canonical transformations) under cellular aggregations and elementary collapses of CW complexes.

In Section \ref{sec: non-ab BF cobordism} we construct the non-abelian cellular $BF$ theory on cobordisms in BV-BFV setting and its quantization. We prove that the quantization satisfies the axioms of a quantum BV-BFV theory and is independent (modulo canonical transformations) of the cellular decomposition of the cobordism.

\subsection*{Acknowledgements}
We thank Nikolai Mnev for inspiring discussions, crucial to this work. 
We are grateful to the anonymous referee for insightful comments and questions that helped improve the paper.
P. M. thanks the University of Zurich and the Max Planck Institute of Mathematics in Bonn, where he was affiliated during different stages of this work, for providing the excellent research environment.

\section{Reminder: Poincar\'e duality for cellular decompositions of manifolds}
\label{sec: reminder on Poincare duality}

\subsection{Case of a closed manifold}\label{sec: Poincare duality, closed mfd}
Let $M$ be a compact oriented\footnote{This assumption is made for convenience and can be dropped, see Remark \ref{rem: orientability} below.} 
piecewise-linear\footnote{
Throughout this paper we will be working in the piecewise-linear category. One can replace PL manifolds with smooth manifolds everywhere, but then instead of gluing of manifolds along a common boundary, one should talk about cutting a manifold along a submanifold of codimension~$1$ or work with manifolds with collars in order to achieve the correct gluing of smooth structures. 
For details on oriented intersection of chains in piecewise-linear setting, we refer the reader to \cite{McClure}.
} (PL) $n$-dimensional manifold without boundary, endowed with  a cellular decomposition $X$ (with cells being finite unions of simplices of a triangulation compatible with the PL structure), which we assume to be a regular CW complex.\footnote{Recall that a 
CW complex is said to be \emph{regular} if the characteristic maps from standard open balls to open cells $\chi: \mr{int}(B^k) \xra{\sim} e\subset X$ extend to homeomorphisms of closed balls to corresponding closed cells $\bar\chi: B^k \xra{\sim} \bar{e}\subset X$.
Another term for a regular CW complex is ``ball complex.''
} One can construct the {\it dual} cellular decomposition of $X^\vee$, uniquely defined up to PL homeomorphism, 
such that:
\begin{itemize}
\item There is a bijection $\varkappa$ between $k$-cells of $X$ and $(n-k)$-cells of $X^\vee$. (One calls $\varkappa(e)$ the {\it dual cell} for $e$.)
\item For a cell $e$ of $X$, $\varkappa(e)\subset \mr{star}(e)\subset M$.\footnote{The standard terminology is that for a cell $e$ of any CW complex $X$, the {\it star} of $e$ is the subcomplex of $X$ consisting of all cells of $X$ containing $e$. The {\it link} of $e$ is the union of cells of $\mr{star}(e)$ which do not intersect $e$.}
\item $e$ intersects $\varkappa(e)$ transversally and at a unique point.
\end{itemize}
Choosing (arbitrarily) orientations of cells of $X$, we can infer the choice of orientations of cells of $X^\vee$ in such a way that the intersection pairing is $e\cdot \varkappa(e)=+1$. More generally, for $e_i$ running over $k$-cells of $X$, we have
$e_i\cdot \varkappa(e_j)=+\delta_{ij}$.

On the level of cellular chains, we have a non-degenerate intersection pairing
\be \cdot:\quad C_k(X;\bZ)\otimes C_{n-k}(X^\vee;\bZ)\ra \bZ \label{intersection pairing closed}\ee
which induces a {\it chain} isomorphism between cellular chains and cochains
\be C_\bt(X;\bZ),\dd\quad \xra{\sim}\quad  C^{n-\bt}(X^\vee;\bZ),d \label{Poincare chains to cochains}\ee
which in turn induces the Poincar\'e duality between homology and cohomology
$$\underbrace{H_\bt(X;\bZ)}_{=H_\bt(M;\bZ)}\xra{\sim} \underbrace{H^{n-\bt}(X^\vee;\bZ)}_{=H^{n-\bt}(M;\bZ)}$$

\begin{remark}\label{rem: barycentric subdiv}
One construction of the dual cellular decomposition $X^\vee$ is via the barycentric subdivision $\beta(X)$ of $X$ -- a simplicial complex, constructed combinatorially as the nerve of the partially ordered set of cells of $X$ (with ordering given by adjacency). The combinatorial simplex $\sigma=(e_0<\ldots < e_k)$ has dimension $k$ and can be geometrically realized as a simplex inside $e_k$ with vertices $\dot{e}_0,\ldots\dot{e}_k$, where $\dot{e}$ is some a priori fixed point in an open cell $e$ -- the {\it barycenter} of $e$. Next, one constructs $X^\vee$ out of $\beta(X)$ as follows. For $v$ a vertex ($0$-cell) of $X$, we set $\varkappa(v)$ to be the star of $v$ in $\beta(X)$. For $e$ a $k$-cell of $X$, we set
\be \varkappa(e)=\bigcap_{v\in e} \mr{star}_{\beta(X)}(v)\quad \cap \mr{star}_X(e) \label{varkappa(e) via barycentric subdivision}\ee
where the first intersection runs over vertices of $e$.
\end{remark}

\begin{remark}\label{rem: orientability} Orientability of $X$ is not required to define the complex $X^\vee$. However, global orientation is necessary to 
define the intersection pairing between cells of $X$ and cells of $X^\vee$
in such a way that (\ref{Poincare chains to cochains}) becomes a chain map.
In a more general setup, we can allow $X$ to be possibly non-orientable. Denote by $\mr{Or}$ the orientation $\mathbb{Z}_2$-local system on $X$ (see Section \ref{sec: reminder on cell local systems} below for a reminder on cellular local systems); the role of orientation is played by a choice of a primitive element $\sigma\in H_n(X,\mr{Or};\mathbb{Z})$. Then (\ref{intersection pairing closed}) becomes $\cdot:C_k(X;\mathbb{Z})\otimes C_{n-k}(X^\vee,\mr{Or};\mathbb{Z})\ra \bZ$. (We twist one of the two factors by the $\mr{Or}$, it is unimportant which factor is twisted.) This pairing depends on the class $\sigma$. Likewise (\ref{Poincare chains to cochains}) becomes $C_\bt(X;\bZ),\dd\;\; \xra{\sim}\;\;  C^{n-\bt}(X^\vee,\mr{Or};\bZ),d$. This setup can be straightforwardly adapted to the setting of cellular complexes with boundary -- we always have to twist one side in Poincar\'e-Lefschetz duality by the orientation local system. However, for simplicity, in this paper we will always be assuming that $\mr{Or}$ is trivial and $X$ is oriented.
\end{remark}

\begin{remark}
We could require $X$ to be a triangulation and $X^\vee$ the dual cellular complex. We are not imposing this requirement, because later the fields $A,B$ of our theory will be cochains on $X$ and $X^\vee$ and it seems unnecessary to break the symmetry between $A$ and $B$ (present in the abelian theory) by forcing $A$ to live on a triangulation.
\end{remark}

\subsection{Case of a manifold with boundary}\label{sec: Lefschetz for cell decomp}
Let $M$ be a compact oriented 
$n$-manifold with boundary $\dd M$. Assume that we have a cellular decomposition $X$ of $M$, which restricts on the boundary to a cellular decomposition $X_\dd$ of $\dd M$.

We can construct\footnote{We can again use the construction of Remark \ref{rem: barycentric subdiv}. Cells $\varkappa(e)$ are then defined exactly as in (\ref{varkappa(e) via barycentric subdivision}) and cells $\varkappa_\dd(e)$ for boundary cells $e\subset X_\dd$ are constructed as $\varkappa_\dd(e)=\overline{\varkappa(e)}\cap \dd M$.

} a new cellular decomposition $X^{\vee_+}$ of $M$ such that the following holds.
\begin{itemize}
\item For every $k$-cell $e$ of $X$ we have an $(n-k)$-cell $\varkappa(e)\subset \mr{star}_X(e)\subset M-\dd M$ of $X^{\vee_+}$.
\item For every $k$-cell $e$ of $X_\dd$, apart from the $(n-k)$-cell $\varkappa(e)$, $X^{\vee_+}$ contains an $(n-k-1)$-cell $\varkappa_\dd(e)\subset \dd M$ of the dual boundary complex $(X_\dd)^\vee$.
\item Cells of the form $\varkappa(e)$, $\varkappa_{\dd}(e)$ (for boundary cells $e$) exhaust the CW complex $X^{\vee_+}$.
\item For $e$ a cell of $X-X_\dd$, $e$ and $\varkappa(e)$ intersect transversally and at a single point. For $e$ a boundary cell of $X$, $e$ meets the closed cell $\overline{\varkappa(e)}$ at a single point.
\end{itemize}

\begin{figure}[!htbp]
\begin{center}
\includegraphics[scale=0.8]{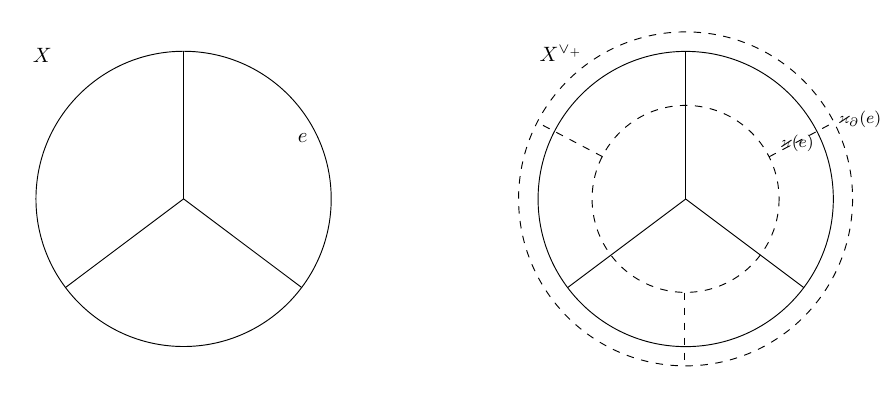}
\caption{A cellular decomposition $X$ of a closed 2-disk $M$ (drawn in solid lines), and the corresponding dual decomposition $X^{\vee_+}$ of a slightly larger disk $M_+$ (drawn in dashed lines).}
\end{center}
\end{figure}

Again, orientations of $X^{\vee_+}$ can be inferred from some chosen orientations of cells of $X$ in such a way that the intersection 
is:
$$e_i\cdot\varkappa(e_j)=+\delta_{ij}$$
In case of $e_j$ being a boundary cell, we have to {\it regularize} the intersection, which we can do by regarding $X^{\vee_+}$ as a cellular decomposition of $M_+$ -- an extension of $M$ by attaching
a collar $\dd M\times [0,\epsilon]$ at the boundary $\dd M$. Then all intersections of cells of $X$ in $M\subset M_+$ and cells of $X^{\vee_+}$ in $M_+$ are transversal. Note that with this regularization, for $e_i$ any cell and $e_j$ a boundary cell of $X$, we have $$e_i\cdot \varkappa_\dd (e_j)=0$$

Intersection pairing defined as above induces a non-degenerate pairing between absolute and relative chains:
\be \cdot:\quad C_k(X;\bZ)\otimes C_{n-k}(X^{\vee_+},X_\dd^\vee;\bZ)\ra \bZ \label{Lefschetz for chains +}\ee
which in turn gives rise to a chain isomorphism between absolute chains and relative cochains
$$C_\bt(X;\bZ),\dd \quad\xra{\sim}\quad C^{n-\bt}(X^{\vee_+},X_\dd^\vee;\bZ),d$$
On the level of homology/cohomology, one obtains the usual Poincar\'e-Lefschetz duality
$$\underbrace{H_\bt(X;\bZ)}_{\cong H_\bt(M;\bZ)}\xra{\sim} \underbrace{H^{n-\bt}(X^{\vee_+},X^\vee_\dd;\bZ)}_{\cong H^{n-\bt}(M,\dd M;\bZ)}$$

Note that unlike the case of closed manifolds, where the operation $X\mapsto X^\vee$ is an involution on cellular decompositions, for manifolds with boundary $X^{\vee_+}$ always has more cells than $X$ and $X\mapsto X^{\vee_+}$ cannot be an involution.

We define $X^{\vee_-}$ as a CW subcomplex of $X^{\vee_+}$ obtained by removing the cells $\varkappa(e)$ and $\varkappa_\dd(e)$ for every boundary cell $e$ of $X$. Topologically, $X^{\vee_-}$ is $X^{\vee_+}$ with a collar near the boundary removed, i.e. the underlying topological space is $M_-\subset M \subset M_+$, where $M-M_-\simeq \dd M \times [-\epsilon,0]$. The counterpart of (\ref{Lefschetz for chains +}) is the non-degenerate intersection pairing
\be \cdot:\quad C_k(X,X_\dd;\bZ)\otimes C_{n-k}(X^{\vee_-};\bZ)\ra \bZ  \label{Lefschetz for chains -}\ee

\begin{figure}[!htbp]
\begin{center}
\includegraphics[scale=0.8]{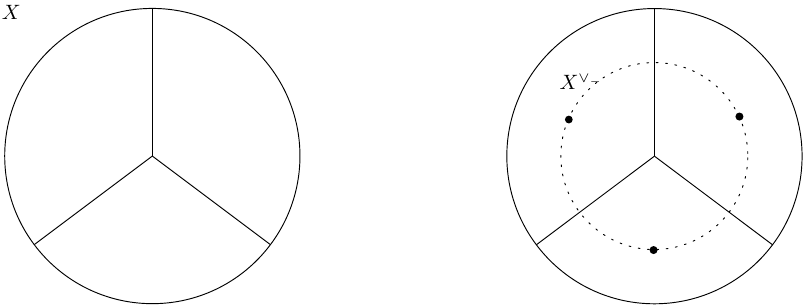}
\caption{A cellular decomposition $X$ of a closed 2-disk $M$ (drawn in solid lines) and the dual decomposition $X^{\vee_-}$ of a smaller disk $M_-$ (drawn in dotted lines).}
\end{center}
\end{figure}

\begin{definition} \label{def: product type}
We will say that a cellular decomposition $X$ of a manifold $M$ with boundary is 
of \emph{product type}
near the boundary if, for any $k$-cell $e_\dd$ of $X_\dd$, there exists a unique $(k+1)$-cell $e$ of $X-X_\dd$ such that $e_\dd\subset \dd e$.\footnote{In other words, we are asking for the intersection of $X$ with a thin tubular neighborhood of the boundary to look like the product CW-complex $X_\dd\times [0,1]$ intersected with $X_\dd\times [0,\epsilon)$. Morally, even though there is no metric in our case, one should think of this property as an analog of the property of a Riemannian metric on a smooth manifold with boundary to be of product form near the boundary.}
\end{definition}

\begin{figure}[!htbp]
\begin{center}
\includegraphics[scale=0.8]{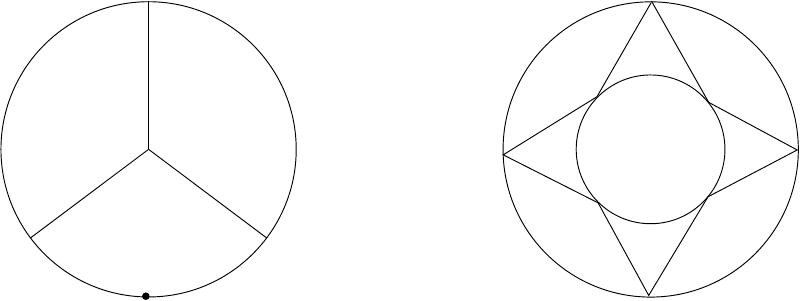}
\caption{Two examples of non-product type behavior of a cellular decomposition (in this case, of a 2-disk) near the boundary. Note the extra vertex on the boundary in the example on the left.}
\end{center}
\end{figure}

There is the obvious geometric inclusion of the boundary
$\iota_+: X^\vee_\dd\hra X^{\vee_+}$. There is also a cellular map 
$\iota_-: X_\dd^\vee \hra X^{\vee_-}$ which sends $\varkappa_\dd(e)$ to $\dd\varkappa(e)\cap \dd M_-$ for any cell $e$ of $X_\dd$. Under the assumption that $X$ is of product type near the boundary, as defined above, $\iota_-$ is in fact an inclusion. In particular, in this case the boundary of the dual complex $X^{\vee_-}\cap \dd M_-$ is isomorphic, as a CW complex, to $X_\dd$.

As opposed to $X^{\vee_+}$, complex $X^{\vee_-}$ has {\it less} cells then $X$, and, in the 
case of product type near the boundary, the ``double duals'' $(X^{\vee_+})^{\vee_-}$, $(X^{\vee_-})^{\vee_+}$ are isomorphic to $X$ as CW complexes.

Note also that $X^{\vee_+}$ is always of product type near the boundary. 

\subsubsection{Cutting a closed manifold into two pieces}\label{sec: cutting a closed mfd}
Let $M$ be a closed manifold cut along a codimension $1$ submanifold $\Sigma$ into two parts -- manifolds with boundary $M_I$, $M_{II}$, with $\dd M_I=\dd M_{II}=\Sigma$. Let $X$ be a cellular decomposition of $M$ such that $X_\Sigma=\Sigma\cap X$ is a subcomplex of $X$. Denote $X_I$, $X_{II}$ the induced cellular decompositions of $M_I$, $M_{II}$. Then one has the obvious (pushout) relation:
$$X=X_I\cup_{X_\Sigma} X_{II}$$
For the dual decompositions, one has
$$X^\vee=X^{\vee_+}_I\cup_{X_\Sigma^\vee} X^{\vee_-}_{II}$$
assuming that $X_{II}$ is of product type near $X_\Sigma$, and
$$X^\vee=X^{\vee_-}_I\cup_{X_\Sigma^\vee} X^{\vee_+}_{II}$$
if $X_I$ is of product type near $X_\Sigma$. It can happen, of course, that $X$ is well-behaved on both sides of $X_\Sigma$, and then both formulae above hold.

\subsection{Case of a cobordism}\label{sec: Lefschetz for cell decomp, cobordism}
A more symmetric version of the construction of Section \ref{sec: Lefschetz for cell decomp} is as follows. Let $M$ be a compact oriented 
manifold with boundary split into two disjoint parts $\dd M=M_\din\sqcup M_\dout$\footnote{By convention, we endow $M_\dout$ with the orientation induced from the orientation of $M$, whereas $M_\din$ is endowed with orientation opposite to the one induced from $M$. Thus, as an oriented manifold the boundary splits as $\dd M=\overline{M_\din}\sqcup M_\dout$ where the overline stands for orientation reversal.} 
(i.e.\ we color the boundary components of $M$ in two colors -- ``in'' and ``out'').
We call this set of data a cobordism and denote it by $M_\din\cob{M} M_\dout$. Let $X$ be a cellular decomposition of $M$ inducing decompositions $X_\din$, $X_\dout$ of the in- and out-boundary, respectively.
When talking about a cobordism with a cellular decomposition, we always make the following assumption.
\begin{assumption}\label{assumption: X of product type near out-boundary}
$X$ is of product type near $M_\dout$.
\end{assumption}
We make no assumption on the behavior of $X$ near $M_\din$. Then we define the dual CW complex as
$$X^\vee:=X^{\vee_+}-\left(\varkappa_\dd(X_\dout)\cup \varkappa(X_\dout)\right)$$
We think of $X$ on the l.h.s. as including the information about which boundary component is ``in'' and which is ``out''. The underlying manifold $\widetilde M$ of $X^\vee$ is $M$ with a collar at $M_\din$ adjoined and a collar at $M_\dout$ removed; we also regard $\widetilde M$ as having the 
in/out coloring of boundary opposite to that of $M$.\footnote{Observe that $X^\vee$ is automatically of product type near the in-boundary of $M$, i.e. near the out-boundary of $\widetilde M$.}

Note that if $\dd M=M_\din$, then $X^\vee=X^{\vee_+}$. If $\dd M=M_\dout$, then $X^\vee=X^{\vee_-}$.
For the numbers of cells (= ranks of chain groups), we have
$$\rk C_k(X^\vee)=\rk C_k(X)+\rk C_k(X_\din)-\rk C_k(X_\dout)$$
Also note that $(X^\vee)^\vee=X$, thus in this formalism taking dual is again an involution.

One has a non-degenerate intersection pairing
\be \cdot:\quad C_k(X,X_\dout;\bZ)\otimes C_{n-k}(X^\vee,X^\vee_\din;\bZ)\ra \bZ \label{Lefschetz for chains mixed}\ee
which gives the chain isomorphism
$$C_\bt(X,X_\dout;\bZ),\dd\quad \xra{\sim}\quad  C^{n-\bt}(X^\vee,X^\vee_\din;\bZ),d$$
and the Poincar\'e-Lefschetz duality between homology and cohomology
$$\underbrace{H_\bt(X,X_\dout;\bZ)}_{=H_\bt(M,M_\dout;\bZ)} \xra{\sim}  \underbrace{H^{n-\bt}(X^\vee,X^\vee_\din;\bZ)}_{=H^{n-\bt}(M,M_\din;\bZ)}$$

\section{Reminder: cellular local systems}\label{sec: reminder on cell local systems}

Let $X$ be a CW complex and $G$ a Lie group.
\begin{definition}
We define a $G$-local system on $X$ as a pair $(\E,\rho)$ consisting of a functor $\E$ from the 
partially ordered set of cells of $X$ (by incidence) 
to the category $G\rightrightarrows \ast$ with single object and with $\mr{Hom}(\ast,\ast)=G$, and a linear representation $\rho:\,G\ra \mr{Aut}(V)$ with $V$ a 
 finite-dimensional vector space.
\end{definition}

One can twist cellular chains of $X$ by $\mc{E}$ as follows. As vector spaces, we set $C_k(X,\mc{E})=V\otimes C_k(X;\bZ)$. Assume all cells of $X$ are oriented (in an arbitrary way). If for a $k$-cell $e$ we have $\dd e=\sum_j \epsilon_j\; e_j$ with $\epsilon_j$ coefficients of the boundary map (in case of a regular complex, $\epsilon_j\in\{\pm 1 ,0\}$), then we set
$$\dd_\mc{E}(v\otimes e)=\sum_j \epsilon_j\; \rho(\mc{E}(e>e_j)) (v)\;\;\otimes e_j$$
where $v\in V$ is an arbitrary vector. 
Functoriality of $\mc{E}$ implies that $\dd_{\mc{E}}^2=0$.

Consider the dual local system $\E^*$ on $X$, which is the same as $\E$ as a functor, but accompanied with the dual vector representation, 
$\rho^*=(\rho^{-1})^T:\,G\ra V^*$. Then the twisted cochains on $X$, $C^\bt(X,\E)$, are constructed as the dual complex to $C_\bt(X,\E^*)$, i.e. $C^k(X,\E)=(C_k(X,\E^*))^*$, with differential $d_\E$ given by the transpose of $\dd_{\E^*}$.

\begin{definition}
We call two local systems $\E$ and $\E'$ on $X$ equivalent, if there exists a natural transformation between them, i.e. for every cell $e\subset X$ there is a group element $g(e)\in G$, so that for any pair of cells $e'\subset e$ we have $$\E'(e>e')=g(e)\;\E(e>e')\; g(e')^{-1}$$
\end{definition}

A local system $\E$ induces a holonomy functor from the fundamental groupoid of $X$ to $G\rightrightarrows \ast$, by associating to a path
$$\chi=(v_0<e_1>v_1<e_2>\cdots <e_N>v_N)$$
from $v_0$ to $v_N$,
a group element
$$\mr{hol}_\E(\chi)=\E(e_N>v_N)^{-1}\cdot\E(e_{N}>v_{N-1})\cdots\E(e_1>v_1)^{-1}\cdot \E(e_1>v_0)\quad \in G$$
Picking a base point $x_0$ -- a vertex of $X$, and restricting the construction to closed paths from $x_0$ to $x_0$ along $1$-cells of $X$, we get a representation of the fundamental group $\pi_1(X,x_0)$ in $G$.

Let $M$ be a 
manifold (possibly, with boundary) endowed with a principal $G$-bundle $P\ra M$ with a flat connection $\nabla_P$.
Let $E=P\times_G V$ be the associated vector bundle with the corresponding flat connection $\nabla_E$. If $X$ is a cellular decomposition of $M$, one can construct a local system $E_X$ on $X$ by marking a point (barycenter) $\dot{e}$ in each cell $e$ and setting
\be E_X(e > e')=\mr{hol}_{\nabla_P}(\gamma_{\dot e'\ra \dot e}) \qquad \in \mr{Hom}(P_{\dot e'}, P_{\dot e})\simeq G \label{cell local system from flat bundle}\ee
-- the parallel transport of the connection $\nabla_P$ along any path $\gamma_{\dot e'\ra \dot e}$ from the barycenter $\dot e'$ to $\dot e$, staying inside the cell $e$.\footnote{Note that functoriality in this construction corresponds to flatness of $\nabla_P$. 
}
 To identify the hom-space between fibers of $P$ over $\dot e'$ and $\dot e$ with the group, we assume that $P$ is trivialized over barycenters of all cells.

The twisted cochain complex $(C^\bt(X,E_X),d_{E_X})$ is quasi-isomorphic to $(\Omega^\bt(M,E),d_{\nabla_E})$ -- the de Rham complex of $M$ twisted by the flat bundle $E$.

In case of a manifold without boundary, the intersection pairing (\ref{intersection pairing closed}) together with the canonical pairing $V\otimes V^*\ra \bR$ induce a non-degenerate pairing
\be \langle,\rangle:\quad C_k(X,E_X)\otimes C_{n-k}(X^\vee,E^*_{X^\vee})\ra \bR \label{intersection pairing closed twisted}\ee
Similarly, in presence of a boundary, one has versions of cellular Poincar\'e-Lefschetz pairing (\ref{Lefschetz for chains +},\ref{Lefschetz for chains -},\ref{Lefschetz for chains mixed}) with coefficients in $E_X, E^*_{X^\vee}$.

The local system $E^*_{X^\vee}$ that we need in (\ref{intersection pairing closed twisted}) can be constructed from a flat principal bundle by applying construction (\ref{cell local system from flat bundle}) to $X^\vee$ and using the dual representation $\rho^*$. Alternatively, $E_{X^\vee}$ can be constructed directly on the combinatorial level, from $E_X$, by the construction below.

\subsection{Local system on the dual cellular decomposition: a combinatorial description}\label{sec: local systems on the dual cell decomp}
The local system $E_{X^\vee}$ on the dual cellular decomposition $X^\vee$ of a closed manifold $M$ can be described combinatorially in terms of $E_X$ as follows:
\be E_{X^\vee}\left(\,\varkappa(e')>\varkappa(e)\,\right)= E_X(e>e')^{-1}\qquad \in G \label{dual loc system}\ee
for $e,e'$ any pair of incident cells of $X$. 
The dual local system $E^*_{X^\vee}$ on $X^\vee$, which appears in the intersection pairing (\ref{intersection pairing closed twisted}) is same as $E_{X^\vee}$, but accompanied by the dual representation $\rho^*:\, G\ra V^*$.

Note that the combinatorial definition (\ref{dual loc system}) agrees with the construction of $E_{X^\vee}$ by calculating holonomies of the original flat $G$-bundle $P$ on $M$ between barycenters of cells of $X^\vee$, as in (\ref{cell local system from flat bundle}). For this to be true it is important that the barycenters are chosen in such a way that $\dot e = e\cap \varkappa(e) = \dot{\varkappa}(e)$.

In the case of a manifold $M$ with boundary, the local system on $X^{\vee_+}$ is constructed by (\ref{dual loc system}) supplemented by definitions
\begin{multline}\label{loc system on X^vee_+}
E_{X^{\vee_+}}\left(\,\varkappa_\dd(e')>\varkappa_\dd(e)\,\right) = E_X(e>e')^{-1},\quad E_{X^{\vee_+}}\left(\,\varkappa(e')>\varkappa_\dd(e)\,\right) =  E_X(e>e')^{-1},\\ E_{X^{\vee_+}}\left(\,\varkappa(e)>\varkappa_\dd(e)\,\right) = 1
\end{multline}
for $e,e'$ cells of $X_\dd$. The third equation above can be regarded as a special case of the second, with $e'=e$.


Assuming that $X$ 
is of product type near the boundary, the local system on $X^{\vee_-}$ is
obtained by restricting $E_{X^{\vee_+}}$ to the CW subcomplex $X^{\vee_-}\subset X^{\vee_+}$.

Associated to the inclusion $\iota_-:\, X_\dd^\vee\hra X^{\vee_-}$ is the pull-back map for cochains
$$ \iota_-^*:\quad  C^k(X^{\vee_-},E_{X^{\vee_-}})\ra  C^k(X^\vee_\dd,E_{X^\vee_\dd}) $$
defined as
\be v\otimes \varkappa(e)^*\mapsto \left\{ \begin{array}{ll} \rho(E_X(e> \Tilde e)^{-1})(v)\otimes\varkappa_\dd(\Tilde e)^* & \mbox{if }e\mbox{ is adjacent to }\dd M\\
0 & \mbox{otherwise}\end{array} \right. \label{iota_-^* def}\ee
extended by linearity to all cochains.
Here $e$ is a $k$-cell of $X-X_\dd$; for  $e$ adjacent to the boundary, $\Tilde e= \dd e\cap \dd M$ is a single $(k-1)$-cell, since $X$ is assumed to be of product type near the boundary; $v\in V$ is an arbitrary coefficient; for $e$ a $k$-cell of $X$, $e^*\in C^k(X,\bZ)$ denotes the corresponding basis integral cochain.
Defined as above, $\iota_-^*$ is a chain map. Parallel transport by the local system appears in (\ref{iota_-^* def}) to account for the collar $M-M_-$ (i.e. because we have to move from the local system trivialized at barycenters $\dot{\varkappa}(e)\in \dd M_-$ to trivialization at barycenters $\dot{\varkappa}_\dd(\Tilde{e})\in \dd M$).

In case of the dual $X^{\vee_+}$, the pull-back to the boundary $\iota_+^*:\quad  C^k(X^{\vee_+},E_{X^{\vee_+}})\ra  C^k(X^\vee_\dd,E_{X^\vee_\dd})$ is simpler:
\be \iota_+^*:\; v\otimes (e^\vee)^* \;\mapsto\; \left\{ \begin{array}{ll}v\otimes (e^\vee)^* & \mbox{if }e^\vee=\varkappa(e_\dd)\mbox{ for }e_\dd\subset X_\dd \\
0 & \mbox{otherwise} \end{array}\right. \label{iota_+^* def}\ee
Absence of the transport by local system here, as opposed to (\ref{iota_-^* def}), corresponds to the fact that we implicitly gave the local system a trivial extension to the collar $M_+-M$ in (\ref{loc system on X^vee_+}).

\begin{remark} 
Let $M=M_I\cup_\Sigma M_{II}$ be a closed manifold cut into two, as in Section \ref{sec: cutting a closed mfd}, with cellular decompositions $X=X_I\cup_{X_\Sigma} X_{II}$, and assume that $X_I$ is of product type near $X_\Sigma$. Let also $E_X$ be a local system on $X$. We can restrict $E_X$ to $X_I,X_\Sigma,X_{II}$.
Then alongside the obvious fiber product diagram for cochains on $X$:
$$
\begin{CD}
C^k(X,E_X) @>>> C^k(X_I,E_{X_I}) \\
@VVV  @VVV \\
C^k(X_{II},E_{X_{II}}) @>>> C^k(X_\Sigma,E_{X_\Sigma})
\end{CD}
$$
we have the fiber product diagram for cochains on $X^\vee$:
$$
\begin{CD}
C^k(X^\vee,E_{X^\vee}) @>>> C^k(X_I^{\vee_-},E_{X_I^{\vee_-}}) \\
@VVV  @V\iota_{I+}^*VV \\
C^k(X_{II}^{\vee_+},E_{X_{II}^{\vee_+}}) @>\iota_{II-}^*>> C^k(X_\Sigma^\vee,E_{X_\Sigma^\vee})
\end{CD}
$$
\end{remark}

For $M$ a general cobordism with a cellular decomposition $X$ and a local system $E_X$, the corresponding local system on $X^\vee$ is given by  combining the two constructions above, for $E_{X^{\vee_+}}$ (used near $M_\din$) and $E_{X^{\vee_-}}$ (used near $M_\dout$); we should assume that $X$ is of product type near $M_\dout$.

\section{Reminder: homological perturbation theory}\label{sec: HPT}
\begin{definition}
Let $C^\bt,d$ and $C'^\bt,d'$ be two cochain complexes of vector spaces (e.g. over $\bR$). We define the
\emph{induction data} from $C^\bt$ to $C'^\bt$  as a triple of linear maps $(i,p,K)$ with
$$
i: C'^\bt \hra C^\bt, \qquad p: C^\bt\proj C'^\bt,\qquad K: C^\bt\ra C^{\bt-1}
$$
satisfying
\be p\,i=\mr{id}_{C'},\quad di=id', \quad p\,d=d'p,\quad dK+Kd=\mr{id}_C-ip,\quad Ki=0,\quad p\,K=0,\quad K^2=0 \label{ind data equations}\ee
\end{definition}

Induction data exist iff $i: C'\ra C$ is a subcomplex such that $i_*$ is the identity in cohomology (i.e. $C'$ is a {\it deformation retract} of $C$).

We will use the notation
\be\label{retraction} C^\bt,d \wavy{(i,p,K)} C'^\bt,d' \ee for 
induction data. We will also sometimes call the whole collection of data - a pair of complexes and the induction data between them - a {\it retraction}.

\begin{lemma}[Homological perturbation lemma \cite{GugenheimLambe}]\label{lemma: homological perturbation}
If $(i,p,K)$ are induction data from $C^\bt,d$ to $C'^\bt, d'$ and $\delta: C^\bt\ra C^{\bt+1}$ is such that $(d+\delta)^2=0$, then the complex
$$C'^\bt, d'+p\;(\delta-\delta K \delta+ \delta K \delta K \delta-\cdots)\;i $$
is a deformation retract of $C^\bt, d+\delta$, with induction data given by
$$(i-K \delta i+K\delta K\delta i-\cdots ,\quad p-p\,\delta K+p\,\delta K \delta K-\cdots,\quad K - K\delta K+K\delta K\delta K-\cdots )$$
assuming all the series above converge.
\end{lemma}
\noindent For the proof, see e.g. \cite{GKR}.

Given induction data $(i,p,K)$, we have a splitting of $C$ into the image of $C'$ and an acyclic complement $C''=\ker p$:
$$C^\bt=i(C'^\bt)\oplus \underbrace{C''^\bt}_{\ker p}$$
The second term can be in turn represented as a sum of images of $d$ and $K$. Putting these splittings together, we have a
(formal) {\it Hodge decomposition}
\be C^\bt=i(C'^\bt)\oplus \underbrace{(\ker p\cap d(C^{\bt-1}))\oplus K(C^{\bt+1})}_{C''^\bt} \label{HPT: Hodge decomp}\ee

An important special case is when $C'^\bt=H^\bt(C)$ -- the cohomology of $C^\bt$. Then, in addition to the axioms above, for the induction data from $C^\bt$ to $H^\bt(C)$ we require the following to hold.
\begin{assumption} \label{assumption on p}
The projection $p$ restricted to closed cochains $\ker d\subset C^\bt$ agrees with the canonical projection associating to a closed cochain its cohomology class.
\end{assumption}
In this case the Hodge decomposition (\ref{HPT: Hodge decomp}) becomes
$$C^\bt = i(H^\bt)\oplus C^\bt_\mr{exact}\oplus \underbrace{C^\bt_{K\mr{-exact}}}_{\mr{im}(K)}$$
with the differential now being an isomorphism $d: C^\bt_{K\mr{-exact}}\xra{\sim} C^{\bt+1}_\mr{exact}$ and vanishing on the first two terms of the decomposition.

\subsection{Composition of induction data}
Two retractions
$$C^\bt_{(0)},d_{(0)} \wavy{(i_{01},p_{01},K_{01})} C^\bt_{(1)},d_{(1)},\qquad  C^\bt_{(1)},d_{(1)} \wavy{(i_{12},p_{12},K_{12})} C^\bt_{(2)},d_{(2)}$$
can be composed as follows:
$$ C^\bt_{(0)},d_{(0)} 
\wavy{(i_{02},p_{02},K_{02})}
C^\bt_{(2)},d_{(2)} $$
with composed induction data
\be (i_{02}=i_{01} i_{12},\; p_{02}= p_{12} p_{01},\; K_{02}=K_{01}+i_{01} K_{12} p_{01} ) \label{composition of ind data}\ee

\subsection{Induction data for the dual complex and for the algebra of polynomial functions}
\label{sec: induction data for dual and sym}
For $C^\bt,d$ a cochain complex, we can construct its dual -- the complex of dual vector spaces $C^{*\;k}=\mr{Hom}(C^{n-k},\bR)$ endowed with the transpose differential 
$(d^*)_k=(-1)^{n-k-1}(d_{n-k-1})^T: C^{*\; k}\ra C^{*\; k+1}$.\footnote{Having in mind Poincar\'e duality, we are incorporating the degree shift by $n$ in the definition of the dual (which is superfluous in purely algebraic setting). In our setup, the canonical pairing $\langle,\rangle: C^{*\;k}\otimes C^{n-k}\ra \bR$ has degree $-n$.}

If $(i,p,K)$ are induction data from $C,d$ to $C',d'$, then we have the {\it dual induction data} 
$C^*\wavy{(i^*,p^*,K^*)}C'^*$
given by
\be (i^*=p^T,\quad p^*=i^T,\quad (K^*)_k=(-1)^k (K_{n-k+1})^T) \label{dual ind data}\ee

Given a finite set of retractions
\be C_{j}^\bt\wavy{(i_j,p_j,K_j)}C'^\bt_j \label{C_j retraction}\ee
for $1\leq j\leq r$, one can construct the induction data for the tensor product $\bigotimes_j C_j\wavy{(i^\otimes,p^\otimes,K^\otimes)} \bigotimes_j C'_j$ where
$$i^\otimes = \bigotimes_j i_j,\quad p^\otimes = \bigotimes_j p_j,\quad K^\otimes=\sum_j (i_1 p_1)\otimes\cdots\otimes (i_{j-1} p_{j-1})\otimes  K_j\otimes \mr{id}_{j+1}\otimes \cdots\otimes \mr{id}_{r}$$
The construction of $K^\otimes$ above corresponds to the composition of a sequence of retractions
$$\bigotimes_{j=1}^r C_j\wavy{}C'_1\otimes\bigotimes_{j=2}^r C_j\wavy{}C'_1\otimes C'_2\otimes\bigotimes_{j=3}^r C_j\wavy{}\cdots  \wavy{}\bigotimes_{j=1}^r C'_j$$
where on each step we use the data (\ref{C_j retraction}) tensored with the identity on the other factors.
One can choose instead to perform the retractions $C_j\wavy{}C'_j$ in a different order, given by a permutation $\sigma$ of $\{1,2,\ldots,r\}$, which yields another formula for $K^\otimes$, depending on the permutation $\sigma$.
Symmetrization over permutations $\sigma$ leads us to the next construction -- retraction between symmetric algebras. In the setting of the present work, we are more interested in the symmetric algebra of dual complexes.

The symmetric algebra of the dual complex $\Sym C^*$ is naturally a differential graded commutative algebra and can be seen as the algebra of polynomial functions on $C$. Induction data from $\Sym C^*$ to $\Sym C'^*$
can be constructed as follows:
\begin{equation}\label{HPT Sym}
i_\Sym = p^*,\quad p_\Sym=i^*, \quad K_\Sym=\int_0^1 ((1-t)+t\, ip+dt\, K)^* 
\end{equation}
Here asterisks stand for pullbacks.
The expression for $K_\Sym$ can be understood in terms of the diagram
$$
\begin{CD}
T[1] [0,1]\times C @>{\lambda=(1-t)+t\, ip+dt\, K}>> C  \\
@V\pi_2 VV \\
C
\end{CD}
$$
where $T[1][0,1]$ is the odd tangent bundle of the interval, with even coordinate $t$ and odd tangent coordinate $dt$. Map $\lambda$ has total degree $0$. Now $K_\Sym$ can be defined as a transgression
$$K_\Sym=(\pi_2)_* \lambda^*:\quad \Sym C^* \ra \Sym C^*$$

Same formulae can be used for more general classes of functions than polynomials (e.g. smooth functions) on $C$, $C'$.

\subsection{Deformations of induction data}\label{sec: deformations of ind data}
Given a retraction $C^\bt,d\wavy{(i,p,K)}C'^\bt,d'$, one can analyze the possible infinitesimal deformations of the induction data $(i,p,K)$, as solutions of the system (\ref{ind data equations}). It turns out (see e.g. \cite{DiscrBF,CSinvar}) that a general infinitesimal deformation $(i,p,K)\mapsto (i+\delta i, p+\delta p, K+\delta K)$ is a sum of the following deformations.
\begin{enumerate}[(I)]
\item \label{HPT deform I} Deformation of $K$, not changing $i$ and $p$:
\be \delta i=\delta p=0,\quad \delta K=d\Lambda-\Lambda d\label{ind data deformation I}\ee
with generator $\Lambda:C''^\bt_{\mr{exact}}\ra C''^{\bt-2}_{K-\mr{exact}}$, extended by zero to the first and third terms of the Hodge decomposition (\ref{HPT: Hodge decomp}).
\item \label{HPT deform II} Deformation of $i$, not changing $p$, and changing $K$ in the ``minimal'' way:
\be \delta i=dI+Id',\quad \delta p=0,\quad \delta K=-I\; p\label{ind data deformation II}\ee
with $I: C'^\bt\ra C''^{\bt-1}_{K-\mr{exact}}$.
\item \label{HPT deform III} Deformation of $p$, not changing $i$, and changing $K$ in the ``minimal'' way:
\be \delta i=0,\quad \delta p=d'P+Pd,\quad \delta K=-i\; P  \label{ind data deformation III}\ee
with $P: C''^\bt_\mr{exact}\ra C'^{\bt-1}$, extended by zero as in (\ref{HPT deform I}).
\item \label{HPT deform IV} Deformation, induced by an (infinitesimal) automorphism of $C',d$:
\be \delta i=i\; \chi,\quad \delta p=-\chi\; p,\quad \delta K=0  \label{ind data deformation IV}\ee
with $\chi:C'^\bt\ra C'^\bt$ a chain map.
\end{enumerate}

In other words, if $\mathfrak{I}$ is the space of all induction data $C\wavy{} C'$, then the tangent space to $\mathfrak{I}$ at a point $(i,p,K)$ splits as a direct sum of four subspaces described by (\ref{ind data deformation I}--\ref{ind data deformation IV}):
$$T_{(i,p,K)}\mathfrak{I}=T_{(i,p,K)}^I\mathfrak{I}\oplus T_{(i,p,K)}^{II}\mathfrak{I}\oplus T_{(i,p,K)}^{III}\mathfrak{I}\oplus T_{(i,p,K)}^{IV}\mathfrak{I}$$

Note that in case of retraction to cohomology, $C^\bt\wavy{}H^\bt(C)$, deformations (\ref{ind data deformation IV}) are prohibited by Assumption \ref{assumption on p}.

\begin{lemma}\label{lemma: space of ind data}
\leavevmode
\begin{enumerate}[(a)]
\item\label{lemma: space of ind data (a)} The space of induction data $C^\bt\wavy{}H^\bt(C)$ satisfying Assumption \ref{assumption on p} is \emph{contractible}.
\item\label{lemma: space of ind data (b)} The space $\mathfrak{I}$ of induction data $C^\bt\wavy{}C'^{\bt}$ for a general deformation retract  $C'$ (without making Assumption \ref{assumption on p}), is  homotopic to $\mr{Aut}(C'^\bt,d')$.
\end{enumerate}
\end{lemma}

\begin{proof}
For (\ref{lemma: space of ind data (a)}), note that induction data $C^\bt\wavy{}H^\bt(C)$ (and the corresponding Hodge decompositions) are in one-to-one correspondence with pairs of 
a right splitting of the short exact sequence
$C^\bt_\mr{exact}\hra C^\bt_\mr{closed}\proj H^\bt(C)$ and 
a left splitting of the short exact sequence $C^\bt_\mr{closed}\hra C^\bt \xra{d} C^{\bt+1}_\mr{exact}$. Thus, it is contractible as a product of contractible spaces (of one-sided inverses of linear inclusions and projections). 

For (\ref{lemma: space of ind data (b)}), note that once chain maps $i,p$ satisfying $p\circ i=\mr{id}$ are fixed, the space of choices of $K$ is contractible, by applying (\ref{lemma: space of ind data (a)}) to $\ker p \wavy{} 0$. Thus, $\mathfrak{I}$ contracts onto the space of choices of pairs $(i,p)$. Fixing some $(i_0,p_0)$, one can obtain any other $(i,p)$ by applying to $(i_0,p_0)$ transformations (\ref{ind data deformation II}), (\ref{ind data deformation III}) (note that these formulae describe not just infinitesimal but also finite transformations; the space of generators $I,P$ is a product of $\mr{Hom}$ spaces of linear spaces and is therefore contractible), and composing with an automorphism of $C'^\bt$ as a complex (which is the finite version of (\ref{ind data deformation IV})). This proves (\ref{lemma: space of ind data (b)}).
\end{proof}

Point (\ref{lemma: space of ind data (a)}) is particularly important for us when discussing quantization, as it will imply that the space of choices of gauge-fixing is contractible.


\section{Cellular abelian $BF$ theory on a closed manifold}\label{sec: closed}
\subsection{Classical theory in BV formalism}\label{sec: classical closed}
In this section we introduce the BV theory for the case at hand. This means an odd (degree $-1$) graded symplectic space together with an even
(degree $0$) function $S$ that Poisson commutes with itself. Such a function is usually called the BV action and the condition
$\{S,S\}=0$
 is called the classical master equation. In addition, we define the second order differential operator $\Delta$, called the BV Laplacian, that generates the Poisson bracket as the defect of the Leibniz identity and we show that the BV action also satisfies the quantum master equation 
 $\frac{1}{2} \{S,S\}-i\hbar \Delta S =0 $, where $\hbar$ is a parameter. The ``Planck constant'' $\hbar$ can be interpreted as the distance from the classical theory (or the strength of quantization). If $\hbar$ is a nonzero real number, instead of a formal parameter, the quantum master equation may also be equivalently written as $\Delta e^{\frac i\hbar S}=0$.

Let $M$ be a closed oriented 
piecewise-linear
$n$-manifold and $X$ a cellular decomposition of $M$.

Let also $E$ be a rank $m$ vector bundle over $M$ endowed with a fiberwise density $\mu_E$ and with a flat connection $\nabla_E$, such that the parallel transport by $\nabla_E$ preserves the density. One can view $E$ as the associated vector bundle $P\times_{G} \bR^m$ for some principal flat $G$-bundle $P$ with $G=SL_\pm (m,\bR)$ the group of $m\times m$ real matrices with determinant $1$ or $-1$.\footnote{A special case of this situation is a flat {\it Euclidean} vector bundle, i.e. with fiberwise scalar product $(,)_E$ and a flat connection preserving it. In this case the structure group reduces to $O(m)\subset SL_\pm(m)$.} By abuse of notations, let $E$ also stand for the corresponding $SL_\pm(m)$-local system on $X$ (i.e. we will suppress the subscripts in $E_X$ and $E^*_{X^\vee}$).

The space of fields is a $\bZ$-graded vector space $\F^\bt$, with degree $k$ component given by
\be \F^k=C^{k+1}(X,E)\oplus C^{k+n-2}(X^\vee,E^*) \label{cell BF fields}\ee
It is concentrated in degrees
$k\in \{-1,\ldots,n-1\}\cup \{2-n,\ldots,2\}$ and is
equipped with a degree $-1$ constant symplectic structure (the BV 2-form)
\be \omega:\quad \F^k\otimes \F^{1-k}\ra \bR \label{omega}\ee
coming from the intersection pairing, see (\ref{omega = <delta B,delta A>}) below.

Introduce the {\it superfields} -- the shifted identity maps pre-composed with projections from $\F$ to the two terms in the r.h.s. of (\ref{cell BF fields}):$$A:\quad \F\ra  C^\bt(X,E)[1] \ra C^\bt(X,E),\qquad  B:\quad \F\ra  C^\bt(X^\vee,E^*)[n-2] \ra C^\bt(X^\vee,E^*)$$
One can also regard 
$A$ and $B$ as coordinate functions on $\F$ taking values in cochains of $X$ and $X^\vee$, respectively, so that the pair $(A,B)$ is a {\it universal} coordinate on $\F$ (i.e. a complete coordinate system). 

We write $A=\sum_{e\subset X} e^* A_e$ -- sum over cells of $X$ of ``local'' superfields $A_e:\F\ra E_{\dot e}$, taking values in the fiber of the local system over the barycenter $\dot e$ of the corresponding cell; $e^*\in C^\bt(X,\bZ)$ stands for the standard basis integral cochain associated to the cell $e$. Similarly, for the second superfield one has $B=\sum_{e^\vee\subset X^\vee}  B_{e^\vee} (e^\vee)^*$, with local components $B_{e^\vee}:\F\ra E^*_{\dot e^\vee}$. (Note that our  convention for ordering the cochain and the superfield component is different between superfields $A$ and $B$.) Internal degrees of field components are $|A_e|=1-\dim e$, $|B_{e^\vee}|=n-2-\dim e^\vee$; in particular, the \emph{total degree} (cellular cochain degree + internal degree) for the superfield $A$ is $1$ and for $B$ is $(n-2)$.

In terms of superfields, the symplectic form (\ref{omega}) is defined as
\be \omega= \langle \delta B, \delta A \rangle\quad \in \Omega^2(\F)_{-1} \label{omega = <delta B,delta A>}\ee
where $\delta$ is the de Rham differential on the space of fields\footnote{
In the language of the variational bicomplex, $\delta$ is the ``vertical differential'' mapping $\Omega^p(\F)\ra \Omega^{p+1}(\F)$. It is formal and we stress its distiction from the ``horizontal differential'' $d$ -- the cellular coboundary operator on cochains of $X$ which does care about the adjacency of cells in $X$.
}
and
$\langle,\rangle: C^{n-\bt}(X^\vee,E^*)\otimes C^\bt(X,E)\ra \bR$ is the inverse of the intersection pairing (\ref{intersection pairing closed twisted}) for chains.\footnote{
We will use the sign convention where the graded binary operation $\langle,\rangle$ is understood as taking a cochain on $X^\vee$ from the left side and a cochain on $X$ from the right side. In other words, the mnemonic rule is that, for the sake of Koszul signs, the {\it comma} separating the inputs in $\langle b,a\rangle$ carries degree $-n$. This pairing is related to the one which operates from the left on two inputs coming from the right by $\langle b,a \rangle=(-1)^{n\deg b}\langle b,a \rangle'$.
}
The symplectic form $\omega$ induces the degree $+1$ Poisson bracket $\{,\}$ and the BV Laplacian\footnote{\label{footnote: Delta_can vs Delta} Recall that, generally, to define the BV Laplacian on {\it functions} (as opposed to half-densities) on an odd-symplectic manifold $\mathcal{M}$, one needs a volume element on $\mathcal{M}$ \cite{SchwarzBV}. In our case, the space of fields is linear, and so possesses a canonical (constant) volume element determined up to normalization. Since the BV Laplacian is not sensitive to rescaling the volume element by a constant factor, we have a preferred BV Laplacian.} $\Delta=\left\langle\frac{\dd}{\dd A},\frac{\dd}{\dd B}\right\rangle$ on the appropriate space of functions on $\F$ which we denote by $\Fun(\F)$. For the purpose of this paper we choose $\Fun(\F)=\widehat{\Sym^\bt}\F^*$ -- the algebra of polynomial functions on $\F$ completed to formal power series.\footnote{In the context of classical abelian $BF$ theory we could instead work with smooth functions on $\F$.}

\begin{remark}\label{rem: closed non-orientable case}
We can allow $M$ to be non-orientable as in Remark \ref{rem: orientability}: we twist the $B$-superfield by the orientation local system $\mr{Or}$ (which superfield to twist is an arbitrary choice). In this case the space of fields becomes $\F=C^\bt(X,E)[1]\oplus C^\bt(X^\vee,E^*\otimes \mr{Or})[n-2]$ and the intersection pairing depends on a choice of primitive top class $\sigma\in H_n(M,\mr{Or})$.
\end{remark}

The BV action of the model is
\be S=\langle B , dA \rangle \quad \in \Fun(\F)_0 \label{S}\ee
where $d$ is the 
coboundary operator in $C^\bt(X,E)$ twisted by the local system.
It satisfies the classical master equation
$$\{S,S\}=0$$
Indeed, the left hand side is $\{S,S\}=2\langle B, d^2 A\rangle=0$. Moreover, one has $\Delta S=\mr{Str}_{C^\bt(X,E)} d=0$ (the supertrace of the coboundary operator vanishes since $d$ changes degree). This implies that the quantum master equation is also satisfied:
\be \Delta e^{\frac{i}{\hbar} S}=0 \quad \Leftrightarrow \quad \frac{1}{2} \{S,S\}-i\hbar \Delta S =0 \label{QME}\ee

The Hamiltonian vector field corresponding to $S$ is the degree $+1$ linear map
$d_X+ d_{X^\vee}: \F\ra \F$
dualized to a map $\F^*\ra \F^*$ and extended to $\Fun(\F)$ as a derivation:
$$Q=\{S,\bt\}=\langle dA,\frac{\dd}{\dd A} \rangle+ \langle dB,\frac{\dd}{\dd B} \rangle$$

The Euler-Lagrange 
equations\footnote{The Euler--Lagrange equations describe the critical locus of $S$ or, equivalently, the zero locus of $Q$.
}
for (\ref{S}) read $dA=0$, $dB=0$. The space of solutions
$$\EL=C^\bt_\mr{closed}(X,E)[1]\oplus C^\bt_\mr{closed}(X^\vee,E^*)[n-2]\qquad \subset \F$$
is coisotropic in $\F$ and its reduction
$$\underline \EL= H^\bt(X,E)[1]\oplus H^\bt(X^\vee,E^*)[n-2]\quad = H^\bt(M,E)[1]\oplus H^\bt(M,E^*)[n-2]$$
 is independent of the cellular decomposition $X$. We will use it as the space of {\it {\zeromodes}} for quantization (in the sense that the partition function will be defined using the framework of effective BV actions, as a fiber BV integral over the space of fields as fibered over {\zeromodes}, cf. \cite{DiscrBF,CSinvar,1DCS,CMRpert}).

\subsection{Quantization}\label{sec: quantization closed}
Our goal in this section is to construct the 
 partition function $Z$ for cellular abelian $BF$ theory on a closed manifold $M$ with cell decomposition $X$, such that $Z$ is invariant under subdivisions of $X$. The 
 partition function will be constructed as a half-density
 on the space of \zeromodes\ $\underline \EL$ via a finite-dimensional  fiber BV integral.

Recall that, for a 
finite dimensional graded vector space $W^\bt$,
one can define the {\it determinant line}
$$\mr{Det}\;W^\bt :=  \bigotimes_{k} \left(\wedge^{\dim W^k} W^k\right)^{(-1)^k} $$
where for $L$ a line (i.e. a $1$-dimensional vector space), $L^{-1}$ denotes the dual line $L^*$. If furthermore $W^\bt$ is based, with $w^k=(w^k_1,\ldots,w^k_{N_k})$ a basis in $W^k$, then one has an associated element
\be\mu=\left(\bigotimes_{k\;\mr{even}}w^k_1\wedge\cdots \wedge w^k_{N_k}\right)\otimes \left(\bigotimes_{k\;\mr{odd}}(w^k_1)^*\wedge\cdots \wedge (w^k_{N_k})^*\right)\quad \in \mr{Det}\; W^\bt \label{coord density}\ee
where $(w^k)^*$ is the basis in $(W^*)^{-k}=(W^k)^*$ dual to $w^k$.

Tensoring the cellular basis in $C^k(X;\bZ)$ with
the standard  basis in $\bR^m$ (or any {\it unimodular} basis, i.e. one on which the standard density on $\bR^m$ evaluates to $1$),
we obtain a preferred basis in $C^\bt(X,E)$. Associated to it by the construction above is an element
$\mu_C \in \mr{Det}\;C^\bt(X,E)$ (well-defined modulo sign).

Passing to densities (for more details see Appendix \ref{AA} and \cite{MnevTorsions}), we have a canonical isomorphism
\be\Det\; C^\bt(X,E)\;/\{\pm 1\} \cong \Dens\;  C^\bt(X,E)[1] 
\xra{(\ast)^{\otimes 2}}
\Dens\;\F
\xra{\sqrt{\ast}}
\Dens^{\frac{1}{2}}\F  \label{det-dens}\ee
where on the r.h.s. we have half-densities on $\F$. 
The middle isomorphism comes from the fact that, for $W^\bt$ a graded vector space,
\be\Det\; (W\oplus W^*[-1])\cong (\Det\; W)^{\otimes 2}\label{Det(W+W^*)}\ee
Denote by $\mu_\F^{1/2}\in \Dens^{\frac{1}{2}}\F$ the image of $\mu_C$ under the isomorphism (\ref{det-dens}).\footnote{The superscript in $\mu_\F^{1/2}$ stands for both the weight of the density and for the square root.}

One can combine the action $S$ with $\mu_\F^{1/2}$ into 
a (coordinate-dependent) half-density $e^{\frac{i}{\hbar}S}\mu_\F^{1/2} 
$ which, as a consequence of (\ref{QME}), satisfies the equation
$$\Delta_\mr{can}(e^{\frac{i}{\hbar}S}\mu_\F^{1/2} )=0$$
where $\Delta_\mr{can}$ is the canonical BV Laplacian on half-densities \cite{Khudaverdian,Severa}.\footnote{The canonical BV Laplacian is related to the BV Laplacian $\Delta=\Delta_{\mu_\F}$ on functions 
by $\Delta_\can (f\, \mu_\F^{1/2})=\Delta(f)\,\mu_\F^{1/2}$, where $f\in \Fun(\F)$.}

\subsubsection{Gauge fixing, perturbative partition function}
Choose representatives of cohomology classes $\ii:H^\bt(M,E)\hra C^\bt_\mr{closed}(X,E)
$
and a right-splitting $\K: C^\bt_\mr{exact}(X,E)\ra C^{\bt-1}(X,E)$ of the short exact sequence
$$C^\bt_\mr{closed}(X,E)\ra C^\bt(X,E)\xra{d} C^{\bt+1}_\mr{exact}(X,E)$$
Thus we have a Hodge decomposition
\be C^\bt(X,E)=\underbrace{\ii(H^\bt(M,E))\oplus C^\bt_\mr{exact}(X,E)}_{C^\bt_\mr{closed}(X,E)}\oplus 
\;\mr{im}(\K)
\label{Hodge decomp closed case}\ee
We extend the domain of $\K$ to the whole of $C^\bt(X,E)$ by defining it to be zero on the first and third terms of (\ref{Hodge decomp closed case}).

Hodge decomposition (\ref{Hodge decomp closed case}) together with the dual decomposition\footnote{Here the second term on the r.h.s. is nondegenerately paired to the third term on the r.h.s. of (\ref{Hodge decomp closed case}) by Poincar\'e duality and vice versa; the first terms are paired between themselves. The map $\K^\vee_k: C^k(X^\vee,E^*)\ra C^{k-1}(X^\vee,E^*)$ is defined as the dual (transpose) of $\K_{n-k+1}: C^{n-k+1}(X,E)\ra C^{n-k}(X,E)$, up to the sign $(-1)^{n-k}$.}
\be C^\bt(X^\vee,E^*)=\ii^\vee(H^\bt(M,E^*))\oplus C^\bt_\mr{exact}(X^\vee,E^*)\oplus \mr{im}(\K^\vee)\label{dual Hodge decomp, closed}\ee
gives the symplectic splitting
$$\F=\ii(\underline\EL)\oplus \F_\mr{fluct}$$
and produces the Lagrangian subspace $\LL=\mr{im}(\K)[1]\oplus \mr{im}(\K^\vee)[n-2]\subset \F_\fluct$.

For half-densities on $\F$, we have
$$\Dens^{\frac12}\F\cong \Dens^{\frac12}\underline\EL\;\otimes\;\underbrace{\Dens^{\frac12}\F_\mr{fluct}}_{\cong \Dens\;\LL}$$
We are using the general fact \cite{Manin} that, for a Lagrangian subspace of an odd-symplectic space $L\subset V$, one has a canonical isomorphism $\Dens^{\frac12}V\ra \Dens\, L$ arising from (\ref{Det(W+W^*)}). 

\begin{remark}
We are free to rescale the  reference half-density $\mu_\F^{1/2}$ on fields by a factor  $\xi_\hbar$. The requirement of having the partition function invariant under subdivisions of $X$ can be achieved, as we will see in Section \ref{sec: normalization},  by introducing such a  factor $\xi_\hbar\in \bC$, which 
is a certain {\it extensive}\footnote{That is, a product over cells of $X$ of certain universal elementary factors, depending only on the dimension of the cell, see Lemma \ref{lemma: zeta = xi/xi} below.} product of powers of $i$, $\hbar$ and $2\pi$. 
\end{remark}

According to the BV quantization scheme, the gauge-fixed partition function on $X$ is defined as
the fiber BV integral\footnote{See \cite[Section 2.2.2]{CMRpert} for details on fiber BV integrals.}
\begin{multline}\label{Z fiber BV}
Z(X,E)=\int_\LL e^{\frac{i}{\hbar}S(\ii(\Azm)+\Afluct,\ii^\vee(\Bzm)+\Bfluct)} \;\xi_\hbar\; \mu_\F^{1/2}
= \\ =\int_\LL e^{\frac{i}{\hbar}S(A_\mr{fluct},B_\mr{fluct})}\;\xi_\hbar\;\mu_\F^{1/2}
\qquad
\in \bC\otimes \Dens^{\frac12}(\underline{\EL})\cong \bC\otimes \Det\; H^\bt(M,E)\;/\{\pm 1\}
\end{multline}
where $\Azm,\Bzm$ are the superfields for $\underline{\EL}$ and $\Afluct,\Bfluct$ are the superfields for $\F_\fluct$. By BV-Stokes' theorem for fiber BV integrals, the value of the integral is independent of the choice of $\ii,\K$.
A special feature of the model at hand is that the value of the integral is a constant (coordinate-independent) half-density. 

\begin{remark}
A Berezin measure $\mathfrak{m}$ on a superspace $V=(V^\mr{even},V^\mr{odd})$ is not exactly the same as a density $\mu$ on $V$. 
Indeed, for a parity-preserving automorphism of $V$, $g=\left(\begin{array}{cc} g^{even} & 0 \\ 0 & g^\mr{odd} \end{array}\right)$, with $g^\mr{even}\in GL(V^\mr{even})$, $g^\mr{odd}\in GL(V^\mr{odd})$, the Berezin integral behaves as
$$\int_V \mathfrak{m}\cdot f=\int_V \underbrace{|\det g^\mr{even}|\cdot (\det g^\mr{odd})^{-1}\;\mathfrak{m}}_{=:(g^{-1})_*\mathfrak{m}}\cdot g^*f$$
for $f\in \Fun(V)$ an integrable 
function. On the other hand, a density on $V$ transforms as
$$\mu\mapsto |\det g^\mr{even}|\cdot |\det g^\mr{odd}|^{-1}\;\mu$$
(see Appendix \ref{appendix: det lines, densities, torsion}). Thus, a Berezin measure changes its sign when acted on by an automorphism which changes the orientation of $V^\mr{odd}$, whereas a density does not. In this work we are calculating partition functions modulo signs,
so we can identify Berezin measures with densities.
\end{remark}

Integral (\ref{Z fiber BV}) is a conditionally convergent\footnote{Convergence is due to the fact that, by construction, the point $(A,B)=0$ is an isolated critical point of the action $S$ restricted to $\LL$.} Gaussian integral over a  finite dimensional superspace; we will show in Section \ref{sec: normalization}, Proposition \ref{prop: Z closed}, that its value 
is
\be Z(X,E)=\xi_\hbar^{H^\bt}\tau(X,E) \label{Z=tau}\ee
where $\tau(X,E)$ is the Reidemeister torsion (or, equivalently, ``$R$-torsion'', see, e.g. \cite{Milnor66, Turaev}) of the CW-complex $X$ with local system $E$ and the coefficient $\xi_\hbar^{H^\bt}\in\bC$ depends only on Betti numbers.
Note that the $R$-torsion for a non-acyclic local system is indeed an element of $\Det\; H^\bt(M,E)\;/\{\pm 1\}$, not a number. By the combinatorial invariance property of the $R$-torsion, the partition function (\ref{Z fiber BV}) depends only on the manifold $M$ and the local system $E$, but not on a particular cellular decomposition $X$.

In the special case of an acyclic local system, $H^\bt(M,E)=0$, the determinant line $\Det\;H^\bt\cong \bR$ is the trivial line and the partition function (\ref{Z fiber BV}) is an actual number, defined modulo sign.

Result (\ref{Z=tau}) can be viewed as a combinatorial version, generalized to possibly non-acyclic local systems, of the result of \cite{Schwarz79}, where analytic torsion was interpreted as a functional BV integral for abelian $BF$ theory.

\begin{remark} We note, anticipating the discussion of the non-abelian case in Section \ref{s:BVpfsheca}, that the partition function (\ref{Z=tau}) is invariant under simple-homotopy equivalence of cellular complexes (the equivalence relation generated by elementary expansions and collapses, see Definition \ref{def: simple-homotopy} below for a reminder), since the Reidemeister torsion is a simple-homotopy invariant, see \cite{Milnor66}.
\end{remark}

\subsubsection{The propagator}
\label{rem: propagator, closed case}
Denote by $p_1, p_2$ the cellular projections from the product CW-complex $X\times X^\vee$ (a cellular decomposition of $M\times M$) to the first and second factor, respectively.
Let $\KK\in C^{n-1}(X\times X^\vee,p_1^* E\otimes p_2^* E^*)$ be the {\it parametrix} for the operator $\K$, i.e. the image of $\K$ under the isomorphism
$$\underbrace{\mr{End}(C^\bt(X,E))_{-1}}_{\ni\;\; \K}\simeq \bigoplus_{k=0}^{n-1} C^k(X,E)\otimes C^{n-k-1}(X^\vee,E^*) \simeq \underbrace{C^{n-1}(X\times X^\vee,p_1^* E\otimes p_2^* E^*)}_{\ni\;\; \KK} $$
Here in the first isomorphism we use the Poincar\'e duality to identify $C^{n-k-1}(X^\vee,E^*)$ with the dual of $C^{k+1}(X,E)$. Then $\KK$ is the propagator of the theory, i.e. (up to a factor of $i\hbar$) the normalized expectation value of the product of evaluations of fluctuations of fields at two cells:
\begin{multline} i\hbar\, \KK(e,e^\vee)= \ll A_\fluct(e)\cdot B_\fluct(e^\vee)\gg:= \\ 
=\frac{1}{Z}\int_\LL e^{\frac{i}{\hbar}S(A,B)} A_\fluct(e)\cdot B_\fluct(e^\vee)\; \xi_\hbar \mu_\F^{1/2}\quad \in E_{\dot e}\otimes E^*_{\dot e^\vee} 
\label{propagator}\end{multline}
Here $e\subset X$, $e^\vee\subset X^\vee$ are two arbitrary cells; $E_{\dot e}$, $E^*_{\dot e^\vee}$ are the fibers of $E$, $E^*$ over the corresponding barycenters; $A_\fluct(e): \F_\fluct\ra E_{\dot e}$ and $B_\fluct(e^\vee): \F_\fluct\ra E^*_{\dot e^\vee}$ are the fluctuations of fields evaluated at the cells $e,e^\vee$. Propagator (\ref{propagator}) between two cells vanishes unless the relation $\dim e+\dim e^\vee=n-1$ holds.
Let furthermore $[h_\alpha]$ be a basis in cohomology $H^\bt(M,E)$, $[h^\vee_\alpha]$ the corresponding dual basis in $H^{n-\bt}(M,E^*)$, and let $\chi_\alpha=\ii [h_\alpha]$, $\chi^\vee_\alpha=\ii^\vee [h^\vee_\alpha]$ be the representatives of cohomology in cochains. Then, due to the equations (\ref{ind data equations}) satisfied by $\K$, we have the following  equations satisfied by the the parametrix:
\begin{multline}\label{propagator equations}
d_{X\times X^\vee} \KK = \sum_{e\subset X} e^*\otimes \varkappa(e)^*\otimes\mathbf{1}-\sum_\alpha \chi_\alpha\otimes \chi^\vee_\alpha, \\
\sum_{e'\subset X}\KK(e,\varkappa(e'))\cdot \chi_\alpha(e') = 0,\quad
\sum_{e'\subset X}\chi^\vee_\alpha(\varkappa(e'))\cdot\KK(e',e^\vee) = 0, \quad
\sum_{e'\subset X}\KK(e,\varkappa(e'))\cdot\KK(e',e^\vee) = 0
\end{multline}
Here $\mathbf{1}\in E_{\dot e}\otimes E^*_{\dot {\varkappa(e)}}$ is the element corresponding to the identity $\mr{id}\in \mr{End}(E_{\dot e})$; $\varkappa(e')$ is the cell dual to $e'$, as in Section \ref{sec: reminder on Poincare duality}. In the last three equations implicit in the notations is the convolution $\mr{tr}\,:\; E^*_{\dot{\varkappa(e')}}\otimes E_{\dot e}\ra \bR$.


\subsubsection{Fixing the normalization of densities}\label{sec: normalization}
Now we 
will focus on
the factors $i$, $\hbar$ and $2\pi$ coming from the Gaussian integral (\ref{Z fiber BV}) and will fix the normalization factors $\xi_\hbar$, $\xi_\hbar^{H^\bt}$ in (\ref{Z fiber BV},\ref{Z=tau}).

The model Gaussian integrals over a pair of even or odd variables are:
\be
\int_{\bR^2}dx\, dy\; e^{\frac{i}{\hbar}yx} = 2\pi\hbar,\qquad 
\int_{\Pi\bR^2}\D\theta\, \D\eta\; e^{\frac{i}{\hbar}\eta \theta}=\frac{i}{\hbar} 
\label{model integrals}
\ee
Note that the first integral is conditionally convergent. In the second integral
$\Pi\bR^2$ stands for the odd plane, with Grassmann coordinates $\theta,\eta$; we write ``$\D$'' in $\D\theta$, $\D\eta$ to emphasize that the Berezin integration measure is not a differential form.

More generally, for $\mathsf{B}$ an even non-degenerate bilinear form on superspace $\bR^{N|N'}$, we have
\be\int_{\bR^{2N|2N'}} 
\prod_{j=1}^N d x^j_\mr{even}\; d y^j_\mr{even}\cdot \prod_{j'=1}^{N'} \D x^{j'}_\mr{odd}\; \D y^{j'}_\mr{odd}\quad
e^{\frac{i}{\hbar}\mathsf{B}(\vec{y},\vec{x})}= (2\pi\hbar)^N \left(\frac{i}{\hbar}\right)^{N'} 
\mr{Sdet}(\mathsf{B})^{-1}
\label{model integral multi-dim}\ee
Note that, if $\mathsf{B}_\mr{even}$, $\mathsf{B}_\mr{odd}$ denote the even-even and odd-odd blocks of the matrix of $\mathsf{B}$, i.e., if $\mathsf{B}=
\left(\begin{array}{c|c} \mathsf{B}_\mr{even} & 0 \\ \hline 0 & \mathsf{B}_\mr{odd} \end{array}\right)
$, then we have $\mr{Sdet}(\mathsf{B})=\frac{{\det}\mathsf{B}_\mr{even}}{ {\det}\mathsf{B}_\mr{odd}}$.

Therefore, for the fiber BV integral in (\ref{Z fiber BV}), without the factor $\xi_\hbar$ (which is yet to be specified), we have the following.
\begin{lemma}\label{lemma: torsion as BV integral}
\be \int_\LL e^{\frac{i}{\hbar}S} \mu_\F^{1/2}= \zeta_\hbar \tau(X,E) \label{fiber BV int with zeta}\ee
where the factor is
\be \zeta_\hbar=(2\pi\hbar)^{\frac{1}{2}\dim\LL^\mr{even}}\left(\frac{i}{\hbar}\right)^{\frac{1}{2}\dim\LL^\mr{odd}}=
(2\pi\hbar)^{\dim\mr{im}(K)^\mr{odd}}\left(\frac{i}{\hbar}\right)^{\dim\mr{im}(K)^\mr{even}} \label{zeta_hbar}\ee
\end{lemma}

\begin{proof} Choose some bases for all terms in the r.h.s. of (\ref{Hodge decomp closed case}): a basis $c_H$ in cohomology $H^\bt(X,E)$, $c_\mr{ex}$ in $C^\bt_\mr{exact}$ and $c_\mr{coex}$ in $\mr{im}(K)$. We can assume that the product of the corresponding coordinate densities agrees with the density $\mu_C\in \Det\; C^\bt(X,E)/\{\pm 1\}$ associated to the cellular basis in cochains:
\be \mu_H\cdot\mu_\mr{ex}\cdot\mu_\mr{coex}=\mu_C \label{mu_H mu_ex mu_coex}\ee
This can always be arranged,  e.g., by rescaling one of the basis vectors in $c_\mr{ex}$.

We have dual bases $c_H^\vee,c_\mr{ex}^\vee,c_\mr{coex}^\vee$ on the terms of the dual Hodge decomposition (\ref{dual Hodge decomp, closed}) for $C^\bt(X^\vee,E^*)$. The corresponding densities $\mu_H^\vee\in\Dens\; H^\bt(X^\vee,E^*)[n-2]$, $\mu_\mr{ex}^\vee\in\Dens\; C^\bt_\mr{ex}(X^\vee,E)[n-2]$, $\mu_\mr{coex}^\vee\in \Dens\; \mr{im}(K^\vee)[n-2]$ are related to the ones on the l.h.s. of (\ref{mu_H mu_ex mu_coex}) by
$$\mu_H^\vee=\mu_H,\quad \mu_\mr{coex}^\vee=\mu_\mr{ex},\quad \mu_\mr{ex}^\vee=\mu_\mr{coex}$$

The integral on the l.h.s. of (\ref{fiber BV int with zeta}) yields
\be\int e^{\frac i\hbar \langle B_\mr{coex},d A_\mr{coex}\rangle}\underbrace{\D A_\mr{coex}}_{\mu_\mr{coex}}\; \underbrace{\D B_\mr{coex}}_{\mu^\vee_\mr{coex}}\;\mu_H=\zeta_\hbar\cdot \mu_H\cdot \mr{Sdet}_{C^\bt_\mr{coex}\ra C^{\bt+1}_\mr{ex}}(d)\label{BV int = torsion}\ee
where the super-determinant appearing on the r.h.s., 
$$\mr{Sdet}(d)=\prod_{k=0}^{n-1}({\det}_{C^k_\mr{coex}\ra C^{k+1}_\mr{ex}}(d))^{(-1)^k}$$ 
is the alternating product of determinants of matrices of isomorphisms $d: C^k_\mr{coex}(X,E)\ra C^{k+1}_\mr{ex}(X,E)$ with respect to the chosen bases $c_\mr{coex}$, $c_\mr{ex}$.
In the last two terms on the r.h.s. of (\ref{BV int = torsion}) one recognizes one of the definitions of $R$-torsion. The coefficient $\zeta_\hbar$ arises as in (\ref{model integral multi-dim}).
\end{proof}

Consider the Hilbert polynomial, packaging the information on dimensions of cochain spaces into a generating function:
\be \Ppol_{C^\bt}(t)=\sum_{k=0}^n t^k\cdot \dim C^k(X,E) \label{P_C}\ee
and the polynomial, counting dimensions of $K$-exact cochains by degree:
$$\Qpol(t)=\sum_{k=0}^n t^k\cdot \dim\, \mr{im}(K)^k $$

Hodge decomposition (\ref{Hodge decomp closed case}) implies the relation $\Ppol_{C^\bt}(t)=\Ppol_{H^\bt}(t)+(1+t)\cdot \Qpol(t)$, or equivalently
\be \Qpol(t)=\frac{\Ppol_{C^\bt}(t)-\Ppol_{H^\bt}(t)}{1+t} \label{Q via P}\ee
where
\be\label{P_H}\Ppol_{H^\bt}(t)=\sum_{k=0}^n t^k\cdot \dim H^k(M,E)\ee
Note that $\Ppol_{C^\bt}(-1)=\Ppol_{H^\bt}(-1)$ is the Euler characteristic $\chi(C^\bt(X,E))=\mr{rk}(E)\cdot\chi(M)$, and hence there is no singularity on the r.h.s. of (\ref{Q via P}).

Exponents in (\ref{zeta_hbar}) can be expressed in terms of values of $\Qpol$ at $t=\pm 1$:
\begin{multline}\label{dim im(K) via P,Q}
\dim\mr{im}(K)^\mr{even}=\frac{\Qpol(1)+\Qpol(-1)}{2}=\frac{1}{4}(\Ppol_{C^\bt}(1)-\Ppol_{H^\bt}(1))+ \frac{1}{2}(\Ppol'_{C^\bt}(-1)-\Ppol'_{H^\bt}(-1))\\
\dim\mr{im}(K)^\mr{odd}=\frac{\Qpol(1)-\Qpol(-1)}{2}=\frac{1}{4}(\Ppol_{C^\bt}(1)-\Ppol_{H^\bt}(1))- \frac{1}{2}(\Ppol'_{C^\bt}(-1)-\Ppol'_{H^\bt}(-1))
\end{multline}
where prime stands for the derivative in $t$ (emerging from evaluating $\Qpol(-1)$ by applying L'H\^{o}pital's rule to the r.h.s. of (\ref{Q via P})). 

\begin{lemma} \label{lemma: zeta = xi/xi}
One can split the coefficient in (\ref{fiber BV int with zeta}) as
\be\zeta_\hbar=\frac{\xi_\hbar^{H^\bt}}{\xi_\hbar} \label{zeta via xi}\ee
with
\be\xi_\hbar=\prod_{k=0}^n (\xi_\hbar^k)^{\dim C^k(X,E)},\qquad \xi_\hbar^{H^\bt}=\prod_{k=0}^n (\xi_\hbar^k)^{\dim H^k(M,E)}\label{xi_hbar}\ee
where we denoted
\be \xi_\hbar^k= (2\pi \hbar)^{-\frac14+\frac12\, k\,(-1)^{k-1}}(e^{-\frac{\pi i}{2}}\hbar)^{\frac14+\frac12 \,k\,(-1)^{k-1}} \label{xi_hbar^k}\ee
\end{lemma}

\begin{proof}
Indeed, (\ref{zeta_hbar}) together with (\ref{dim im(K) via P,Q}) implies that one can write $\zeta_\hbar=\xi_\hbar^{H^\bt}/\xi_\hbar$ with
$$\xi_\hbar=\xi_\hbar^{C^\bt}=(2\pi\hbar)^{-\frac{1}{4}\Ppol_{C^\bt}(1)+\frac{1}{2}\Ppol'_{C^\bt}(-1)} \left(e^{-\frac{\pi i}{2}}\hbar\right)^{\frac{1}{4}\Ppol_{C^\bt}(1)+\frac{1}{2}\Ppol'_{C^\bt}(-1)}$$
and $\xi_\hbar^{H^\bt}$ given by the same formula, replacing cochains by cohomology. Then formulae (\ref{xi_hbar}) follow immediately from the definitions (\ref{P_C},\ref{P_H}) of $\Ppol_{C^\bt}(t)$, $\Ppol_{H^\bt}(t)$.
\end{proof}

Note that $\xi_\hbar$ is a product over cells of $X$ of factors depending only on dimension of the cell. We define the normalized density
$$\mu_\hbar:=\xi_\hbar\cdot \mu\quad \in \bC\otimes\Det\; C^\bt(X,E)/\{\pm 1\}$$
which can be seen as a product of normalized elementary densities for individual cells of $X$,
\be \mu_\hbar=\prod_{e\subset X}\underbrace{(\xi_\hbar^{\dim e})^{\mr{rk}E}\cdot\D A_e}_{=:\D_\hbar A_e}
\label{normalized density}\ee
with $\D A_e\in\Dens\; E_{\dot e}[1-\dim e]
$ the elementary density for the cell $e$, associated to 
a unimodular basis in the fiber of $E$ over the barycenter $\dot e$ of $e$.

On the other hand, $\xi_\hbar^{H^\bt}$ depends exclusively on Betti numbers of cohomology, and as such is manifestly independent under subdivisions of $X$. In particular, for an acyclic local system, $\xi_\hbar^{H^\bt}=1$.
Summarizing this discussion, the result (\ref{fiber BV int with zeta}) can be rewritten as follows.
\begin{Proposition}\label{prop: Z closed}
The perturbative partition function (\ref{Z fiber BV})
with normalization of integration measure fixed by (\ref{normalized density}) is
\be Z(X,E)=\int_\LL e^{\frac{i}{\hbar}S(A,B)}(\mu_\F^\hbar)^{1/2}= 
\xi_\hbar^{H^\bt}\cdot \tau(X,E)\qquad \in \bC\otimes \Dens^{\frac12}(\underline\EL) \label{fiber BV int with xi^H}\ee
with $\xi_\hbar^{H^\bt}$ given by (\ref{xi_hbar}), (\ref{xi_hbar^k}). Here the normalized half-density on the space of fields is $(\mu_\F^\hbar)^{1/2}=\sqrt{(\mu_\hbar)^{\otimes 2}}|_\LL=\xi_\hbar\cdot \mu_\F^{1/2}$. The partition function $Z$ is independent of the details of gauge-fixing and is invariant under subdivisions of $X$.
\end{Proposition}
Formula (\ref{fiber BV int with xi^H}) is indeed just the formula (\ref{Z=tau}), where we have identified the factor in front of the torsion by (\ref{xi_hbar}).

For the later use, alongside with the notation $\D_\hbar A_e$ introduced in (\ref{normalized density}), we also introduce the notation 
\begin{equation}\label{D_h B}
\D_\hbar B_{e^\vee}:=(\xi_\hbar^{n-\dim e^\vee})^{\mr{rk}E} \D B_{e^\vee}
\end{equation}
-- the normalized elementary density for the field $B$, associated to a cell $e^\vee\subset X^\vee$ of the dual CW-complex; $\D B_{e^\vee}\in \Dens\,E^*_{\dot e^\vee}[n-2-\dim e^\vee]$ is the elementary density associated to a unimodular basis in the fiber of $E^*$ over the barycenter of $e^\vee$. With these definitions, for the normalized half-density on bulk fields, appearing in (\ref{fiber BV int with xi^H}), we have
$$(\mu_\F^\hbar)^{1/2}=\prod_{e\subset X}(\D_\hbar A_e)^{1/2} (\D_\hbar B_{\varkappa(e)})^{1/2}$$
with $\varkappa(e)\subset X^\vee$ the dual cell for $e$, as in Section \ref{sec: reminder on Poincare duality}.

\begin{remark}[Phase of the partition function]\label{rem: phase}
By the discussion above, in the case of a non-acyclic local system $E$, the partition function $Z(X,E)$ attains a nontrivial complex phase of the form $e^{\pi i s/8}\in U(1)/\{\pm 1\}$, with
\be s=\sum_k (-1+2k\,(-1)^{k})\cdot\dim H^k(M,E)\mod 8 \label{phase s}\ee
(we do not take it $\mr{mod}\, 16$, since we anyway only define the partition function modulo sign). This looks surprising, since the model integrals (\ref{model integrals}) contain simpler phases (integer powers of $e^{\frac{\pi i}{2}}$). The complicated phase arises because we split the factor $\zeta_\hbar$ in (\ref{fiber BV int with zeta}), which contains only a simple phase, into a factor with cellular locality and a factor depending only on cohomology (\ref{zeta via xi}).
\end{remark}

\begin{remark}
For a closed  manifold of dimension $\dim M=3\;\mr{mod}\; 4$, the phase of the partition function is trivial, $e^{\pi i s/8}=\pm1$, as follows from (\ref{phase s}) and from Poincar\'e duality.
\end{remark}

\begin{remark}[Normalization ambiguities]
One can change the definition (\ref{xi_hbar}) of $\xi_\hbar$ by a factor $\exp(\nu \underbrace{\Ppol(-1)}_{\chi(C^\bt(X,E))})$ with $\nu\in\bC$ a parameter. Performing the rescaling
\be \xi_\hbar\wavy{} e^{\nu \Ppol_{C^\bt}(-1)}\cdot\xi_\hbar, \qquad \xi_\hbar^{H^\bt}\wavy{} e^{\nu \Ppol_{H^\bt}(-1)}\cdot\xi_\hbar^{H^\bt}\label{xi_hbar rescaling}\ee
(or, equivalently, redefining $\xi_\hbar^k\wavy{}\xi_\hbar^k\cdot e^{(-1)^k\nu}$) does not change the quotient (\ref{zeta via xi}). For example, one can choose $\nu=\frac{i\pi}{8}$, which has the effect of changing the phase of the partition function of Remark \ref{rem: phase} from $e^{\pi i s/8}$ to $e^{\pi i s'/4}$ with
$$s'= \sum_k \left(-\frac{1-(-1)^k}{2}+k\,(-1)^{k}\right)\cdot\dim H^k(M,E)\mod 4$$
Another ambiguity in the phase of Remark \ref{rem: phase} stems from the possibility to change values of the model integrals (\ref{model integrals}) by some integral powers of $e^{2\pi i}$, which results in the shift of $s$ in (\ref{phase s}) by a multiple of $4\cdot\sum_k\dim H^k(M,E)$. Since we only consider $s\;\mr{mod}\; 8$, this shift can be viewed as a special case of the transformation (\ref{xi_hbar rescaling}), with $\nu=r\cdot\frac{i\pi}{2}$ for some $r\in\bZ$.  
\end{remark}

\section{Cellular abelian $BF$ theory on manifolds with boundary: classical theory}\label{sec: classical cell ab BF on mfd with bdry}

\subsection{Classical theory on a cobordism}

Let $M$ be a compact oriented piecewise-linear $n$-manifold with boundary $\dd M=\overline{M_\din}\sqcup M_\dout$. Overline indicates that we take $M_\din$ with the orientation opposite to that induced from $M$, whereas the orientation of $M_\dout$ agrees with that of $M$. Let $(E,\mu_E,\nabla_E)$ be a flat vector bundle over $M$ of rank $m$ with a horizontal fiberwise density $\mu_E$, and let $X$ be a cellular decomposition of $M$.

We define the space of fields to be the graded vector space
\be \F=C^\bt(X,E_X)[1]\oplus C^\bt(X^\vee,E^*_X)[n-2] \label{cell BF fields on a cobordism}\ee
where $X^{\vee}$ is defined as in Section \ref{sec: Lefschetz for cell decomp, cobordism}, $E_X$ is the cellular $SL_\pm(m)$-local system on $X$ induced by the vector bundle $E$, and $E^*_X$ is the dual local system on the dual cellular decomposition (cf. Section \ref{sec: local systems on the dual cell decomp}). We will suppress the local system in the notation for fields onwards: cochains on $X$ are always taken with coefficients in $E_X$, cochains on $X^\vee$ -- with coefficients in $E^*_X$.

The space of fields is equipped with a 
constant pre-symplectic structure of degree $-1$,
$$\omega: \F^k\otimes \F^{1-k}\ra \bR$$
which is degenerate if and only if $\dd M$ is non-empty.
We construct $\omega$ by 
combining
the non-degenerate pairing
\be  C^{n-\bt}(X^{\vee},X_\din^\vee)\otimes C^\bt(X,X_\dout)\ra \bR \label{nondeg pairing cochains}\ee
(the inverse of intersection pairing (\ref{Lefschetz for chains mixed})) with the zero maps
$$C^{n-\bt}(X^{\vee})\otimes C^\bt(X_\dout) \xra{0} \bR,\qquad  C^{n-\bt}(X^\vee_\din) \otimes C^\bt(X) \xra{0}\bR$$
to a 
pairing
\be \langle,\rangle:\quad  C^{n-\bt}(X^{\vee})\otimes C^\bt(X) \ra \bR \label{degen pairing for cochains}\ee
In terms of superfields
\be A: \F\ra C^\bt(X),\quad B: \F\ra C^\bt(X^{\vee}) \label{superfields}\ee
the presymplectic form is
$$\omega=\langle \delta B, \delta A\rangle\quad \in \Omega^2(\F)_{-1}$$

The space of boundary fields is defined as
\be \F_\dd=C^\bt(X_\dd)[1]\oplus C^\bt(X_\dd^\vee)[n-2] \label{cell boundary fields}\ee
(
with coefficients in
the pullback of the local system to the boundary); a (non-degenerate) degree $0$ symplectic form (the {\it BFV 2-form}) on $\F_\dd$ is given by
\begin{eqnarray}
\omega_\dd&=&\delta \alpha_\dd=\langle \delta B_\dout , \delta A_\dout \rangle_\dout - \langle \delta B_\din, \delta A_\din\rangle_\din\quad \in \Omega^2(\F_\dd)_{0}, \label{omega_bdry}\\
\alpha_\dd&=&\langle  B_\dout , \delta A_\dout \rangle_\dout - \langle \delta B_\din, A_\din\rangle_\din\quad \in \Omega^1(\F_\dd)_0 \label{alpha_bdry}
\end{eqnarray}
where $\langle,\rangle_{\mr{in/out}}: C^{n-1-k}(X_{\mr{in/out}}^\vee)\otimes C^k(X_{\mr{in/out}})\ra \bR$ is the inverse intersection pairing on the in/out boundary. Boundary superfields $A_\dd=(A_\din,A_\dout),B_\dd=(B_\din,B_\dout)$ in the formulae above are defined similarly to (\ref{superfields}). 
The projection
\be \pi:\F\twoheadrightarrow \F_\dd\label{pi}\ee
is defined as $\pi=\iota^*\oplus(\iota^\vee)^*$, where $\iota^\vee$ is defined to be $\iota_+$ for in-boundary and $\iota_-$ for the out-boundary (cf. Section \ref{sec: Lefschetz for cell decomp} for definition of cellular inclusions $\iota_\pm:X_{\mr{in/out}}\ra X$), i.e. cochains of $X$ are restricted to $X_\dd$, whereas cochains of $X^{\vee}$ are first restricted to $\dd \widetilde M$ (i.e. $M$ with a collar at $M_\dout$ removed and a collar at $M_\din$ added) and then parallel transported, using holonomy of $E^*$, to $\dd M$, cf. (\ref{iota_-^* def},\ref{iota_+^* def}) in Section \ref{sec: local systems on the dual cell decomp}.

We define the action as
\be S=\langle B, dA\rangle+ \langle (\iota^\vee)^* B,\iota^* A \rangle_\din\quad \in \Fun(\F)_0 \label{S with bdry term}\ee

Since $\omega$ is degenerate in the presence of a boundary, one cannot invert it to construct a  Poisson bracket on $\F$. Instead, following the logic of the BV-BFV formalism \cite{CMR}, we introduce a degree $+1$ vector field $Q$ as the map $d_X+d_{X^{\vee}}:\F\ra \F$ dualized to a map $\F^*\ra \F^*$ and extended by Leibnitz rule as a derivation on $\Fun(\F)$:
$$Q=\langle dA ,\frac{\dd}{\dd A}\rangle+\langle dB ,\frac{\dd}{\dd B}\rangle \quad \in \mathfrak{X}(\F)_1$$
(i.e. $QA=dA$, $QB=dB$, where $Q$ acts on functions on $\F$ while $d$ acts on cochains where the superfields take values).
This vector field is cohomological, i.e. $Q^2=0$, and projects to a cohomological vector field on $\F_\dd$,
$$\pi_*Q =Q_\dd=\langle dA_\dd ,\frac{\dd}{\dd A_\dd}\rangle+\langle dB_\dd ,\frac{\dd}{\dd B_\dd}\rangle \quad \in \mathfrak{X}(\F_\dd)_1$$
The projected vector field on the boundary is Hamiltonian w.r.t. to the BFV 2-form $\omega_\dd$, with degree $1$ Hamiltonian
\be S_\dd=\langle B_\dout  , dA_\dout  \rangle_\dout -
\langle B_\din , dA_\din \rangle_\din
\quad \in \Fun(\F_\dd)_1 \label{BFV action}\ee
-- the {\it BFV action} (i.e. the relation is: $\iota_{Q_\dd}\omega_\dd=\delta S_\dd$). On the other hand, $Q$ itself, instead of being the Hamiltonian vector field for $S$, satisfies the following.
\begin{Proposition}
\label{lemma: i_Q=dS+alpha_bdry}
The data of the cellular abelian $BF$ theory satisfy the equation
\be \delta S= \iota_Q \omega-\pi^*\alpha_\dd \label{i_Q=dS+alpha_bdry}\ee
\end{Proposition}
This relation 
is a consequence of
the following. 
\begin{lemma}[Cellular Stokes' formula]\label{lemma: Stokes}
\be (-1)^{n+\deg b}\langle d b, a \rangle +\langle b, da \rangle= \langle (\iota^\vee)^* b, \iota^* a \rangle_\dout  - \langle (\iota^\vee)^* b, \iota^* a \rangle_\din \label{Stokes}\ee
with $a\in C^k(X)$, $b\in C^{n-k-1}(X^{\vee})$ any cochains.
\end{lemma}
\begin{proof}
For $a_\dout\in C^k(X_\dout)$, denote $\til {a_\dout}\in C^k(X)$ the extension of $a_\dout$ by zero on cells of $X-X_\dout$. Likewise denote $\til {b_\din}$ an extension of a cochain on $X^\vee_\din$ to $X^\vee$ by zero on cells of $X^\vee-X^\vee_\din$.
Define two degree $1$ maps
\be\begin{array}{lllllllll}\phi: & C^k(X_\dout) &\ra & C^{k+1}(X,X_\dout)  &,\qquad&
\phi^\vee: & C^k(X^\vee_\din) &\ra & C^{k+1}(X^\vee,X^\vee_\din)\\
& a_\dout &\mapsto  & d \til {a_\dout} -\til{da_\dout}   & & & b &\mapsto  & d \til{b_\din} -\til{db_\din} \end{array}
\label{phi def}\ee
Note that $\phi,\phi^\vee$ are chain maps:
$$d\phi+\phi d=0,\quad  d\phi^\vee+\phi^\vee d=0$$
and induce on the level of cohomology the standard homomorphisms $\phi_*: H^\bt(M_\dout)\ra H^{\bt+1}(M,M_\dout)$, $(\phi^\vee)_*: H^\bt(M_\din)\ra H^{\bt+1}(M,M_\din)$ -- connecting homomorphisms in the two long exact sequences of cohomology of pairs $(M,M_\dout)$ and $(M,M_\din)$.

Next, we have
$$\langle b, \phi(a_\dout) \rangle = \langle b|_\dout, a_\dout \rangle_\dout, \quad \langle \phi^\vee (b_\din), a\rangle = -(-1)^{n+\deg b}\langle b_\din, a|_\din \rangle_\din  $$
--- both right hand sides are sums of intersections in cells adjacent to the boundary and result in boundary terms on the left.\footnote{We are using the natural notations $a|_\din, a|_\dout$ for the components of the image of a cochain $a$ under restriction $\iota^*:C^\bt(X)\ra C^{\bt}(X_\din)\oplus C^\bt(X_\dout)$, and likewise $b|_\din,b|_\dout$ are the components of the image of $b$ under $(\iota^\vee)^*:C^\bt(X^\vee)\ra C^\bt(X_\din^\vee)\oplus C^\bt(X_\dout^\vee)$.}

To prove (\ref{Stokes}), we calculate
\begin{multline*}
(-1)^{n+\deg b+1}\langle db, a\rangle= (-1)^{n+\deg b+1}\langle d(b-\til{b|_\din})+\phi^\vee(b|_\din) \;,\; a- \til{a|_\dout} \rangle= \\ =
\langle b-\til{b|_\din} \;,\; d(a- \til{a|_\dout} )\rangle_\mr{int}+ \langle b|_\din , a|_\din \rangle_\din \\
= \langle b,da \rangle-\langle b, \phi(a|_\dout)\rangle + \langle b|_\din , a|_\din\rangle_\din \\
=\langle b,da \rangle- \langle b|_\dout  , a|_\dout \rangle_\dout+ \langle b|_\din , a|_\din\rangle_\din
\end{multline*}
We put the subscript ``int'' for the non-degenerate (inverse) intersection pairing (\ref{nondeg pairing cochains}) for which the respective coboundary maps $d_X$ and $d_{X^{\vee}}$ are mutually adjoint (up to a sign).  
\end{proof}

\begin{proof}[Proof of Proposition \ref{lemma: i_Q=dS+alpha_bdry}] Indeed, let us calculate the differential of the action (\ref{S with bdry term}) using cellular Stokes' theorem (\ref{Stokes}):
\begin{multline*}
\delta S= \\ =
\langle \delta B,dA \rangle - \underbrace{\langle B,d\delta A \rangle}_{-\langle dB,\delta A \rangle+ \langle B|_\dout,\delta A|_\dout \rangle_\dout-\langle B|_\din,\delta A|_\din\rangle_\din} + 
\langle \delta B|_\din,A|_\din \rangle_\din - \langle B|_\din,\delta A|_\din\rangle_\din   \\[1em] 
=\underbrace{\langle\delta B,dA\rangle+\langle dB,\delta A\rangle}_{\iota_Q \omega}\underbrace{-\langle B|_\dout,\delta A|_\dout \rangle_\dout+\langle \delta B|_\din,A_\din \rangle_\din}_{-\pi^*\alpha_\dd}
\end{multline*}
\end{proof}

\begin{remark} The boundary term in the action (\ref{S with bdry term}) was introduced so that equation (\ref{i_Q=dS+alpha_bdry}) is satisfied for the boundary primitive 1-form (\ref{alpha_bdry}). The latter is chosen so as to agree with $A$-polarization for the out-boundary and $B$-polarization for the in-boundary, which we are going to use in quantization of the model.
\end{remark}

\subsection{Euler-Lagrange spaces, reduction}
The Euler-Lagrange subspaces in $\F$, $\F_\dd$ are defined as zero-loci of $Q$, $Q_\dd$ respectively:
$$\EL= C^\bt_\mr{closed}(X)[1]\oplus C^\bt_\mr{closed}(X^{\vee})[n-2],\qquad
\EL_\dd= C^\bt_\mr{closed}(X_\dd)[1]\oplus C^\bt_\mr{closed}(X^{\vee}_\dd)[n-2]
$$
The respective EL moduli spaces ($Q$-reduced zero-loci of $Q$) are independent of the cellular decomposition:
$$
\begin{CD}
\EL/Q=H^\bt(M)[1]\oplus H^\bt(M)[n-2] \\
@V\pi_*VV \\
\EL_\dd/Q_\dd=\underline\EL_\dd=H^\bt(\dd M)[1]\oplus H^\bt(\dd M)[n-2]
\end{CD}
$$
The boundary moduli space inherits a (non-degenerate, degree $0$) symplectic structure $\underline\omega_\dd$, and the bulk moduli space inherits a degree $+1$ Poisson structure (cf. \cite{CMR}), with symplectic foliation given by fibers of $\pi_*$, which are isomorphic to
$$(\pi^*)^{-1}\{0\}=\frac{H^\bt(M,\dd M)}{H^{\bt-1}(\dd M)}[1]\oplus \frac{H^\bt(M,\dd M)}{H^{\bt-1}(\dd M)}[n-2]$$
Here quotients are over the image of the connecting homomorphism in the long exact sequence in cohomology of the pair $(M,\dd M)$.
Image of $\pi_*$ is a Lagrangian subspace of $\underline\EL_\dd$. The Hamilton-Jacobi action $S|_{\EL/Q}$ on the bulk moduli space is identically zero.

We refer the reader to  \cite{CMR} for generalities on Euler-Lagrange moduli spaces in the BV-BFV framework.

\subsection{Classical ``$A$-$B$'' gluing}\footnote{``$A$-$B$'' means that we stay in the setting of cobordisms and only allow attaching out-boundary (or ``$A$-boundary'', for the polarization we are going to put on it in the quantization procedure to follow) of one cobordism to in- (or ``$B$''-) boundary of the next one.}
\label{sec: class gluing}
Let $M$ be an $n$-dimensional cobordism from $M_1$ to $M_3$, cut by a codimension $1$ submanifold $M_2$ into cobordisms $M_1\cob{M_I} M_2$ and $M_2\cob{M_{II}}M_3$ (we use Roman numerals for $n$-manifolds and arabic numerals for $(n-1)$-manifolds). 

Let $X$ be a cellular decomposition of $M$ for which $M_2\cap X$ is a CW-subcomplex. Thus we have cellular decompositions $X_{1,2,3}, X_{I,II}$ of $M_{1,2,3}$ and $M_{I,II}$, respectively.
As usual, we assume that $X_I$ is of product type near $M_2$ and $X_{II}$ is of product type near $M_3$.
We also have Poincar\'e dual decompositions $X_{1,2,3}^\vee$ of $M_{1,2,3}$ and $X_I^\vee,X_{II}^\vee$ of $\til M_I$, $\til M_{II}$ - displaced versions of $M_{I,II}$ (cf. Section \ref{sec: Lefschetz for cell decomp, cobordism}). On the level of CW-complexes, we have both $X=X_I\cup_{X_2} X_{II}$ and $X^\vee=X^\vee_I\cup_{X_2^\vee} X_{II}^\vee$.

The space of fields associated to $(M,X)$ is expressed in terms of spaces of fields for $(M_I,X_I)$ and $(M_{II},X_{II})$ as
$$\F=\F_I\times_{\F_2} \F_{II}$$
-- the fiber product w.r.t. the projections $\F_I\stackrel{\pi_{I,2}}{\twoheadrightarrow}\F_2 \stackrel{\pi_{II,2}}{\twoheadleftarrow}\F_{II}$ -- ``out-part'' of projection (\ref{pi}) for $(M_I,X_I)$ and ``in-part'' of projection (\ref{pi}) for $(M_{II},X_{II})$, respectively. Recalling that we also have projections to boundary fields associated to $(M_{1,3},X_{1,3})$, we have the following diagram.
$$
\begin{CD}
\F @>>> \F_{II} @>\pi_{II,3}>> \F_3 \\
@VVV @V\pi_{II,2}VV \\
\F_{I} @>\pi_{I,2}>> \F_2 \\
@V\pi_{I,1}VV \\
\F_1
\end{CD}
$$

The presymplectic BV 2-form on $\F$ is recovered as the sum of pullbacks of presymplectic forms on $\F_I$ and $\F_{II}$.

For the action, we have
$$S=S_I+S_{II}-\langle B_2,A_2 \rangle$$
where the third term, associated to the gluing interface $(M_2,X_2)$, compensates for the boundary term in $S_{II}$. 

\section{Quantization in $A/B$-polarization}
\label{sec: quantum cell ab BF on mfd with bdry}
We choose the following Lagrangian fibration of the space of boundary fields:
\be\label{polarization}
\begin{CD}
\F_\dd=C^\bt(X_\dd)[1]\oplus C^\bt(X^{\vee}_\dd)[n-2] \\
@VpVV  \\
\B_\dd=C^\bt(X_\dout )[1]\oplus C^\bt(X_\din^\vee)[n-2]
\end{CD}
\ee
Notation $\B_\dd$ comes from ``base'' of the fibration. Pre-composing with $\pi: \F\ra \F_\dd$, we get the projection $p\circ \pi: \F\ra \B_\dd$. The presymplectic structure $\omega$ restricts to a {\it symplectic} (non-degenerate) structure on the fibers of $p\circ\pi$ in $\F$. Thus, for $b=(A_\dout ,B_\din)\in \B_\dd$, the fiber
\be \F_b=\pi^{-1}p^{-1}\{b\}\simeq C^\bt(X,X_\dout )[1]\oplus  C^\bt(X^\vee,X^\vee_\din)[n-2] \label{F_b}\ee
carries the degree $-1$ symplectic structure $\omega_b=\omega|_{\F_b}=\langle \delta B_{X^\vee-X^\vee_\din},\delta A_{X-X_\dout}\rangle_\mr{int}$ and  the BV Laplacian
$$\Delta_\mr{bulk}=\left\langle \frac{\dd}{\dd A_{X-X_\dout }} \;,\; \frac{\dd}{\dd B_{X^{\vee}-X^\vee_\din}}\right\rangle_\mr{int}:\quad  \Fun(\F_b)_k\ra \Fun(\F_b)_{k+1}$$
satisfying
$$\Delta_\mr{bulk}^2=0$$
As in the proof of Lemma \ref{lemma: Stokes}, we are emphasizing the non-degenerate intersection pairing, and the inverse one, with subscript ``int''.

Geometric (canonical) quantization of the
space of boundary fields $\F_\dd$ (with symplectic structure $\omega_\dd$ and the trivial prequantum line bundle with global connection $1$-form $\frac{i}{\hbar}\alpha_\dd$)
w.r.t. the real polarization given by the vertical tangent bundle  to the fibration (\ref{polarization}),
yields the {\it space of states}
\be\HH_\dd=\bC\otimes \Fun(\B_\dd) 
\label{H}\ee
associated to the boundary. 
\begin{remark} The splitting of the space of boundary fields $\F_\dd$ into contributions of in- and out-boundary induces a splitting of the space of states as
\be \HH_\dd=\underbrace{\HH_\din^{(B)}}_{\Fun_\bC(C^\bt(X^\vee_\din)[n-2])}\widehat\otimes \underbrace{\HH_\dout^{(A)}}_{\Fun_\bC(C^\bt(X_\dout)[1])} \label{H = H_in otimes H_out}\ee
where the superscripts $(A)$, $(B)$ stand for the respective polarizations (``$A$ fixed'' on the out-boundary and ``$B$ fixed'' on the in-boundary). Subscript $\bC$ corresponds to taking complex-valued functions.
\end{remark}

\begin{remark}\label{rem: pairing H^A with H^B}
For $N$ a closed $(n-1)$-manifold with a cellular decomposition $Y$, one can introduce a pairing between the spaces of states corresponding to $A$- and $B$-polarizations on $Y$, $\HH^{(A)}_Y\otimes \HH^{(B)}_{\bar Y}\ra \bC$ given by
\be (\phi(A_Y),\psi(B_Y))=\int_{\B^{(A)}_Y\times \B^{(B)}_Y} \phi(A_Y)\; \left(\D_\hbar A_Y\; e^{-\frac{i}{\hbar}\langle B_Y,A_Y \rangle}\; \D_\hbar B_Y\right)\; \psi(B_Y)\label{pairing H^A with H^B}\ee
for a pair of states $\phi\in \HH^{(A)}_Y=\Fun_\bC(\B^{(A)}_Y)$, $\psi\in \HH^{(B)}_{\bar Y}=\Fun_\bC(\B^{(B)}_Y)$, where $\B^{A}_Y=C^\bt(Y)[1]$, $\B^{(B)}_Y= C^\bt(Y^\vee)[n-2]$ are the bases of $A$- and $B$-polarizations on the space $\F_Y$, respectively.
Normalized densities $\D_\hbar A_Y$, $\D_\hbar B_Y$ are defined as products over cells of $Y$ or $Y^\vee$ of respective normalized densities, cf. (\ref{normalized density}), (\ref{D_h B}).
We use the bar to denote the orientation reversal\footnote{
We also understand that the orientation reversing identity map $Y\ra \bar Y$ acts on states by complex conjugation $\phi(A_Y)\mapsto \overline{\phi(A_Y)}$, $\psi(B_Y)\mapsto \overline{\psi(B_Y)}$. In particular, we have a sesquilinear pairing $\HH^{(A)}_Y\otimes \HH^{(B)}_Y\ra \bC$ (note that here $Y$ has the same orientation in both factors), given by the formula (\ref{pairing H^A with H^B}) with $\psi$ replaced by the complex conjugate $\bar\psi$.
}; in these notations, the spaces appearing on the r.h.s. of (\ref{H = H_in otimes H_out}) are $\HH^{(B)}_\din=\HH^{(B)}_{\bar X_\din}$
and $\HH^{(A)}_\dout=\HH^{(A)}_{X_\dout}$.\footnote{
Recall that $\bar{X}_\din$ is the CW complex $X_\din$ endowed with the orientation induced from the orientation of $X$.
}
The pairing (\ref{pairing H^A with H^B}) will play a role when we discuss the behavior of partition functions under gluing of cobordisms (Section \ref{sec: quantum gluing}, Proposition \ref{prop: gluing}) (with $N$ being the gluing interface between two cobordisms and $Y$ its cellular decomposition).
Using the pairing (\ref{pairing H^A with H^B}) for $N=M_\din$, $Y=X_\din$, the space of states associated to the boundary of a cobordism (\ref{H = H_in otimes H_out}) can be identified with the $\mr{Hom}$-space
\begin{equation}\label{HH as Hom-space}
\HH_\dd=\mr{Hom}(\HH^{(A)}_\din,\HH^{(A)}_\dout)
\end{equation}
\end{remark}
Quantization of boundary BFV action (\ref{BFV action}) yields
$$\widehat S_\dd= -i\hbar\;Q_{\B_\dd}\quad \in \mr{End}(\HH_\dd)_1,\qquad\mbox{where}\quad Q_{\B_\dd}=\langle dA_\dout ,\frac{\dd}{\dd A_\dout }\rangle+
\langle dB_\din,\frac{\dd}{\dd B_\din}\rangle\quad \in \mathfrak{X}(\B_\dd)_1$$
This odd, degree $+1$ operator on the space of states satisfies
$$\widehat S_\dd^2=0$$
Thus, $(\HH,\widehat S_\dd)$ is a cochain complex.


\begin{lemma}\label{lemma 7.3}
The action satisfies the following version of the quantum master equation modified by the boundary term:
\be \left(\frac{i}{\hbar}\widehat S_\dd-i\hbar\; \Delta_\mr{bulk}\right)\;e^{\frac{i}{\hbar}S(A,B)}=0 \label{QME with bdry term}\ee
where the exponential of the action is regarded as an element of $\Fun_\bC(\F)\cong\HH_\dd\widehat\otimes\, \Fun(\F_b)$. (Here we are exploiting the fact that 
all fibers $\F_b$ for different $b$ are isomorphic.)
\end{lemma}

\begin{proof}\footnote{This Lemma follows from the general treatment in \cite{CMRpert}, Section 2.4.1. For reader's convenience, we give an adapted proof here.}
First calculate the Poisson bracket (corresponding to the fiber symplectic structure $\omega_b$ for some $b\in\B_\dd$) of the action with itself:
\begin{multline*}
\frac{1}{2}\{S,S\}_{\omega_b}=\langle dB-\til{dB} \;,\; dA-\til{dA} \rangle_\mr{int}
=\\
=\langle dB , dA\rangle
=\underbrace{-\langle B,d^2A \rangle}_0+\langle B|_\dout,d A|_\dout\rangle_\dout-\langle B|_\din, dA|_\din\rangle_\din=\pi^*S_\dd
\end{multline*}
Consider the splitting of superfields into $\B_\dd$-components and ``bulk'' components (corresponding to fibers of $p\circ \pi:\F\ra \B_\dd$):
$$A=\til{A|_\dout}+\underbrace{(A-\til{A|_\dout})}_{A_\bulk},\qquad B=\til{B|_\din}+\underbrace{(B-\til{B|_\din})}_{B_\bulk}$$
Substituting this splitting into the action, we have
$$S(A,B)=\langle B_\bulk, \phi(A|_\dout)+d A_\bulk\rangle+ \langle B|_\din, A_\bulk|_\din \rangle_\din$$
Next, calculate
$$Q_{\B_\dd} S=-\langle B_\bulk|_\dout, dA|_\dout\rangle_\dout+\langle B|_\din, dA_\bulk|_\din \rangle_\din=-\pi^* S_\dd$$
Therefore, we have
$$\left(\frac{i}{\hbar}\widehat S_\dd-i\hbar\Delta_\bulk\right)e^{\frac{i}{\hbar}S(A,B)}=\frac{i}{\hbar}\left(\underbrace{Q_{\B_\dd} S}_{-\pi^* S_\dd}+\underbrace{\frac{1}{2}\{S,S\}_{\omega_b}}_{\pi^* S_\dd}\underbrace{-i\hbar \Delta_\bulk S}_0\right)\cdot e^{\frac{i}{\hbar}S(A,B)}=0$$
\end{proof}

Denote $\mu_{\B_\dd}=\D \Bin\cdot \D \Aout\in \Dens\; \B_\dd$ the density on $\B_\dd$ associated to the cellular basis.
The corresponding normalized density is
$$\mu_{\B_\dd}^\hbar=\D_\hbar \Bin\cdot \D_\hbar \Aout=\prod_{e^\vee_\din\subset X^\vee_\din}\D_\hbar B_{e^\vee_\din}\cdot \prod_{e_\dout\subset X_\dout}\D_\hbar A_{e_\dout}$$
with $\D_\hbar B_{e^\vee}$, $\D_\hbar A_e$ defined as in Section \ref{sec: normalization}.

Denoting by $\mu_\mr{bulk}^{1/2}\in \Dens^{\frac12}\F_b$ the half-density on $\F_b$ associated to the cellular basis, 
by $\mu^{1/2}=\mu_{\B_\dd}^{1/2}
\cdot \mu_\mr{bulk}^{1/2}$ the half-density on $\F$,
and by $\mu_\hbar^{1/2}=(\mu_{\B_\dd}^\hbar)^{1/2}\cdot (\mu_\mr{bulk}^\hbar)^{1/2}$ the corresponding normalized half-density,
we can rewrite (\ref{QME with bdry term}) as
\be \left(\frac{i}{\hbar}\widehat S_\dd^\mr{can}-i\hbar\; \Delta_\mr{bulk}^\mr{can}\right)\;\underbrace{\left(e^{\frac{i}{\hbar}S(A,B)}\mu^{1/2}_\hbar\right)}_{\in \HH^\mr{can}_\dd\widehat\otimes \Dens^{\frac12,\Fun}_\bC(\F_b)}=0\label{QME with bdry term, half-densities}\ee
Here $\widehat S_\dd^\mr{can}$, $\Delta_\mr{bulk}^\mr{can}$ are the half-density (``canonical'') versions of the quantum BFV action and of the BV Laplacian, defined by
$$\widehat S_\dd^\mr{can}\left(\phi\cdot (\mu_\B^\hbar)^{1/2}\right)=\widehat S_\dd\left(\phi\right)\cdot (\mu_\B^\hbar)^{1/2},\qquad \Delta_\mr{bulk}^\mr{can}\left(f\cdot (\mu_\bulk^\hbar)^{1/2}\right)=\Delta_\mr{bulk}\left(f\right)\cdot (\mu_\bulk^\hbar)^{1/2}$$
for $\phi\in \HH$ and $f\in \Fun(\F_b)$.
We denote by
$$\HH^\mr{can}_\dd=\Dens^{\frac12,\Fun}_\bC(\B_\dd)$$
the version of the space of states where states are regarded as half-densities on $\B_\dd$; 
the superscript $\Fun$ means that we allow half-densities to be depend on 
coordinates (i.e. $\Fun$ stands for tensoring constant half-densities with functions), the subscript $\bC$ stands for tensoring the space of half-densities with complex numbers.

\subsection{Bulk gauge fixing}\label{sec 7.1 gauge-fixing}
Let us choose a realization of relative cohomology
$\ii: H^\bt(M,M_\dout)\ra C^\bt_\mr{closed}(X,X_\dout )$
and a section $\K$ of $C^\bt(X,X_\dout )\xra{d} C^{\bt+1}_\mr{exact}(X,X_\dout )$. We have the Hodge decomposition
\be C^\bt(X,X_\dout )=\mr{im}(\ii)\oplus C^\bt_\mr{exact}(X,X_\dout )\oplus \mr{im}(\K) \label{Hodge decomp}\ee
We extend $\K$ by zero on the first two terms on the r.h.s. to a map $\K: C^\bt(X,X_\dout )\ra C^{\bt-1}(X,X_\dout )$. Denote by $\pp: C^\bt(X,X_\dout )\ra H^\bt(X,X_\dout )$ the projection to cohomology arising from the decomposition (\ref{Hodge decomp}).
Using Poincar\'e-Lefschetz duality for cohomology and cochains, we construct the dual (transpose) maps
\begin{align*}
\ii^\vee=\pp^*:&\quad H^\bt(M,M_\din)\ra C^\bt_\mr{closed}(X^\vee,X^\vee_\din), \\
\pp^\vee=\ii^*:&\quad  C^\bt(X^\vee,X^\vee_\din)\ra H^\bt(M,M_\din), \\
\K^\vee_k=(-1)^{n-k} (\K_{n-k+1})^*:&\quad  C^k(X^\vee,X^\vee_\din)\ra C^{k-1}(X^\vee,X^\vee_\din)
\end{align*}
Then we have the dual Hodge decomposition for cochains of the dual complex:
\be C^\bt(X^\vee,X^\vee_\din)=\mr{im}(\ii^\vee)\oplus C^\bt_\mr{exact}(X^\vee,X^\vee_\din)\oplus \mr{im}(\K^\vee) \label{Hodge decomp dual}\ee

In other words, we chose some induction data $C^\bt(X,X_\dout)\wavy{(\ii,\pp,\K)} H^\bt(M,M_\dout)$ and inferred the dual one $C^\bt(X^\vee,X^\vee_\din)\wavy{(\ii^\vee,\pp^\vee,\K^\vee)} H^\bt(M,M_\din)$ using the construction (\ref{dual ind data}).

Hodge decompositions (\ref{Hodge decomp},\ref{Hodge decomp dual}) together
yield the symplectic splitting
$$\F_b=(\ii\oplus\ii^\vee)
\F^\zm_b
\oplus \F_\fluct$$
where
\be \F^\zm_b=H^\bt(M,M_\dout)[1]\oplus H^\bt(M,M_\din)[n-2]\label{bulk zero-modes}\ee
is our choice for the space of bulk {\zeromodes},
and we have a Lagrangian subspace $\LL=\mr{im}(\K)[1]\oplus \mr{im}(\K^\vee)[n-2]\subset \F_\fluct$.

\subsection{Perturbative partition function: integrating out bulk fields}
\label{sec 7.2 Z pert}

Substituting in the action (\ref{S with bdry term}) the decomposition of fields into coordinates on $\B_\dd$, bulk {\zeromodes} and fluctuations, we have
\begin{multline*}
S(A,B)=S(\til{\Aout} +\Azm+\Afluct \;,\; \til{ \Bin}+\Bzm+\Bfluct) = \\
=\langle \Bzm, \phi(\Aout) \rangle - \langle \phi^\vee (\Bin), \Azm  \rangle+
\langle \Bfluct, \phi(\Aout) \rangle - \\
- \langle \phi^\vee (\Bin), \Afluct  \rangle+
\langle \Bfluct, d\Afluct\rangle= \\
=\langle \left.\Bzm\right|_\dout, \Aout \rangle_\dout + \langle \Bin, \left.\Azm\right|_\din  \rangle_\din + \\ +
\left\langle \Bfluct+\K^\vee\phi^\vee(\Bin) \;,\;d(\Afluct+\K\phi(\Aout))  \right\rangle-\langle \phi^\vee(\Bin)\;,\; \K\phi(\Aout)\rangle
\end{multline*}
We are suppressing the inclusions $\ii,\ii^\vee$ in bulk \zeromodes\  $\ii\Azm,\ii^\vee\Bzm$ in the notation.

The fiber BV integral over bulk fields yields
\begin{multline}
Z(\underbrace{A_\dout ,B_\din}_b; \Azm,\Bzm)=
\\
=\int_{\LL\subset\F_\fluct\subset  \F_b} e^{\frac{i}{\hbar}S(\til{A_\dout} +\Azm+\Afluct \;,\; \til{ B_\din}+\Bzm+\Bfluct)}\mu^{1/2}_\hbar = \\
= e^{\frac{i}{\hbar} \langle  \left.\Bzm\right|_\dout, A_\dout  \rangle_\dout +
 \langle B_\din, \left. \Azm\right|_\din\rangle_\din - \langle \phi^\vee(B_\din) \;,\; \K \phi(A_\dout)  \rangle
}\cdot \\
\cdot \left(\int_\LL e^{\frac{i}{\hbar}S(\Afluct+\K\phi(\Aout),\Bfluct+\K^\vee\phi^\vee(\Bin))
} (\mu_\bulk^\hbar)^{1/2}\right) (\mu_\B^\hbar)^{1/2}
\label{Z integrating out bulk fields}
\end{multline}
\begin{Proposition}\label{prop: Z}
\begin{enumerate}[(i)]
\item \label{prop: Z (i)} Explicitly, the partition function (\ref{Z integrating out bulk fields}) is
\begin{multline} Z(\Bin,\Aout;\Azm,\Bzm)= \\ =
e^{\frac{i}{\hbar}  \langle \left. \Bzm \right|_\dout, A_\dout  \rangle_\dout +
\langle B_\din, \left. \Azm \right|_\din\rangle_\din
- \langle \phi^\vee(B_\din) \;,\; \K \phi(A_\dout)  \rangle}
\;\xi_\hbar^{H^\bt(M,M_\dout)}\cdot\underbrace{\tau(M,M_\dout)}_{\in \Dens^{\frac12}\F_b^\zm}\cdot(\mu_{\B_\dd}^\hbar)^{1/2}       
\\
\in \HH^\mr{can}_\dd\widehat\otimes \Dens^{\frac12,\Fun}_\bC(\F_b^\zm)
\label{Z integrating out bulk fields result}
\end{multline}
where $\tau(M,M_\dout )\in \Det\; H^\bt(M,M_\dout)/\{\pm 1\}$ is the R-torsion and normalization factor is $\xi_\hbar^{H^\bt(M,M_\dout)}=\prod_{k=0}^n (\xi_\hbar^k)^{\dim H^k(M,M_\dout)}$ with $\xi_\hbar^k$ as in (\ref{xi_hbar^k}).
\item \label{prop: Z (ii)} The partition function satisfies the quantum master equation
\be\left(\frac{i}{\hbar}\widehat S_\dd^\mr{can}-i\hbar\; \Delta_\zm^\mr{can}\right) Z(A_\dout,B_\din;\Azm,\Bzm) =0\label{QME with bdry term for Z}\ee
where $\Delta_\zm=\left\langle \frac{\dd}{\dd \Azm} , \frac{\dd}{\dd \Bzm} \right\rangle$ is the BV Laplacian on $\Fun(\F_b^\zm)$ and $\Delta_\zm^\mr{can}$ is the corresponding BV Laplacian on half-densities.
\item \label{prop: Z (iii)} Deformation of the induction data $C^\bt(X,X_\dout)\wavy{(\ii,\pp,\K)} H^\bt(M,M_\dout)$ induces a 
transformation of the partition function of the form
    $$Z \;\mapsto\; Z+ \left(\frac{i}{\hbar}\widehat S_\dd^\mr{can}-i\hbar\; \Delta_\zm^\mr{can}\right) (\cdots) $$
\end{enumerate}
\end{Proposition}

\begin{proof} Part (\ref{prop: Z (i)}) is a straightforward computation of the Gaussian integral over fluctuations in (\ref{Z integrating out bulk fields}).

Part (\ref{prop: Z (ii)}) is a consequence of (\ref{QME with bdry term, half-densities}) and the BV Stokes' theorem: one has
\begin{multline*}
\left(\frac{i}{\hbar}\widehat S_\dd^\mr{can}-i\hbar\; \Delta_\zm^\mr{can}\right) Z = \int_{\LL\subset\F_\fluct} \left(\frac{i}{\hbar}\widehat S_\dd^\mr{can}-i\hbar\; \Delta^\mr{can}_\bulk\right)\left(e^{\frac{i}{\hbar}S(A,B)}\mu_\hbar^{1/2}\right)+\\
+\int_{\LL\subset\F_\fluct} i\hbar\;\Delta^\can_\fluct \left(e^{\frac{i}{\hbar}S(A,B)}\mu_\hbar^{1/2}\right) =0
\end{multline*}
with $\Delta^\can_\fluct=\left\langle \frac{\dd}{\dd A_\fluct} , \frac{\dd}{\dd B_\fluct} \right\rangle$ the BV Laplacian on coordinate-dependent half-densities on the space of bulk fluctuations. 

For part (\ref{prop: Z (iii)}), it suffices to consider a general {\it infinitesimal} deformation of the induction data, given as a sum of deformations of the three types  (\ref{ind data deformation I},\ref{ind data deformation II},\ref{ind data deformation III}), cf. Section \ref{sec: deformations of ind data}, with
$\Lambda: C^\bt_\mr{exact}(X,X_\dout)\ra \mr{im}(\K)^{\bt-2}$,  $I: H^\bt(M,M_\dout)\ra \mr{im}(\K)^{\bt-1}$, $P:C^\bt_\mr{exact}(X,X_\dout)\ra H^{\bt-1}(M,M_\dout)$ the corresponding generators. Using the explicit formula for the partition function (\ref{Z integrating out bulk fields result}), one can check directly that the effect of such a deformation on $Z$ is given by
$$
Z\;\mapsto Z +\left(\frac{i}{\hbar}\widehat S_\dd^\mr{can}-i\hbar\; \Delta_\zm^\mr{can}\right) R
$$
with
$$
R= Z\cdot \left(-\langle \left. I^\vee(\Bzm) \right|_\dout,A_\dout  \rangle_\dout +
\langle B_\din, \left. I(\Azm) \right|_\din\rangle_\din-\langle \phi^\vee(B_\din) \;,\; \til\Lambda \phi(A_\dout)  \rangle \right)
$$
Here $\til\Lambda=\Lambda d K: C^\bt(X,X_\dout)\ra C^{\bt-2}(X,X_\dout)$ is the extension of $\Lambda$ from exact to all cochains, by zero on the first and third terms of (\ref{Hodge decomp}); the map $I^\vee:H^\bt(M,M_\din)\ra \mr{im}(\K)^{\bt-1}$ is defined as $(I^\vee)_k=(-1)^{n-k} (P_{n-k+1})^*$.
\end{proof}

\begin{remark}
Recall that, by Poincar\'e duality for torsions \cite{Milnor62}, one can relate the relative and absolute torsions as $\tau(M,M_\dout)=\tau(M,M_\din)^{(-1)^{n-1}}$. In particular, a more symmetric way to write the evaluation of the Gaussian integral over the fluctuation part of fields in (\ref{Z integrating out bulk fields}) is:
\begin{multline*}
(\tau(M,M_\dout)\cdot \tau(M,M_\din)^{(-1)^{n-1}})^{\frac12}\;\;\in \\
\in\;\; \left(\Det\; H^\bt(M,M_\dout)\otimes (\Det\; H^\bt(M,M_\din))^{(-1)^{n-1}}\right)^{\otimes \frac{1}{2}}/\{\pm 1\}
\end{multline*}
\end{remark}


\begin{remark} In the case $H^\bt(M,M_\dout)=0$ (and hence $H^\bt(M,M_\din)=0$, too), formula (\ref{Z integrating out bulk fields}) simplifies to
$$Z(A_\dout,B_\din)=e^{-\frac{i}{\hbar}\langle \phi^\vee(B_\din) \;,\; \K \phi(A_\dout)  \rangle}  \cdot \underbrace{\tau(M,M_\dout)}_{\in \bR/\{\pm 1\}}\cdot(\mu_{\B_\dd}^\hbar)^{1/2}$$
The $R$-torsion in this formula is a nonzero real number defined modulo sign. In the case $M_\din=\varnothing$, we also have a simplification:
$$Z(\Aout; \Azm,\Bzm)=
e^{\frac{i}{\hbar} \langle \left.\Bzm\right|_\dout, \Aout  \rangle_\dout }
\cdot\xi_\hbar^{H^\bt(M,M_\dout)}\tau(M,M_\dout)\cdot(\mu_{\B_\dd}^\hbar)^{1/2}
$$
and similarly in case $M_\dout=\varnothing$.
\end{remark}

\begin{remark}\label{rem: propagator cob}
The propagator can be introduced as the parametrix  $\KK\in C^{n-1}(X\times X^\vee; p_1^* E_X\otimes p_2^* E^*_{X^\vee})$ for $\K$, exactly as in Section \ref{rem: propagator, closed case}. Equations (\ref{propagator},\ref{propagator equations}) hold without changes (with the correction that $\chi_\alpha$, $\chi^\vee_\alpha$ are now the representatives of relative cohomology $H^\bt(M,M_\dout)$, $H^{n-\bt}(M,M_\din)$). Now the propagator satisfies additionally the boundary conditions
$$\KK(e,e^\vee)=0\quad\mbox{ if }\; e \in X_\dout \;\mbox{ or}\;\; e^\vee\in X_\din^\vee $$
\end{remark}

\subsection{Gluing}
\label{sec: quantum gluing}
Consider the situation of Section \ref{sec: class gluing}, i.e. a glued cobordism
\be M_1\cob{M}M_3\quad =\quad M_1\cob{M_I}M_2\cob{M_{II}}M_3 \label{glued cobordism}\ee
with a glued cellular decomposition
$$X_1\cob{X}X_3\quad =\quad X_1\cob{X_I}X_2\cob{X_{II}}X_3$$ (and the dual one).
Our goal of this Section is obtain an Atiyah-Segal-type gluing formula, expressing the partition function for $(M,X)$ in terms of partition functions for $(M_I,X_I)$ and $(M_{II},X_{II})$.

The zeroth approximation to the expected formula is:
$$Z(B_1,A_3)=\int_{\B_2^{(A)}\times \B_2^{(B)}} Z_I(B_1,A_2) \left[(\D_\hbar A_2)^{1/2}\cdot e^{-\frac{i}{\hbar}\langle B_2,A_2 \rangle} \cdot (\D_\hbar B_2)^{1/2}\right] Z_{II}(B_2,A_3) $$
where the integral is taken over the space of leaves (base) $\B_2^{(A)}$ of $A$-polarization on the interface $(M_2,X_2)$, parameterized by $A_2$, and over the space of leaves $\B_2^{B}$ of $B$-polarization on the interface, parameterized by $B_2$. In the formula above we are ignoring (for the moment) the issue that spaces of 
residual fields 
on left and right hand sides are generally different.

Consider the expression\footnote{
We are putting asterisks on boundary conditions $A_2^*$, $B_2^*$ at $X_2$ to distinguish them from the components $A_2^{II},B_2^I$ of bulk fields -- coordinates on fibers of $\F_{II}\ra \B_2^{(B)}$ and $\F_{I}\ra \B_2^{(A)}$, respectively. In other words, in (\ref{S_I+S_II-(B,A)}) we are counting each cell of $X_2$, $X_2^\vee$ twice: once as a boundary condition and once as a part of bulk fields.}
\be S(\underbrace{A_I}_{\mr{s.t.\;} A_I|_{X_2}=A_2^*},B_I) \;+\; S(A_{II},\underbrace{B_{II}}_{\mr{s.t.\;} B_{II}|_{X_2^\vee}=B_2^*})-\langle B_2^*,A_2^* \rangle \label{S_I+S_II-(B,A)}\ee
The second term here  
contains a boundary term $\langle B_2^*,A_2^{II}\rangle$ and no other terms dependent on $B_2^*$. Therefore, integrating out $B_2^*$, we impose the constraint $A_2^*=A_2^{II}$.
Integrating out both $A_2^*$ and $B_2^*$, we obtain the action on the whole (glued) cobordism, 
i.e.
\begin{multline} 
e^{\frac{i}{\hbar}S(A,B)} =\\
=\int_{\B_2^{(A)}\times \B_2^{(B)}} \D_\hbar A_2^*\; \D_\hbar B_2^*\;\; e^{-\frac{i}{\hbar}\langle B_2^*,A_2^* \rangle}\cdot e^{\frac{i}{\hbar} S(\til A_2^*+p_I^\rel(A),p_I(B))}\cdot e^{\frac{i}{\hbar}S( p_{II}(A), \til B_2^*+p_{II}^\rel(B) ) } \label{gluing e^S}
\end{multline}
Here $p_I$, $p_{II}$ is the projection of fields on $X$ to fields on $X_I$, $X_{II}$, respectively (by restriction); $p_{I,II}^\rel$ is the projection to cochains on $X_I$ or $X_{II}$ (resp. the dual complexes)  vanishing on $X_2$, e.g. $p_I^\rel(A)=p_I(A)-\til {A|_{X_2}}\in C^\bt(X_I,X_2)$. Note that normalization of the integration measure coming from conventions of Section \ref{sec: normalization} works 
correctly here.\footnote{\label{footnote: gluing normalization check}Indeed, for a $k$-cell $e\subset X_2$ we have $\int \D_\hbar B_{\varkappa_2(e)}^* \D_\hbar A_e^*\quad e^{-\frac{i}{\hbar} \langle B^*_{\varkappa_2(e)} , A^*_e \rangle_E}=(\underbrace{\xi_\hbar^{n-\dim \varkappa_2(e)} \xi_\hbar^{\dim e}}_{\xi_\hbar^{k+1}\xi_\hbar^k})^{\mr{rk} E}\cdot \left\{ \begin{array}{cl} (2\pi \hbar)^{\mr{rk}E} & \mr{if}\; k \;\mr{odd} \\ \left(\frac{i}{\hbar}\right)^{\mr{rk}E} & \mr{if}\; k \;\mr{even}\end{array} \right.=1$}

If we fix boundary conditions for fields in (\ref{gluing e^S}), $B_1$ on $X_1^\vee$ and $A_3$ on $X_3$, the l.h.s. of (\ref{gluing e^S}) becomes a function on $\F_{B_1,A_3}\simeq \F_{\B_1,0}^I\times \F_{0,A_3}^{II}$ (where subscripts denote the boundary conditions, as in (\ref{F_b}), and $\simeq$ is a symplectomorphism). Using the gauge-fixing on $X_I$, $X_{II}$, we can evaluate the fiber BV integral of (\ref{gluing e^S}), yielding a function of {\it composite} {\zeromodes} $\F^\mr{comp.\;\zm}_{B_1, A_3}=\F^{I\;\zm}_{B_1,0}\times \F^{II\;\zm}_{0,A_3}$:
\begin{multline}\label{Z composite}
Z^\mr{comp.\;\zm}(B_1,A_3;A_\zm^I,A_\zm^{II},B_\zm^I,B_\zm^{II})
=\int_{\B^{(A)}_2\times \B^{(B)}_2} Z_I(B_1,A_2;A_\zm^I,B_\zm^I) \cdot\\
\cdot \left[(\D_\hbar A_2)^{1/2}\cdot e^{-\frac{i}{\hbar}\langle B_2,A_2 \rangle} \cdot (\D_\hbar B_2)^{1/2}\right] \cdot Z_{II}(B_2,A_3;A_\zm^{II},B_\zm^{II})
\end{multline}
The next step is to pass from the composite {\zeromodes} in the expression above to standard bulk {\zeromodes} (\ref{bulk zero-modes}) on $X$.

\subsubsection{Gluing bulk {\zeromodes}}
Consider the cochain complex
\be C^\bt(X,X_3)\simeq C^\bt(X_I,X_2)\oplus  C^\bt(X_{II},X_3) \label{gluing cochains}\ee
Note that this is an isomorphism of (based) graded vector spaces but not of cochain complexes: the differential on the l.h.s. has the block triangular form
\be d_{C^\bt(X,X_3)}=\left(\begin{array}{cc} d_I & \phi_I p_2 \\ 0 & d_{II} \end{array}\right) \label{d upper triangular}\ee
where $p_2: C^\bt(X_{II},X_3)\ra C^\bt(X_2)$ is the restriction to the interface; $\phi_I: C^\bt(X_2)\ra C^{\bt+1}(X_I,X_2)$ is as in (\ref{phi def}). Similarly, for the dual cochain complex, we have
\be C^\bt(X^\vee,X_1^\vee)\simeq C^\bt(X_I^\vee,X_1^\vee)\oplus C^\bt(X_{II}^\vee,X_2^\vee),\qquad
d_{C^\bt(X^\vee,X_1^\vee)}=\left(\begin{array}{cc} d_I^\vee & 0 \\ \phi_{II}^\vee p_2^\vee & d_{II}^\vee \end{array}\right) \label{gluing dual cochains}\ee
with $p_2^\vee:C^\bt(X_I^\vee,X_1^\vee)\ra C^\bt(X_2^\vee) $ and $\phi_{II}^\vee: C^\bt(X_2^\vee)\ra C^{\bt+1}(X_{II}^\vee,X_2^\vee)$.

Returning to decomposition (\ref{gluing cochains}), we would like to view the total differential (\ref{d upper triangular}) as the diagonal part (the standard differential on the r.h.s. of (\ref{gluing cochains})) plus a strictly upper-triangular perturbation:

\be d_{C^\bt(X,X_3)}=\left(\begin{array}{cc} d_I & 0 \\ 0 & d_{II} \end{array}\right)+\left(\begin{array}{cc} 0 & \phi_I p_2 \\ 0 & 0 \end{array}\right) \label{d decomposition}\ee
By homological perturbation lemma (Lemma \ref{lemma: homological perturbation}),
using the direct sum induction data
\be C^\bt(X_I,X_2)\oplus C^\bt(X_{II},X_3),\;d_I\oplus d_{II}\quad \wavy{(\ii_I\oplus \ii_{II},\pp_I\oplus \pp_{II},\K_I\oplus \K_{II})}\quad H^\bt(M_I,M_2)\oplus H^\bt(M_{II},M_3) \label{C_I+C_II retraction}\ee
we 
construct the induced differential 
\be \ddelta=  \left(\begin{array}{cc} 0 & \pp_I (\phi_I p_2) \ii_{II} \\ 0 & 0 \end{array}\right) \label{delta}\ee
on the cohomology $H^\bt(M_I,M_2)\oplus H^\bt(M_{II},M_3)$ of the first term in r.h.s. of (\ref{d decomposition}), such that cohomology of $\ddelta$ is isomorphic to $H^\bt(M,M_3)$. Higher terms in the series for the induced differential vanish due to the special (upper-triangular) form of the perturbation of the differential.

In purely cohomological terms, without resorting to cochains, $\ddelta$ is the composition of the natural map $H^\bt(M_{II},M_3)\ra H^\bt(M_2)$ (pullback by the inclusion $M_2\hra M_{II}$ ) and the connecting homomorphism $H^\bt(M_2)\ra H^{\bt+1}(M_I,M_2)$ in the long exact sequence of the pair $(M_I,M_2)$.

Choose some induction data
\be H^\bt(M_I,M_2)\oplus H^\bt(M_{II},M_3),\;\ddelta \quad \wavy{(i_\gzm,p_\gzm,K_\gzm)} \quad  H^\bt(M,M_3) \label{gzm retraction}\ee
(``$\gzm$'' stands for ``gluing of {\zeromodes}''). 
This, together with the dual induction data (in the sense of Section \ref{sec: induction data for dual and sym}),
\be H^\bt(M_I,M_1)\oplus H^\bt(M_{II},M_2),\;\ddelta^\vee \quad \wavy{(i_\gzm^\vee,p_\gzm^\vee,K_\gzm^\vee)} \quad  H^\bt(M,M_1) \label{gzm dual}\ee
gives a splitting of the composite \zeromodes\  into standard \zeromodes\  plus a complement, and a Lagrangian in the complement:
\begin{multline}
\F^\mr{comp.\;z.m.}_{B_1,A_3}=(i_\gzm\oplus i_\gzm^\vee)\F^\zm_{B_1,A_3}\oplus\\
\oplus
\rlap{$\overbrace{\phantom{\left(\mr{im}(\ddelta)[1]\oplus \mr{im}(\ddelta^\vee)[n-2]\right)\oplus \left(\mr{im}(K_\gzm)[1]\oplus \mr{im}(K_\gzm^\vee)[n-2]\right)}}^{\F_\gzm^\mr{fluct}}$}
\left(\mr{im}(\ddelta)[1]\oplus \mr{im}(\ddelta^\vee)[n-2]\right)\oplus \underbrace{\left(\mr{im}(K_\gzm)[1]\oplus \mr{im}(K_\gzm^\vee)[n-2]\right)}_{\LL_\gzm} 
\end{multline}

Using this gauge-fixing data, we can construct the pushforward of (\ref{Z composite}) to the standard \zeromodes\ using the fiber BV integral.
\begin{Proposition}\label{prop: gluing}
The partition function of the glued cobordism $(M_1,X_1)\cob{(M,X)} (M_3,X_3)$ can be expressed in terms of the partition functions for constituent cobordisms $(M_1,X_1)\cob{(M_I,X_I)}(M_2,X_2)$, $(M_2,X_2)\cob{(M_{II},X_{II})}(M_3,X_3)$ as an integral over boundary conditions on the interface $(M_2,X_2)$ and the fiber BV integral for gluing the bulk residual fields: 
\begin{multline}\label{gluing formula}
Z(B_1,A_3; \Azm,\Bzm)= \\
=\int_{\LL_\gzm\subset \F_\gzm^\mr{fluct}}
\int_{\B^{(A)}_2\times \B^{(B)}_2} Z_I\left(\rlap{$\phantom{e^{\frac{i}{\hbar}}}$}
B_1,A_2;i_\gzm^I(\Azm)+A_\gzm^{\mr{fluct},I},i_\gzm^{\vee,I}(\Bzm)+B_\gzm^{\mr{fluct},I}\right)\cdot \\
\cdot\left[(\D_\hbar A_2)^{1/2}\cdot e^{-\frac{i}{\hbar}\langle B_2,A_2 \rangle} \cdot (\D_\hbar B_2)^{1/2}\right]\cdot \\
\cdot Z_{II}\left(
\rlap{$\phantom{e^{\frac{i}{\hbar}}}$}
B_2,A_3;i_\gzm^{II}(\Azm)+A_\gzm^{\mr{fluct},II},i_\gzm^{\vee,II}(\Bzm)+B_\gzm^{\mr{fluct},II}\right)
\end{multline}
The equality is modulo $\left(\frac{i}{\hbar}\hat S_\dd^\can-i\hbar\,\Delta_\zm^\can\right)$-coboundaries.
\end{Proposition}

Superscripts $I$, $II$ correspond to projections to the first and second terms of
the splittings in the l.h.s. of (\ref{gzm retraction},\ref{gzm dual}) in the obvious way. Schematically, the formula (\ref{gluing formula}) can be written as
$$Z=(p_\gzm)_*(Z_I\ast Z_{II})$$
where $Z_I\ast Z_{II}$ stands for the convolution as in (\ref{Z composite}) and $(p_\gzm)_*$ stands for the BV pushforward\footnote{We use the term BV pushforward as a synonym for the fiber BV integral.}
from composite {\zeromodes} to the standard {\zeromodes}.

The statement (\ref{gluing formula}) follows by construction from general properties of fiber BV integrals and from (\ref{gluing e^S}); we will give a proof by direct computation below, after describing the gluing on the level of induction data.

\begin{Proposition}\label{prop: glued ind data}
Induction data $ C^\bt(X,X_3)\wavy{(\ii_\gl,\pp_\gl,\K_\gl)} H^\bt(M,M_3)$ for the glued cobordism can be constructed in terms of the induction data for constituent cobordisms and the data (\ref{gzm retraction}) by the following formulae:
\begin{multline}\label{glued ind data}
\ii_\gl= \left(\begin{array}{cc} \ii_I & -\K_I \phi_I p_2 \ii_{II} \\ 0 & \ii_{II} \end{array}\right) \left(\begin{array}{c} i_\gzm^I \\ i_\gzm^{II} \end{array}\right),\;
\pp_\gl=\left(\begin{array}{cc} p_\gzm^I & p_\gzm^{II} \end{array}\right) \left(\begin{array}{cc} \pp_I & -\pp_I \phi_I p_2 \K_{II} \\ 0 & \pp_{II} \end{array}\right),\\
\K_\gl=\left(\begin{array}{cc} \K_I & -\K_I \phi_I p_2 \K_{II} \\ 0 & \K_{II} \end{array}\right)-\left(\begin{array}{cc} \ii_I & -\K_I \phi_I p_2 \ii_{II} \\ 0 & \ii_{II} \end{array}\right) K_\gzm \left(\begin{array}{cc} \pp_I & -\pp_I \phi_I p_2 \K_{II} \\ 0 & \pp_{II} \end{array}\right)
\end{multline}
\end{Proposition}
(Subscript ``$\gl$'' here stands for ``glued''.)

\begin{proof}
Indeed, we first deform the induction data in (\ref{C_I+C_II retraction}) by the upper-triangular perturbation of the differential in (\ref{d decomposition}), using the homological perturbation lemma, which yields a retraction $C^\bt(X,X_3)\wavy{(\ii_\comp,\pp_\comp,\K_\comp)} H^\bt(M_I,M_2)\oplus H^\bt(M_{II},M_3),\ddelta $ with
\be \ii_\comp=\left(\begin{array}{cc} \ii_I & -\K_I \phi_I p_2 \ii_{II} \\ 0 & \ii_{II} \end{array}\right),\;
\pp_\comp=\left(\begin{array}{cc} \pp_I & -\pp_I \phi_I p_2 \K_{II} \\ 0 & \pp_{II} \end{array}\right),\;
\K_\comp=\left(\begin{array}{cc} \K_I & -\K_I \phi_I p_2 \K_{II} \\ 0 & \K_{II} \end{array}\right) \label{glued ind data (composite)}\ee
(Subscript ``$\comp$'' for stands for ``composite {\zeromodes}''.)
Then we compose it with  the retraction (\ref{gzm retraction}), using construction (\ref{composition of ind data}), which yields the retraction $ C^\bt(X,X_3)\wavy{(\ii_\gl,\pp_\gl,\K_\gl)} H^\bt(M,M_3)$ given by (\ref{glued ind data}).
\end{proof}

\begin{proof}[Proof of Proposition \ref{prop: gluing}]
Let us check by a direct computation that the
l.h.s and r.h.s. of (\ref{gluing formula}) coincide exactly (not modulo coboundaries)
for a special choice of gauge-fixing data in the integral (\ref{Z integrating out bulk fields}) on the glued cobordism, -- the one associated to the induction data (\ref{glued ind data}).
Indeed, substituting explicit expressions (\ref{Z integrating out bulk fields}) into (\ref{Z composite}), we have
\begin{multline*}
Z^\mr{comp.\; \zm}(B_1,A_3;A_\zm^I,A_\zm^{II},B_\zm^I,B_\zm^{II})=\\=
\int_{\B^{(A)}_2\times \B^{(B)}_2}
e^{\frac{i}{\hbar}\left(\langle \left.B^I_\zm\right|_2,A_2 \rangle_2+\langle B_1 ,\left.A^I_\zm\right|_1\rangle_1-\langle \phi^\vee_{1\ra I}(B_1),\K_I\phi_{2\ra I}(A_2)\rangle_I\right)}\cdot \\ \cdot
(\D_\hbar B_1)^{1/2}\cdot \tau_\hbar(M_I,M_2)\cdot (\D_\hbar A_2)^{1/2} 
\cdot\left[(\D_\hbar A_2)^{1/2}\cdot e^{-\frac{i}{\hbar}\langle B_2,A_2 \rangle} \cdot (\D_\hbar B_2)^{1/2}\right]\cdot \\
\cdot e^{\frac{i}{\hbar}\left(\langle \left.B^{II}_\zm\right|_3,A_3 \rangle_3+\langle B_2 ,\left.A^{II}_\zm\right|_2\rangle_2-\langle \phi^\vee_{2\ra II}(B_2),\K_{II}\phi_{3\ra II}(A_3)\rangle_{II}\right)} \cdot \\
\cdot (\D_\hbar B_2)^{1/2}\cdot \tau_\hbar(M_{II},M_3)\cdot (\D_\hbar A_3)^{1/2}=
\\
=e^{\frac{i}{\hbar} \left(\langle B_1, \left.A^I_\zm\right|_1\rangle_1+ \langle \left. B^{II}_\zm\right|_3,A_3 \rangle_3+ \langle \left.\left(B^I_\zm-\K_I^\vee\phi_{1\ra I}^\vee(B_1)\right)\right|_2 , \left.\left( A^{II}_\zm-\K_{II}\phi_{3\ra II}(A_3) \right)\right|_2 \rangle_2 \right) }
\cdot \\
\cdot(\D_\hbar B_1)^{1/2}\cdot \tau_\hbar(M_I,M_2)\cdot \tau_\hbar(M_{II},M_3)\cdot (\D_\hbar A_3)^{1/2}
\end{multline*}
Here we use the notation $\tau_\hbar(M,M_\dout)=\xi_\hbar^{H^\bt(M,M_\dout)}\tau(M,M_\dout)$.
Taking the Gaussian fiber BV integral over $\LL_\gzm\subset \F_\gzm$, we obtain
\begin{multline}\label{Z glued}
Z(B_1,A_3;A_\zm,B_\zm)=\\
{\scriptstyle \exp \frac{i}{\hbar} \left( \left\langle B_1, \left.\left( i_\gzm^I(A_\zm)- \K_I\phi_{2\ra I} p_2 i^{II}_\gzm(A_\zm) \right)\right|_1 \right\rangle_1 + \left\langle  \left.\left( i^{\vee,II}_\gzm(B_\zm)-\K_{II}^\vee\phi^\vee_{2\ra II} p_2^\vee i_\gzm^{\vee,I}(B_\zm) \right)\right|_3  ,A_3 \right\rangle_3 + \right.} \\ {\scriptstyle\left. +
\langle \left.\K_I^\vee\phi_{1\ra I}^\vee(B_1)\right|_2 , \left.\K_{II}\phi_{3\ra II}(A_3)\right|_2 \rangle_2
-\right.} \\  {\scriptstyle \left. -\left\langle (\mr{id}-\phi_{2\ra II}^\vee p^\vee_2 \K_I^\vee)\phi_{1\ra I}^\vee(B_1),\;\; (\ii_I\oplus \ii_{II})K_\gzm (\pp_I\oplus \pp_{II})\;\;(\mr{id}-\phi_{2\ra I}p_2 \K_{II})\phi_{3\ra II}(A_3)  \right\rangle_{I\cup II}
\right)  }\cdot\\
\cdot (\D_\hbar B_1)^{1/2}\cdot \underbrace{\frac{\xi_\hbar^{H^\bt(M,M_3)}}{\xi_\hbar^{H^\bt(M_I,M_2)}\cdot \xi_\hbar^{H^\bt(M_{II},M_3)}}\;\bT(\tau_\hbar(M_I,M_2)\cdot \tau_\hbar(M_{II},M_3))}_{\tau_\hbar(M,M_3)}\cdot (\D_\hbar A_3)^{1/2}
\end{multline}
Here $\bT: \Det \left(H^\bt(M_I,M_2)\oplus H^\bt(M_{II},M_3)\right)\xra{\cong} \Det H^\bt(M,M_3) $
is the canonical isomorphism between the determinant line of a cochain complex and the determinant line of cohomology, associated to the retraction (\ref{gzm retraction}). The factor $\frac{\xi}{\xi\cdot \xi}$  in front of $\bT(\cdots)$ appears as in Lemma \ref{lemma: torsion as BV integral}. In the expression (\ref{Z glued}) one recognizes the r.h.s. of (\ref{Z integrating out bulk fields}) with the gauge-fixing associated to the induction data of Proposition \ref{prop: glued ind data}.
\end{proof}

\begin{remark}
Note that the formula for $\K_\gl$ in (\ref{glued ind data}) is the gluing formula for propagators (cf. the analogous formula obtained in a different language in \cite{CMRpert}). In the special case $H^\bt(M_I,M_2)=H^\bt(M_{II},M_3)=0$ (or, more generally, $\ddelta=0$, or equivalently $H^\bt(M,M_3)\simeq H^\bt(M_I,M_2)\oplus H^\bt(M_{II},M_3)$), the formula simplifies to $\K_\comp$ of (\ref{glued ind data (composite)}). In terms of parametrices, in the latter case one has
\begin{multline*}
\KK_\comp(e_I,e_{I}^\vee)=\KK_I(e_I,e_I^\vee), \quad \KK_\comp(e_{II},e_{II}^\vee)=\KK_{II}(e_{II},e_{II}^\vee),\quad
\KK_\comp(e_{II},e_I^\vee)=0,\\ 
\KK_\comp(e_I,e_{II}^\vee)= - \sum_{e_2\in X_2} \KK_I(e_I,\varkappa_2(e_2))\cdot \KK_{II}(e_2,e_{II}^\vee)
\end{multline*}
\end{remark}

\subsection{Passing to the reduced space of states}\label{sec: reduced space of states}
If a state $\psi\in \HH_\dd=\Fun_\bC(\B_\dd)$ satisfies $\hat S_\dd \psi=0$, then it can be projected to the {\it reduced space of states} -- cohomology of the quantum BFV operator:
$$[\psi]\in\quad  \HH^\red_\dd=H^\bt_{\hat S_\dd}(\HH_\dd)\cong \Fun_\bC(
\B^\red_\dd
)$$
where $\B^\red_\dd=H^\bt(M_\dout)[1]\oplus H^\bt(M_\din)[n-2]$ is the moduli space of the $Q$-manifold $(\B_\dd,Q_{\B_\dd})$, i.e. the
zero-locus of the cohomological vector field $Q_{\B_\dd}$ reduced modulo the distribution induced by $Q_{\B_\dd}$ on the zero-locus (see \cite{CMR}).

\begin{remark} Returning to the setup of Remark \ref{rem: pairing H^A with H^B}, let $N$ be a closed $(n-1)$-manifold with a cellular decomposition $Y$. Given two states $\phi\in \HH^{(A)}_Y$, $\psi\in \HH^{(B)}_{\bar Y}$ that are annihilated by respective quantum BFV operators $\widehat S_Y^{(A)}\in \mr{End}(\HH^{(A)}_Y)_{1}$, $\widehat S_Y^{(B)}\in \mr{End}(\HH^{(B)}_{\bar Y})_{1}$, the pairing (\ref{pairing H^A with H^B}) between them can be expressed in terms of classes of states $\phi$, $\psi$ in $\hat S_Y$-cohomology, $[\phi]\in \HH^{(A),\red}_Y$, $[\psi]\in \HH^{(B),\red}_{\bar Y}$:
\be (\phi,\psi)=\int_{\B_Y^{(A),\red}\times \B_Y^{(B),\red}}[\phi]\cdot\left( \D_\hbar [A_Y]\; e^{-\frac{i}{\hbar}\langle [B_Y] , [A_Y] \rangle} \;\D_\hbar [B_Y] \right)\cdot  [\psi] \label{pairing H^A with H^B reduced}\ee
with $[A_Y], [B_Y]$ the superfields for $\B^{(A),\red}_Y=H^\bt(N)[1]$ and $\B^{(B),\red}_Y=H^\bt(N)[n-2]$, respectively. Here the normalized densities on reduced spaces
$\B_Y^{(A),\red}$,  $\B_Y^{(B),\red}$  
are defined as:
\be \D_\hbar [A_Y]= \underbrace{\prod_{k=0}^{n-1}\left(\xi_\hbar^k\right)^{\dim H^k(N)}}_{\xi_\hbar^{H^\bt(N)}}\cdot\tau(N),\qquad  \D_\hbar [B_Y]= \underbrace{\prod_{k=0}^{n-1}\left(\xi_\hbar^{n-k}\right)^{\dim H^k(N)}}_{\til \xi_\hbar^{H^\bt(N)}}\cdot\tau(N)^{(-1)^{n-1}} \label{normalized densities on Y}\ee
with $\tau(N)\in \Det\, H^\bt(N)/\{\pm 1\}\simeq \Dens\, B^{(A),\red}_Y$ the $R$-torsion of $N$ and factors $\xi_\hbar^k$ as in (\ref{xi_hbar^k}).\footnote{
Note that, using Poincar\'e duality on $N$, the factor $\til\xi_\hbar^{H^\bt(N)}$ in (\ref{normalized densities on Y}) can be expressed as $\prod_k \left(\xi_\hbar^k\right)^{\dim H^{k-1}(N)}=\xi_\hbar^{H^\bt(N)[-1]}$.
} Equality (\ref{pairing H^A with H^B reduced}) is checked straightforwardly (cf. Footnote \ref{footnote: gluing normalization check}) for states of the ``plane wave'' form, $\phi=e^{\frac{i}{\hbar}\langle \beta_Y , A_Y \rangle_Y}$, $\psi=e^{\frac{i}{\hbar}\langle B_Y , \alpha_Y \rangle_Y}$, with parameters $\beta_Y$ a closed cochain on $Y^\vee$ and $\alpha_Y$ a closed cochain on $Y$. Then (\ref{pairing H^A with H^B reduced}) follows by extension by bilinearity to all pairs of $\widehat S_\dd$-closed states.
\end{remark}

In terms of half-densities, the reduction sends 
a $\hat S_\dd^\can$-closed state $\psi\cdot \left(\mu_{\B_\dd}^\hbar\right)^{1/2}\in \HH^\can_\dd$ to
$$[\psi]\cdot \left(\mu^\hbar_{\B^\red_\dd}\right)^{1/2}
\quad \in \HH^{\red,\can}_\dd=\Dens^{\frac12,\Fun}_\bC(\B^\red_\dd)$$
where we define the normalized density on $\B^\red_\dd$ as
$$\mu^\hbar_{\B^\red_\dd} = \D_\hbar [A_\dout]\cdot \D_\hbar [B_\din] = \xi_\hbar^{H^\bt(M_\dout)}\til \xi_\hbar^{H^\bt(M_\din)}\cdot \tau(M_\dout)\cdot \tau(M_\din)^{(-1)^{n-1}}$$
with normalization factors as in (\ref{normalized densities on Y}).

If the partition function (\ref{Z integrating out bulk fields result}) were a $\hat S_\dd$-cocycle, we could construct the reduced partition function as the $\hat S_\dd$-cohomology class of $Z$. However, $Z$ generally only satisfies $\left(\frac{i}{\hbar}\hat S_\dd - i\hbar\,\Delta_\zm\right)Z=0$, but does not satisfy $\hat S_\dd Z =0$.
This problem is easily overcome as follows.


Fix some induction data 
$(\B_\dd,d_{X_\dout}\oplus d_{X^\vee_\din}) \wavy{(i_\B,p_\B,K_\B)} \B^\red_\dd=H^\bt(\B_\dd)$.
By the construction of Section \ref{sec: induction data for dual and sym}, we can infer the induction data $(i_\HH=p_\B^*,p_\HH=i_\B^*,K_\HH=\cdots)$ from the space of states $(\HH_\dd,\hat S_\dd)$ to the reduced space $\HH^\red_\dd$.
(For the moment we are discussing the non-canonical picture, where states are functions on $\B_\dd$ or $\B^\red_\dd$; we will switch to half-densities later.)
Denote
$$Z^\mr{mod}=i_\HH p_\HH Z = Z-\hat S_\dd (K_\HH Z)-K_\HH (\underbrace{\hat S_\dd Z}_{\hbar^2\Delta_\zm Z})=Z+\left(\frac{i}{\hbar}\hat S_\dd - i\hbar\,\Delta_\zm\right)(\cdots)$$
where we used the quantum master equation (\ref{QME with bdry term for Z}) and $(\cdots) = i\hbar\; K_\HH Z$.
By construction, $\hat S_\dd Z^\mr{mod}=0$. Also note that $Z^\mr{mod}$ differs from $Z$ by a
$\left(\frac{i}{\hbar}\hat S_\dd - i\hbar\,\Delta_\zm\right)$-exact term, i.e. by a BV canonical transformation, and we are ultimately only interested in partition functions modulo BV canonical transformations.

In this sense, the reduced partition function is simply
$$Z^\red:=[Z^\mr{mod}]=p_\HH Z\quad \in \HH^\red_\dd\widehat\otimes\, \Fun_\bC(\F_b^\zm)$$
i.e. the evaluation of (\ref{Z integrating out bulk fields result}) on chosen representatives of cohomology of $M_\dout$, $M_\din$, as boundary fields $\Aout,\Bin$.

In terms of half-densities, the pushforward of $\left(\mu_\B^\hbar\right)^{1/2}$ to the reduced space of states results in the appearance of square roots of torsions of the boundary in the canonical reduced partition function $Z^\red=(p_\HH)_*Z$. More precisely, we have the following, as a corollary of Proposition \ref{prop: Z}.
\begin{Proposition}
The canonical reduced partition function is
\begin{multline}\label{Z reduced}
Z^\red([\Aout],[\Bin];\Azm,\Bzm)
=\\
=
e^{\frac{i}{\hbar} \langle \left. \Bzm\right|_\dout, i_\B[\Aout]  \rangle_\dout +
\langle i_\B[\Bin], \left. \Azm \right|_\din\rangle_\din
- \langle \phi^\vee(i_\B[\Bin]) \;,\; \K \phi(i_\B[\Aout])  \rangle} \cdot \\
\cdot
\xi_\hbar{\left(M_\din \cob{M} M_\dout\right)}
\cdot\tau(M,M_\dout)\cdot \tau(M_\dout)^{\frac12}\cdot \tau(M_\din)^{\frac{(-1)^{n-1}}{2}}\qquad\in\\
\in \HH^{\red,\can}_\dd\widehat\otimes\; \Dens^{\frac12,\Fun}_\bC(\F_b^\zm)
\end{multline}
Here $[\Aout],[\Bin]$ are the superfields for $H^\bt(M_\dout)[1]$, $H^\bt(M_\din)[n-2]$. The normalization factor is
\begin{multline*}
\xi_\hbar{\left(M_\din \cob{M} M_\dout\right)}=\xi_\hbar^{H^\bt(M,M_\dout)}\left(\xi_\hbar^{H^\bt(M_\dout)}\right)^{\frac12}\left(\til \xi_\hbar^{H^\bt(M_\din)}\right)^{\frac12}=\\
=\prod_{k=0}^n (\xi_\hbar^k)^{\dim H^k(M,M_\dout)+\frac12 \dim H^k(M_\dout)+\frac12 \dim H^{n-k}(M_\din)}\quad \in\bC
\end{multline*}
\end{Proposition}
The first term in the exponential in (\ref{Z reduced}) is the pairing of $H^k(M,M_\din)$ with $H^{n-k-1}(M_\dout)$ via the natural map $H^k(M,M_\din)\ra H^k(M)\ra H^k(M_\dout)$
and Poincar\'e pairing in cohomology of $M_\dout$, and similarly for the second term in the exponential. Third term generally depends on the details of
 gauge-fixing.

\begin{remark}\label{rem: K term in case of no zero modes}
In the case $H^\bt(M,M_\dout)=0$ (or equivalently, $H^\bt(M,M_\din)=0$), one has isomorphisms $H^\bt(M)\xra{\sim} H^\bt(M_\dout)$, $H^\bt(M)\xra{\sim} H^\bt(M_\din)$ (arising from long exact sequences of pairs $(M,M_\dout)$ and $(M,M_\din)$, respectively). In this case the pairing in the third term in the exponential in (\ref{Z reduced}) is the composition of the isomorphism $\theta:H^\bt(M_\dout)\xrightarrow{\sim}H^\bt(M)\xra{\sim} H^\bt(M_\din)$ with Poincar\'e pairing on $H^\bt(M_\din)$, i.e. $\langle [B_\din],\theta [A_\dout]\rangle_\din$.
\end{remark}

\begin{example}\label{example: Z red for only out-bdry} For $M$ with $M_\din=\varnothing$,
$$Z^\red=e^{\frac{i}{\hbar}\langle \iota^*B_\zm, [\Aout]\rangle}\cdot\xi_\hbar^{H^\bt(M,M_\dout)}\left(\xi_\hbar^{H^\bt(M_\dout)}\right)^{\frac12}\cdot\tau(M,M_\dout)\cdot\tau(M_\dout)^{1/2}$$
where $\iota^*:H^\bt(M)\ra H^\bt(M_\dout)$ is the pullback by the inclusion $\iota: M_\dout\hra M$ in cohomology; $\langle,\rangle$ is the Poincar\'e pairing in boundary cohomology. An analogous consideration applies in the case $M_\dout=\varnothing$.
\end{example}

\begin{example}[cylinder] Let $M=N\times [0,1]$, with in-boundary $M_\din=N\times \{0\}$ and out-boundary $M_\dout=N\times \{1\}$, endowed with arbitrary cellular decomposition. Then there are no 
residual fields
 and, by Remark \ref{rem: K term in case of no zero modes}, we have
\be Z^\red=e^{\frac{i}{\hbar}\langle [\Bin],[\Aout]\rangle}\;(\D_\hbar[\Bin])^{1/2}\cdot(\D_\hbar[\Aout])^{1/2} \label{Z reduced cylinder}\ee
with $\langle,\rangle$ the Poincar\'e pairing on $H^\bt(N)$.
Note that this partition function represents the \emph{identity} in $\mr{Hom}(\HH^{(A),\red}_N,\HH^{(A),\red}_N)$, cf. (\ref{HH as Hom-space}).
\end{example}

\begin{remark}\label{rem 7.16}
By construction, the reduced partition function satisfies the quantum master equation without boundary term:
\begin{equation}\label{QME reduced}
\Delta_\zm^\can Z^\red=0
\end{equation}
(which can be thought of as the equation (\ref{QME with bdry term for Z}) where the boundary BFV operator is killed by passing to the reduced space of states).
Changing the details of gauge-fixing, i.e. maps $\ii,\pp,\K,i_\B,p_\B,K_\B$, results in a 
transformation of the partition function of the form
$$Z^\red\mapsto Z^\red+\Delta_\zm^\can(\cdots)$$
These properties follow immediately from (\ref{prop: Z (ii)}), (\ref{prop: Z (iii)}) of Proposition \ref{prop: Z}.
\end{remark}

\begin{remark}
Observe the similarity between integration over bulk fluctuations of fields (\ref{Z integrating out bulk fields}) and pushforward to the reduced space of states on the boundary (\ref{Z reduced}). Both procedures involve similar sets of gauge-fixing/induction data, both deal with half-densities and produce the $R$-torsion in bulk/boundary.
\end{remark}

Equality (\ref{pairing H^A with H^B reduced}) implies that the gluing formula (\ref{gluing formula}) holds also in the setting of reduced boundary states.
\begin{corollary}[of Proposition \ref{prop: gluing}]
For a glued cobordism (\ref{glued cobordism}), we have
\begin{multline}\label{gluing formula bdry reduced}
Z^\red([B_1],[A_3]; \Azm,\Bzm)= \\
=\int_{\LL_\gzm\subset \F_\gzm^\fluct}
\int_{\B^{(A),\red}_2\times \B^{(B),\red}_2} Z_I^\red\left(\rlap{$\phantom{e^{\frac{i}{\hbar}}}$}
[B_1],[A_2];i_\gzm^I(\Azm)+A_\gzm^{\mr{fluct},I},i_\gzm^{\vee,I}(\Bzm)+B_\gzm^{\fluct,I}\right)\cdot \\
\cdot\left[(\D_\hbar [A_2])^{1/2}\cdot e^{-\frac{i}{\hbar}\langle [B_2],[A_2] \rangle} \cdot (\D_\hbar [B_2])^{1/2}\right]\cdot \\
\cdot Z_{II}^\red\left(
\rlap{$\phantom{e^{\frac{i}{\hbar}}}$}
[B_2],[A_3];i_\gzm^{II}(\Azm)+A_\gzm^{\mr{fluct},II},i_\gzm^{\vee,II}(\Bzm)+B_\gzm^{\fluct,II}\right)
\end{multline}
modulo $\Delta^\can_\zm$-coboundaries.
\end{corollary}

Note that $Z^\red$ is an element of the space which is expressed in terms of cohomology of $M$ and its boundary, and thus is manifestly independent on the cellular decomposition $X$ of $M$. More precisely, one has the following.
\begin{Proposition}
The class of $Z^\red$ in cohomology of $\Delta_\zm^\can$ is independent of the cellular decomposition $X$.
\end{Proposition}

\begin{proof}
First, observe that if we glue to $(M_\din,X_\din)\cob{(M,X)}(M_\dout,X_\dout)$ at the in-boundary a cylinder $
(M_\din,\cdots)\cob{(M_\din\times [0,1],\cdots)}(M_\din,X_\din)$ (with arbitrary cellular decomposition inducing $X_\din$ on the out-boundary of the cylinder), this procedure does not change the reduced partition function:
$$Z_{\mr{cylinder}\;\cup\; M}^\red=Z_M^\red$$
This follows directly from the gluing formula (\ref{gluing formula bdry reduced}) and the explicit result for the cylinder (\ref{Z reduced cylinder}).

Now, let $X$ be a cellular decomposition of a cobordism $M_\din\cob{M} M_\dout$. Consider the ``out-out'' cylinder
$$\mr{Cyl}_\mr{out-out}\quad =\qquad \varnothing\cob{(M_\din\times [0,1],\cdots)} (M_\din\times\{0\},X_\din)\sqcup (M_\din\times\{ 1 \},Y) $$
and attach its $M_\din\times\{0\}$-boundary to the in-boundary of $M$. The result is a cobordism $\til M=\mr{Cyl}_\mr{out-out}\cup M$ with only out-boundary, thus for $\til M$ the reduced partition function is independent of cellular decomposition (cf. Example \ref{example: Z red for only out-bdry}). On the other hand, we can attach to $M_\din\times\{1\}\subset \mr{Cyl}_\mr{out-out}$ an ``in-in'' cylinder
$$\mr{Cyl}_\mr{in-in}\quad =\qquad (M_\din\times\{0\},Y)\sqcup (M_\din\times\{ 1 \},\cdots) \cob{(M_\din\times [0,1],\cdots)} \varnothing $$

\begin{figure}[!htbp]
\begin{center}
\includegraphics[scale=1]{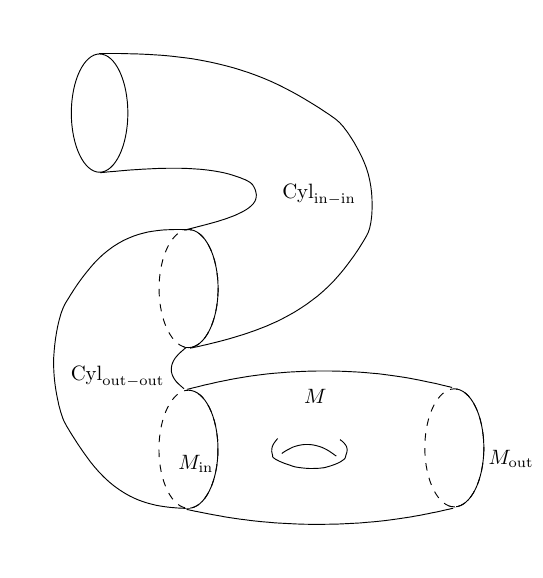}
\caption{A cobordism with a ``tail'' of two cylinders (out-out and in-in) attached.}
\end{center}
\end{figure}

The resulting cobordism $\til{\til M}=\mr{Cyl}_\mr{in-in}\cup\mr{Cyl}_\mr{out-out}\cup M$ differs from $M$ by attaching an ``in-out'' cylinder, thus $Z_{\til{\til M}}^\red=Z_M^\red$ by discussion above. On the other hand $Z_{\til M}^\red$ is independent of cellular decomposition $X$, as is, of course, $Z_{\mr{Cyl}_\mr{in-in}}^\red$ (which is also a case of 
Example \ref{example: Z red for only out-bdry}). Therefore $Z_M^\red=Z_{\til{\til M}}^\red=\mr{Gluing}(Z_{\mr{Cyl}_\mr{in-in}}^\red,Z_{\til M}^\red)$ is independent on $X$.  ($\mr{Gluing}(-,-)$ is a schematic notation for the r.h.s. of (\ref{gluing formula bdry reduced}).)

\end{proof}

\section{Non-abelian cellular $BF$ theory, I: ``canonical setting''}
\label{sec: non-ab BF I}
The goal of this section is the construction of a ``canonical" version  of non-abelian cellular $BF$ theory, where
the fields are a cochain and a chain of the same CW complex $X$ which is not required to be a manifold.
We construct cellular actions (Theorem~\ref{thm: cellBF}) that deform the abelian action, satisfy the BV quantum master equation, are compatible with restrictions to subcomplexes, and on $0$-cells have the canonical form for a $0$-dimensional non-abelian $BF$ theory (this may be viewed as a cellular replacement for the AKSZ construction \cite{AKSZ} for topological quantum field theories).

Next, in Section~\ref{s:BVpfsheca}, we prove that these cellular actions are
compatible with elementary collapses (Lemma \ref{lemma: collapse}) and
that the partition function, defined via BV pushforward to cohomology,  is a simple-homotopy invariant (Proposition \ref{prop: non-ab BF I simple homotopy invariance}). 


The version of non-abelian cellular $BF$ theory presented in this section is not of Segal type, since one of the fields has ``wrong''  (covariant) functoriality. On the other hand, we have a version of Mayer-Vietoris gluing formula for the cellular actions (see (\ref{simpBF (vi)}) of Theorem \ref{thm: simpBF}).

Throughout this section we adopt the formalism of \cite{SimpBF,DiscrBF} where the field $B$ of $BF$ theory is treated as covariant (as a cellular chain in discrete setting and as a de Rham current in continuous setting\footnote{This setting for (continuum) $BF$ theory is known as ``canonical $BF$ theory'' in the literature, cf. e.g. \cite{CR}.}) whereas the field $A$ is contravariant as usual. To reflect this, we denote the fields (and the space of fields) in a different font. Later, in Section \ref{sec: non-ab BF cobordism}, we will return to the formalism where both fields are contravariant.

Let $X$ be a finite CW 
complex (it is not required to be a triangulation of a manifold) and let $\g$ be a unimodular Lie algebra.
We introduce the graded vector space of fields $\mathsf{F}_{X}=C^\bt(X,\g)[1]\oplus C_\bt(X,\g^*)[-2]$.\footnote{Here we are not introducing the twist by a local system and the notation is simply $C^\bt(X,\g):=\g\otimes C^\bt(X,\mathbb{R})$ -- cellular cochains with coefficients in $\g$. Likewise, for the chains, $C_\bt(X,\g^*):=\g^*\otimes C_\bt(X,\mathbb{R})$.} As in Section \ref{sec: classical closed}, it is spanned by the superfields 
\begin{align*}
\sA_X  =\sum_{e\subset X} e^*\cdot\sA_e & \quad\in  \mr{Hom}_1(\sF_X,C^\bt(X,\g)),\\
\sB_X  =\sum_{e\subset X} \sB_e\cdot e & \quad \in  \mr{Hom}_{-2}(\sF_X,C_\bt(X,\g^*))
\end{align*}
Here $e^*, e$ are the standard basis integral cochain and chain, respectively, associated to a cell $e\subset X$; components $\sA_e$ are $\g$-valued functions on $\sF_X$ of degree $1-\dim e$ and components $\sB_e$ are $\g^*$-valued functions on $\F_X$ of degree $-2+\dim e$.

Note that here, unlike in the rest of the paper, the field $\sB_X$ is a \textit{chain} of $X$, as opposed to a \textit{cochain} of $X^\vee$ (moreover, in the setup of this subsection, $X^\vee$ is meaningless, as $X$ is not required to be a cellular decomposition of a manifold).

The canonical pairing $\langle,\rangle$ between cochains and chains induces a degree $-1$ symplectic form 
$$\omega_X=\langle \delta \sB_X, \delta \sA_X \rangle=\sum_{e\subset X}(-1)^{\dim e}  \langle\delta \sB_e\stackrel{\wedge}{,} \delta \sA_e\rangle_\g$$ and a BV Laplacian 
$$\Delta_X=\left\langle \frac{\dd}{\dd \sB_X}, \frac{\dd}{\dd \sA_X}\right\rangle=\sum_{e\subset X} (-1)^{\dim e} \left\langle\frac{\dd}{\dd \sB_e},\frac{\dd}{\dd \sA_e}\right\rangle_\g $$ on $\mr{Fun}(\sF_X)=\widehat{\Sym^\bt}\sF_X^*$ (the formal power series in fields). Symbol $\langle,\rangle_\g$ stands for the canonical pairing between $\g$ and $\g^*$.

In our notation $e\subset X$ stands for an \textit{open} cell; its closure $\bar{e}=e\cup\dd e\subset X$ is a closed ball (for $X$ a regular CW complex, which we always assume unless stated otherwise) inheriting a structure of CW complex from $X$. Thus, e.g., $\sA_e$ is a component of the superfield $\sA_X$, whereas $\sA_{\bar{e}}=\sum_{e'\subset \bar{e}}(e')^*\cdot \sA_{e'}$ is the entire superfield for the subcomplex $\bar{e}\subset X$ containing components of $\sA_X$ for the cell $e$ itself and cells belonging to the boundary $\dd e$. Likewise, $\dd e\subset X$ is a subcomplex and $\sA_{\dd e}=\sum_{e\subset \dd e}(e')^*\cdot\sA_{e'}$ is the corresponding $\sA$-superfield, and we have $\sA_{\bar e}=e^*\cdot \sA_e+\sA_{\dd e}$.

\subsection{Non-abelian $BF$ theory on a simplicial 
complex, 
after \cite{SimpBF,DiscrBF}
}\label{sec: simp BF reminder}
Let $X$ be a simplicial complex.

We denote by $\bar{\DDelta}^N$ the standard closed simplex of dimension $N\geq 0$, endowed with standard triangulation We view it as a simplicial complex with the top cell the open simplex $\DDelta^N$. 

Let $\Omega^\bt(X)$ stand for the complex of continuous piecewise polynomial differential forms on (the geometric realization of) $X$. Cochains $C^\bt(X,\mathbb{R})$ can be quasi-isomophically embedded into $\Omega^\bt(X)$ as \textit{Whitney forms} (continuous piecewise linear forms, linear on every simplex of $X$ w.r.t. its barycentric coordinates), see \cite{Whitney,Getzler} for details. We denote $\mathsf{i}_X: C^\bt(X)\hra \Omega^\bt(X)$ the realization  of cochains as Whitney forms. This embedding has a natural left inverse, the \textit{Poincar\'e map} $\mathsf{p}_X:\Omega^\bt(X)\ra C^\bt(X)$ which integrates a form over simplices of $X$, i.e. maps $\alpha\mapsto \sum_{e\subset X} e^*\cdot \left(\int_e \alpha \right)$. Dupont has constructed an explicit chain homotopy operator $\sfK_X$, contracting $\Omega^\bt(X)$ onto Whitney forms, see \cite{Dupont, Getzler}. Thus, in the terminology of Section \ref{sec: HPT}, we have a retraction $\Omega^\bt(X)\wavy{(\sfi_X,\sfp_X,\sfK_X)}C^\bt(X)$. It is glued (by fiber products) out of building blocks -- ``standard'' retractions for $\bar\DDelta^N$ for different $N$ (i.e.  when restricted to any simplex of $X$, it reduces to a ``standard'' retraction for a standard simplex).

We will denote by $\wed$ the pushforward of the wedge product of forms (defined piecewise on simplices of $X$) to cochains, $a\wed b:= \sfp_X (\sfi_X(a)\wedge \sfi_X(b))$ -- this is a 
graded-commutative 
non-associative product on $C^\bt(X)$. Tensoring this operation with the Lie bracket in $\g$, one gets a bilinear operation $[-\stackrel{\wed}{,}-]$ on $C^\bt(X,\g)$ defined by $[x\otimes a\stackrel{\wed}{,}y\otimes b]=[x,y]\otimes (a\wed b)$ for $x,y\in \g$ and $a,b\in C^\bt(X)$.

The following is a reformulation of one of the main results of \cite{DiscrBF}.

\begin{theorem} \label{thm: simpBF}
There exists a sequence of elements
$\bar{S}_{\DDelta^N}\in \mr{Fun}(\mathsf{F}_{\bar\DDelta^N})[[\hbar]]$, for $N=0,1,2,\ldots$, of the form
\begin{multline} \label{Sbar}
\bar{S}_{\DDelta_N}(\sA_{\bar\DDelta^N},\sB_{\DDelta^N};\hbar)=\\
=\sum_{\n=1}^\infty\sum_{\Gamma_0}\sum_{e_1,\ldots,e_\n\subset \bar\DDelta^N}\frac{1}{|\mr{Aut}(\Gamma_0)|} C^{\DDelta^N}_{\Gamma_0,e_1,\ldots,e_\n}
\left\langle \sB_{\DDelta^N},\mr{Jacobi}_{\,\Gamma_0}(\sA_{e_1},\ldots,\sA_{e_\n})\right\rangle_\g-\\
-i\hbar \sum_{\n= 2}^\infty \sum_{\Gamma_1} 
\sum_{e_1,\ldots,e_\n\subset \bar\DDelta^N}\frac{1}{|\mr{Aut}(\Gamma_1)|} C^{\DDelta^N}_{\Gamma_1,e_1,\ldots,e_\n}
\mr{Jacobi}_{\,\Gamma_1}(\sA_{e_1},\ldots,\sA_{e_\n})
\end{multline}
for some values of ``structure constants'' $C^{\DDelta^N}_{\Gamma_l,e_0,\ldots,e_\n}\in \mathbb{R}$, $l=0,1$, 
such that for any finite simplicial complex $X$ and any unimodular Lie algebra $\g$ the element
\begin{equation}\label{simpBF locality}
S_X(\sA_X,\sB_X;\hbar)=\sum_{e\subset X}\bar{S}_e(\sA_X|_{\bar{e}},\sB_e;\hbar)\;\;\in \mr{Fun}(\sF_X)[[\hbar]]
\end{equation}
satisfies
\begin{enumerate}[(a)]
\item\label{simpBF item QME} the quantum master equation $\Delta_X e^{\frac{i}{\hbar}S_X}=0$,
\item\label{simpBF item init cond} the 
property 
\begin{equation}\label{simpBF init cond}
S_X(\sA_X,\sB_X;\hbar)=
\langle \sB_X, d \sA_X \rangle+ \frac12 \langle \sB_X, [\sA_X \stackrel{\wed}{,} \sA_X] \rangle+ \mathcal{R}
\end{equation}
with the ``error term'' 
$\mathcal{R}\in \widehat{\Sym^{\geq 4}}\sF_X^*\oplus \hbar\cdot \widehat{\Sym^{\geq 2}}\sF_X^*$.
\end{enumerate}
Here the notations are:
\begin{itemize}
\item summation in (\ref{Sbar}) is over binary rooted trees $\Gamma_0$ (oriented towards the root) with $\n$ leaves and 1-loop connected 3-valent graphs $\Gamma_1$ (with every vertex having two incoming and one outgoing half-edge) with $\n$ leaves. Leaves of the graph are decorated by faces $e_1,\ldots, e_\n$ (of arbitrary codimension) of $\bar\DDelta^N$.
\item $\mr{Jacobi}_{\,\Gamma_0}(\sA_{e_1},\ldots, \sA_{e_\n})$ is the nested Lie bracket in $\g$, associated to the tree $\Gamma_0$, evaluated on elements $\sA_{e_1},\ldots, \sA_{e_\n}\in\g$. 
\item $\mr{Jacobi}_{\,\Gamma_1}(\sA_{e_1},\ldots, \sA_{e_\n})$ is the number obtained by cutting the loop of $\Gamma_1$ anywhere (resulting in a tree $\widetilde{\Gamma_1}$ with $\n+1$ leaves, one of which is marked) and taking the trace $\mr{tr}_\g\,\mr{Jacobi}_{\,\widetilde{\Gamma_1}}(\sA_{e_1},\ldots, \sA_{e_\n},\bullet) $ of the endomorphism of $\g$ corresponding to the tree.
\end{itemize}

\end{theorem}

\begin{remark}
The proof we present below differs from \cite{SimpBF,DiscrBF} in its treatment of $S^{(1)}$: here we avoid using regularized infinite-dimensional supertraces over the space of forms $\Omega^\bt(X)$ and instead construct $S^{(1)}$ by purely finite-dimensional methods, using homological perturbation theory. The proof is constructive and, in particular, we can make choices (the only ambiguity in the construction is the choice of induction data in part (\ref{simpBF (iv)}) in the proof below) that give \emph{rational} structure constants $C^{\DDelta^N}_{\Gamma_l,e_1,\ldots,e_\n}\in \mathbb{Q}$, $l=0,1$.
\end{remark} 
 
\begin{proof}[Sketch of proof] We split $S_X$ as $S_X=S_X^{(0)}-i\hbar S_X^{(1)}$ and treat the components $S_X^{(0)}$, $S_X^{(1)}$ separately. Likewise, in (\ref{Sbar}) we split $\bar{S}_{\DDelta^N}=\bar{S}_{\DDelta^N}^{(0)}-i\hbar \bar{S}_{\DDelta^N}^{(1)}$. We break the proof in several steps; we address the construction and properties of $S^{(0)}$ in (\ref{simpBF (i)}--\ref{simpBF (iii)}), construction and properties of $S^{(1)}$ for a single simplex in (\ref{simpBF (iv)},\ref{simpBF (v)}), and finally we put everything together in (\ref{simpBF (vi)},\ref{simpBF (vii)}).
\begin{enumerate}[(i)]
\item \label{simpBF (i)}
We construct $S_X^{(0)}:=\sum_{\Gamma_0}\Phi_X^{\Gamma_0}(\sA_X,\sB_X)$ where the sum is over binary rooted trees $\Gamma_0$ and the contributions $\Phi_X^{\Gamma_0}(\sA_X,\sB_X)$ are defined as follows, by putting decorations on half-edges of $\Gamma_0$ starting from leaves and going inductively to the root.\footnote{Our convention is that the leaves and the root are loose half-edges of the graph.}
\begin{itemize}
\item Leaves of $\Gamma_0$ are decorated by $\sfi_X(\sA_X)$.
\item In the internal vertices of $\Gamma_0$ one  
calculates the Lie bracket on $\Omega^\bt(X,\g)$ (coming from the wedge product on forms and the Lie bracket in $\g$) applied to the decorations of the two incoming half-edges and puts the result on the outgoing half-edge.
\item On internal edges one evaluates $-\sfK_X$ applied to the decoration of the in-half-edge and puts the result on the out-half-edge.
\item Finally, we define $\Phi_X^{\Gamma_0}(\sA_X,\sB_X):=\frac{1}{|\mr{Aut}(\Gamma_0)|}\langle \sB_X, \sfp_X ( \mathfrak{R}_X^{\Gamma_0}(\sA_X) ) \rangle$ where $\mathfrak{R}_X^{\Gamma_0}(\sA_X)$ is the decoration of the root coming from the assignments above.
\end{itemize}
By convention, the contribution of the ``trivial tree'' (with one leaf and no internal vertices or edges) is
$\Phi_X^{\mr{triv}}:= \langle \sB_X, d\sA_X \rangle$. We split $S_X^{(0)}=
S_{X}^{(0),2}
+S_{X}^{(0),\geq 3}$  where $S_{X}^{(0),2}$ is another notation for $\Phi_X^{\mr{triv}}$ 
and the second term is the contribution of non-trivial trees (the new superscripts $2$, $\geq 3$ denote the degree in fields).

\item \label{simpBF (ii)}
It follows from the fact that the data $\Omega^\bt(X)\wavy{(\sfi_X,\sfp_X,\sfK_X)} C^\bt(X)$ are assembled from the standard building blocks $\Omega^\bt(\bar\DDelta^N)\wavy{(\sfi_{\bar\DDelta^N},\sfp_{\bar\DDelta^N},\sfK_{\bar\DDelta^N})} C^\bt(\bar\DDelta^N)$, 
and from the factorization $\Omega^\bt(X,\g)=\Omega^\bt(X)\otimes\g$ into a tensor product of a cdga and a Lie algebra, 
that $S^{(0)}_X$ has the form
$S^{(0)}_X(\sA_X,\sB_X)=\sum_{e\subset X}\bar{S}^{(0)}_e(\sA_X|_{\bar{e}},\sB_e)$ where $\bar{S}^{(0)}_{\DDelta^N}$ satisfies the $\bmod\;\hbar$ part of the ansatz (\ref{Sbar}).

\item \label{simpBF (iii)}
Using Leibniz rule 
in $\Omega^\bt(X)$ and the identity $d\sfK_X+\sfK_X d=\mr{id}-\sfi_X\sfp_X$, one calculates the odd Poisson bracket 
$$\{\Phi_X^\mr{triv},\Phi_X^{\Gamma_0}\}=\sum_{\mr{edges}\;e\mr{\;of\;}\Gamma_0} \left(- \Phi_X^{\Gamma_0'} \frac{\overleftarrow{\dd}}{\dd \sA_X} \frac{\overrightarrow{\dd}}{\dd \sB_X} \Phi_X^{\Gamma_0''} +\Phi_X^{\Gamma_0,e} \right) $$ where on the r.h.s. we sum over edges $e$ of $\Gamma_0$; removing this edge splits $\Gamma_0$ into  trees $\Gamma_0'$ and $\Gamma_0''$. The term $\Phi_X^{\Gamma_0,e} $ is the contribution of $\Gamma_0$ with edge $e$ contracted; such contributions cancel when we sum over trees $\Gamma_0$ due to the combinatorics of trees and Jacobi identity in $C^\bt(X,\g)$. As a result, summing over $\Gamma_0$, we obtain 
$\{S_{X}^{(0),2},S_{X}^{(0),\geq 3}\}=-\frac12 \{S_{X}^{(0),\geq 3},S_{X}^{(0),\geq 3}\}$ which together with the obvious identity $\{S_{X}^{(0),2},S_{X}^{(0),2}\}=0 $ (a guise of $d^2=0$ on cochains) gives the classical master equation
\begin{equation}\label{simpBF CME} 
\{ S^{(0)}_X,S^{(0)}_X \}=0 
\end{equation}

\item \label{simpBF (iv)} For $N\geq 0$, denote by $Q_{\bar\DDelta^N}=\{S_{\bar\DDelta^N}^{(0)},\bt\}$ the cohomological vector field generated by $S_{\bar\DDelta^N}^{(0)}$ (the fact that it squares to zero follows from (\ref{simpBF CME}) for $X=\bar\DDelta^N$); it is tangent to $C^\bt(\bar\DDelta^N,\g)[1]\subset \sF_{\bar\DDelta^N}$. Observe that $\Delta S_{\bar\DDelta^N}^{(0)}\in \mr{Fun}_1(C^\bt(\bar\DDelta^N,\g)[1])$ is $Q_{\bar\DDelta^N}$-closed, as follows from (\ref{simpBF CME}) 
and the fact that $\Delta$ is a bi-derivation of $\{\bt,\bt\}$. 
Using homological perturbation theory (Lemma \ref{lemma: homological perturbation}), one constructs a retraction $\mr{Fun}(C^\bt(\bar\DDelta^N,\g)[1]),Q_{\bar\DDelta^N}\wavy{(\iota,\pi,\kappa)}\mr{Fun}(\g[1]),d_{CE}$ where the retract is the Chevalley-Eilenberg cochain complex of the Lie algebra $\g$.
By definition of chain homotopy, one has $Q_{\bar\DDelta^N}\,\kappa+\kappa\, Q_{\bar\DDelta^N}=\mr{id}-\iota\circ\pi$. Applying both sides of this equation to $\Delta S^{(0)}_{\bar\DDelta^N}$, and noticing that $\Delta S^{(0)}_{\bar\DDelta^N}$ is annihilated by $\pi$ (this follows from unimodularity of $\g$) and is $Q_{\bar\DDelta^N}$-closed, we obtain $Q_{\bar\DDelta^N}\kappa \Delta S^{(0)}_{\bar\DDelta^N}= \Delta S^{(0)}_{\bar\DDelta^N}$. Therefore, the element 
$S_{\bar\DDelta^N}^{(1)}(\sA_{\bar\DDelta^N})\in \mr{Fun}_0(C^\bt(\bar\DDelta^N,\g)[1])$ constructed as 
\begin{equation}\label{simpBF S^1 on Delta^N}
S_{\bar\DDelta^N}^{(1)}:=-\kappa\Delta S_{\bar\DDelta^N}^{(0)} 
\end{equation}
is a solution  of the equation 
\begin{equation}\label{simpBF QME Delta^N}
\Delta S_{\bar\DDelta^N}^{(0)}+ \{S_{\bar\DDelta^N}^{(0)}, S_{\bar\DDelta^N}^{(1)}\}=0
\end{equation}
%

\item \label{simpBF (v)} One defines 
\begin{equation}\label{simpBF barS^1 via S^1}
\bar{S}_{\DDelta^N}^{(1)}(\sA_{\bar\DDelta^N})=\sum_{e\subset \bar\DDelta^N} (-1)^{\mr{codim}(e)} S_{\bar{e} 
}^{(1)}(\sA_{\bar\DDelta^N}|_{\bar{e}})
\end{equation}
with summands as in (\ref{simpBF (iv)}) for faces $e$ of $\Delta^N$ of arbitrary codimension. One can check from the construction (\ref{simpBF S^1 on Delta^N}), that  $\bar{S}_{\DDelta^N}^{(1)}$ satisfies the part of ansatz (\ref{Sbar}) linear in $\hbar$; the property 
$S_{\bar\DDelta^N}^{(1)}=\sum_{e\subset \bar\DDelta^N} \bar{S}^{(1)}_e$ is obvious from (\ref{simpBF barS^1 via S^1}). 

\item \label{simpBF (vi)}
Assume that a simplicial complex $X$ is given as union of two simplicial subcomplexes $X_1$ and $X_2$ intersecting over $Y=X_1\cap X_2$ and assume that the elements $S_{X_1}$, $S_{X_2}$, $S_{Y}$ as defined by (\ref{simpBF locality}) satisfy the quantum master equation on $\sF_{X_1}$, $\sF_{X_2}$, $\sF_{Y}$ respectively. Then it is straightforward to check that $S_X=S_{X_1}+S_{X_2}-S_{Y}$ (which indeed also satisfies the ansatz (\ref{simpBF locality})) is a solution of the quantum master equation on $\sF_X$.\footnote{
Indeed, we already know by (\ref{simpBF CME}) that $\{S^{(0)}_X,S^{(0)}_X\}=0$; we are left to check that $\Delta S^{(0)}_X+\{S^{(0)}_X,S^{(1)}_X\}=0$. We calculate $\Delta S^{(0)}_X=\Delta_{X_1} S^{(0)}_{X_1}+\Delta_{X_2} S^{(0)}_{X_2}-\Delta_{Y} S^{(0)}_{Y}= -\{S^{(0)}_{X_1},S^{(1)}_{X_1}\}_{X_1}- \{S^{(0)}_{X_2},S^{(1)}_{X_2}\}_{X_2} +\{S^{(0)}_{Y},S^{(1)}_{Y}\}_{Y}=-\{S^{(0)}_X,S^{(1)}_{X_1}+S^{(1)}_{X_2}-S^{(1)}_Y\}=-\{S^{(0)}_X,S^{(1)}_X\}$. (Here we indicate explicitly where we calculate the odd Poisson brackets and BV Laplacians; no index means $X$.)
}

\item \label{simpBF (vii)}
We break the simplicial complex $X$ into individual simplices $e$; on each of them we have a ``building block'' solution of the quantum master equation of form $S_{\bar\DDelta^N}:=S^{(0)}_{\bar\DDelta^N}-i\hbar S^{(1)}_{\bar\DDelta^N}$ as constucted in (\ref{simpBF (i)}) (setting $X=\bar\DDelta^N$) and (\ref{simpBF (iv)}), with $N=\dim(e)$. The fact that $S_{\bar\DDelta^N}$ satisfies the quantum master equation is (\ref{simpBF CME}) specialized to $\bar\DDelta^N$, put together with (\ref{simpBF QME Delta^N}). 
Then, by (\ref{simpBF (vi)}), the element $S_X$ as defined by (\ref{simpBF locality}) satisfies the quantum master equation on $\sF_X$.
Finally, property (\ref{simpBF item init cond}) is obvious from the construction.
\end{enumerate}
\end{proof}

\begin{remark} Examples of values of structure constants $C$ for a simplex $\DDelta^N$ in Theorem \ref{thm: simpBF}:
\begin{center}
\begin{tabular}{c|l}
$\Gamma,\{e_i\}$ & $C^{\DDelta^N}_{\Gamma,\{e_i\}}$ \\
\hline
\includegraphics[scale=0.6]{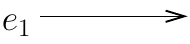} & 
$\left\{\begin{array} {ll}
\pm 1 & \mr{if}\; |e_1|=N-1 \\
0 & \mr{otherwise}
\end{array}\right.$
\\
$\vcenter{\hbox{\includegraphics[scale=0.6]{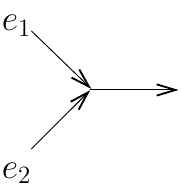}}}$ & 
$\left\{\begin{array} {ll}
\pm\frac{|e_1|!|e_2|!}{(N+1)!} & \mr{if}\; |e_1|+|e_2|=N\;\mr{and}\; |e_1\cap e_2|=0 \\
0 & \mr{otherwise}
\end{array}\right.$
\\
$\vcenter{\hbox{\includegraphics[scale=0.4]{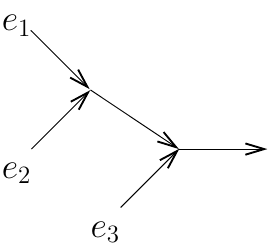}}}$ & 
$\left\{\begin{array} {ll}
\pm\frac{|e_1|!|e_2|!|e_3|!}{(|e_1|+|e_2|+1)\cdot(N+2)!} & \mr{if}\; |e_1|+|e_2|+|e_3|=N+1,\\ & \mr{and}\;|e_1\cap e_2|=0,\\ &  \mr{and}\;\{|e_1\cap e_3|,|e_2\cap e_3|\}=\{0,1\} \\
0 & \mr{otherwise}
\end{array}\right.$
\\
$\vcenter{\hbox{\includegraphics[scale=0.35]{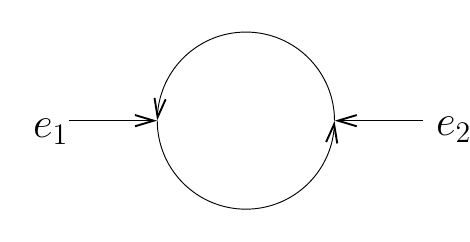}}}\qquad$ & 
$\left\{\begin{array} {ll}
\frac{(-1)^{N+1}}{(N+1)^2(N+2)} & \mr{if}\; e_1=e_2\;\mr{and}|e_1|=1  \\
0 & \mr{otherwise}
\end{array}\right.$
\end{tabular}
\end{center}
Here the signs and nonvanishing conditions are formulated in terms of combinatorics and orientations of the $\n$-tuple of (arbitrary codimension) faces $e_1,\ldots,e_\n$ of $\bar\DDelta^N$; $|\cdots|$ stands for dimension. The top two graphs correspond, upon summing over the simplices, as in (\ref{simpBF locality}), to the cellular differential and the bracket $[\bt\stackrel{*}{,}\bt]$ (the projected wedge product of forms tensored with the Lie bracket in $\g$) on cochains and thus to the first two terms of (\ref{simpBF init cond}). The two bottom graphs give first nontrivial contributions to $\mathcal{R}$ in (\ref{simpBF init cond}).
\end{remark}

\begin{remark} \label{rem: low dim simplices}
Building block (\ref{Sbar}) for a $0$-simplex $\DDelta^0$ is simply
$\bar{S}_{\DDelta^0}=\frac12 \left\langle \sB_{[0]},[\sA_{[0]},\sA_{[0]}]\right\rangle$ where we denoted the only cell $[0]:=\DDelta^{0}$ (and we are suppressing the subscript $\g$ in $\lan,\ran$ and $\tr$). This corresponds to having a single nonvanishing structure constant $C^{\DDelta^0}_{\vcenter{\hbox{\includegraphics[scale=0.2]{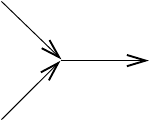}}},[0],[0]}=1$.

The nonvanishing structure constants $C$ of Theorem \ref{thm: simpBF} for the $1$-simplex are as follows. We denote the top $1$-cell as $[01]$ and the boundary $0$-cells as $[0]$ and $[1]$.
\begin{center}
\begin{tabular}{c|l}
$\Gamma,\{e_i\}$ & $C^{\DDelta^1}_{\Gamma,\{e_i\}}$ \\
\hline
$\vcenter{\hbox{\includegraphics[scale=0.4]{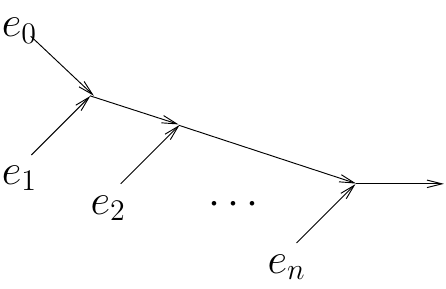}}}$ & 
$\left\{\begin{array}{ll} 
(-1)^n\frac{B_n}{n!} & \mr{if}\; \{e_0,e_1\}=\{[1],[01]\},\;e_2=\cdots=e_n=[01] \\
-\frac{B_n}{n!} & \mr{if}\; \{e_0,e_1\}=\{[0],[01]\},\;e_2=\cdots=e_n=[01]\\
0 & \mr{otherwise}
\end{array}\right.$
\\
$\vcenter{\hbox{\includegraphics[scale=0.35]{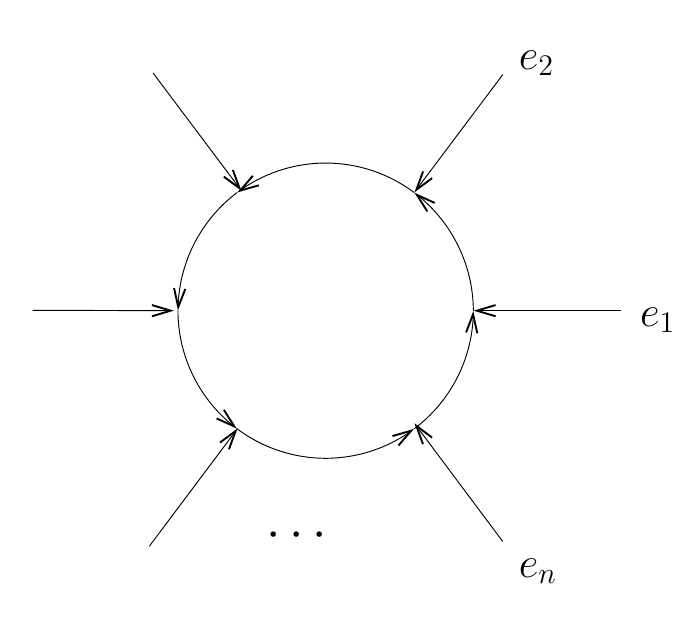}}}\qquad$ & 
$\left\{\begin{array}{ll}
\frac{B_n}{n!} & \mr{if}\; e_1=\cdots=e_n=[01]\\
0 & \mr{otherwise}
\end{array}\right.$
\end{tabular}
\end{center}
Here $B_n$ are the Bernoulli numbers, $B_0=1,\, B_1=-\frac12,\, B_2=\frac16,\, B_3=0, B_4=-\frac{1}{30},\ldots$ In particular, the building block 
(\ref{Sbar}) for $\DDelta^1$ is
\begin{multline}\label{Sbar interval} 
\bar{S}_{\DDelta^1}=\lan \sB_{[01]}, \sA_{[1]}-\sA_{[0]} \ran+\frac12 \lan \sB_{[01]}, [\sA_{[01]},\sA_{[0]}+\sA_{[1]}] \ran-\frac{1}{12}\lan \sB_{[01]}, [\sA_{[01]},[\sA_{[01]},\sA_{[1]}-\sA_{[0]}]] \ran + \cdots -\\
-i\hbar\left(\frac{1}{24} \mr{tr}\, [\sA_{[01]},[\sA_{[01]},\bt]]+\cdots \right)=\\
=\left\langle \sB_{[01]}, \frac12\left[\sA_{[01]},\sA_{[0]}+\sA_{[1]}\right]+\mathbb{F}(\mr{ad}_{\sA_{[01]}})\circ (\sA_{[1]}-\sA_{[0]}) \right\rangle -i\hbar\; \mr{tr} \log \mathbb{G}(\mr{ad}_{\sA_{[01]}}) 
\end{multline}
where we introduced the two functions 
\begin{equation}\label{simpBF interval F,G}
\mathbb{F}(x)=\frac{x}{2}\coth\frac{x}{2},\qquad \mathbb{G}(x)=\frac{2}{x}\sinh\frac{x}{2}
\end{equation}
\end{remark}

\begin{remark}
\label{rem: prismatic}
We have required $X$ to be a \emph{simplicial} complex rather than a general $CW$ complex because for the 
Theorem \ref{thm: simpBF} we need retractions $\Omega^\bt(\bar{e})\wavy{}C^\bt(\bar{e})$ for cells $e$ of $X$, compatible with restriction to cells of $\partial e$ (if we want ``standard building blocks'' as in Theorem \ref{thm: simpBF}, we should also require that the retraction depends only on the combinatorial type of $\bar{e}$ and is compatible with combinatorial symmetries of the cell). In the case of simplices, such retractions are provided by Whitney forms and Dupont's chain homotopy operator (whereas Poincar\'e map works for cells of any type). More generally, we can allow $X$ to be a \emph{prismatic complex}, with cells $e=\sigma_1\times\cdots\times \sigma_r$ being \emph{prisms} -- products of simplices (of arbitrary dimension). The respective retraction $\Omega^\bt(\bar{e})\wavy{}C^\bt(\bar{e})$ is constructed from Whitney/Dupont retractions for simplices by the tensor product construction of Section \ref{sec: induction data for dual and sym} (whenever we have several simplices of same dimension in $e$, we average over the order in which we contract the factors, in order to have the retraction compatible with the symmetries of the prism). The special case when the prisms are \emph{cubes} (products of 1-simplices) was considered in detail in \cite{DiscrBF}. As a result, one can allow $X$ in Theorem \ref{thm: simpBF} to be a prismatic complex; the building blocks $\bar{S}_e$ then depend on the combinatorial type of the prism $e$ (dimensions of the simplex factors) and the structure constants $C^{e}_{\Gamma_0,e_1,\ldots,e_n}$, $C^e_{\Gamma_1,e_1,\ldots,e_n}$ depend on the combinatorics of an $n$-tuple of faces of the prism $e$. In Section \ref{sec: non-ab CW} we will construct a further generalization of Theorem \ref{thm: simpBF} to general regular CW complexes.
\end{remark}

\subsection{Case of general CW complexes}\label{sec: non-ab CW}
One can extend Theorem \ref{thm: simpBF} to a general regular CW complex, with cells not required to be simplices or prisms.

\begin{theorem}\label{thm: cellBF}
Let $X$ be a finite regular CW complex and $\g$ be a unimodular Lie algebra. Then to every cell $e\subset X$ one can associate an element (a \emph{local building block})
$\bar{S}_e\in\mr{Fun}(\sF_{\bar{e}})[[\hbar]]$ of the form
\begin{multline} \label{Sbar cell}
\bar{S}_{e}(\sA_{\bar{e}},\sB_{e};\hbar)=\\
=\sum_{n=1}^\infty\sum_{\Gamma_0}\sum_{e_1,\ldots,e_n\subset \bar{e}}\frac{1}{|\mr{Aut}(\Gamma_0)|} C^e_{\Gamma_0,e_1,\ldots,e_n}
\left\langle \sB_{e},\mr{Jacobi}_{\,\Gamma_0}(\sA_{e_1},\ldots,\sA_{e_n})\right\rangle_\g-\\
-i\hbar \sum_{n= 2}^\infty \sum_{\Gamma_1} 
\sum_{e_1,\ldots,e_n\subset \bar{e}}\frac{1}{|\mr{Aut}(\Gamma_1)|} C^e_{\Gamma_1,e_1,\ldots,e_n}
\mr{Jacobi}_{\,\Gamma_1}(\sA_{e_1},\ldots,\sA_{e_n})
\end{multline}
with some real coefficients $C^e_{\Gamma_l,e_1,\ldots,e_n}\in \mathbb{R}$ for $l=0,1$, in such a way that the element (the \emph{cellular action})
\begin{equation}\label{cellBF locality}
S_X(\sA_X,\sB_X;\hbar)=\sum_{e\subset X}\bar{S}_e(\sA_X|_{\bar{e}},\sB_e;\hbar)\;\;\in \mr{Fun}(\sF_X)[[\hbar]]
\end{equation}
satisfies the quantum master equation $\Delta_X e^{\frac{i}{\hbar}S_X}=0$ and the following conditions:
\begin{enumerate}[(a)]
\item \label{cellBF thm a} the property 
\begin{equation}\label{cellBF: a eq}
S_X(\sA_X,\sB_X;\hbar)=\lan \sB_X, d\sA_X \ran+\mathfrak{r}
\end{equation}
with $\mathfrak{r}\in \widehat{\Sym^{\geq 3}} \sF_X^*\oplus \hbar\cdot \widehat{\Sym^{\geq 2}} \sF^*_X$,
\item \label{cellBF thm b} for $e$ any $0$-cell of $X$, one has
\begin{equation}\label{cellBF S_pt}
\bar{S}_e=\lan \sB_e,\frac12[\sA_e,\sA_e] \ran_\g
\end{equation}
\item \label{cellBF thm c} For $Y\subset X$ any subcomplex, $S_Y(\sA_Y,\sB_Y;\hbar):=\sum_{e\subset Y}\bar{S}_e\;\in\mr{Fun}(\sF_Y)[[\hbar]]$ satisfies the quantum master equation on $\sF_Y$.
\end{enumerate}
\end{theorem}

\begin{proof}
Choose an ordering of
the cells $e_1,\ldots,e_N$ of $X$ in such a way that $\dim e_1\leq\cdots\leq \dim e_N$. Then $X$ admits an increasing filtration by CW subcomplexes 
\begin{equation}\label{cellBF CW filtration}
X_1\subset \cdots \subset X_{N-1}\subset X_N=X
\end{equation}
with $X_k:=\cup_{i\leq k} e_i$. Then $X_k=X_{k-1}\cup e_k$. That is, $X_k$ is $X_{k-1}$ with a single new cell $e_k$ adjoined; its boundary $\dd e_k$ lies entirely in $X_{k-1}$.

Proceeding by induction in $k$, assume that a solution $S_{X_{k-1}}$ of the QME  on $\sF_{X_{k-1}}$ is constructed and we want to extend it to $\sF_{X_k}=\sF_{X_{k-1}}[\sA_{e},\sB_{e}]$ (we temporarily denote $e:=e_k$). Denote $n=\dim e$. First, we look for a solution of the classical master equation $S^{(0)}_{X_k}$ of the following form
\begin{equation}\label{cellBF: S induction}
S^{(0)}_{X_k}=S^{(0)}_{X_{k-1}}
+\underbrace{
\sum_{j\geq 1} \lan \sB_{e}, \sigma_{j}(\sA_{e},\sA_{\dd e}) \ran}_{\bar{S}^{(0)}_{e}} 
\end{equation}
where $\sigma_j$ is of polynomial degree $j$ in $\sA$-variables (we call this polynomial degree \textit{weight} to distinguish it from other gradings). We assume $n\geq 1$ (otherwise we simply set $\bar{S}_e=\lan \sB_e,\frac12 [\sA_e,\sA_e] \ran$ as prescribed by (\ref{cellBF thm b})).

We denote by
$$d_{\dd e\ra e}= e^*\cdot \mathsf{d}_{\dd e\ra e}:\quad  C^{n-1}(\dd e) \;\; \ra \;\; C^{n}(\bar{e},\dd e)=\mr{Span}(e^*)$$ the component of the cellular coboundary operator on $X_k$ proportional to $e^*$ (and $\sd_{\dd e\ra e}: C^{n-1}(\dd e)\ra \mathbb{R}$ picks the coefficient of $e^*$). Thus, the cellular coboundary operator on $C^\bt(\bar{e})$ splits as $d_{\bar{e}}=d_{\dd e}+d_{\dd e\ra e}$ with $d_{\dd e}$ the cellular coboundary operator on $C^\bt(\dd e)$. 
Let us choose some induction data 
\begin{equation}\label{cellBF: cell ind data}
C^\bt(\bar{e}),d_{\bar{e}} 
\wavy{(i,p,K)} H^\bt(\bar{e})
=\left\{\begin{array}{ll} \mathbb{R}, & \bt=0\\ 0,& \bt\neq 0 \end{array}\right. 
\end{equation}
Consider the space 
$$\EE^j:=\Sym^j(C^\bt(\bar{e},\g)[1])^* \qquad \ni f(\sA_{e},\sA|_{\dd e})$$
The operator $d_{\dd e\ra e}$ lifts to a weight zero operator\footnote{The sign is chosen in such a way that $D\sA=d_{\dd e\ra e}\sA$ holds.} 
$$D=(-1)^n\lan  \sd_{\dd e\ra e}\; \sA_{\dd e},\frac{\dd}{\dd \sA_{e}} \ran :\quad  \EE^{\bt}\ra \EE^{\bt}$$
Denote by $Q_{\dd e}$ the differential on $\Sym (C^\bt(\dd e,\g)[1])$ induced by $\{S_{X_{k-1}}^{(0)},\bt\}$; it extends (via $Q_{\dd e}\sA_e=0$) to $\EE^\bt$. We split $Q_{\dd e}$ according to weight as $Q_{\dd e}=Q_{\dd e}^0+Q_{\dd e}^1+\cdots$. In particular, $Q^0_{\dd e}+D$ is simply the lift of the cellular coboundary operator $d_{\bar{e}}$ to $\EE^\bt$.
Thus, using the construction (\ref{HPT Sym}), one produces, out of the triple $(i,p,K)$ chosen above, the induction data 
\begin{equation}\label{cellBF: E ind data}
\EE^\bt,Q^0_{\dd e}+D \wavy{(i_\EE,p_\EE,K_\EE)} H^\bt(\EE)\cong \Sym^\bt (\g[1])^*
\end{equation} Note that $i$ in (\ref{cellBF: cell ind data}) is canonical: it has to represent $H^0(\bar{e})$ by constant $0$-cochains; thus $p_\EE=i^*$ is also canonical and is given by evaluation on constant $0$-cochains.

We set in (\ref{cellBF: S induction}) the ``initial condition'' $\sigma_1:=\sd_{\dd e\ra e}\, \sA_{\dd e}$, which is forced by (\ref{cellBF: a eq}).
The classical master equation $\{S^{(0)}_{X_k},S^{(0)}_{X_k}\}=0$ is equivalent to a sequence of equations for the functions $\sigma_j$ for $j\geq 2$:
\begin{equation}\label{cellBF CME induction}
(D+Q^0_{\dd e}) \sigma_{j}=-\sum_{i=1}^{j-1}Q^i_{\dd e}\sigma_{j-i}-(-1)^n\sum_{i=2}^{j-1} \lan \sigma_{j+1-i},\frac{\dd}{\dd \sA_{e}} \ran \sigma_{i} 
\end{equation}
Note that the r.h.s. of (\ref{cellBF CME induction}) depends only on $\sigma_1,\ldots,\sigma_{j-1}$. We solve (\ref{cellBF CME induction}), as an equation for $\sigma_{j}\in \g\otimes \EE^{j}$, by induction in $j$. The r.h.s. is $(D+Q^0_{\dd e})$-closed\footnote{
Indeed, using the induction hypothesis we calculate
{\Tiny
\begin{multline*}
(D+Q^0_{\dd e})\left(\mbox{r.h.s. of }(\ref{cellBF CME induction})
\right)=
\sum_{i=1}^{j-1}\sum_{l=0}^{i-1}Q^{i-l}_{\dd e}Q^{l}_{\dd e}\sigma_{j-i}+\sum_{i=1}^{j-1}Q_{\dd e}^i D \sigma_{j-i}-(-1)^n\sum_{i=1}^{j-1}\lan Q^i_{\dd e}\sigma_1, \frac{\dd}{\dd \sA_e} \ran \sigma_{j-i}-\\
-(-1)^n \sum_{i=2}^{j-1} \lan (D+Q^0_{\dd e})\sigma_{j+1-i},\frac{\dd}{\dd \sA_e} \ran \sigma_i 
+(-1)^n \sum_{i=2}^{j-1} \lan \sigma_{j+1-i},\frac{\dd}{\dd \sA_e} \ran (D+Q^0_{\dd e})\sigma_i=:a+b+c+d+e
\end{multline*}
}
Further, we have 
{\tiny
\begin{multline*}
a+b=\sum_{i=1}^{j-1}Q^i_{\dd e}\left((D+Q^0_{\dd e})\sigma_{j-i}+\sum_{l=1}^{j-i-1}Q_{\dd e}^l \sigma_{j-i-l}\right)=-(-1)^n\sum_{i=1}^{j-1}\,\sum_{r,s\geq 2,r+s=j-i+1}Q^i_{\dd e}\left( \lan \sigma_r, \frac{\dd}{\dd \sA_e} \ran \sigma_s \right)\\
= -(-1)^n\sum_{i=1}^{j-1}\,\sum_{r,s\geq 2,r+s=j-i+1}\lan Q^i_{\dd e}\sigma_r, \frac{\dd}{\dd \sA_e} \ran \sigma_s +(-1)^n\sum_{i=1}^{j-1}\,\sum_{r,s\geq 2,r+s=j-i+1}\lan \sigma_r, \frac{\dd}{\dd \sA_e} \ran Q^i_{\dd e}\sigma_s =:f+g
\end{multline*}
}
Next we note that $c+d+f=\sum\limits_{r,s,t\geq 2, r+s+t=j+2}\lan \lan \sigma_r ,\frac{\dd}{\dd \sA_e} \ran \sigma_s, \frac{\dd}{\dd \sA_e}\ran \sigma_t$ and
$e+g=-\sum\limits_{r,s,t\geq 2, r+s+t=j+2} \lan \sigma_r ,\frac{\dd}{\dd \sA_e} \ran \left( \lan\sigma_s, \frac{\dd}{\dd \sA_e}\ran \sigma_t\right)$, and thus 
$c+d+f+e+g=\sum\limits_{r,s,t\geq 2, r+s+t=j+2}  \lan\sigma_r \sigma_s, \frac{\dd}{\dd \sA_e}\frac{\dd}{\dd \sA_e} \ran \sigma_t=0$ -- vanishes as a contraction of a symmetric and a skew-symmetric tensor. Thus we proved that $(D+Q^0_{\dd e})\left(\mbox{r.h.s. of }(\ref{cellBF CME induction})\right)=0$.
} and is annihilated by $p_\EE$.\footnote{
To see that the r.h.s. of (\ref{cellBF CME induction}) is annihilated by $p_\EE$, note that in $H^\bt(\EE)$, the weight coincides with the internal degree. On the other hand, the weight of (\ref{cellBF CME induction}) is $j\geq 2$ while the internal degree is $3-n\leq 2$. Thus $p_\EE(\mbox{r.h.s. of }(\ref{cellBF CME induction}))$ is zero for degree reasons, except for the case $n=1$ (i.e. $e$ is an interval) with $j=2$, where one sees explicitly that $-Q^1_{\dd e}\sigma_1=-\frac12 \left[\sA_{[1]},\sA_{[1]}\right]+\frac12 \left[\sA_{[0]},\sA_{[0]}\right]$ vanishes on constant $0$-cochains (we denoted the endpoints of the interval $e$ by $[0]$ and $[1]$, as in Remark \ref{rem: low dim simplices}).
} Therefore, it is exact, and one can construct the primitive as
$$\sigma_j:= K_\EE \left(\mbox{r.h.s. of }(\ref{cellBF CME induction})\right)$$

Thus we have constructed a solution of the classical master equation  (\ref{cellBF: S induction}) on $X_k$ by extension of the known one (by induction hypothesis) on $X_{k-1}$. 

Next, we want to construct $S^{(1)}_{X_k}=S^{(1)}_{X_{k-1}}+\bar{S}^{(1)}_e(\sA_e,\sA_{\dd e})$ in such a way that $S_{X_k}=S^{(0)}_{X_k}-i\hbar\, S^{(1)}_{X_k}$ satisfies the quantum master equation; we assume that $S^{(1)}_{X_{k-1}}$ is already constructed. We use the strategy of (\ref{simpBF (iv)}) of the proof of Theorem \ref{thm: simpBF}: the QME can be written as 
\begin{equation}\label{cellBF QME on e}
Q_{\bar{e}} S^{(1)}_{\bar{e}}=-\Delta S^{(0)}_{\bar{e}}
\end{equation} 
where $Q_{\bar{e}}=Q_{\dd e}+D+(-1)^n\sum_{j\geq 2}\lan\sigma_j,\frac{\dd}{\dd \sA_e}\ran$ is the differential on $\EE^\bt$ induced by $\{S^{(0)}_{X_k},\bt\}$ and $S^{(0,1)}_{\bar{e}}:=\sum_{e'\subset e}\bar{S}^{(0,1)}_{e'}$.  We deform the induction data (\ref{cellBF: E ind data}), using Lemma \ref{lemma: homological perturbation}, to
$$\EE^\bt,Q_{\bar{e}} \wavy{(\widetilde{i}_\EE,\widetilde{p}_\EE,\widetilde{K}_\EE)} \Sym^\bt(\g[1])^*,d_{CE}$$
The r.h.s. of (\ref{cellBF QME on e}) is $Q_{\bar{e}}$-closed, as follows from the classical master equation, and is annihilated by $\widetilde{p}_\EE$ (from 
unimodularity of $\g$), thus we construct a solution of (\ref{cellBF QME on e}) as
$S^{(1)}_{\bar{e}}:=-\widetilde{K}_\EE\Delta S^{(0)}_{\bar{e}}$ and set $$\bar{S}^{(1)}_e:=S^{(1)}_{\bar{e}}-\sum_{e'\subset \dd e}\bar{S}^{(1)}_{e'}$$

This finishes the construction of a solution of the quantum master equation on $X_k$ of the form 
$S_{X_k}=\sum_{i\leq k} \bar{S}_{e_i}$
with $\bar{S}_{e_i}=\bar{S}^{(0)}_{e_i}-i\hbar\,\bar{S}^{(1)}_{e_i}$; taking $k=N$ we obtain the statement of the Theorem. 

Property (\ref{cellBF thm c}) follows from the possibility to choose the ordering of cells $e_1,\ldots,e_N$ differently (preserving the nondecreasing dimension property) while preserving the choice of induction data (\ref{cellBF: cell ind data}) for the cells. Each $Y\subset X$ arises as $X_k$ for some ordering of cells and some $k$, which implies the QME on $Y$.
\end{proof}

\begin{lemma}\label{lemma: cellBF well-definedness}
Cellular action $S_X$ of the Theorem \ref{thm: cellBF} is well-defined up to a canonical BV transformation. More precisely, let $S_X$ and $S'_X$ be two cellular actions fulfilling the conditions of Theorem \ref{thm: cellBF} -- the quantum master equation,  the ansatz (\ref{cellBF locality},\ref{Sbar cell}) and  properties (\ref{cellBF thm a})--(\ref{cellBF thm c}). 
Then one can construct a family $S_{X,t}\in \mr{Fun}_0(\sF_X)[[\hbar]]$ of solutions of QME for $t\in [0,1]$ together with a generator $R_{X,t}\in \mr{Fun}_{-1}(\sF_X)[[\hbar]]$ such that
\begin{enumerate}[(i)]
\item \label{lemma 8.7 (i)} $S_{X,0}=S_X$, $S_{X, 1}=S'_X$,
\item \label{lemma 8.7 (ii)} $\frac{\dd}{\dd t}S_{X,t}=\{S_{X,t},R_{X,t}\}-i\hbar \Delta R_{X,t}$,
\item \label{lemma 8.7 (iii)} both $S_{X,t}$ and $R_{X,t}$ satisfy the ansatz (\ref{cellBF locality},\ref{Sbar cell}), with $t$-dependent structure constants $C^{S,e}_{\Gamma,e_1,\ldots, e_n}(t)$, $C^{R,e}_{\Gamma,e_1,\ldots, e_n}(t)$ for the action $S_{X,t}$ and the generator of infinitesimal canonical transformation $R_{X,t}$. 
Moreover, the \emph{trivial tree} $\Gamma_0^\mr{triv}$ with a single leaf and no internal vertices has coefficient $C^{R,e}_{\Gamma_0^\mr{triv},e_1}=0$).
\end{enumerate}
\end{lemma}

\begin{proof}
Consider the filtration (\ref{cellBF CW filtration}). We proceed by induction in $X_k$: assuming that the Lemma holds for $X_{k-1}$, we aim to prove it for $X_k$. (Note that the Lemma holds trivially for $X_1$, since $S_{X_1}$ is fixed uniquely by (\ref{cellBF thm b}) of Theorem \ref{thm: cellBF}.) 
Set $S_{X_k}:=\sum_{l=1}^k \bar{S}_{e_l}$, $S'_{X_k}:=\sum_{l=1}^k \bar{S}'_{e_l}$. 
By induction hypothesis, cellular actions $S_{X_{k-1}}$ and $S'_{X_{k-1}}$ 
(defined 
as above, omitting the $k$-th term in the sums) 
are connected by a canonical BV transformation, which we denote schematically by $\mathcal{R}_{X_{k-1}}$, i.e.  $S'_{X_{k-1}}=\mathcal{R}_{X_{k-1}}\circ S_{X_{k-1}}$. 

Define 
\begin{equation}\label{cellBF lemma: S''}
S''_{X_k}:=\R_{X_{k-1}}^{-1}S'_{X_k}
\end{equation} 
We have then $S_{X_k}=S_{X_{k-1}}+\bar{S}_e$ and $S''_{X_k}=S_{X_{k-1}}+\bar{S}''_e$ (for some $\bar{S}''_e(\sA_{\bar{e}},\sB_e;\hbar)$ satisfying the ansatz (\ref{Sbar cell})). 
We can connect\footnote{
Indeed, the most general construction of extension of a solution of QME from $X_{k-1}$ to $X_k$ is as in our proof of Theorem \ref{thm: cellBF}, where on each step of induction in $j$ we can shift $\sigma_j\ra \sigma_j+(D+Q^{(0)}_{\dd e})(\cdots)$, also we can shift $\bar{S}^{(1)}_e\ra \bar{S}^{(1)}_e+Q_{\bar{e}}(\cdots)$. This amounts to a contractible space of choices. Thus the space of solutions of QME on $X_k$ of form $S_{X_{k-1}}+\bar{s}_{e}$, with a fixed solution $S_{X_{k-1}}$ of QME on $X_{k-1}$ and an indeterminate function $\bar{s}_{e}$ satisfying ansatz (\ref{Sbar cell}), is contractible; in particular, it is path-connected.
} $S_{X_k}$ and $S''_{X_k}$ by a path of solutions of QME on $\sF_{X_k}$ of the form $S_{X_{k},t}=S_{X_{k-1}}+\bar{S}_{e,t}$ for $t\in [0,1]$, with $\bar{S}_{e,t}=\bar{S}_{e,t}^{(0)}-i\hbar\, \bar{S}_{e,t}^{(1)}=\lan \sB_e, \sigma_t(\sA_{\bar{e}}) \ran-i\hbar \, \rho_t (\sA_{\bar{e}})$ for some $t$-dependent functions $\sigma_t\in \g\otimes \EE^{\geq 2}$, $\rho_t\in \EE^{\geq 2}$ of $\sA$-variables of weight $\geq 2$. (We are borrowing the notations of the proof of Theorem \ref{thm: cellBF}, in particular $e:=e_k$.)

Differentiating QME for $S_{X_t}$ in $t$, we obtain
\begin{eqnarray}
\{S_{X_k,t}^{(0)},\dot{\bar{S}}_{e,t}^{(0)}\}&=&0 \label{cellBF lemma: CME} \\
\{S^{(1)}_{X_k,t},\dot{\bar{S}}_{e,t}^{(0)}\}+ \{S^{(0)}_{X_k,t},\dot{\bar{S}}_{e,t}^{(1)}\}+\Delta \dot{\bar{S}}_{e,t}^{(0)} &=& 0 \label{cellBF lemma: QME}
\end{eqnarray}
where the dot stands for the derivative in $t$.

Observe that: 
\begin{enumerate}[(a)]
\item\label{cellBF lemma: coh a} The cohomology of the differential $\{\lan \sB_{X_k}, d\sA_{X_k}\ran,\bt\}$ on the subcomplex $\Xi^{(0)}:=\{\lan \sB_e, f(\sA_{\bar{e}}) \ran\,|\,f\in \g\otimes\EE^{\geq 2}\}\subset \mr{Fun}(\sF_{X_k})$ vanishes in internal degree zero.\footnote{
For this we assume $\dim e\geq 1$; in the case of $\dim e=0$ the induction step of the Lemma works trivially as $\bar{S}_e$ is fixed uniquely by (\ref{cellBF thm b}) of Theorem \ref{thm: cellBF}.
} 
This implies, by homological perturbation theory, that degree zero cohomology of $\{S^{(0)}_{X_k,t},\bt\}$ on $\Xi^{(0)}$ vanishes as well. 
\item\label{cellBF lemma: coh b} By a similar argument, degree zero cohomology of $\{S^{(0)}_{X_k,t},\bt\}$ on $\Xi^{(1)}:=\{g(\sA_{\bar{e}})\in \EE^{\geq 2}\}\subset \mr{Fun}(\sF_{X_k})$ also vanishes. 
\end{enumerate}

Thus, by (\ref{cellBF lemma: CME}) and (\ref{cellBF lemma: coh a}), we have $\dot{\bar{S}}_{e,t}^{(0)}=\{S_{X_k,t}^{(0)},r^{(0)}_{e,t}\}$ for some degree $-1$ element $r^{(0)}_{e,t}\in \Xi^{(0)}$. Substituting this into (\ref{cellBF lemma: QME}), we obtain, using Jacobi identity for $\{,\}$ and the QME on $S_{X_k,t}$, that
$\left\{S^{(0)}_{X_k,t},\dot{\bar{S}}_{e,t}^{(1)}-\{S^{(1)}_{X_k,t},r^{(0)}_{e,t}\}-\Delta r^{(0)}_{e,t}\right\}=0$. This implies, by (\ref{cellBF lemma: coh b}), that 
$\dot{\bar{S}}_{e,t}^{(1)}=\{S^{(1)}_{X_k,t},r^{(0)}_{e,t}\}+\Delta r^{(0)}_{e,t}+\{S^{(0)}_{X_k,t},r^{(1)}_{e,t}\}$
for some degree $-1$ element $r^{(1)}_{e,t}\in \Xi^{(1)}$. Thus we have proven that 
\begin{equation}\label{cellBF lemma: can transf}
\dot{S}_{X_k,t}=\{S_{X_k,t},r_{e,t}\}-i\hbar\,\Delta r_{e,t}
\end{equation}
with the generator $r_{e,t}=r^{(0)}_{e,t}-i\hbar\, r^{(1)}_{e,t}\in \, \Xi^{(0)}-i\hbar\,\Xi^{(1)}$. One can use homological perturbation theory to construct an explicit chain contractions of $\left(\Xi^{(0)},\{S^{(0)}_{X_k,t},\bt\}\right)$ and $\left(\Xi^{(1)},\{S^{(0)}_{X_k,t},\bt\}\right)$ and use them to construct $r^{(0)}_{e,t}$, $r^{(1)}_{e,t}$. By inspection of the construction, the resulting generator $r_{e,t}$ satisfies the ansatz (\ref{Sbar cell}) for some $t$-dependent structure constants $C^{r,e}_{\Gamma,e_1,\ldots,e_n}(t)$.

Combining (\ref{cellBF lemma: can transf}) and (\ref{cellBF lemma: S''}), we obtain that $S_{X_k}$ and $S'_{X_k}$ can be connected by a canonical transformation (\ref{lemma 8.7 (ii)}) with generator
$$R_{X_k,t}=\left\{\begin{array}{ll} 
2 r_{e,2t}, & t\in [0,\frac12) \\
2 R_{X_{k-1},2t-1}, & t\in [\frac12 ,1] 
\end{array} \right.$$
This proves the induction step $X_{k-1}\ra X_k$.
\end{proof}

\begin{remark} One can put together $S_{X,t}$ and $R_{X,t}$ of Lemma \ref{lemma: cellBF well-definedness} into a single non-homogeneous differential form on the interval,
$$\widetilde{S}_X:=S_{X,t}+dt\cdot R_{X,t}\in \Omega^\bt([0,1])\widehat{\otimes}\mr{Fun}(\sF_X)[[\hbar]]$$ 
Then the quantum master equation on $S_{X,t}$ together with (\ref{lemma 8.7 (ii)}) can be packaged as an \emph{extended} quantum master equation for $\Tilde{S}_X$,
$$ (d_t-i\hbar\,\Delta)\, e^{\frac{i}{\hbar}\widetilde{S}_X}=0 $$
where $d_t=dt\cdot \frac{\dd}{\dd t}$ -- the de Rham differential in $t$. By the Lemma, $\widetilde{S}_{X}$ satisfies the ansatz (\ref{cellBF locality},\ref{Sbar cell}) with structure constants $C$ taking values in $\Omega^\bt([0,1])$.
\end{remark}

\begin{remark}
The building block (\ref{Sbar cell}) of the cellular action is defined uniquely for $\dim e=0$ (fixed by (\ref{cellBF S_pt})) and for $\dim e=1$ (as follows from Lemma \ref{lemma: cellBF well-definedness}: for degree reasons, the generator of the canonical transformation has to vanish; $\bar{S}_e$ in this case is given by (\ref{Sbar interval})). For cells of dimension $\geq 2$, $\bar{S}_e$ is not uniquely defined. 
\end{remark}

\begin{remark}
A different approach to the proof of Theorem \ref{thm: cellBF} is to fix a simplicial refinement $W$ of $X$ (i.e. cells of $X$ are triangulated in $W$).
Then, proceeding again by induction in $X_k$, as in (\ref{cellBF CW filtration}), one constructs $S_{X_k}$ as a BV pushforward of $S_{W_k}$, as constructed in Theorem \ref{thm: simpBF} ($W_k$ here is the restriction of $W$ to $X_k$), using a choice of gauge-fixing compatible with one for the BV pushforward $W_{k-1}\ra X_{k-1}$ used in the previous step and involving a choice of induction data for relative cochains $C^\bt(W|_{e_k},\dd e_k)\wavy{}C^\bt(\bar{e}_k,\dd e_k)$.
\end{remark}

\begin{remark}[Standard building blocks for cells]
\label{rem: standard building blocks}
Assume that we have fixed some ``standard'' choice of the induction data (\ref{cellBF: cell ind data}) for all possible regular CW decompositions $\bar{e}$ of a closed $n$-ball with single top cell, considered up to homeomorphisms, for $n\geq 1$. Denote this collection of choices $\theta$. Then, by the construction of Theorem \ref{thm: cellBF}, we have a ``standard building block'' $\bar{S}_e^\theta$ satisfying ansatz (\ref{Sbar cell}) for any possible combinatorics of $\bar{e}$. For any finite regular CW complex $X$, we then have a ``standard'' cellular action $S^\theta_X=\sum_{e\subset X}\bar{S}^\theta_e$ satisfying the quantum master equation and properties (\ref{cellBF thm a}--\ref{cellBF thm c}) of Theorem \ref{thm: cellBF}. In particular, we can choose $\theta$ on simplices (and, more generally, on prisms, see Remark \ref{rem: prismatic}) to give the building blocks (\ref{Sbar}) of Theorem \ref{thm: simpBF}, coming from Whitney forms/Dupont's chain homotopy. 
\end{remark}

\begin{remark}[Rationality] We can choose the induction data (\ref{cellBF: cell ind data}) for cells to be \emph{rational}, i.e. to factor through induction data $C^\bt(\bar{e},\mathbb{Q})\wavy{(i_\mathbb{Q},p_\mathbb{Q},K_\mathbb{Q})}H^\bt(\bar{e},\mathbb{Q})$. Then, by inspection of the proof of Theorem \ref{thm: cellBF}, all the structure constants of the cellular action -- coefficients $C^e_{\Gamma_0,e_1,\ldots,e_n}$, $C^e_{\Gamma_1,e_1,\ldots,e_n}$ in (\ref{Sbar cell}) -- are \emph{rational} as well.
\end{remark}

\subsubsection{Cellular $BF$ action as a ``generating function'' of a unimodular $L_\infty$ algebra on cochains}\label{sec: uL_infty}
\begin{definition}\footnote{
This algebraic structure (and the example coming from Theorem \ref{thm: simpBF}) was introduced in \cite{SimpBF,DiscrBF} under the name of a \emph{quantum} $L_\infty$ algebra.  
In \cite{Granaker} it was named a \emph{unimodular} $L_\infty$ algebra and was studied as an algebra over a particular Merkulov's wheeled operad.
} \label{def: uL_infty}
A \emph{unimodular $L_\infty$ algebra} is a graded vector space $V^\bt$ endowed with two sequences of skew-symmetric multilinear operations,
\begin{itemize}
\item \emph{classical operations}  $l_n:\,\wedge^n V \ra V$ for $n\geq 1$ of degree $2-n$ and
\item \emph{quantum operations} $q_n:\, \wedge^n V\ra \mathbb{R}$ for $n\geq 1$ of degree $-n$,
\end{itemize}
such that the following two sequences of identities hold: homotopy Jacobi identities
\begin{equation}\label{L_infty relations}
\sum_{\sigma\in \Sigma_{n}}\sum_{r+s=n}\pm\frac{1}{r!s!}\,l_{r+1}(x_{\sigma_1},\ldots,x_{\sigma_r},l_s(x_{\sigma_{r+1}},\ldots,\sigma_{n}))=0
\end{equation}
and homotopy unimodularity relations
\begin{equation} \label{uL_infty homotopy unimodularity}
\sum_{\sigma\in\Sigma_n}\left(\pm\frac{1}{n!}\,\mr{Str}_V\;l_{n+1}(x_{\sigma_1},\ldots,x_{\sigma_n},\bt)+\sum_{r+s=n}\pm\frac{1}{r!s!}\,q_{r+1}(x_{\sigma_1},\ldots,x_{\sigma_r},l_s(x_{\sigma_{r+1}},\ldots,x_{\sigma_n}))\right)=0 
\end{equation}
Here $x_1,\ldots,x_n$ is an $n$-tuple of elements of $V$ and $\sigma$ runs over permutations of this $n$-tuple; $\pm$ are the Koszul signs. 

\end{definition}

\begin{definition}[\cite{DiscrBF}] For $V^\bt$ a graded vector space, consider the odd-symplectic space $\sF=V[1]\oplus V^*[-2]$. We say that an element $f\in \mr{Fun}(\sF)[[\hbar]]$ satisfies the ``$BF_\infty$ ansatz'' if 
\begin{equation}\label{BF_infty ansatz general}
f=\lan \sB , \alpha(\sA) \ran - i\hbar \beta(\sA)
\end{equation}
where $\sA\in \mr{Hom}_1(\sF,V)$, $\sB\in \mr{Hom}_{-2}(\sF,V^*)$ are the superfields (projections to summands of $\sF$ composed with shifted identity map) and  $\alpha(\sA) \in \widehat{\Sym^{\geq 1}}(V[1])\otimes V^*$, $\beta(\sA)\in \widehat{\Sym^{\geq 1}}(V[1])$ arbitrary elements.
\end{definition}
We have the following properties (see \cite{DiscrBF} for details and proofs):
\begin{enumerate}[(i)]
\item If $f$ and $g$ satisfy the $BF_\infty$ ansatz then $\{f,g\}$ and $\Delta f$ also satisfy it.
\item If $S$ satisfies the $BF_\infty$ ansatz and is of internal degree zero, then one can write 
\begin{equation}\label{BF_infty ansatz via l,q}
S(\sA,\sB;\hbar)=\sum_{n\geq 1}\frac{1}{n!} \langle \sB, l_n(\underbrace{\sA,\ldots,\sA}_n) \rangle- i\hbar \sum_{n\geq 1}\frac{1}{n!}q_n(\underbrace{\sA,\ldots,\sA}_n)
\end{equation}
where $l_n:\, \wedge^n V\ra V$ and $q_n:\, \wedge^n V\ra \mathbb{R}$ are certain multilinear operations on $V$. They endow $V$ with the structure of a unimodular $L_\infty$ algebra, as in Definition \ref{def: uL_infty}. Relations (\ref{L_infty relations},\ref{uL_infty homotopy unimodularity}) can be conveniently written in terms of the superfield $\sA_X$ as
\begin{align*}
\sum_{r+s=n}\frac{1}{r!s!}\,l_{r+1}(\underbrace{\sA,\ldots,\sA}_r,l_s(\underbrace{\sA,\ldots,\sA}_s))=0,\\
\frac{1}{n!}\,\mr{Str}\;l_{n+1}(\underbrace{\sA,\ldots,\sA}_n,\bt)+\sum_{r+s=n}\frac{1}{r!s!}\,q_{r+1}(\underbrace{\sA,\ldots,\sA}_r,l_s(\underbrace{\sA,\ldots,\sA}_s))=0 
\end{align*}
and are equivalent to the quantum master equation satisfied by $S$.
\item If $S_t$, for $t\in [0,1]$, is a family of solutions of the QME satisfying the $BF_\infty$ ansatz such that $S_t$ for different $t$ are related by a canonical BV transformation 
\begin{equation}\label{can transformation}
\frac{\dd}{\dd t}\, S_t=\{S_t,R_t\}-i\hbar\, R_t
\end{equation} 
then the degree $-1$ generator $R_t$ has to satisfy the $BF_\infty$ ansatz as well.
\item One can introduce a natural notion of \emph{equivalence of unimodular $L_\infty$ structures} on $V$: two structures $\{l_n,q_n\}$, $\{l'_n,q'_n\}$ are called \emph{equivalent} if the corresponding solutions (\ref{BF_infty ansatz via l,q}) of QME on $\sF$ can be related by a canonical BV transformation (\ref{can transformation}).
\end{enumerate}

In particular, for $V=C^\bt(X,\g)$, 
the action $S_X$ constructed in Theorem \ref{thm: cellBF} can be expanded as
\begin{equation}\label{BF_infty ansatz for S_cell}
S_X(\sA_X,\sB_X;\hbar)=\sum_{n\geq 1}\frac{1}{n!} \langle \sB_X, l_n^X(\underbrace{\sA_X,\ldots,\sA_X}_n) \rangle- i\hbar \sum_{n\geq 2}\frac{1}{n!}q_n^X(\underbrace{\sA_X,\ldots,\sA_X}_n)
\end{equation}
where $l_n^X:\, \wedge^n C^\bt(X,\g)\ra C^\bt(X,\g)$ and $q_n^X:\, \wedge^n C^\bt(X,\g)\ra \mathbb{R}$ are multilinear operations on $\g$-valued cochains,  endowing the space of cochains $C^\bt(X,\g)$ with the structure of a unimodular $L_\infty$ algebra.

We can split operations $l_n^X,q_n^X$ into contributions of individual cells:
$$l_n^X=\sum_{e\subset X} e^*\cdot l_n^e(\sA_{\bar{e}},\ldots,\sA_{\bar{e}}),\qquad q_n^X=\sum_{e\subset X} q_n^e(\sA_{\bar{e}},\ldots,\sA_{\bar{e}})$$
where $l_n^e:\, \wedge^n C^\bt(\bar{e},\g)\ra \g$ and $q_n^e:\,\wedge^n C^\bt(\bar{e},\g)\ra \mathbb{R}$ are the terms of the Taylor expansion in $\sA$-fields of the building block (\ref{Sbar cell}),
\begin{equation}\label{Sbar via l and q}
\bar{S}_e=\sum_{n\geq 1} \frac{1}{n!} \lan B_e, l_n^e(\sA_{\bar{e}},\ldots,\sA_{\bar{e}}) \ran_\g -i\hbar\, \sum_{n\geq 2}\frac{1}{n!}\, q_n^e(\sA_{\bar{e}},\ldots,\sA_{\bar{e}}) 
\end{equation}

\begin{example}[From \cite{SimpBF,DiscrBF}]
The explicit answer for the action $S_X$ for $X=I$ the interval 
(\ref{Sbar interval})
corresponds to the following unimodular $L_\infty$ algebra structure on $\g$-valued cochains $C^\bt(I,\g)=\mr{Span}_\g(\epsilon_0,\epsilon_1,\epsilon_{01})$ of the interval (we denote $e_0,e_1,e_{01}$ the basis cochains associated to the left/right endpoints and the bulk):
\begin{align*}
l_2(\alpha_1\otimes\epsilon_0,\alpha_2\otimes \epsilon_0)&=[\alpha_1,\alpha_2]\otimes \epsilon_0\\
l_2(\alpha_1\otimes\epsilon_1,\alpha_2\otimes \epsilon_1)&=[\alpha_1,\alpha_2]\otimes \epsilon_1 \\
l_{n+1}(\alpha_1\otimes \epsilon_{01},\ldots, \alpha_n\otimes \epsilon_{01},\beta\otimes \epsilon_1)&=
\frac{B_n^+}{n!}\;\sum_{\sigma}\mr{ad}_{\alpha_{\sigma_1}}\cdots\mr{ad}_{\alpha_{\sigma_n}}(\beta) \otimes \epsilon_{01} \\
l_{n+1}(\alpha_1\otimes \epsilon_{01},\ldots, \alpha_n\otimes \epsilon_{01},\beta\otimes \epsilon_0)&=
-\frac{B_n^-}{n!}\;\sum_{\sigma}\mr{ad}_{\alpha_{\sigma_1}}\cdots\mr{ad}_{\alpha_{\sigma_n}}(\beta)  \otimes \epsilon_{01} \\
q_{n}(\alpha_1\otimes \epsilon_{01},\ldots,\alpha_{n}\otimes \epsilon_{01})&=
\frac{B_n}{n\cdot n!}\; \sum_\sigma \mr{tr}_\g (\mr{ad}_{\alpha_{\sigma_1}}\cdots\mr{ad}_{\alpha_{\sigma_n}})
\end{align*}
Here we wrote out all the nonvanishing operations; $n\geq 0$ for $l_{n+1}$ and $n\geq 2$ for $q_n$; $\sigma$ runs over permutations of numbers $1,\ldots,n$;  $\alpha_i$ and $\beta$ are arbitrary elements in $\g$; $B_n^\pm$ are Bernoulli numbers with $B_1^\pm=\pm\frac12$ (and $B_0^\pm=1,\;B_2^\pm=\frac16,\ldots$ the standard Bernoulli numbers). Forgetting about the quantum operations $q_n$, we have an $L_\infty$ algebra on $\g$-valued cochains of the interval -- the Lie version of the ``algebra of  the interval''\footnote{This algebraic structure appeared  independently and nearly simultaneously in \cite{Lawrence-Sullivan}, in the preprint of \cite{Cheng-Getzler} and, in its Lie form, in the preprint of \cite{SimpBF}.} in Lawrence-Sullivan \cite{Lawrence-Sullivan} (more precisely, the $L_\infty$ algebra $C^\bt(I,\g),\{l_n\}$ is the $C_\infty$ algebra of Lawrence-Sullivan on $C^\bt(I)$ tensored with $\g$).
\end{example}

\begin{remark}
Theorem \ref{thm: cellBF} possesses a straightforward generalization whereby one replaces the unimodular Lie algebra $(\g,[,])$ of coefficients of cochains/chains by any finite-dimensional unimodular $L_\infty$ algebra $(\g,\{l_n^\g\},\{q_n^\g\})$. In this case, instead of (\ref{cellBF S_pt}), for $e$ a $0$-cell, we have $\bar{S}_{e}=\lan B_e,\sum_{n\geq 1}\frac{1}{n!}\,l_n^\g(\sA_e,\ldots,\sA_e) \ran_\g-i\hbar\, \sum_{n\geq 1} \frac{1}{n!}\, q_n^\g(\sA_e,\ldots,\sA_e)$; instead of (\ref{cellBF: a eq}) we have $S_X=\lan \sB_X,(d+l_1^\g)\sA_X \ran+\mathfrak{r}$. In (\ref{Sbar cell}), we then allow $\Gamma_0$ to be any rooted tree with in-valencies $\geq 2$ at vertices (not necessarily binary), with $\mr{Jacobi}_{\Gamma_0}$ defined as an nested composition of operations $l_n^\g$ (with $n$ the in-valence of a vertex of the tree) associated to $\Gamma_0$. Graph $\Gamma_1$ can be either a $1$-loop graph with every vertex having out-valency $1$ and in-valency $\geq 2$ (and then $\mr{Jacobi}_{\Gamma_1}$ is a supertrace of a nested classical operation), or a rooted tree with the root vertex decorated by $q_n^\g$, with $n$ the valency of the root vertex.
\end{remark}

\subsection{BV pushforward to cohomology, simple-homotopy equivalence, cellular aggregations
}
\label{s:BVpfsheca}

Similarly to Section \ref{sec: quantization closed}, we consider the BV pushforward of the half-density $e^{\frac{i}{\hbar}S_X}(\mu^\hbar_{\sF_X})^{1/2} \in \Dens^{\frac12,\Fun}_\mathbb{C}(\sF_X)$ to residual fields $\sF^{\zm}=H^\bt(X,\g)[1]\oplus H_\bt(X,\g^*)[-2]$. The factor $(\mu^\hbar_{\sF_X})^{1/2}=\xi_\hbar\cdot \mu_{\sF_X}^{\frac12}$ is the normalized cellular half-density on the space of fields as in Section \ref{sec: quantization closed} (for cochains with coefficients in a trivial local system 
with fiber $\g$).\footnote{
Half-density $\mu_{\sF_X}^{\frac12}$ is constructed using an a priori fixed density (i.e. a fixed normalization of the Lebesgue measure) $\mu_\g$ on $\g$, instead of the standard density on $\mathbb{R}^m$ as in Section \ref{sec: quantization closed}.
} Explictly, we define the partition function by the fiber BV integral
\begin{equation}\label{Z sec 8.3}
Z(\sA_\zm,\sB_\zm):=\int_{\LL\subset \sF_\fluct } e^{\frac{i}{\hbar}S_X}(\mu^\hbar_{\sF_X})^{1/2} \quad \in \Dens^{\frac12,\Fun}_\mathbb{C}(\sF^\zm) 
\end{equation}
where the gauge fixing data -- the splitting $\sF_X=\sF^\zm\oplus \sF_\fluct$ and the Lagrangian $\LL\subset  \sF_\fluct$ -- are constructed, as in Section \ref{sec: quantization closed}, from a choice of induction data $C^\bt(X)\wavy{} H^\bt(X) $ (tensored with identity in $\g$). 

By the general properties of BV pushforwards and by Theorem \ref{thm: cellBF}, we have $\Delta_\zm Z=0$ and a change of the data $C^\bt(X)\wavy{} H^\bt(X)$ changes $Z\mapsto Z+\Delta_\zm(\cdots)$.\footnote{By an abuse of notations, throughout Sections \ref{sec: non-ab BF I} and \ref{sec: non-ab BF cobordism} we are suppressing the superscript $\can$ for the BV Laplacian on half-densities. Similarly, in Section \ref{sec: non-ab BF cobordism} we will suppress this superscript for the operator $\widehat{S}_\dd$ acting on half-densities.}
More precisely, 
a change of induction data induces a canonical BV transformation,\footnote{This is a general property of BV pushforwards for a change of gauge-fixing data in a smooth family, cf e.g. \cite{CMRpert}, Section 2.2.2. Note that by the discussion of Section \ref{sec: deformations of ind data}, in our case the space of gauge-fixing data is contractible and in particular path-connected, thus any two choices of gauge-fixing can be connected by a smooth family.
We are slightly abusing the term ``canonical BV transformation'': for us it has  two related meanings -- for $\Delta$-closed half-densities, as in (\ref{Z non-ab can transf}), and for actions (functions solving QME), as in 
(\ref{can transformation}).
These meanings are equivalent for half-densities satisfying exponential ansatz $Z=e^{\frac{i}{\hbar}S}\mu^{1/2}$.
} i.e. if $Z$ and $Z'$ are constructed using two different gauge-fixings, we can construct a family $Z_t$, $t\in [0,1]$, such that $Z_0=Z$, $Z_1=Z'$ and 
\begin{equation}\label{Z non-ab can transf}
\frac{\dd}{\dd t}Z_t=\Delta_\zm(Z_t\cdot R_t)
\end{equation} 
for some $t$-dependent generator $R_t\in \mr{Fun}_{-1}(\sF_\zm)[[\hbar]]$ of form $R_t=\lan \sB_\zm, \alpha_t(\sA_\zm) \ran-i\hbar\, \beta_t(\sA_\zm)$. 

Computing the BV pushforward yields the result of the following form:
$$Z(\sA_\zm,\sB_\zm)= e^{\frac{i}{\hbar}S_\zm(\sA_\zm,\sB_\zm;\hbar)}\;\xi_\hbar^{H^\bt(X,\g)}\underbrace{\tau(X,\g)}_{\in \HDens \sF^\zm}$$
where the factor $\xi_\hbar^{H^\bt}$ and the torsion $\tau(X,\g)$ are for a trivial local system on $X$ of rank $\dim\g$. The effective action on residual fields $S_\zm$ is computed as a sum of Feynman diagrams 
and satisfies the $BF_\infty$ ansatz (\ref{BF_infty ansatz via l,q}) for some multilinear operations $l_n^\zm:\, \wedge^n H^\bt(X,\g)\ra H^\bt(X,\g)$ and $q_n^\zm:\, \wedge^n H^\bt(X,\g)\ra\mathbb{R}$ which endow the cohomology $H^\bt(X,\g)$ with the structure of a unimodular $L_\infty$ algebra. In particular, the classical operations $\{l_n\}$ determine the \textit{rational homotopy type} of $X$.\footnote{
Here we mean that we need to know $l_n$'s (Massey-Lie brackets) for a general Lie algebra of coefficients $\g$, 
which is tantamount to knowing the $C_\infty$ operations (Massey products) on $H^\bt(X)$, see (\ref{H(X,g)= H(X) otimes g}) below. In fact, one can recover the $n$-ary $C_\infty$ operation $m_n$ on $H^\bt(X)$ from $l_n$ with $\g=\mathfrak{b}^+_{n+1}$ the algebra of upper-triangular matrices of size $n+1$, simply from ${l_n(w_1\otimes t_{12},\ldots w_{n}\otimes t_{n\, n+1})} =m_n(w_1,\ldots,w_n)\otimes t_{1\, n+1}$ with $w_i\in H^\bt(X)$ and $t_{ij}$ the matrix with entry $1$ at $(ij)$-th place and all other entries being zero. 
We stress that one does not recover the $C_\infty$ structure on $H^\bt(X)$ by plugging $\g=\mathbb{R}$ into the formulae for operations $l_n$ on $H^\bt(X,\g)$  -- that would just kill all the operations.}

\begin{remark} \label{rem: L_infty and RHT}
Some comments on the $L_\infty$ algebra $H^\bt(X,\g),\{l_n\}_{n\geq 2}$ and its relation to the rational homotopy type:
\begin{enumerate}[(a)]
\item \emph{\textbf{Factorization.}}   
One has the following factorization property: the space of $\g$-valued cochains, viewed as an $L_\infty$ algebra (disregarding the quantum operations $q_n$), can be written as
\begin{equation} \label{C(X,g)=C(X) otimes g}
C^\bt(X,\g)=C^\bt(X)\otimes \g 
\end{equation}
-- the Lie algebra of coefficients $\g$ tensored with the $C_\infty$ algebra\footnote{Recall, see e.g. \cite{Cheng-Getzler} for details, that an $A_\infty$ algebra is a $\mathbb{Z}$-graded vector space $W$ together with a sequence of multilinear operations $m_n:W^{\otimes n}\ra W$, $n\geq 1$, satisfying the quadratic associativity-up-to-homotopy identities. An $A_\infty$ algebra $(W,\{m_n\})$ is called a $C_\infty$ algebra if in addition operations $m_n$ vanish on shuffle-products.
We refer the reader to Appendix \ref{Appendix: C-infty otimes Lie} on how to construct the tensor product $L_\infty$ structure on $W\otimes\g$, with $W$ a $C_\infty$ algebra and $\g$ a Lie algebra.
The $C_\infty$ structure on $C^\bt(X)$, for $X$ a simplicial complex, coinciding with the one read off of the tree part of (\ref{Sbar})  and constructed via homotopy transfer from piecewise-polynomial forms by Kontsevich-Soibelman formula using Dupont's chain homotopy operator
was considered in \cite{Cheng-Getzler}.

} 
$C^\bt(X)$ of cochains with coefficients  in $\mathbb{R}$ or $\mathbb{Q}$.
This follows by inspection of (\ref{Sbar cell}). Similarly, one has
\begin{equation}\label{H(X,g)= H(X) otimes g}
H^\bt(X,\g)=H^\bt(X)\otimes \g 
\end{equation}
-- the $L_\infty$ algebra of $\g$-valued cohomology equals  the coefficient Lie algebra $\g$ tensored with the cohomology $C_\infty$ algebra, regarded as a homotopy transfer of the $C_\infty$ algebra $C^\bt(X)$ given by Kontsevich-Soibelman sum-over-trees formula \cite{Kontsevich-Soibelman}. See also \cite{Cheng-Getzler} for the fact that the sum-over-trees formula  transfers  $C_\infty$ algebras
into $C_\infty$ algebras.\footnote{
Here the remark (see \cite{SimpBF,DiscrBF}) is that the perturbative evaluation of the integral (\ref{Z sec 8.3}) in the lowest order in $\hbar$ corresponds to the homotopy transfer formula for $L_\infty$ algebras to a subcomplex as a sum over (non-planar) rooted trees. This is the $L_\infty$ version of Kontsevich-Soibelman formula for homotopy transfer of $A_\infty$ algebras where the sum is over planar rooted trees. Also, one has that the homotopy transfer commutes with tensoring with a Lie algebra $\g$:
$$ \begin{CD} W,\{m_n\} @>{\otimes \g}>> W\otimes \g, \{l_n\}\\
@VVV  @VVV \\
W',\{m'_n\}  @>{\otimes \g}>> W'\otimes \g, \{l'_n\}
\end{CD}
$$
}
\item \emph{\textbf{Uniqueness of the $C_\infty$ structure:}}
\begin{itemize}
\item The $C_\infty$ algebra structure on $C^\bt(X)$ appearing in the r.h.s. of (\ref{C(X,g)=C(X) otimes g})  is inductively unique up to $C_\infty$ isomorphism (cf. Lemma \ref{lemma: cellBF well-definedness}). The $C_\infty$ structure on $H^\bt(X)$ appearing in the r.h.s. of (\ref{H(X,g)= H(X) otimes g}) is $C_\infty$ quasi-isomorphic to it. 
\item For $X$ a simplicial complex, $C^\bt(X)$ and $H^\bt(X)$ are both $C_\infty$ quasi-isomorphic to Sullivan's algebra $\Omega^\bt_\mr{poly}(X)$ of piecewise-polynomial differential forms on $X$. 
\item Also, for $X$ a cellular decomposition of a manifold $M$, $C^\bt(X,\mathbb{R})$ and $H^\bt(X,\mathbb{R})$ are both $C_\infty$ quasi-isomorphic to de Rham  algebra of smooth differential forms $\Omega^\bt(M)$.
\end{itemize}
\item \emph{\textbf{Massey products and rational homotopy type.}}
For $X$ a simply-connected CW complex, the $C_\infty$ algebra structure on $H^\bt(X,\mathbb{Q})$ determines the rational homotopy type of $X$ by a theorem of Kadeishvili \cite{Kadeishvili08}, \cite{Kadeishvili93}. In particular, this $C_\infty$ algebra is quasi-isomorphic (in the category of $C_\infty$ algebras) to Sullivan's minimal model cdga  \cite{Sullivan} of the space $X$, from which the rational homotopy groups $\mathbb{Q}\otimes \pi_*(X)$ can be directly recovered.
\end{enumerate}
\end{remark}

\begin{definition}[Whitehead \cite{Whitehead}, see also \cite{Cohen}] \label{def: simple-homotopy}
\begin{enumerate}[(i)]
\item Let $Y$ be a CW complex
containing an $n$-cell $e$ and an $(n-1)$-cell $e'\subset \dd e$ for some $n\geq 1$, such that $e'$ is a \emph{free face} of $e$, i.e. $e'$ is cobounded only by $e$. Let $X\subset Y$ be the subcomplex obtained by removing the pair $e,e'$ from $X$. Then one calls $X$ an \emph{elementary collapse} of $Y$ and  $Y$ an \emph{elementary expansion} of $X$. The customary notation is $X\nearrow Y$ or $Y\searrow X$.
\item Two CW complexes $X$ and $Y$ are called \emph{simple-homotopy equivalent} if they can be connected by a sequence of elementary expansions and collapses.\footnote{For example, Pachner's moves of triangulations of an $n$-manifold can be realized as a sequence of elementary expansions followed by a sequence of elementary collapses. Also, cellular subdivisions and aggregations can be realized as sequences of expansions and collapses.}
\end{enumerate}
\end{definition}


Let us assume that the collection $\theta$ of standard induction data on cells is chosen, as in Remark \ref{rem: standard building blocks}, so that we have standard cellular actions $S_X^\theta$. We will omit the superscript $\theta$, implying that $S_X$ always stands for $S_X^\theta$ (and similarly for the building block  $\bar{S}_e$) throughout this section.

\begin{lemma}\label{lemma: collapse}
Let $X$ and $Y$ be two regular CW complexes such that $Y$ is an elementary expansion of $X$ obtained by adjoining a pair of cells $e$, $e'\subset \dd e$. 
Then the BV pushforward $P_*^{Y\ra X}$  
from $\sF_Y$ to $\sF_X$ relates the standard cellular actions (\ref{cellBF locality}) 
as follows:
\begin{equation}\label{collapse: P_*}
P_*^{Y\ra X} \left( e^{\frac{i}{\hbar}S_Y}(\mu_{\sF_Y}^\hbar)^{1/2}\right)=
e^{\frac{i}{\hbar}S_X}(\mu_{\sF_X}^\hbar)^{1/2}+\Delta_X(\rho)
\end{equation}
There is a preferred (canonical) choice of gauge-fixing for the BV pushforward $P_*^{Y\ra X}$, for which 
the half-density $\rho\in\Dens^{\frac12,\Fun}_\mathbb{C}(\sF_X)$ above attains the form
\begin{equation}\label{collapse: rho}
\rho= -i\hbar\; r\cdot e^{\frac{i}{\hbar}S_X}(\mu_{\sF_X}^\hbar)^{1/2}
\end{equation}
with $r\in \mr{Fun}_{-1}(C^\bt(\dd e-e',\g)[1])$ -- a function of the $\sA$-field on $\dd e-e'$.
\end{lemma}

\begin{proof}
Define $h:=\dd e-e'\subset X$. 
Let  $k:=\dim (e')=\dim(e)-1$. 
Our 
preferred gauge-fixing for the BV pushforward $P_*^{Y\ra X}$ is associated to the induction data $C^\bt(Y)\simeq C^\bt(X)\oplus \mr{Span}(e^*,(e')^*)\wavy{(i,p,K)}C^\bt(X)$ with 
\begin{multline}\label{collapse ind data} 
p=\left\{\begin{array}{ll} \mr{id} & \mr{on}\; C^\bt(X)\\0 & \mr{on}\; e^*,(e')^* \end{array}\right.,\;
i:\; \varepsilon^*\mapsto \left\{\begin{array}{ll} \varepsilon^*-(e')^* & \mr{if}\; \varepsilon \;\mr{is\; a}\; k\mr{-cell\;of}\; h \\
\varepsilon^* & \mr{otherwise}  \end{array}\right.,\\
K:\;\left\{\begin{array}{l} e^* \mapsto (e')^* \\
C^\bt(X)\oplus \mr{Span}((e')^*) \mapsto 0  \end{array}\right.
\end{multline}
Here $\varepsilon$ is a cell of $X$. The map $i$ is the pullback by the projection $C_\bt(Y)\ra C_\bt(X)$ which sends $e'\mapsto -h$ and $e\mapsto 0$. 
The result of the BV pushforward has the form $e^{\frac{i}{\hbar}S'_X}(\mu^\hbar_{\sF_X})^{1/2}$ with $S'_X$ that can be different from 
the standard action $S_X$. 
However the difference $S'_X-S_X$ is a function of the form  $-i\hbar\, \phi(\sA_h)$ depending only on the field $\sA$ on cells of $h$. Thus we have two solutions of the quantum master equation $S_h$ (standard) and $S_h-i\hbar\, \phi$ (obtained by BV pushforward from $e$) on $\sF_h$.
One can connect them by a path of solutions of the quantum master equation $S_{h,t}=S_h-i\hbar\, t\cdot \phi$ with $t\in [0,1]$ such that $S_{h,0}=S_h$, $S_{h,1}=S_h-i\hbar\,\phi$. Similarly to the argument in the proof of Lemma \ref{lemma: cellBF well-definedness}, the fact that $\dd_t S_{h,t}$ can be written as an infinitesimal canonical BV transformation
\begin{equation}\label{can transf horn}
\dd_t S_{h,t}=\{S_{h,t},R_t\}-i\hbar\,\Delta R_t
\end{equation}
for some 
generator $R_t\in \mr{Fun}_{-1}(\sF_h)[[\hbar]]$ 
of form $R_t=-i\hbar\,\chi(\sA_h)$,  
follows from the computation of the zeroth cohomology of $Q_h=\{S_h^{(0)},\bt\}$ on functions of field $\sA$: 
\begin{align}
\label{Chevalley H^0} H^0_{Q_h}(\mr{Fun}(C^\bt(h,\g)[1]))&\simeq H^0_{CE}(\g)=\mathbb{R}
\end{align}
where we use 
the contractibility of $h$.\footnote{
In more detail, the quantum master equation for $S_h$ and for $S_h-i\hbar\,\phi$ implies 
$Q_h\phi=0$. 
Function $\phi$ vanishes at $\sA=0$ (this is the point where we use the normalization of the half-densities $(\mu^\hbar_{\sF_Y})^{1/2},(\mu^\hbar_{\sF_X})^{1/2}$) which implies, together with (\ref{Chevalley H^0}), that the obstruction to $Q_h$-exactness of $\phi$ vanishes, i.e. we have $\phi=Q_h \chi$ for some $\chi$, a function of $\sA_h$ of degree $-1$. Thus we have proved (\ref{can transf horn}) for the generator $R_t=-i\hbar \chi$.
}
Equation (\ref{can transf horn}) is equivalent to $\frac{\dd}{\dd t}e^{\frac{i}{\hbar}S_{h,t}}=\Delta\left(e^{\frac{i}{\hbar}S_{h,t}}R_t\right)$; integrating over $t\in [0,1]$, we obtain $e^{\frac{i}{\hbar}(S_h-i\hbar\,\phi)}-e^{\frac{i}{\hbar}S_h}=\Delta\left(\int_0^1 dt\, e^{\frac{i}{\hbar}S_{h,t}}R_t \right)
=\Delta\left( e^{\frac{i}{\hbar}S_h}\frac{e^\phi-1}{\phi}(-i\hbar\,\chi)\right)
$. Therefore, returning to the pair of complexes $X,Y$, we have obtained that the BV pushforward $e^{\frac{i}{\hbar}S_Y}(\mu^\hbar_{\sF_Y})^{1/2}$ to $\sF_X$ differs from $e^{\frac{i}{\hbar}S_X}(\mu^\hbar_{\sF_X})^{1/2}$ by $\Delta_X(\cdots)$ with $(\cdots)$ given by (\ref{collapse: rho}) with $r=\frac{e^\phi-1}{\phi}\,\chi 
$.
By general properties of BV pushforwards, another choice of gauge-fixing for $P_*^{Y\ra X}$ preserves the result (\ref{collapse: P_*}) but may change the ansatz for $\rho$.
\end{proof}

\begin{Proposition}\label{prop: non-ab BF I simple homotopy invariance}
Assume that two 
regular CW complexes $X$ and $Y$ are simple-homotopy equivalent. 
Let $Z_X$ and $Z_Y$ be the respective partition functions on $\sF^\zm_X\cong \sF^\zm_Y$. Then
\begin{equation}\label{Z_Y-Z_X simple homotopy}
Z_Y-Z_X=\Delta_\zm(\cdots)
\end{equation}
More precisely, $Z_X$ and $Z_Y$ can be connected by a canonical transformation, as in (\ref{Z non-ab can transf}).
\end{Proposition}

\begin{proof}
We can assume without loss of generality that $Y$ is an elementary expansion of $X$ obtained by attaching a 
pair of cells $e$, $e'\subset \dd e$. 
Since the value of the BV pushforward from $\sF_Y$ to residual fields is independent of the gauge-fixing data 
when considered modulo canonical transformations, 
one can choose the gauge-fixing corresponding to first pushing forward from $\sF_Y$ to $\sF_X$ and then to $\sF^\zm$ (using the construction of composition of induction data (\ref{composition of ind data})). On the other hand, by Lemma \ref{lemma: collapse}, the pushforward $\sF_Y\ra \sF_X$, $\zeta_1:=P_*^{Y\ra X}\left(e^{\frac{i}{\hbar}S_Y}(\mu^\hbar_{\sF_Y})^{1/2}\right)$, differs from $\zeta_0:=e^{\frac{i}{\hbar}S_X}(\mu^\hbar_{\sF_X})^{1/2}$ by a canonical transformation, 
i.e. we have a family of half-densities $\zeta_t$ on $\sF_X$ with $\frac{\dd}{\dd t}\zeta_t=\Delta_X(\zeta_t R_{X,t})$. Hence, $P_*^{X\ra \zm}P_*^{Y\ra X}\left(e^{\frac{i}{\hbar}S_Y}(\mu^\hbar_{\sF_Y})^{1/2}\right)=P_*^{X\ra \zm}\zeta_1 $ is connected to $P_*^{X\ra\zm}\zeta_0$ by a family $Z_t=P_*^{X\ra\zm}\zeta_t$ satisfying $\frac{\dd}{\dd t}Z_t = \Delta_\zm(Z_t R_t)$ with $R_t:=Z_t^{-1}P_*^{X\ra \zm}(\zeta_t R_{X,t})$.  We used here the property of BV pushforward that it commutes with BV Laplacians. This proves the Proposition.
%
\end{proof}

\begin{remark} As we remarked above, in the case of simply-connected $X$, the partition function (if we know it for all $\g$) determines the rational homotopy type of $X$. One might ask, what kind of topological information is contained in the partition function for $X$ non-simply connected? The partition function $Z$ contains the following:
\begin{enumerate}[(a)]
\item\label{rem: non-sc (a)} The deformation-theoretic model (given by the homotopy Maurer-Cartan equation for the $L_\infty$ algebra on cohomology, or on cochains) for the singularity of the moduli space of flat connections at zero connection  (or at the chosen local system, if we twist the construction by it, as explained in Section \ref{sec: non-ab BF cobordism}).
\item\label{rem: non-sc (b)} A formal infinitesimal thickening of the moduli space - its graded/supermanifold part, corresponding to writing the homotopy Maurer-Cartan equation above without requiring the unknown (the tangent vector to the moduli space) to be in degree $1$ and allowing it to be an inhomogeneous element (and also allowing the generators of gauge transformations to be inhomogeneous rather than in degree $0$).
\item\label{rem: non-sc (c)} The part coming from unimodular/quantum operations $q_n$ on cohomology -- they encode the singular behavior of $R$-torsion as a function on the moduli space, near (not just at) the zero connection (or, more generally, a given local system if we twist by it).
\end{enumerate}
In the case $\pi_1(X)=0$, only (\ref{rem: non-sc (b)}) survives and gives the rational homotopy type. In the case $\pi_1(X)\neq 0$, (\ref{rem: non-sc (c)}) pertains to the simple-homotopy type, (\ref{rem: non-sc (a)}) is (a local model for) the character variety of the group $\pi_1(X)$ and one expects (\ref{rem: non-sc (b)}) to be again related to the rational homotopy type.
\end{remark}

\begin{Proposition}\label{prop: aggregations}
Let $X,Y$ be two 
regular CW
complexes such that $X$ is a \emph{cellular aggregation} of $Y$ (or, equivalently, $Y$ is a \emph{subdivision} of $X$). Then the BV pushforward $P_*^{Y\ra X}$ from $\sF_Y$ to $\sF_X$ relates the standard cellular actions $S_Y$ and $S_X$ by
\begin{equation}\label{aggreg}
P_*^{Y\ra X} \left( e^{\frac{i}{\hbar}S_Y}(\mu_{\sF_Y}^\hbar)^{1/2} \right)
= e^{\frac{i}{\hbar}S_X}(\mu_{\sF_X}^\hbar)^{1/2}+\Delta_X(\cdots)
\end{equation}
More precisely, 
the r.h.s. of (\ref{aggreg}) has the form $e^{\frac{i}{\hbar}S'_X}(\mu_{\sF_X}^\hbar)^{1/2}$ and for special ``geometric'' choices of gauge-fixing for the BV pusforward $P_*^{Y\ra X}$, we have a canonical transformation
\begin{equation}\label{aggreg can transf}
S'_X=S_X-i\hbar\{S_X,R^{(1)}\}
\end{equation}
with the generator $R^{(1)}(\sA_X)\in \mr{Fun}_{-1}(C^\bt(X,\g)[1])$  given by the linear in $\hbar$ term of the ansatz (\ref{Sbar cell},\ref{cellBF locality}) with coefficients $C^e_{\Gamma_1,e_1,\ldots,e_n}$ for a cell $e\subset X$ depending on the combinatorics of $Y|_e$ and the particular choice of geometric gauge-fixing. In this case, the $\Delta_X(\cdots)$ term in (\ref{aggreg}) is in fact  $\Delta_X\left( \frac{e^\phi-1}{\phi}\cdot(-i\hbar\, R^{(1)})\cdot e^{\frac{i}{\hbar}S_X}(\mu_{\sF_X}^\hbar)^{1/2} \right)$ where we denoted $-i\hbar\,\phi(\sA_X):=S'_X-S_X$.
\end{Proposition}

\begin{proof}
Let us write $Y\succ X$ if $X$ an aggregation of $Y$. We can always decompose $Y\succ X$ as a sequence 
\begin{equation}\label{aggreg decomp into single-cell aggreg}
Y=X_m\succ X_{m-1}\succ\cdots \succ X_0=X 
\end{equation}
where for each $0\leq k<m$, $X_{k+1}$ is a subdivision of $X_k$ where only a single cell $e$ of $X_k$ is subdivided and others (including cells of $\dd e$) are untouched.\footnote{
To obtain such a decomposition, order the cells $e_1,\ldots, e_m$ of $X$ in the order of non-decreasing dimension, as in the proof of Theorem \ref{thm: cellBF}. Then set $X_k:=\left(\bigcup_{i\leq k} Y|_{e_i}\right)\cup \left(\bigcup_{i>k} e_i\right)$.
}

First consider the case when 
$Y$ is a subdivision of a single cell $e$ of $X$. Then we can represent the aggregation $Y\succ X$ as a simple-homotopy equivalence -- a single elementary expansion (given by adjoining to $Y$ the pair of cells $\widetilde{e},e\subset \dd\widetilde{e}$) followed by a sequence of collapses
\begin{equation}\label{aggreg via simple-homotopy}
Y\nearrow\; (W=W_0)\; \searrow W_1\searrow\cdots \searrow\; (W_p=X)
\end{equation}
\begin{figure}[!htbp]
\begin{center}
\includegraphics[scale=0.5]{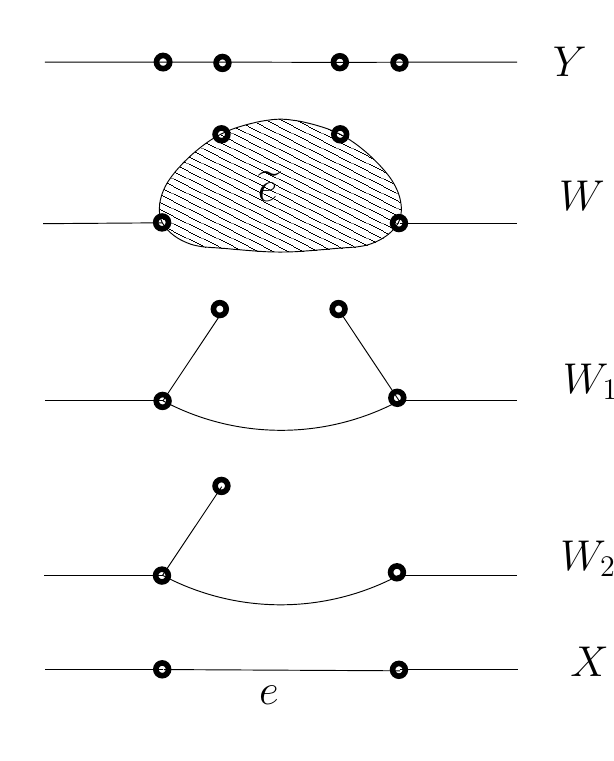}
\caption{Example of a single-cell aggregation presented as an expansion followed by a sequence of collapses.}
\end{center}
\end{figure}
Here one can regard $W$ as $(Y\cup_{X-e} X)\cup \widetilde{e}$; the boundary of the cell $\widetilde{e}$ is $\dd \widetilde{e}=\bar{e}\cup_{\dd e} \bar{e}_Y$, where $e_Y=Y|_e$ is a copy $e$ subdivided in $Y$, which is glued to a non-subdivided copy of $e$ along the equator $\dd e=\dd e_Y$ in the sphere $\dd \widetilde{e}$. 

Consider the standard cellular action $S_W$ for $W$. 
We have the standard induction data (\ref{collapse ind data}) for the collapses $W\searrow\cdots\searrow X$ which compose to $C^\bt(W)\wavy{(i_{W\ra X},p_{W\ra X},K_{W\ra X})} C^\bt(X)$ and also a standard induction data $C^\bt(W)\wavy{(i_{W\ra Y},p_{W\ra Y},K_{W\ra Y})} C^\bt(Y) $ for the collapse $W\searrow Y$. By Lemma \ref{lemma: collapse}, we have 
{\small
\begin{align}
\label{aggreg P_* W to X}
P_*^{W\ra X}\left( e^{\frac{i}{\hbar}S_W}(\mu_{\sF_W}^\hbar)^{1/2} \right) &= e^{\frac{i}{\hbar}S_X}(\mu_{\sF_X}^\hbar)^{1/2}+\Delta_X\left(-i\hbar\, 
r_{W\ra X}\cdot e^{\frac{i}{\hbar}S_X}(\mu_{\sF_X}^\hbar)^{1/2}\right),\\
\label{aggreg P_* W to Y}
P_*^{W\ra Y}\left( e^{\frac{i}{\hbar}S_W}(\mu_{\sF_W}^\hbar)^{1/2} \right) &= e^{\frac{i}{\hbar}S_Y}(\mu_{\sF_Y}^\hbar)^{1/2}+\Delta_Y\left(-i\hbar\, 
r_{W\ra Y}\cdot e^{\frac{i}{\hbar}S_Y}(\mu_{\sF_Y}^\hbar)^{1/2}\right)
\end{align}
}
for some degree $-1$ functions $r_{W\ra X}(\sA_X|_{\bar{e}})$ and $r_{W\ra Y}(\sA_Y|_{\bar{e}_Y})$. 
Here for the BV pushforwards we use the gauge-fixing associated with the standard induction data above. For the pushforward $P_*^{Y\ra X}$ corresponding to the aggregation $Y\ra X$, we will use the induction data 
\begin{equation}\label{aggreg ind data}
C^\bt(Y)\wavy{(p_{W\ra Y}\circ i_{W\ra X}, p_{W\ra X}\circ i_{W\ra Y}, p_{W\ra Y}\circ K_{W\ra X}\circ i_{W\ra Y})}C^\bt(X)
\end{equation}
-- this is the ``geometric'' gauge-fixing mentioned in the Proposition; it depends on choosing a particular presentation (\ref{aggreg via simple-homotopy}) of an aggregation as a simple-homotopy equivalence.\footnote{
Note also that induction data (\ref{aggreg ind data}) factors through \emph{integral} cochains, $C^\bt(Y,\mathbb{Z})\wavy{}C^\bt(X,\mathbb{Z})$.
}

Composition of this data with $C^\bt(W)\wavy{}C^\bt(Y)$ above yields the data $C^\bt(W)\wavy{} C^\bt(X)$ above. Thus with these choices one has 
$P_*^{Y\ra X}P_*^{W\ra Y}=P_*^{W\ra X}$ (a precise identity, not modulo canonical transformations); applying this to $e^{\frac{i}{\hbar}S_W}(\mu_{\sF_W}^\hbar)^{1/2}$ and using (\ref{aggreg P_* W to X},\ref{aggreg P_* W to Y}), we obtain 
\begin{multline}\label{aggreg P_* Y to X}
P_*^{Y\ra X}\left(e^{\frac{i}{\hbar}S_Y}(\mu_{\sF_Y}^\hbar)^{1/2}\right)=\\= 
e^{\frac{i}{\hbar}S_X}(\mu_{\sF_X}^\hbar)^{1/2}+\Delta_X\left(-i\hbar\, r_{Y\ra X}\cdot e^{\frac{i}{\hbar}S_X}(\mu_{\sF_X}^\hbar)^{1/2}\right)=: e^{\frac{i}{\hbar}S'_X}(\mu_{\sF_X}^\hbar)^{1/2}
\end{multline}
with $r_{Y\ra X}(\sA_X|_{\bar{e}}):=r_{W\ra X}-e^{-\frac{i}{\hbar}S_X} P_*^{Y\ra X}\left(r_{W\ra Y}\cdot e^{\frac{i}{\hbar}S_Y}(\mu_{\sF_Y}^\hbar)^{1/2}\right)\in \mr{Fun}_{-1}(C^\bt(\bar{e},\g)[1])$.
We then define $R^{(1)}(\sA_X|_{\bar{e}})$ as $R^{(1)}:=r_{Y\ra X}\cdot\frac{\phi}{e^{\phi}-1}$ with $\phi(\sA_X|_{\bar{e}}):=\frac{i}{\hbar}(S'_X-S_X)\in \mr{Fun}_0(C^\bt(\bar{e},\g)[1])$. Then, by a calculation similar to the one in the proof of Lemma \ref{lemma: collapse}, (\ref{aggreg P_* Y to X}) is equivalent to (\ref{aggreg can transf}). 
This proves the Proposition in the case when $Y$ is a single-cell subdivision of $X$.

In the case of a general (non single-cell) aggregation (\ref{aggreg decomp into single-cell aggreg}), 
we notice that $S'_X$, defined via $e^{\frac{i}{\hbar}S'_X}(\mu_{\sF_X}^\hbar)^{1/2}:= P_*^{Y\ra X}\left(e^{\frac{i}{\hbar}S_Y}(\mu_{\sF_Y}^\hbar)^{1/2}\right)$, with $P_*^{Y\ra X}=P_*^{X_1\ra X}\cdots P_*^{Y\ra X_{m-1}}$, satisfies the assumptions of Lemma \ref{lemma: cellBF well-definedness} and thus is connected to $S_X$ by a canonical transformation with generator satisfying the ansatz (\ref{Sbar cell},\ref{cellBF locality}). From the single-cell aggregation case we infer that the generator has only the linear in $\hbar$ part (while the $\hbar$-independent part vanishes). Such a transformation with a $t$-dependent generator $-i\hbar\, R^{(1)}_t$, $t\in [0,1]$ is equivalent to a transformation with constant ($t$-independent) generator $-i\hbar\,R^{(1)}=-i\hbar\,\int_0^1 dt\, R^{(1)}_t$.
%
%
%
\end{proof}

\subsection{Remarks}\label{sec: 8.4 remarks}
In this section we comment on the BV cohomology defined by the theory. We also prove that the theory converges in an appropriate sense, in the limit of dense triangulation (here we restrict to the simplicial case), to the standard continuum $BF$ theory on a manifold.

\subsubsection{BV cohomology} 

\begin{definition}
Let $\sF,\omega$ be an odd-symplectic space and $S\in \mr{Fun}_0(\sF)[[-i\hbar]]$ a solution of the quantum master equation on $\sF$. We define the \emph{perturbative BV cohomology at $S$}\footnote{We introduce this term to avoid confusion with the cohomology of the BV Laplacian $\Delta$ itself which is quite different (perturbative BV cohomology contains more information), see (\ref{BV coh 1}) vs. (\ref{BV coh 7}) below.} 
as the cohomology of the differential $\delta_S=\{S,\bt\}-i\hbar \Delta = e^{-\frac{i}{\hbar}S} \Delta e^{\frac{i}{\hbar}S}$ on $\mr{Fun}(\sF)[[-i\hbar]]$.
\end{definition}

In particular, cohomology of $\delta_S$ in degree zero controls infinitesimal deformations of $S$ as a solution of QME modulo canonical transformations (or, in other words, gives observables $O\in \mr{Fun}_0 (\sF)[[-i\hbar]]$ such that $\Delta(O\, e^{\frac{i}{\hbar}S})=0$ considered modulo infinitesimal equivalence $O\,e^{\frac{i}{\hbar}S}\sim O\, e^{\frac{i}{\hbar}S}+\epsilon\,\Delta(-i\hbar\, R\, e^{\frac{i}{\hbar}S})$).

Assume that we are in the setting of Section \ref{sec: uL_infty}, with $\sF=V[1]\oplus V^*[-2]$ for $V=V^\bt$ a graded vector space and $S$ satisfying the $BF_\infty$ ansatz (\ref{BF_infty ansatz general}). We note the following.  
\begin{enumerate}[(i)]
\item \label{BV coh 1}
Cohomology of $\Delta$ on $\mr{Fun}(\sF)$ is a line spanned by the element $\nu\in \bigotimes_k \mr{Det}(V^{2k})^* \otimes \bigotimes_k \mr{Det}\, V^{2k+1}=\mr{Det}(\sF^\mr{odd})^*$ given by the product of all components of $\sA$ ,$\sB$ of \emph{odd} internal degree, $H_\Delta (\mr{Fun}(\sF))=\nu\cdot \mathbb{R}$, with $\nu$ having internal degree $|\nu|=\sum_{k}(1-2k)\dim V^{2k}+\sum_k (-2+ (2k+1))\dim V^{2k+1}$.\footnote{This follows from writing $\sF$ as an odd cotangent bundle $T^*[-1]N$ of an evenly graded vector space $N=\bigoplus_k V^{2k+1}\oplus \bigoplus_k (V^{2k})^*$ spanned by \emph{even} components of $\sA$, $\sB$. Then $\mr{Fun}(\sF),\Delta$ is identified with the space of polyvector fields on $N$ with differential given by the divergence operator; this complex in turn is isomorphic to the de Rham complex of $N$ via \emph{odd Fourier transform} \cite{SchwarzBV}. Since $N$ is a vector space, its de Rham cohomology is given by constant $0$-forms on $V$; their odd Fourier transform gives $\nu\cdot \mathbb{R}$.
}  This cohomology was considered in \cite{Gwilliam}.
\item \label{BV coh 2}
We have 
\begin{equation}\label{BV coh classical}
H^k_{\{S^{(0)},\bt\}}(\mr{Fun}(\sF))= H^k_{CE}(V,\Sym V[2])
\end{equation}
where on the r.h.s. we have the Chevalley-Eilenberg cohomology of the $L_\infty$ algebra $V,\{l_n\}$ (with $l_n$ the $L_\infty$ operations corresponding to the $O(\hbar^0)$ term in $S$ via (\ref{BF_infty ansatz via l,q})) with coefficients in the sum of  symmetric powers of the adjoint module. We denote the cohomology (\ref{BV coh classical}) by $\mathfrak{H}$ -- one can view it as the ``classical'' ($\hbar$-independent) counterpart of perturbative BV cohomology.
\item \label{BV coh 3}
The operator $\{S^{(1)},\bt\}+\Delta$ can be viewed as a perturbation of the differential $\{S^{(0)},\bt\}$ on $\mr{Fun}(\sF)$ and thus induces, by homological perturbation theory (HPT) a differential $\mathfrak{D}$ on $\mathfrak{H}$. We then have
\begin{equation}
H^\bt_{\delta_S}(\mr{Fun}(\sF)[[-i\hbar]])=H^\bt_{-i\hbar \mf{D}}(\mf{H}[[-i\hbar]])=\left(\mf{H}^\bt\cap \ker\mf{D}\right)\oplus (-i\hbar)\cdot H^\bt_{\mf{D}}(\mf{H})[[-i\hbar]]
\end{equation}
This means that the perturbative BV cohomology in the order $O(\hbar^0)$ is given by $\mf{D}$-closed elements of classical BV cohomology (\ref{BV coh classical}), whereas in positive orders in $\hbar$ it is given by copies of cohomology of $\mf{D}$, one copy per order.
\item \label{BV coh 4}
Consider the ``abelian $BF$'' action $S^\mr{ab}:=\lan \sB, d\sA\ran $ on $\sF$, for $d=l_1$ the differential on $V$, -- it is the quadratic $\hbar$-independent term of the full $BF_\infty$ action $S$. We have $H^\bt_{\{S^\mr{ab},\bt\}}(\mr{Fun}(\sF))=\mr{Fun}(\sF_\zm)$ with $\sF_\zm=H^\bt(V)[1]\oplus (H^\bt(V))^*[-2]$ and the perturbation of the differential  $\{S^\mr{ab},\bt\} \mapsto \{S^\mr{ab},\bt\}+\Delta$ produces, as the induced differential, the standard BV Laplacian $\Delta_\zm$ on $\mr{Fun}(\sF_\zm)$. In particular, we have 
\begin{equation}\label{BV coh abelian}
H^\bt_{\{S^\mr{ab},\bt\}-i\hbar \Delta}=H^\bt_{-i\hbar \Delta_\zm}(\mr{Fun}(\sF_\zm)[[-i\hbar]])=\left(\mr{Fun}(\sF_\zm)\cap \ker \Delta_\zm\right)\oplus (-i\hbar)\nu_\zm\cdot \mathbb{R}[[-i\hbar]]
\end{equation}
with $\nu_\zm$ as in (\ref{BV coh 1}), with $V$ replaced by $H^\bt(V)$.

\item \label{BV coh 5} Assume that $V',d'$ is a retract of the complex $V,d$ and assume that $S'$ is the effective action on  $\sF'=V'[1]\oplus (V')^*[-2]$ induced from $S$ via BV pushforward. Note that $S'$ automatically satisfies $BF_\infty$ ansatz. 
Then, 
treating $\delta_S$ via HPT as a perturbation of the differential $\{S^\mr{ab},\bt\}-i\hbar\Delta$ on  $\mr{Fun}(\sF)[[-i\hbar]]$, one can prove that the induced differential on $\mr{Fun}(\sF')[[-i\hbar]]$ is precisely $\delta_{S'}$, with $S'$ given by Feynman diagrams for the BV pushforward (a 
version of this observation was made in \cite{GJF}; this statement is not specific to $BF_\infty$ ansatz).
Therefore, we have 
\begin{equation}
H^\bt_{\delta_S}(\mr{Fun}(\sF)[[-i\hbar]])=H^\bt_{\delta_{S'}}(\mr{Fun}(\sF')[[-i\hbar]])
\end{equation}

\item \label{BV coh 6} Let $S_\zm=S_\zm^{(0)}-i\hbar\, S_\zm^{(1)}\in \mr{Fun}(\sF_\zm)[[-i\hbar]]$ be the effective action induced on $\sF_\zm$ from $S$. Denote $w:=\left. S_\zm\right|_{\hbar=i}= S_\zm^{(0)}+ S_\zm^{(1)} \in \mr{Fun}(\sF_\zm)$. Using (\ref{BV coh 5}) and (\ref{BV coh 4}), we see that the cohomology
\begin{equation}
H^\bt_{\mf{D}}(\mf{H})=H^\bt_{\{w,\bt\}+\Delta_\zm}(\mr{Fun}(\sF_\zm))= e^{-w}\nu_\zm\cdot\mathbb{R}
\end{equation}
has rank $1$ and is concentrated in degree $|\nu_\zm|=\sum_k (1-2k)(\dim H^{2k}(V)-\dim H^{2k+1}(V))$.
\item \label{BV coh 7} 
Putting together (\ref{BV coh 3}) and (\ref{BV coh 6}), we obtain that
\begin{equation}
H^\bt_{\delta_S}(\mr{Fun}(\sF)[[-i\hbar]])=(\mf{H}^\bt\cap\ker\mf{D})\oplus (-i\hbar)e^{-w}\nu_\zm\cdot \mathbb{R}[[-i\hbar]]
\end{equation}

\end{enumerate}

Summarizing the observations above and applying to the case of cellular action of Theorem \ref{thm: cellBF}, and using Proposition \ref{prop: non-ab BF I simple homotopy invariance}, we have the following.

\begin{Proposition} For $X$ a regular CW complex, perturbative BV cohomology associated to the cellular action (\ref{cellBF locality}) has the form
\begin{equation}
H^\bt_{\delta_{S_X}}(\mr{Fun}(\sF_X)[[-i\hbar]])=(\mf{H}^\bt_X\cap\ker\mf{D}_X)\oplus (-i\hbar)\cdot \mathbb{R}[-\mf{p}][[-i\hbar]]
\end{equation}
where 
\begin{itemize}
\item $\mf{H}^\bt=H^\bt_{CE}(C(X,\g),\Sym C(X,\g)[2])\cong H^\bt_{CE}(H(X,\g),\Sym H(X,\g)[2])$ -- Chevalley-Eilenberg cohomology of cellular cochains and cellular cohomology, understood as $L_\infty$ algebras, with coefficients in the sum of symmetric powers of the adjoint module.
\item $\mf{D}_X$ is the differential on $\mf{H}^\bt$ induced by homological perturbation theory from the perturbation $\{S^{(0)},\bt\}\mapsto \{S^{(0)},\bt\}+\{S^{(1)}_X,\bt\}+\Delta_X$ of the differential on $\mr{Fun}(\sF_X)$.
Equivalently, $\mf{D}_X$ is induced from the perturbation  $\{S^{(0)}_\zm,\bt\}\mapsto \{S^{(0)}_\zm,\bt\}+\{S^{(1)}_\zm,\bt\}+\Delta_\zm$ of the differential on $\mr{Fun}(\sF_\zm)$.
\item Perturbative BV cohomology in higher orders in $\hbar$ has rank $1$ and is concentrated in degree
$$\mf{p}=\sum_k (1-2k)\left(\dim H^{2k}(X,\g)-\dim H^{2k+1}(X,\g)\right)$$
\item The generator $1\in\mathbb{R}$ of perturbative BV cohomology in positive degree in $\hbar$ is represented in $\mr{Fun}(\sF_X)$ by the element $p^* (e^{-w}\nu_\zm)$ where $p:\sF_X\ra \sF_\zm$ is the projection to residual fields 
used as a part of gauge-fixing in (\ref{Z sec 8.3}), $\nu_\zm\in \bigotimes_k\mr{Det}(H^{2k}(X,\g))^*\otimes \bigotimes_k \mr{Det}(H_{2k+1}(X,\g^*))^*=\mr{Det}(\sF_\zm^\mr{odd})^*$ is the product of all components of fields $\sA_\zm,\sB_\zm$ of odd internal degree; $w=S^{(0)}_\zm+S^{(1)}_\zm$ is the effective action on residual fields evaluated at $\hbar=i$.
\end{itemize}
If the complexes $X$ and $Y$ are simple-homotopy equivalent, the respective BV cohomology is canonically isomorphic,
$$H^\bt_{\delta_{S_X}}(\mr{Fun}(\sF_X)[[-i\hbar]])\cong H^\bt_{\delta_{S_Y}}(\mr{Fun}(\sF_Y)[[-i\hbar]])$$
\end{Proposition}

\begin{example}
For $X$ contractible, perturbative BV cohomology is the same as for $X$ a point, and thus we have $H^\bt_{\delta_{S_X}}=(H^\bt_{CE}(\g,\Sym \g[2])\cap \ker \mf{D})\oplus (-i\hbar)\mathbb{R}[-\dim\g][[-i\hbar]]$. One can understand $H^\bt_{CE}(\g,\Sym \g[2])\cap \ker \mf{D}$ as the subspace of \emph{unimodular classes} in Lie algebra cohomology. For example, in degree zero we have
$H^0_{\delta_{S_X}}=\bigoplus_k H^{2k}_{CE}(\g,\Sym^k \g)\cap \ker \mf{D}$. Contribution of $k=1$ here is the space of unimodular deformations of the Lie bracket on $\g$ modulo inner automorphisms (note that for $\g$ semi-simple, it vanishes since $H^2_{CE}(\g,\g)$ vanishes as a whole).
\end{example}

\begin{example} For $X$ arbitrary and $\g$ an abelian Lie algebra, perturbative BV cohomology is given by (\ref{BV coh abelian}). If we identify $\mr{Fun}(\sF_\zm)$ with polyvector fields on $H^\bt(X,\g)[1]$, then $\mr{Fun}(\sF_\zm)\cap \ker\Delta_\zm$ is the subspace of \emph{divergence-free} polyvector fields.
\end{example}

Returning to the setup of Section \ref{sec: uL_infty}, we can introduce a subspace 
\begin{multline*}
\Xi:=
\\:=\{\lan \sB, \alpha(\sA)\ran-i\hbar\, \beta (\sA)  \;\Big{|}\; \alpha(\sA)\in \widehat{\Sym^{\geq 2}} (V[1])^*\otimes V^*, \beta(\sA)\in \widehat{\Sym^{\geq 1}} (V[1])^* \}\\
=\Xi^{(0)}\oplus (-i\hbar)\cdot\Xi^{(1)}\quad \subset \mr{Fun}(\sF)[[-i\hbar]]
\end{multline*}
-- elements satisfying the $BF_\infty$ ansatz (\ref{BF_infty ansatz general}) with terms bi-linear in $\sA$ and $\sB$ prohibited. Note that $\Xi$ is a subcomplex w.r.t. $\delta_S$. Repeating the argument above, we obtain that the cohomology of $\delta_S$ on $\Xi$ is
$$H^\bt_{\delta_S}(\Xi)=H^\bt_{\{S^{(0)},\bt\}}(\Xi^{(0)})\cap \ker \mf{D}$$
with $\mf{D}$ as above; note that this cohomology does not have $O(\hbar^{\geq 1})$ terms.

\begin{remark} Note that $-i\hbar\cdot 1$ is a $\delta_S$-exact element in $\mr{Fun}(\sF)[[-i\hbar]]$. By inspecting the HPT argument above, one can construct an explicit primitive $\phi(\sA,\sB)$ with 
$\delta_S(\phi)=-i\hbar\cdot 1$; this element necessarily contains a component bi-linear in $\sA$ and $\sB$. Therefore, for any $c\in\mathbb{R}$ one can connect $S$ and $S-i\hbar\cdot c$ (or, equivalently, half-densities $e^{\frac{i}{\hbar}S}(\mu_\sF^\hbar)^{1/2}$ and $e^c\cdot e^{\frac{i}{\hbar}S}(\mu_\sF^\hbar)^{1/2}$) by a canonical transformation. But if the generator of the canonical transformation is prohibited to have a term bi-linear in $\sA$, $\sB$, then such transformation is prohibited. Thus our rigid normalization of half-densities $e^{\frac{i}{\hbar}S_X}(\mu_{\sF_X}^\hbar)^{1/2}$ and $Z$ (as defined by (\ref{Z sec 8.3})) is meaningful, since the restriction above for canonical transformations holds in all cases of relevance for us: in Lemmata \ref{lemma: cellBF well-definedness}, \ref{lemma: collapse}, in (\ref{Z non-ab can transf}) and Propositions \ref{prop: non-ab BF I simple homotopy invariance} and \ref{prop: aggregations}.
\end{remark}

\subsubsection{Continuum limit}\label{sss: cont limit}
Simplicial $BF$ action (\ref{simpBF locality}) approximates the standard action of non-abelian $BF$ theory defined on differential forms on a manifold, in the limit of ``dense'' triangulation, 
in the following sense.

Let $M$ be a smooth compact oriented $n$-dimensional manifold endowed with a Riemannian metric $g$.
Let $\{X_N\}$ for $N=1,2,\ldots$ be a sequence of triangulations of $M$ with the property that metric diameters of simplices of $X_N$ are bounded from above by $c/N$ for $c$ a constant. Fix $(A,B)\in \Omega^\bt(M,\g)[1]\oplus \Omega^\bt(M,\g)[n-2]$ a pair of \emph{smooth} non-homogeneous differential forms. We project this pair to an element of the cellular space of fields, $(\sA_N,\sB_N)\in \sF_{X_N}$, with $\sA_N:=\sum_{e\subset X_N} e^*\cdot \left(\int_e A\right)$, $\sB_N:=\sum_{e\subset X_N} \left(\int_M B\wedge \chi_e\right) \cdot e$ where $\chi_e$ is a piecewise-linear form on $M$ of degree $\dim e$ -- the Whitney form associated to $e$.

\begin{lemma} \label{lemma: cont limit}
We have the following asymptotics for the value of the $\hbar$-independent part of the simplicial action $S^{(0)}_{X_N}$  on $(\sA_N,\sB_N)$:
\begin{equation}
S_{X_N}^{(0)}(\sA_N,\sB_N)\underset{N\ra \infty}{\sim} \int_M \lan B \stackrel{\wedge}{,}dA+\frac12 [A\stackrel{\wedge}{,}A] \ran_\g + O\left(\frac{1}{N}\right)
\end{equation}
\end{lemma}
The first term on the r.h.s. is the standard action of non-abelian $BF$ theory in BV formalism, see e.g. \cite{CR,SimpBF,CMR}.

\begin{proof}
Let $i_N:C^\bt(X_N)\ra \Omega^\bt(M)$ be the inclusion of cellular cochains of the triangulation into piecewise-polynomial forms given by Whitney forms, and let $p_N:  \Omega^\bt(M)\ra C^\bt(X_N)$ be the Poincar\'e projection, as in Section \ref{sec: simp BF reminder}. Then, for $\alpha$ any smooth form on $M$, we have 
\begin{equation}\label{cont limit eq1}
\alpha - i_N\circ p_N(\alpha)=O\left(\frac1N\right)
\end{equation}
In particular, $A-i_N(\sA_N)=O\left(\frac1N\right)$ and $[A\stackrel{\wedge}{,}A]-i_N([\sA_N,\stackrel{*}{,}\sA_N])=O\left(\frac1N\right)$. 

Note that components of the cellular field $(\sA_N,\sB_N)$  on simplices of $X_N$ behave in the asymptotics $N\ra\infty$ as $\sA_e=O(N^{-\dim e})$, $\sB_e=O(N^{-n+\dim e})$. The number of simplices $e$ of $X^N$ of any fixed dimension behaves as $O(N^n)$.  Thus, we can estimate the term 
$$S^{(0),k}:=\sum_{e\subset X_N}\sum_{\Gamma_0}\sum_{e_1,\ldots,e_k\subset \bar{e}} \frac{1}{|\mr{Aut}(\Gamma_0)|}C^{e}_{\Gamma_0,e_1,\ldots, e_k} \lan \sB_e,\mr{Jacobi}_{\Gamma_0}(\sA_{e_1},\ldots, \sA_{e_k})\ran $$
in the cellular action $S_{X_N}$ as $O(N^{n}\cdot N^{-n+\dim e- \dim e_1-\cdots - \dim e_k})=O(N^{2-k})$ (where we use the relation $\dim e-\dim e_1-\cdots - \dim e_k=2-k$ between dimensions of the $k$-tuple of faces of $e$, which arises from the fact that the associated monomial in components of fields, $\lan \sB_e,\mr{Jacobi}_{\Gamma_0}(\sA_{e_1},\ldots, \sA_{e_k})\ran$, should have internal degree zero).  
In particular, we have\footnote{For this, we observe that the bound above can be improved to be uniform in $k$: structure constants $C^{\DDelta^m}_{\Gamma_0,e_1,\ldots,e_k}$ (for a simplex of fixed dimension $m$) have at most exponential growth in $k$, as follows from analyzing the explicit integral expressions for the structure constants arising from the construction (\ref{simpBF (i)}) of the proof of Theorem \ref{thm: simpBF}. Thus, one has $|S^{(0),k}|<\gamma^k N^{2-k}$ for some $\gamma$ depending on $(A,B)$ but independent of $k$.} 
\begin{equation}\label{cont limit eq2}
\sum_{k\geq 3}S^{(0),k}(\sA_N,\sB_N)=O\left(\frac{1}{N}\right)
\end{equation}

On the other hand, by (\ref{simpBF init cond}) we have $S^{(0)}_{X_N}=\lan \sB_N,d\sA_N \ran+\frac12 \lan \sB_N,[\sA_N\stackrel{*}{,} \sA_N]\ran+\sum_{k\geq 3}S^{(0),k}(\sA_N,\sB_N)$ and
\begin{multline*}
\left(\lan \sB_N,d\sA_N \ran+\frac12 \lan \sB_N,[\sA_N\stackrel{*}{,} \sA_N]\ran \right)-
\int_M \lan B \stackrel{\wedge}{,}dA+\frac12 [A\stackrel{\wedge}{,}A] \ran_\g\\
=\int_M \lan B \stackrel{\wedge}{,} (i_N\circ p_N-\mr{id})\left(dA+\frac12 [A\stackrel{\wedge}{,}A]\right) \ran_\g\quad = O\left(\frac{1}{N}\right)
\end{multline*}
where we used (\ref{cont limit eq1}). Together with the estimate (\ref{cont limit eq2}), this finishes the proof of the Lemma.

\end{proof}

\begin{remark}
For the linear in $\hbar$ part of the simplicial action $S^{(1)}_{X_N}$, similar estimates yield $S^{(1),k}(\sA_N)=O(N^{n-k})$ (for the part of homogeneous degree $k$ in $\sA_N$) and $S^{(1)}_{X_N}(\sA_N)=O(N^{n-2})$ for the whole.
\end{remark}

\begin{remark}
In Lemma \ref{lemma: cont limit}, one can allow $X_N$ to be a sequence of prismatic cellular decompositions of $M$ instead of triangulations (with the same bound $c/N$ on the diameters of prisms), according to Remark \ref{rem: prismatic}.\footnote{
The argument we used to prove Lemma \ref{lemma: cont limit} uses Whitney forms, and it is not clear how to extend it to CW decompositions with arbitrary cells.
}
\end{remark}


\section{Non-abelian cellular $BF$ theory, II: case of a cobordism }
\label{sec: non-ab BF cobordism}
In this section we address the construction of non-abelian cellular $BF$ theory in the formalism of Section \ref{sec: classical cell ab BF on mfd with bdry}.

Fix $G$ a Lie group with bi-invariant Haar measure $\mu_G$. The Lie algebra $\g=\mr{Lie}(G)$ is automatically unimodular and carries the induced density $\mu_\g$.

Let $M$ be a compact oriented piecewise-linear $n$-manifold with boundary $\dd M=\overline{M_\din}\sqcup M_\dout$. Let $(P,\nabla_P)$ be a flat $G$-bundle over $M$ and $(E,\nabla_E)$ the vector bundle associated to $P$ via adjoint representation, $E=\mr{Ad}(P)=P\times_G \g$; by construction it carries a horizontal fiberwise density $\mu_E$ induced from $\mu_G$. Let $X$ be a cellular decomposition of $M$;  we assume that $X$ is a regular CW complex of product type near the boundary $X_\dout$ (see Definition \ref{def: product type}).
We call such $X$ an \textit{admissible} cellular decomposition of $M$.


We import directly from Section \ref{sec: classical cell ab BF on mfd with bdry}, without changes, the construction of: 
\begin{itemize}
\item The space of fields 
$\F=C^\bt(X,E)[1]\oplus C^\bt(X^\vee,E^*)[n-2]$. (By a slight abuse of notation, we omit subscripts in $E_X$ and $E^*_{X^\vee}$.)
\item The superfields 
$$A=\sum_{e\subset X}e^*\cdot A_e:\;\F\ra C^\bt(X,E),\qquad B=\sum_{e^\vee\subset X^\vee} B_{e^\vee}\cdot (e^\vee)^*:\;\F\ra C^\bt(X^\vee,E^*)$$ 
Note that $B$ is again a cochain, unlike in the formalism of Section \ref{sec: simp BF reminder}.
\item The presymplectic 2-form $\omega=\lan \delta B,\delta A\ran$. 
\item The boundary fields $\F_\dd=C^\bt(X_\dd,E)[1]\oplus C^\bt(X_\dd^\vee,E^*)[n-2]$. 
\item The symplectic form $\omega_\dd$ and its primitive 1-form $\alpha_\dd$.
\item The projection $\pi:\F\twoheadrightarrow \F_\dd$. Later we will need to deform it, see (\ref{Pi^*}) below.
\end{itemize}

We define the action $S(A,B;\hbar)\in \mr{Fun}(\F)_0[[\hbar]]$ as follows:
\begin{multline}\label{S cell non-ab}
S(A,B;\hbar):=
\sum_{e\subset X-X_\dout} \bar{S}_e(A|_{\bar{e}}, B_{\varkappa(e)};\hbar)+ \lan B,A\ran_\din\\
=\sum_{e\subset X-X_\dout}\sum_{\nnn\geq 1}\frac{1}{\nnn!} \lan B_{\varkappa(e)}, l_\nnn^e\left( 
A|_{\bar{e}},\cdots,  
A|_{\bar{e}}\right)\ran
-i\hbar \sum_{e\subset X-X_\dout}\sum_{\nnn\geq 2}\frac{1}{\nnn!} q_\nnn^e\left( 
A|_{\bar{e}},\cdots,  
A|_{\bar{e}}\right)+ \lan B, A\ran_\din
\end{multline}
with $\bar{S}_e$ the building blocks (\ref{Sbar cell}) 
of Theorem \ref{thm: cellBF} and $l_\nnn^e,q_\nnn^e$ the corresponding components of the operations of the unimodular $L_\infty$ algebra 
coming from (\ref{Sbar via l and q}). 
In (\ref{S cell non-ab}) we use the parallel transport $E(e>e')$  to trivialize the coefficient local system over the cell $e$ (see (\ref{cell local system from flat bundle})) and we set 
\begin{equation}\label{A|_e}
A|_{\bar{e}}:=\sum_{e'\subset e} (e')^*\cdot \mr{Ad}_{E(e>e')} A_{e'}:\quad \F\ra C^\bt(\bar{e})\otimes \g
\end{equation}
We also use a shorthand notation $\lan B,A\ran_\din:= \lan (\iota^\vee)^*B,\iota^* A\ran_\din=\sum_{e\subset X_\din}\lan B_{\varkappa_\din(e)},A_e\ran$ for the boundary term of the action, as in (\ref{S with bdry term}). As in Section \ref{sec: non-ab BF I}, the action has the structure  $S=S^{(0)}-i\hbar\, S^{(1)}$.

We also define $S_\dd=S_\dout-S_\din \in \mr{Fun}_1(\F_\dd)$ by
\begin{equation}\label{S bdry non-ab}
\!\!\!\!\!\!\!\!
S_\dd(A_\dd,B_\dd):=\sum_{e\subset X_\dout} \sum_{\nnn\geq 1}\frac{1}{\nnn!} 
\lan B_{\varkappa_\dout(e)}, l_\nnn^e\left( 
A|_{\bar{e}},\cdots,
A|_{\bar{e}}\right) \ran-
\sum_{e\subset X_\din} \sum_{\nnn\geq 1}\frac{1}{\nnn!} 
\lan B_{\varkappa_\din(e)}, l_\nnn^e\left( 
A|_{\bar{e}},\cdots,
A|_{\bar{e}}\right) \ran
\end{equation}

Further, we introduce the vector field $Q^{(0)}\in \mathfrak{X}(\F)_1$,
\begin{multline}\label{Q non-ab cob}
Q^{(0)}=\sum_{e\subset X}\sum_{\nnn\geq 1}(-1)^{\dim e}\frac{1}{\nnn!} \lan  l_\nnn^{e}(A|_{\bar{e}},\cdots,A|_{\bar{e}}),\frac{\dd}{\dd A_e} \ran-
\\-
\sum_{e\subset X-X_\dout}\sum_{\nnn\geq 0} \frac{1}{\nnn!} \lan B_{\varkappa(e)}, l_{\nnn+1}^e\left(A|_{\bar{e}},\cdots,A|_{\bar{e}},\sum_{e'\subset \bar{e}-X_\dout}(e')^*\cdot \mr{Ad}_{E(e>e')}\frac{\dd}{\dd B_{\varkappa(e')}} \right) \ran\\
+\sum_{e\subset X_\din}\sum_{\nnn\geq 0} \frac{1}{\nnn!} \lan B_{\varkappa_\din(e)}, l_{\nnn+1}^e\left(A|_{\bar{e}},\cdots,A|_{\bar{e}},\sum_{e'\subset \bar{e}}(e')^*\cdot \mr{Ad}_{E(e>e')}\frac{\dd}{\dd B_{\varkappa_\din(e')}} \right) \ran\\
-\sum_{e\subset X_\din} \lan B_{\varkappa_\din(e)},\frac{\dd}{\dd B_{\varkappa(e)}} \ran
\end{multline}

Introduce the algebra homomorphism $\Pi^*:\; \mr{Fun}(\F_\dd)\ra \mr{Fun}(\F)$ deforming the pullback by the standard projection $\pi^*$, defined on the generators as follows:
\begin{equation}\label{Pi^*}
\Pi^*:\left\{\begin{array}{l}  
A_\din\mapsto A_\din \\
B_\din \mapsto B_\din \\
A_\dout \mapsto A_\dout \\
\lan B_\dout, a\ran_\dout \mapsto 
\sum_{e\supset e'} \sum_{\nnn\geq 0} \frac{1}{\nnn!} \lan B_{\varkappa(e)},l_\nnn^{e}(A|_{\bar{e}},\cdots, A|_{\bar{e}}, (e')^*\cdot \mr{Ad}_{E(e>e')} a_{e'}  ) \ran  
\end{array}\right. 
\end{equation}
where $a=\sum_{e\subset X_\dout}e^*\cdot a_e\in C^\bt(X_\dout)$ is a test cochain; the sum is over pairs of cells $e\supset e'$ such that $e'\subset X_\dout$ while $e\subset X-X_\dout$. 

The following is proven by a straightforward (but lengthy) computation, using the $\bmod\,\hbar$ part of the result of Theorem \ref{thm: simpBF}.
\begin{Proposition}\label{prop: non-ab class BV-BFV}
The data $(\F,\omega,Q^{(0)},S^{(0)},\Pi^*)$, $(\F_\dd,\alpha_\dd,S_\dd)$ define a classical BV-BFV theory \cite{CMR}, i.e.
the following relations hold:
\begin{enumerate}[(a)]
\item $(Q^{(0)})^2=0$.
\item $\iota_{Q^{(0)}}\omega = \delta S^{(0)}+\Pi^* \alpha_\dd$.
\item $Q^{(0)}\Pi^*=\Pi^* Q_\dd$ where $Q_\dd \in \mathfrak{X}(\F_\dd)_1$ is the Hamiltonian vector field generated by $S_\dd$, i.e. is defined by $\iota_{Q_\dd}\omega_\dd=\delta S_\dd$. 
\end{enumerate}
\end{Proposition}

The space of states $\HH_\dd=\mr{Fun}_\mathbb{C}(\mathcal{B}_\dd)\ni \psi(B_\din,A_\dout)$ is defined as in Section \ref{sec: quantum cell ab BF on mfd with bdry}, with the base of polarization $\mathcal{B}_\dd$ the same as in (\ref{polarization}).
It splits as $\HH_\dd=\HH_\din^{(B)}\widehat{\otimes} \HH_{\dout}^{(A)}$, as in (\ref{H = H_in otimes H_out}).\footnote{
Recall from Section \ref{sec: quantum cell ab BF on mfd with bdry} that we have, in fact, two models for the space of states, $\HH_\dd$ (functions on $\B_\dd$) and $\HH_\dd^\can$ (half-densities on $\B_\dd$). The two models are isomorphic and the comparison goes via multiplication by the (appropriately normalized) reference half-density $\psi\mapsto (\mu_{\B_\dd}^\hbar)^{1/2}\psi $. 
} 

We define the degree $1$ operator $\widehat{S}_\dd=\widehat{S}_\dout+\widehat{S}_\din $ on $\HH_\dd$ as 
the quantization 
of $S_\dd$, obtained by replacing 
$B_{\varkappa_\dout(e)}\mapsto -i\hbar (-1)^{\dim e}\frac{\dd}{\dd A_e}$ 
for $e\subset X_\dout$ and $A_e\mapsto -i\hbar \frac{\dd}{\dd B_{\varkappa_\din(e)}}
$ for $e\subset X_\din$ in  (\ref{S bdry non-ab}) and putting 
all derivatives in $A$ or $B$ to the right.

Explicitly, for $\widehat{S}_\dout,\widehat{S}_\din$ we have
\begin{eqnarray}\label{Shat out}
\widehat{S}_\dout &=& \sum_{e\subset X_\dout} \sum_{\nnn\geq 1}\frac{1}{\nnn!} 
\lan  l_\nnn^e\left( 
A|_{\bar{e}},\cdots,
A|_{\bar{e}}\right), -i\hbar\,(-1)^{\dim e}\frac{\dd}{\dd A_e} \ran \\
\label{Shat in}
\widehat{S}_\din &=& \sum_{e\subset X_\din} \sum_{\nnn\geq 1}
\frac{1}{\nnn!} 
\lan B_{\varkappa_\din(e)}, l_\nnn^e\left( 
\widehat{A}|_{\bar{e}},
\cdots,
\widehat{A}|_{\bar{e}}\right) \ran
\end{eqnarray}
where $\widehat{A}|_{\bar{e}}:=-i\hbar\sum_{e'\subset \bar{e}} (e')^*\cdot \mr{Ad}_{E(e>e')}\frac{\dd}{\dd B_{\varkappa_\din(e')}}$.

It is convenient to introduce, alongside $\widehat{S}_\dout$, its version acting on out-states \emph{from the right}, 
\begin{equation}\label{Shat out right}
\ola{\widehat{S}}_\dout= \sum_{e\subset X_\dout} \sum_{\nnn\geq 1}\frac{1}{\nnn!} 
\lan -i\hbar\,\frac{\ola\dd}{\dd A_e}, l_\nnn^e\left( 
A|_{\bar{e}},\cdots,
A|_{\bar{e}}\right) \ran
\end{equation} 
It satisfies $\widehat{S}_\dout \psi=(-1)^{|\psi|+1}\psi \ola{\widehat{S}}_{\dout}$ for any $\psi=\psi(A_\dout)\in \HH_\dout^{(A)}$.

\begin{remark} For $Y$ a cellular decomposition of a closed $(n-1)$-manifold, the operators  $\ola{\widehat{S}}_Y^{(A)}$ and $\widehat{S}_Y^{(B)}$ defined by formulae (\ref{Shat out right},\ref{Shat in}) on $Y$, are mutually adjoint w.r.t. the pairing (\ref{pairing H^A with H^B}): 
$$(\phi\circ \ola{\widehat{S}}^{(A)}_Y,\psi)=(\phi, \widehat{S}^{(B)}_Y\circ \psi)$$
for $\phi(A_Y)\in \HH_Y^{(A)}$ and $\psi(B_Y)\in \HH_Y^{(B)}$.
\end{remark}

\begin{Proposition}
\begin{enumerate}
[(a)]
\item \label{prop non-ab: item a} The operator $\widehat{S}_\dd$ defined as above squares to zero.
\item \label{prop non-ab: item b} The action (\ref{S cell non-ab}) on a  cobordism $M$ endowed with an admissible cellular decomposition $X$ satisfies the modified quantum master equation:
\begin{equation}\label{mQME non-ab}
\left(\frac{i}{\hbar}\widehat{S}_\dd-i\hbar\Delta_\bulk\right)\,e^{\frac{i}{\hbar}S} = 0
\end{equation}
\end{enumerate}
\end{Proposition}

\begin{proof} 
Part (\ref{prop non-ab: item a}) is an immediate consequence of the classical $L_\infty$ relations (\ref{L_infty relations}) for the boundary complex $X_\dd$, which in turn follow from Theorem \ref{thm: simpBF} applied to $X_\dd$.\footnote{
Note that in our case the cochains on $X_\dd$ are twisted by a local system but the action (\ref{cellBF locality}) still satisfies the master equation, with the new definition (\ref{A|_e}) of $A|_{\bar{e}}$, as can be seen by inspecting the proof of Theorem \ref{thm: cellBF}: we have quantum master equations on cells, where the local system is trivialized and this implies (by the gluing procedure (\ref{simpBF (vi)}) of the proof of Theorem \ref{thm: simpBF}) that (\ref{cellBF locality}) is a solution of the master equation.
}

Let us prove part (\ref{prop non-ab: item b}). First, observe from (\ref{S cell non-ab}) that $S$ depends on $B_\din$ only via the boundary term $\lan B,A\ran_\din$. This implies that $  \widehat{S}_\din\circ e^{\frac{i}{\hbar}S}= S_\din\cdot e^{\frac{i}{\hbar}S}$.\footnote{
In more detail, 
we have
{\scriptsize
\begin{multline*}
\widehat{S}_\din\circ e^{\frac{i}{\hbar}S}= \left(\sum_{e\subset X_\din}\sum_{\nnn\geq 1}\frac{1}{\nnn!} 
\lan B_{\varkappa_\din(e)},l_\nnn^e\left(\widehat{A}|_{\bar{e}},\cdots, \widehat{A}|_{\bar{e}}\right) \ran
\right)\circ e^{\frac{i}{\hbar}S} \\=
\left(\sum_{e\subset X_\din}\sum_{\nnn\geq 1}\frac{1}{\nnn!} 
\lan B_{\varkappa_\din(e)},l_\nnn^e\left(\widehat{A}|_{\bar{e}}\circ\frac{i}{\hbar}S,\cdots, \widehat{A}|_{\bar{e}}\circ\frac{i}{\hbar}S\right) \ran
\right)\cdot e^{\frac{i}{\hbar}S} \\=
\left(\sum_{e\subset X_\din}\sum_{\nnn\geq 1}\frac{1}{\nnn!} \lan B_{\varkappa_\din(e)},l_\nnn^e\left( A|_{\bar{e}},\cdots,  A|_{\bar{e}}\right) \ran
\right)\cdot e^{\frac{i}{\hbar}S}= S_\din\cdot e^{\frac{i}{\hbar}S}.
\end{multline*}
}
} Also note that $\widehat{S}_\dout$ is a first order differential operator.
Therefore, we have the following
\begin{multline}\label{prop 8.8 eq1}
-i\hbar\; e^{-\frac{i}{\hbar}S}\left(\frac{i}{\hbar}\widehat{S}_\dd-i\hbar\Delta_\bulk\right)\,e^{\frac{i}{\hbar}S}=\frac12 \{S,S\}_{\omega_b}-\frac{i}{\hbar} S\circ\ola{\widehat{S}}_\dout+S_\din-i\hbar\Delta_\bulk S
\end{multline}
Here $\{,\}_{\omega_b}$ is the same Poisson bracket as in the proof of Lemma \ref{lemma 7.3}. 
We calculate 
\begin{multline*} 
\frac12 \{S^{(0)},S^{(0)}\}_{\omega_b}=
-\sum_{e\subset X-X_\dout}S^{(0)}\lan \frac{\ola\dd}{\dd A_e} ,\frac{\ora\dd}{\dd B_{\varkappa(e)}} \ran S^{(0)} 
\\
=-\sum_{e, e' \subset X-X_\dout,\;e\subset \bar{e}'}\; \sum_{r,s\geq 1} \frac{1}{r!s!} \lan B_{\varkappa(e')}, l^{e'}_{r+1}(A|_{\bar{e}'},\cdots,A|_{\bar{e}'},e^*\cdot\mr{Ad}_{E(e'>e)} l_s^{e}(A|_{\bar{e}},\cdots, A|_{\bar{e}})) \ran\\
-\underbrace{\sum_{e\subset X_\din}\lan B_{\varkappa_\din(e)} ,\frac{\ora \dd}{\dd B_{\varkappa(e)}} \ran S^{(0)}}_{S_\din}
\end{multline*}
Last term on the right is the contribution of the in-boundary term in $S$ to the Poisson bracket. If the first sum above were over all pairs of $e'$ -- a cell of $X-X_\dout$ and $e$ its (arbitrary codimension) face (which can be on $X_\dout$), the sum would vanish by classical $L_\infty$ relations on $X$ following from Theorem \ref{thm: cellBF}. Therefore, we continue:
\begin{multline}\label{prop 8.8 eq2}
\frac12 \{S^{(0)},S^{(0)}\}_{\omega_b}=\\
=\sum_{e'\subset X-X_\dout}\; \sum_{e\subset \bar{e}'\cap X_\dout}\; \sum_{r,s\geq 1} \frac{1}{r!s!} \lan B_{\varkappa(e')}, l^{e'}_{r+1}(A|_{\bar{e}'},\cdots,A|_{\bar{e}'},e^*\cdot\mr{Ad}_{E(e'>e)}l_s^{e}(A|_{\bar{e}},\cdots, A|_{\bar{e}})) \ran-S_\din \\
=\sum_{e\subset X_\dout}S^{(0)}\lan \frac{\ola\dd}{\dd A_e} ,\frac{\ora\dd}{\dd B_{\varkappa_\dout(e)}} \ran S_\dout-S_\din=\frac{i}{\hbar}S^{(0)}\circ \ola{\widehat{S}}_\dout-S_\din
\end{multline}
Similarly, we have
\begin{multline*}
\{S^{(1)},S^{(0)}\}_{\omega_b}= -\sum_{e\subset X-X_\dout}S^{(1)}\lan \frac{\ola\dd}{\dd A_e} ,\frac{\ora\dd}{\dd B_{\varkappa(e)}} \ran S^{(0)}\\
=-\sum_{e, e' \subset X-X_\dout,\;e\subset \bar{e}'}\; \sum_{r,s\geq 1} \frac{1}{r!s!}\, q^{e'}_{r+1}(A|_{\bar{e}'},\cdots,A|_{\bar{e}'},e^*\cdot\mr{Ad}_{E(e'>e)}l_s^{e}(A|_{\bar{e}},\cdots, A|_{\bar{e}}))
\end{multline*}
Which, by the unimodular $L_\infty$ relations (\ref{uL_infty homotopy unimodularity}) on $X$, and the argument as above, yields
\begin{equation}\label{prop 8.8 eq3}
\{S^{(1)},S^{(0)}\}_{\omega_b}=-\Delta_\bulk S^{(0)}+\frac{i}{\hbar}S^{(1)}\circ \ola{\widehat{S}}_\dout
\end{equation}
Putting (\ref{prop 8.8 eq2}) and (\ref{prop 8.8 eq3}) into (\ref{prop 8.8 eq1}), we obtain the modified quantum master equation (\ref{mQME non-ab}).

\end{proof}

\begin{example}
Consider $M$ an interval viewed as a cobordism between an in-point and an out-point, with $X$ a cellular decomposition with $N\geq 1$ $1$-cells which we denote  $[01],[12],\ldots,[N-1,N]$ and $N+1$ $0$-cells denoted $[0],[1],\ldots,[N]$. 
The dual CW complex $X^\vee$ has $0$-cells $[0]^\vee,\ldots, [N]^\vee$ and $1$-cells $[01]^\vee,\ldots, [N-1,N]^\vee$.
$$\includegraphics[scale=0.6]{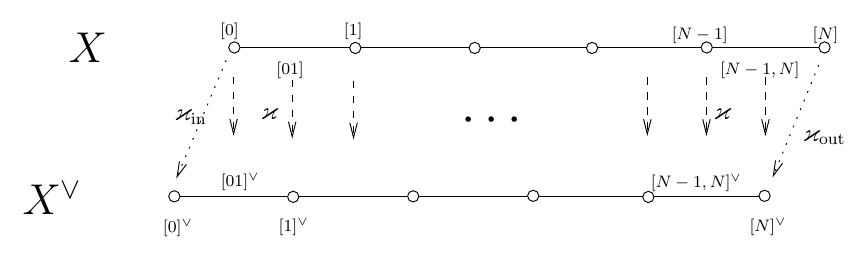}$$
The data of the local system $E_X$ is a collection of group elements 
$$u_k:= E([k,k+1]>[k])\;\in G,\qquad v_{k+1}:=E([k,k+1]>[k+1])\;\in G$$ 
for $k=0,1,\ldots,N-1$.
The superfields are:
$$ A= \sum_{k=0}^N [k]^* A_{[k]} + \sum_{k=0}^{N-1} [k,k+1]^* A_{[k,k+1]},\qquad
B=\sum_{k=0}^N ([k]^\vee)^* B_{[k]^\vee} + \sum_{k=0}^{N-1} ([k,k+1]^\vee)^* B_{[k,k+1]^\vee} $$
where $A_{[k]},A_{[k,k+1]}\in \g$ have degrees $1$ and $0$, respectively; degrees of $B_{[k]^\vee},B_{[k,k+1]^\vee}\in\g^*$ are $-1$ and $-2$, respectively.
We have 
\begin{multline*} 
S=\sum_{k=0}^{N-1} \lan B_{[k,k+1]^\vee},\frac12[A_{[k]},A_{[k]}] \ran+\\ 
+\sum_{k=1}^N \lan B_{[k]^\vee},\left[A_{[k-1,k]},\frac{\Ad_{u_{k-1}}A_{[k-1]}+\Ad_{v_k} A_{[k]}}{2}\right]+\mathbb{F}(\mr{ad}_{A_{[k-1,k]}})\circ (\Ad_{v_k} A_{[k]}-\Ad_{u_{k-1}}A_{[k-1]}) \ran\\
 -i\hbar \sum_{k=1}^N \tr \log 
 \mathbb{G}(\mr{ad}_{A_{[k-1,k]}}) + \lan B_{[0]^\vee},A_{[0]} \ran
\end{multline*}
with functions $\mathbb{F}$ and $\mathbb{G}$ as in (\ref{simpBF interval F,G}).
The nontrivial component of the map $\Pi^*$ (\ref{Pi^*}) is 
$$\Pi^*:\;\; B_{[N]^\vee}\mapsto \frac{\dd }{\dd A_{[N]}}S=\Ad^*_{v_N^{-1}}\mathbb{H}(\mr{ad}^*_{A_{[N-1,N]}}) B_{[N]^\vee}$$ where $\mathbb{H}(x):=-\frac{x}{2}+\mathbb{F}(-x)=\frac{x}{e^x-1}$.
Boundary action is
$$ S_\dd = \lan B_{[N]^\vee},\frac12 [A_{[N]},A_{[N]}] \ran - 
\lan B_{[0]^\vee},\frac12 [A_{[0]},A_{[0]}] \ran $$
\end{example}
The space of states
$$\HH_\dd=\mr{Fun}(\g^*[-1])\otimes \mr{Fun}(\g[1]) =  C_{-\bt}^{CE}(\g)\otimes C^\bt_{CE}(\g)\quad \ni \psi(B_{[0]^\vee},A_{[N]})$$
can be identified with the Chevalley-Eilenberg cochain complex tensored with its dual (Lie algebra chains with opposite grading). The differential on states is
$$ \widehat{S}_\dd= -i\hbar\; \lan \frac12 [A_{[N]},A_{[N]}], \frac{\dd}{\dd A_{[N]}} \ran 
-i\hbar \lan B_{[0]^\vee},\frac12 \left[\frac{\dd}{\dd B_{[0]^\vee}},\frac{\dd}{\dd B_{[0]^\vee}}\right] \ran $$
-- the sum of standard Lie algebra cochain and chain differentials (up to normalization). Its cohomology is 
$H^i_{\widehat{S}_\dd}(\HH_\dd)=\bigoplus_{-j+k=i} H_j^{CE}(\g)\otimes H^k_{CE}(\g) $.

\subsection{Perturbative partition function on a cobordism: pushforward to cohomology in the bulk}\label{sec: non-ab quantization cob}
We proceed as in Sections \ref{sec 7.1 gauge-fixing}, \ref{sec 7.2 Z pert} to define the perturbative partition function as the BV pushforward of the non-abelian $BF$ theory on a cobordism $M$ endowed with admissible cellular decomposition $X$ from ``cellular bulk fields'' $\F_b=C^\bt(X,X_\dout)[1]\oplus C^\bt(X^\vee,X^\vee_\din)[n-2]$ to $\F_b^\zm=H^\bt(M,M_\dout)[1]\oplus H^\bt(M,M_\din)[n-2]$, with gauge-fixing inferred from a choice of induction data $C^\bt(X,X_\dout)\wavy{(\ii,\pp,\K)}H^\bt(M,M_\dout)$. Namely, we define 
$Z(B_\din,A_\dout;A_\zm,B_\zm) \in \HH_\dd^\can \widehat\otimes \Dens^{\frac12,\Fun}_\mathbb{C}(\F_b^\zm)$ -- thought of as a boundary state with coefficients in half-densities of bulk residual fields -- by formula (\ref{Z integrating out bulk fields}), for the non-abelian cellular action (\ref{S cell non-ab}).

The following statement generalizes (\ref{Z integrating out bulk fields result}) to the non-abelian setting and is the result of a straightforward perturbative computation of the fiber integral defining $Z$.
\begin{Proposition} \label{prop: Z non-ab}
For a cobordism $M$ endowed with an admissible cellular decomposition $X$ and a $G$-local system $E$ in adjoint representation, we have the following.
Explicitly, the partition function $Z$ has the form
\begin{equation}\label{Z non-ab cob explicit}
Z=e^{\frac{i}{\hbar}S_\mr{eff}(B_\din,A_\dout;A_\zm,B_\zm)}\;\xi_\hbar^{H^\bt(M,M_\dout)}\cdot \tau(M,M_\dout)\cdot (\mu_{\B_\dd}^\hbar)^{1/2}
\end{equation}
with the constant factor as in (\ref{Z integrating out bulk fields result}) and with 
\begin{equation}\label{S_eff non-ab cob}
S_\mr{eff}=\sum_{\Gamma}\frac{(-i\hbar)^{\mr{loops}(\Gamma)+V^q(\Gamma)}}{|\mr{Aut}(\Gamma)|}\;\varphi_\Gamma(B_\din,A_\dout;A_\zm,B_\zm)
\end{equation}
where the sum runs over connected oriented graphs $\Gamma$ on $M$ with:
\begin{itemize}
\item Oriented edges, with source and target half-edge placed\footnote{We talk here about ``placing'' elements of the graph $\Gamma$ at cells and decorating them with particular tensors depending on the placement. In the end, to obtain the Feynman weight of the graph $\varphi_\Gamma$, we sum over placements the contraction of the respective tensors. Sum over placements here is a cellular analog of configuration space integrals defining the weights of Feynman graphs in \cite{CMRpert}.} at cells $e^\vee\subset X^\vee-X_\din^\vee$ and $e\subset X-X_\dout$ respectively, 
decorated with minus the propagator $-\KK(e,e^\vee)\in E_{\dot{e}}\otimes E^*_{\dot{e}^\vee}$ (see Remark \ref{rem: propagator cob}).
\item  $V_\dout$ univalent vertices (with outgoing half-edge), placed at cells $e\subset X_\dout$ with the adjacent half-edge placed at $\varkappa_\dout(e)$; such a vertex is decorated with $\mr{Ad}_{E(\varkappa^{-1}\varkappa_\dout(e)>e)}\circ (A_\dout)_e$.  
\item $V_\din$ univalent vertices (with incoming half-edge), placed at cells $e^\vee\subset X_\din^\vee$ with the adjacent half-edge placed at $\varkappa_\din^{-1}(e^\vee)$; such a vertex is decorated with $(B_\din)_{e^\vee}$.
\item $V^l$  
bulk vertices placed at cells $e\subset X-X_\dout$, with one outgoing half-edge also placed at $e$ and with $k\geq 2$ incoming half-edges placed at faces of arbitrary codimension $e_1,\ldots,e_k\subset e$; the decoration is: $$l^e_{k,e_1,\ldots,e_k}\circ(\mr{Ad}_{E(e>e_1)}\otimes\cdots \otimes \mr{Ad}_{E(e>e_k)}) \in \mr{Hom}(\bigotimes_{j=1}^k E_{\dot{e}_j}, E_{\dot{e}} )\cong E_{\dot{e}}\otimes \bigotimes_{j=1}^k E^*_{\dot{e}_j} $$
with $l^e_{k,e_1,\ldots,e_k}\in \mr{Hom}(\g^{\otimes k},\g)$ a local component of the $k$-ary $L_\infty$ operation on $C^\bt(X,\g)$ determined by (\ref{Sbar cell}).
\item $V^q\in\{0,1\}$ bulk vertices placed at cells $e\subset X-X_\dout$, with no outgoing and $k\geq 2 $ incoming half-edges placed at faces of arbitrary codimension $e_1,\ldots, e_k\subset e$; the decoration is:
$$q^e_{k,e_1,\ldots,e_k}\circ(\mr{Ad}_{E(e>e_1)}\otimes\cdots \otimes \mr{Ad}_{E(e>e_k)}) \in \mr{Hom}(\bigotimes_{j=1}^k E_{\dot{e}_j}, \mathbb{R} )\cong \bigotimes_{j=1}^k E^*_{\dot{e}_j}$$
with $q^e_{k,e_1,\ldots,e_k}\in \mr{Hom}(\g^{\otimes k},\mathbb{R})$ a local component of the $k$-ary \emph{unimodular} $L_\infty$ operation on $C^\bt(X,\g)$ determined by (\ref{Sbar cell}).
\item $V_\zm^A$ leaves (loose half-edges), oriented towards the vertex and placed at a cell $e$. Decoration: $(\ii A_\zm)_e$,
\item $V_\zm^B$ leaves, oriented from the vertex and placed at a cell $e^\vee$. Decoration: $(\pp^\vee B_\zm)_{e^\vee}$,
\end{itemize}
The value of $\varphi_\Gamma$ in (\ref{S_eff non-ab cob}) is the sum over all placements of half-edges and vertices of the graph at cells of $X$ (subject to restrictions to boundary strata and local relations between placement of half-edges and vertices as above), 
of products of all decorations (with tensors in the fibers of the local system contracted in the way prescribed by the graph $\Gamma$).\footnote{
Note that $\varphi_\Gamma$ is a polynomial in the variables $(B_\din,A_\dout,A_\zm,B_\zm)$ of degree $(V_\din,V_\dout,V^A_\zm,V^B_\zm)$.}
\end{Proposition}

\begin{remark} Graphs contributing to (\ref{S_eff non-ab cob}) fall into four types (we provide each type with a picture of a typical example):
\begin{enumerate}[(I)]
\item Rooted trees with the root decorated with $B_\din$ and leaves either decorated by $A_\zm$ or by $A_\dout$.
$$\includegraphics[scale=0.35]{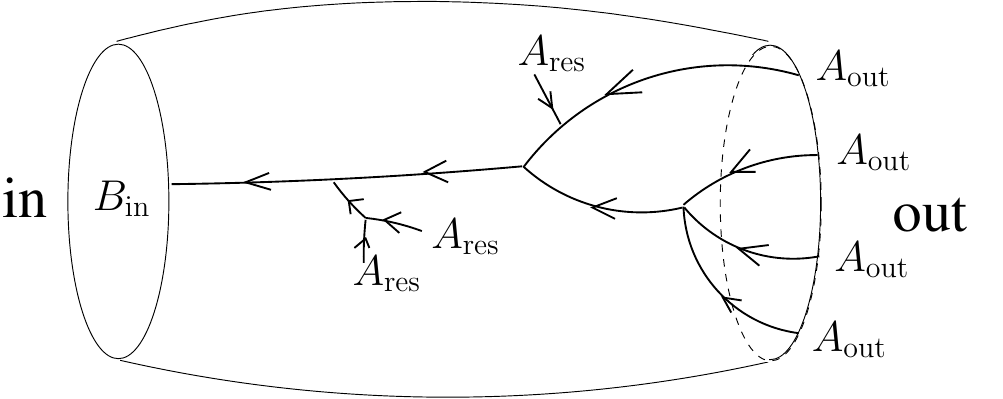}$$
\item   Rooted trees with the root decorated with $B_\zm$ and leaves decorated by $A_\zm$ or  $A_\dout$.
$$\includegraphics[scale=0.35]{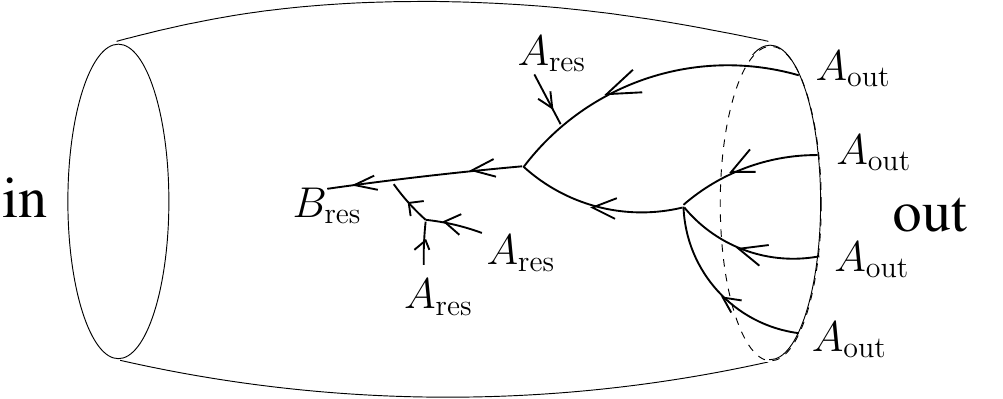}$$
\item Rooted trees with the root decorated with a quantum operation $q_k^e$ and leaves decorated by $A_\zm$ or $A_\dout$.
$$\includegraphics[scale=0.35]{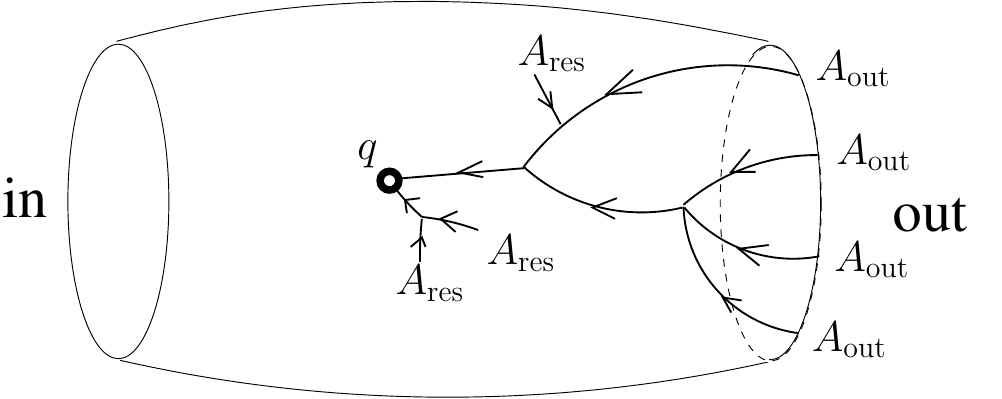}$$
\item One-loop graphs (a cycle with several trees attached to the cycle at the root) with leaves decorated by $A_\zm$ or $A_\dout$.
$$\includegraphics[scale=0.35]{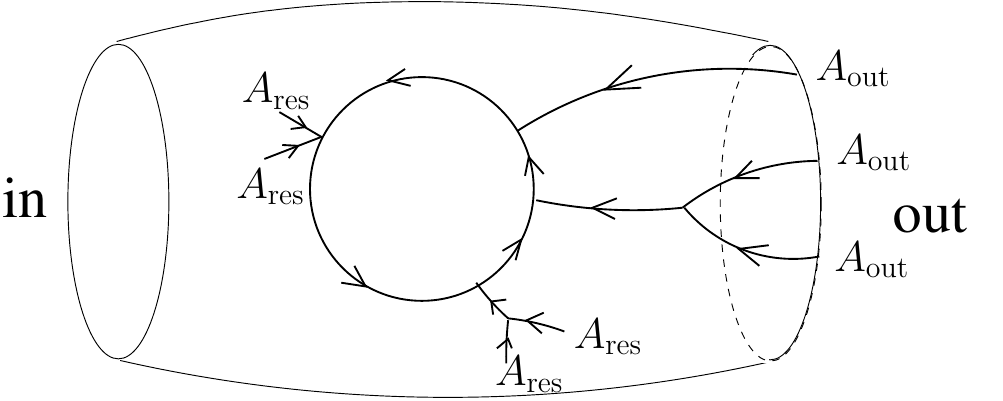}$$
\end{enumerate}
This classification of graphs yields the following ansatz for $S_\mr{eff}$:
$$S_\mr{eff}=\lan B_\din, \varphi^\mr{I}(A_\dout,A_\zm)\ran_{\din}+
\lan B_\zm, \varphi^\mr{II}(A_\dout,A_\zm)\ran_{\zm}-i\hbar\; \varphi^\mr{III+IV}(A_\dout,A_\zm)
$$
with $\varphi^\mr{I},\varphi^\mr{II},\varphi^\mr{III+IV}$ generally of unbounded degree in the variables $A_\dout,A_\zm$.
\end{remark}

\begin{theorem}
\label{prop 8.15} \leavevmode
\begin{enumerate}[(i)]
\item \label{prop 8.15 (i)} The partition function (\ref{Z non-ab cob explicit}) satisfies the modified quantum master equation\footnote{
Note that operators $\widehat{S}_\dd,\Delta_\zm$ here act on half-densities. So, in the conventions of Section \ref{sec: quantum cell ab BF on mfd with bdry}, we should be writing $\widehat{S}_\dd^\can,\Delta_\zm^\can$. We omit here the superscript $\can$ to lighten the notation.}
$$\left(\frac{i}{\hbar}\widehat{S}_\dd-i\hbar \Delta_\zm\right)Z=0$$
\item \label{prop 8.15 (ii)} A change of gauge-fixing data $(\ii,\pp,\K)$  changes $Z$ by a $\left(\frac{i}{\hbar}\widehat{S}_\dd-i\hbar \Delta_\zm\right)$-exact term.
\item \label{prop 8.15 (iii)} Considered modulo $\left(\frac{i}{\hbar}\widehat{S}_\dd-i\hbar \Delta_\zm\right)$-exact terms, $Z$ is independent of the cellular decomposition $X$ of $M$, provided that the cellular decomposition of the boundary is fixed.
\item \label{prop 8.15 (iv)} Assume, as in Proposition \ref{prop: gluing} that the cobordism $(M_1,X_1)\cob{(M,X)}(M_3,X_3)$ is obtained by composing (gluing) cobordisms $(M_1,X_1)\cob{(M_I,X_I)}(M_2,X_2)$ and $(M_2,X_2)\cob{(M_{II},X_{II})}(M_3,X_3)$. Then the partition function (\ref{Z non-ab cob explicit}) for the glued cobordism can be recovered from the partition functions for the constituent cobordisms by the same gluing formula (\ref{gluing formula}), as in abelian case. 

\end{enumerate}
\end{theorem}


\begin{proof}
Items (\ref{prop 8.15 (i)}), (\ref{prop 8.15 (ii)}) are an immediate consequence of (\ref{mQME non-ab}) and the general properties of finite-dimensional BV puhforwards (that the ``family over $\B$ version'' of BV pushforward is a chain map w.r.t. $\left(\frac{i}{\hbar}\Omega_\B-i\hbar\Delta\right)$ and that a change of gauge-fixing induces a change of the pushforward of a closed element by an exact element, see Theorem 2.14 in \cite{CMRpert}; in our case $\Omega_\B=\widehat{S}_\dd$).

To prove (\ref{prop 8.15 (iii)}), we (partially) switch back to the formalism of Section \ref{sec: simp BF reminder}. Let us regard the cellular action 
(\ref{S cell non-ab}) as a function on $\sF_{X,X_\dout}:=C^\bt(X,X_\dout)[1]\oplus C_\bt(X,X_\dout)[-2]$ (which is canonically symplectomorphic to $C^\bt(X,X_\dout)[1]\oplus {C^\bt(X^\vee,X_\din^\vee)[n-2]}$) with $A_\dout$ and $B_\din$ as external parameters -- this point of view allows us to forget about the dual CW complex in the bulk. Then we can perform elementary expansions and collapses on $X-X_\dd$ and by Lemma \ref{lemma: collapse}, together with the chain map property of BV pushforwards ``in a family'' quoted above, we obtain (\ref{prop 8.15 (iii)}) for a change of $X$ by a simple-homotopy relative to $X_\dd$. Since such simple-homotopy allows one to pass between any two CW decompositions of $M$ restricting to $X_\dd$ at the boundary, this proves (\ref{prop 8.15 (iii)}).

For the item (\ref{prop 8.15 (iv)}), we note that the entire discussion of gluing of Section \ref{sec: quantum gluing}, starting with the gluing formula (\ref{gluing e^S}) for the exponential of cellular action, works in non-abelian case exactly as in abelian case.
\end{proof}

\subsubsection{Reduction of the spaces of states (passage to the cohomology of the abelian part of BFV operators)}
Let us split the boundary BFV operator as $\widehat S_\dd=\widehat S_\dd^\mr{ab}+\widehat S_\dd^\mr{pert}$ with $\widehat{S}_\dd^\mr{ab}$ the abelian part (i.e. $n=1$ term in (\ref{Shat in},\ref{Shat out})). We can pass to the reduced space of states 
$$
\HH^\red_\dd:=H^\bt_{\widehat{S}_\dd^\mr{ab}}(\HH_\dd)=
\Fun\left(H^\bt(M_\dout,E)[1]\oplus H^\bt(M_\din,E^*)[n-2]\right)
$$
as in Section \ref{sec: reduced space of states}; i.e., reduced states are functions $[\psi]$ of cohomology classes $[A_\dout]$ and $[B_\din]$. 

Unlike in abelian case, we have a nonzero BFV operator $\widehat{S}_\dd^\red=\widehat{S}_\dout^\red+\widehat{S}_\din^\red$ on $\HH^\red_\dd$ induced from $\widehat S_\dd$ via homological perturbation theory, with $\widehat{S}_\dout^\red$, $\widehat{S}_\din^\red$ satifying the ansatz
\begin{eqnarray}
\widehat{S}_\dout^\red &=&\sum_{\nnn\geq 2}\frac{1}{\nnn!} \lan l_\nnn^{\red,\dout}([A_\dout],\ldots, [A_\dout]),-i\hbar \frac{\dd}{\dd [A_\dout]} \ran, \label{Shat reduced out}\\
\widehat{S}_\din^\red &=& \sum_{\nnn\geq 2} \frac{1}{\nnn!} \lan [B_\din], l_\nnn^{\red,\din}(\widehat{[A_\din]},\ldots, \widehat{[A_\din]}) \ran \label{Shat reduced in}
\end{eqnarray}
with $\widehat{[A_\din]}:=-i\hbar \frac{\dd}{\dd [B_\din]} $. 
Here $l_\nnn^{\red,\dout}$ are the $L_\infty$ algebra operations on $H^\bt(M_\dout,E)$ induced, via homotopy transfer, from the cellular $L_\infty$ structure on $C^\bt(X_\dout,E)$ produced by Theorem \ref{thm: cellBF}. It can also be viewed as induced by homotopy transfer from dg Lie version of de Rham algebra $\Omega^\bt(M_\dout,E)$; case of $L_\infty$ opertaions on in-boundary is similar. In particular, note that the cochain complex $(\HH^\red_\dd,\widehat{S}^\red_\dd)$, regarded modulo chain isomorphisms, is independent on the cellular decomposition of the boundary.

\begin{remark}
Reduced BFV differential $\widehat{S}^\red_\dd$ can be viewed as a generating function for Massey operations on cohomology of the boundary and thus 
determines
the rational homotopy type of the boundary (at least, in the case of simply-connected boundary).
\end{remark}

\begin{remark}
One can consider the total reduction of the space of states -- the cohomology of $\widehat{S}_\dd^\red$ on $\HH^\red_\dd$ which coincides, by homological perturbation lemma, with cohomology of the total BFV differential $\widehat{S}_\dd$ on $\HH_\dd$. This total reduction is isomorphic to 
$$\HH_\dd^\mr{tot.\;red.}\cong  H^\bt_{CE}\left(H^\bt(M_\dout,E),\{l_k^{\red,\dout}\}\right)\otimes \left(H^\bt_{CE}\left(H^\bt(M_\din,E^*),\{l_k^{\red,\din}\}\right)\right)^*$$
-- the  Chevalley-Eilenberg cohomology of the $L_\infty$ structure on the de Rham cohomology of the out-boundary, tensored with the Chevalley-Eilenberg homology of the respective $L_\infty$ structure associated to the in-boundary.
\end{remark}

Theorem \ref{prop 8.15} holds for the reduced partition function $Z^\red([B_\din],[A_\dout];A_\zm,B_\zm)\in \HH_\dd^{\red,\can}\widehat\otimes \mr{Dens}^{\frac12,\Fun}(\F_b^\zm)$ (defined by evaluating the partition function $Z$ on representatives of classes $[A_\dout]$, $[B_\din]$ in cellular cochains of the boundary, as in Section \ref{sec: reduced space of states}. Here we replace the BFV operator by its reduced version $\widehat{S}^\red_\dd$. In part (\ref{prop 8.15 (ii)}) of the Theorem in addition to changes of $(\ii,\pp,\K)$ we are now also allowing 
changes of $(i_\B,p_\B,K_\B)$ -- the HPT induction data from cellular cochains of the boundary to cohomology, as in Remark \ref{rem 7.16}.\footnote{Note that only $i_\B$ (choice of cellular representatives for cohomology classes) is relevant for the construction of $Z^\red= i_\B^* Z$ whereas $p_\B$ and $K_\B$ are manifestly irrelevant for $Z^\red$; however, the whole package $(i_\B,p_\B,K_\B)$ is involved in the construction of $\widehat{S}^\red_\dd$.}

In part (\ref{prop 8.15 (iii)}) of the Theorem we can now allow changes of cellular decomposition of $X$ that change the decomposition of the boundary.\footnote{Sketch of proof: for $Y$ an arbitrary cell decomposition of the cylinder $\Sigma\times [0,1]$ (regarded as a cobordism from $\Sigma$ to $\Sigma$) the reduced partition function $Z^\red: \HH^\red_{\Sigma}\ra \HH^\red_\Sigma$ is chain homotopic to identity (proven from the gluing property -- (\ref{prop 8.15 (iv)}) of Theorem \ref{prop 8.15}). Now let $X$ and $X'$ be two cellular decompositions of $M$. We can attach two cylinders at in- and out-boundaries of $X'$ to obtain a cell decomposition $\widetilde X$ of $\widetilde M$ -- a copy of $M$ with collars attached at in- and out-boundary, such that $\widetilde{X}_\din\simeq X_\din$ and $\widetilde{X}_\dout\simeq X_\dout$. By the previous observation about cylinders yielding identity up to homotopy, and by gluing formula, we have $Z_{\widetilde{X}}^\red\sim Z_{X'}^\red$ (where $\sim$ stands for 
equality up to $(\frac{i}{\hbar}\widehat{S}_\dd^\red-i\hbar \Delta_\zm)$-exact terms). On the other hand, we can view $\widetilde{X}$ and $X$ as two cellular decompositions of $M$ coinciding on the boundary, thus (\ref{prop 8.15 (iii)}) of Theorem \ref{prop 8.15} applies and we have $Z_{\widetilde{X}}^\red\sim Z_X^\red$. Thus, we have $Z_{X}^\red\sim Z^\red_{X'}$.}

\appendix


\section{Determinant lines, densities, $R$-torsion}\label{appendix: det lines, densities, torsion}\label{AA}
\subsection{Determinant lines, torsion of a complex of vector spaces}\label{appendix: det lines}
In what follows, {\it line} stands for an abstract 1-dimensional real vector space.
\begin{definition}For $V$ a finite-dimensional real vector space, 
the determinant line is defined as the top exterior power $\Det\; V=\wedge^{\dim V} V$. For $V^\bt$ a $\bZ$-graded vector space, one defines the determinant line as
$$\Det\; V^\bt=\bigotimes_k (\Det\; V^k)^{(-1)^k}$$
where for $L$ a line, $L^{-1}=L^*$ denotes the dual line.
\end{definition}

Here are several useful properties of determinant lines.
\begin{enumerate}[(i)]
\item The determinant line of the dual graded vector space is
$$\Det\; V^*\cong\left(\Det\; V\right)^{-1}$$
(with the grading convention $(V^*)^k=(V^{-k})^*$). In the case of a vector space concentrated in degree $0$, the pairing between $\Det\; V^*$ and $\Det\;V$ is given by
$$\langle v_n^*\wedge\cdots\wedge v_1^*\;,\;v_1\wedge\cdots\wedge v_n \rangle={\det}\langle v_i^*,v_j\rangle$$
with $v_i\in V$, $v_i^*\in V^*$ for $i=1,\ldots,n=\dim V$. Extension to the graded case is straightforward.
\item Determinant line of the degree-shifted vector space is
$$\Det\; V^\bt[k]\cong\left(\Det\; V^\bt\right)^{(-1)^k}$$
\item\label{det lines: multiplicativity} Given a short exact sequence of graded vector spaces $0\ra U^\bt\ra V^\bt\ra W^\bt\ra 0$, one has
\be \Det\;V^\bt\cong \Det\;U^\bt\otimes \Det\;W^\bt \label{Det V = Det U otimes Det W}\ee
In the case of non-graded vector spaces, the isomorphism sends
$$\underbrace{(u_1\wedge\cdots\wedge u_{\dim U})}_{\in\; \Det\;U}\otimes \underbrace{(w_1\wedge\cdots\wedge w_{\dim W})}_{\in\;\Det\; W}\quad\mapsto\quad \underbrace{u_1\wedge\cdots\wedge u_{\dim U}\wedge w'_1\wedge\cdots \wedge w'_{\dim W}}_{\in\; \Det\;V}$$
where on the right, $w'_i$ is some lifting of the element $w_i$ to $V$. Extension to the graded case is, again, straightforward.
\item If $V^\bt,d$ is a cochain complex with cohomology $H^\bt(V)$, there is a canonical isomorphism of determinant lines
    \be \bT:\; \Det \; V^\bt \xra{\cong} \Det\; H^\bt(V) \label{T: Det V to Det H}\ee
    Indeed, one applies property (\ref{Det V = Det U otimes Det W}) to the two short exact sequences
    $$ V_\mr{closed}^\bt\hra V^\bt \xra{d} V^{\bt+1}_\mr{exact},\qquad V^\bt_\mr{exact}\hra V^\bt_\mr{closed}\ra H^\bt(V) $$
    to obtain isomorphisms
    $$ \Det\; V^\bt\cong\Det\;V^\bt_\mr{closed} \otimes (\Det\;V^\bt_\mr{exact})^{-1},\qquad \Det\;V^\bt_\mr{closed}\cong \Det\;V^\bt_\mr{exact}\otimes \Det\; H^\bt(V) $$ which combine to (\ref{T: Det V to Det H}).

\end{enumerate}

All isomorphisms above are canonical (functorial).

It is convenient to work with determinant lines modulo signs, so that one can ignore the question of orientations and Koszul signs. We will use the notation $\Det\;V^\bt/\{\pm 1\}$ for {\it non-zero} elements of the determinant line considered modulo sign; so the precise notation should have been $(\Det\;V^\bt-\{0\})/\{\pm 1\}$.

\begin{definition}\label{def: torsion of a cx of vect spaces} For $V^\bt,d$ a cochain complex and $\mu\in \Det\;V^\bt/\{\pm 1\}$ a preferred element of the determinant line, defined up to sign, the torsion is defined as
$$\tau(V^\bt,d,\mu)=\bT(\mu)\quad\in \Det\;H^\bt(V)/\{\pm 1\} $$
with $\bT$ as in (\ref{T: Det V to Det H}).
\end{definition}

The following Lemma has important consequences in the setting of $R$-torsion (Section \ref{appendix: R-torsion}).
\begin{lemma}[Multiplicativity of torsions with respect to short exact sequences]\label{lemma: torsions multiplicativity}
Let $0\ra U^\bt\ra V^\bt \ra W^\bt\ra 0$ be a short exact sequence of complexes, equipped with elements $\mu_U,\mu_V,\mu_W$ in respective determinant lines, such that $\mu_V=\mu_U\cdot \mu_W$. Then for the torsions we have
$$\bT_\mr{LES}(\bT_U(\mu_U)\cdot \bT_V(\mu_V)^{-1}\cdot \bT_W(\mu_W))=1\quad \in \bR/\{\pm 1\}$$
where $\bT_{U},\bT_V,\bT_W$ are the maps (\ref{T: Det V to Det H}) for $U^\bt,V^\bt,W^\bt$. We denoted $\mr{LES}$ the induced long exact sequence in cohomology $\cdots\ra H^k(U)\ra H^k(V)\ra H^k(W)\ra H^{k+1}(U)\ra\cdots$ viewed as an acyclic cochain complex, and
$$\bT_\mr{LES}: \Det\; H^\bt(U)\otimes (\Det\; H^\bt(V))^{-1}\otimes \Det\; H^\bt(W)\ra \bR$$
is the corresponding isomorphism (\ref{T: Det V to Det H}).
\end{lemma}
See \cite{Milnor66} for details; cf. also \cite{MnevTorsions} for discussion in the language of determinant lines.

\subsection{Densities} \label{appendix: densities}
\begin{definition} For $\alpha\in\bR$ and $V$ a finite-dimensional real vector space, the space $\Dens^\alpha (V)$ of $\alpha$-densities on $V$ is defined as
the space of maps $\phi:F(V)\ra \bR_+$ from the space of bases (frames) in $V$ to positive half-line satisfying the equivariance property: for any automorphism $g\in GL(V)$ and any frame $\underline{v}=(v_1,\ldots,v_{\dim V})\in F(V)$, one has
$$\phi(g\cdot \underline{v})=|\det g|^\alpha\cdot \phi(\underline{v})$$
$\Dens^\alpha (V)$ is a torsor over $\bR_+$ (viewed as a multiplicative group), and in the setting of $\bZ$-graded vector spaces, one defines
$$\Dens^\alpha (V^\bt)=\bigotimes_k \left(\Dens^\alpha (V^k)\right)^{(-1)^k}$$
(tensor product is over $\bR_+$); $\alpha$ is called the {\it weight} of the density.
\end{definition}
By default a ``density'' has weight $\alpha=1$ (and then we write $\Dens$ instead of $\Dens^1$), and a ``half-density'' has, indeed, $\alpha=1/2$.

If $\phi_\alpha$, $\phi_\beta$ are two densities on $V^\bt$ of weights $\alpha,\beta$, then the product $\phi_\alpha\cdot\phi_\beta$ is an $(\alpha+\beta)$-density. Also, $\phi_\alpha$ can be raised to any real power $\gamma\in\bR$ to yield a density $(\phi_\alpha)^\gamma$  of weight $\alpha\cdot\gamma$. In particular, one has mutually inverse 
maps
$$\Dens^{1/2}V^\bt\xra{(\ast)^2}\Dens\;V^\bt,\quad \Dens\;V^\bt\xra{\sqrt{\ast}}\Dens^{1/2}V^\bt$$


Evaluation pairing $(\Det\; V^\bt/\{\pm 1\})\otimes \Dens\;V^\bt\ra\bR_+$ induces a canonical isomorphism of $\bR_+$-torsors
$$\Det\;V^\bt/\{\pm 1\}\cong \Dens\;V^\bt[1]$$

\subsection{$R$-torsion}\label{appendix: R-torsion}
Let $X$ be a finite CW-complex and $Y\subset X$ a CW-subcomplex. Let
\be h:\pi_1(X)=\pi_1\ra SL_{\pm}(m,\bR)\label{h}\ee
be some representation of the fundamental group of $X$ by real matrices of determinant $\pm 1$. It extends to a ring homomorphism $h:\bZ[\pi_1]\ra \mr{Mat}(m,\bR)$ from the group ring of $\pi_1$ to all real matrices of size $m$. Let $p: \til X\ra X$ be the universal cover of $X$ and denote $\til Y=p^{-1}(Y)\subset \til X$. Consider the cochain complex of vector spaces
\be C^\bt(X,Y;h)=\bR^m\otimes_{\bZ[\pi_1]} C^\bt(\til X,\til Y;\bZ) \label{C(X,Y;h)}\ee
where on the right we have integral cellular cochains of the pair $(\til X,\til Y)$, which is a complex of free $\bZ[\pi_1]$-modules with elements of $\pi_1$ acting on cells of $\til X$ by covering transformations, tensored with $\bR^m$ using the representation $h$. In $C^\bt(X,Y;h)$ one has a preferred basis of the form
\be \{v_i\otimes (\til e)^* \}_{1\leq i\leq m,\; e\subset X-Y} \label{cell basis}\ee
where $\{v_i\}$ is the standard basis on $\bR^m$ (or any unimodular basis, i.e. such that the standard density on $\bR^m$ evaluates on it to $\pm 1$) and $\til e$ are some liftings of cells $e$ of $X$ not lying in $Y$ to the universal cover; $(\til e)^*$ is the corresponding basis cochain.

Associated to the basis (\ref{cell basis}) by construction (\ref{coord density}) is an element
$\mu\in \Det \;C^\bt(X,Y;h)/\{\pm 1\}$, independent of the choices of liftings of cells $e\mapsto \til e$ and independent of the choice of unimodular basis in $\bR^m$.

\begin{definition} The $R$-torsion of the pair $(X,Y)$ of CW-complexes, associated to the representation (\ref{h}), is defined as the torsion (in the sense of Definition \ref{def: torsion of a cx of vect spaces}) of the complex $C^\bt(X,Y;h)$ equipped with element $\mu$:
$$\tau(X,Y;h)=\bT(\mu)\quad \in \Det \;H^\bt(X,Y;h)/\{\pm 1\}$$
Torsion of a single CW-complex $X$ is defined as $\tau(X;h):=\tau(X,\varnothing;h)$.
\end{definition}
Of particular importance (and historically the most studied) is the acyclic case, when $H^\bt(X,Y;h)=0$. Then the torsion takes values in the trivial line and thus is a number (modulo sign).

Instead of choosing a representation $h$ of $\pi_1$, one can choose a cellular local $SL_\pm(m)$-system $E$ on $X$, in the sense of Section \ref{sec: reminder on cell local systems}, and define $h$ as the holonomy of $E$. Cochain complex $C^\bt(X,Y;E)$ (dual to the chain complex $C_\bt(X,Y;E^*)$ constructed in Section \ref{sec: reminder on cell local systems}) is isomorphic to (\ref{C(X,Y;h)}). When we prefer to think in terms of a local system $E$ rather than a representation $h$ of $\pi_1$ (e.g. when we consider restriction to a CW-subcomplex, or gluing of two complexes along a subcomplex), we will write the torsion as $\tau(X,Y;E)$.

The following two properties are consequences of the multiplicativity of the algebraic torsion  with respect to short exact sequences of cochain complexes (Lemma \ref{lemma: torsions multiplicativity}).
\begin{enumerate}[(A)]
\item \label{torsions (A)} For $X\supset Y$ a pair of CW-complexes, one has
\be \tau(X;E)=\tau(X,Y;E)\cdot \tau(Y;E|_Y) \label{tau(X) = tau(X,Y) tau(Y)}\ee
The formula makes sense because $\Det\;H^\bt(X;E)\cong \Det\;H^\bt(X,Y;E)\otimes \Det\;H^\bt(Y;E|_Y)$, since the determinant line of the long exact sequence in homology of the pair $(X,Y)$ (regarded itself as a complex) is
$\Det\;H^\bt(X,Y;E)\otimes(\Det\;H^\bt(X;E))^{-1} \otimes\Det\;H^\bt(Y;E|_Y)$ and, on the other hand, is the trivial line, by (\ref{T: Det V to Det H}) applied to the long exact sequence.
\item For $Z=X\cup Y$ a CW-complex represented as a union of two intersecting subcomplexes, one the gluing (inclusion/exclusion) formula
    \be\tau(X\cup Y;E)=\tau(X;E|_X)\cdot \tau(Y;E|_Y)\cdot \tau(X\cap Y;E|_{X\cap Y})^{-1} \label{tau inclusion-exclusion}\ee
    The reason why l.h.s. and r.h.s. can be at all compared is as in (\ref{torsions (A)}), but one replaces the long exact sequence of a pair by Mayer-Vietoris sequence.\footnote{
    More pedantically, (\ref{tau(X) = tau(X,Y) tau(Y)}) should be written as
$\bT_\mr{LES}(\tau(X,Y;E)\cdot \tau(X;E)^{-1}\cdot \tau(Y;E|_Y))=1$,
with $\bT_\mr{LES}$ the isomorphism (\ref{T: Det V to Det H}) between the determinant line of the long exact sequence of cohomology of the pair and the standard line $\bR$. Likewise, (\ref{tau inclusion-exclusion}) should be written as $\bT_\mr{MV}(\tau(X\cup Y;E)\cdot \tau(X;E|_X)^{-1}\cdot \tau(Y;E|_Y)^{-1}\cdot \tau(X\cap Y;E|_{X\cap Y}))=1$ with $\bT_\mr{MV}$ the isomorphism (\ref{T: Det V to Det H}) for the Mayer-Vietoris long exact sequence.
    }
\end{enumerate}
In the acyclic case (i.e. when all relevant cohomology spaces vanish), (\ref{tau(X) = tau(X,Y) tau(Y)},\ref{tau inclusion-exclusion}) are equalities of numbers.

\begin{theorem}[Combinatorial invariance of $R$-torsion] If $(X',Y')$ is a cellular subdivision of the pair $(X,Y)$, then
$$\tau(X',Y';h)=\tau(X,Y;h)$$
\end{theorem}
For the proof, see e.g. \cite{Milnor66}. The case $Y=Y'=\varnothing$ is due to 
Reidemeister, Franz and de Rham.

The combinatorial invariance theorem implies in particular that, for $M$ a compact PL manifold with two different cellular decompositions $X$ and $Y$, one has
$\tau(X;h)=\tau(Y;h)$. Thus in this case it makes sense to talk about the $R$-torsion of a manifold $M$, $\tau(M;h)$, forgetting about the cellular subdivision.

\begin{theorem}[Milnor, \cite{Milnor62}] If $M$ is a piecewise-linear compact oriented $n$-manifold with boundary $\dd M=\dd_1 M\sqcup \dd_2 M$, one has
$$\tau(M,\dd_1 M;h)=(\tau(M,\dd_2 M;h^*))^{(-1)^{n-1}}$$
where $h^*$ is the dual representation to $h$.
\end{theorem}
Note that the l.h.s. belongs to $\Det\; H^\bt(M,\dd_1 M;h)$ while the r.h.s. belongs to $(\Det\; H^\bt(M,\dd_2 M;h^*))^{(-1)^{n-1}}$ (modulo signs); these determinant lines are canonically isomorphic due to Poincar\'e-Lefschetz duality $H^k(M,\dd_1 M;h)\cong (H^{n-k}(M,\dd_2 M;h^*))^*$. Thus it does make sense to compare the two torsions.

\begin{corollary} For $M$ a closed even-dimensional manifold and $h$ such that $H^\bt(M;h)=0$, the torsion is trivial, $\tau(M;h)=1$.
\end{corollary}



\section{Two points of view on ``$C_\infty\otimes \mr{Lie}=L_\infty$''}\label{Appendix: C-infty otimes Lie}
In connection with Remark \ref{rem: L_infty and RHT}, we recall two ways to see the $L_\infty$ algebra structure on the tensor product of a $C_\infty$ algebra and a Lie algebra. 

Given a $C_\infty$ algebra $W$ with multlinear operations $m_n:W^{\otimes n}\ra W$,
and given a Lie algebra $\g$, one can construct the tensor product $L_\infty$ algebra structure on the graded vector space $W\otimes \g$ by defining
\begin{equation}\label{l_n from m_n}
l_n(w_1\otimes \alpha_1,\ldots,w_n\otimes\alpha_n)=\sum_{\sigma\in S_n}\pm m_n(w_{\sigma_1},\ldots,w_{\sigma_n})\otimes (\alpha_{\sigma_1}\cdots \alpha_{\sigma_n})
\end{equation}
with $w_1,\ldots,w_n\in W$ and $\alpha_1,\ldots,\alpha_n\in\g$  arbitrary elements. The sum on the r.h.s. is over permuations $\sigma$ of $1,\ldots,n$. Here the product of $\alpha_i$'s is seen as a product in the universal enveloping algebra $U\g$. The $C_\infty$ property of the operation $m_n$  (vanishing on shuffle-products) implies that the result lands in $W\otimes \g\subset W\otimes U\g$.
 
Another way to present the same  tensor product $L_\infty$ structure on $W\otimes \g$ is as follows.
The $C_\infty$ operations $m_n$ can be written in the form 
\begin{equation} \label{m_n from m_n^T}
m_n(w_1,\ldots,w_n)=\sum_{T,\pi} m_n^T\circ\pi^{-1}(w_1\otimes \cdots \otimes w_n)
\end{equation} 
where the sum runs over binary rooted trees $T$ with $n$ leaves (viewed up to graph automorphism; for each $T$ we fix arbitrarily a ``standard'' planar realization) and their planar realizations $\pi$; $m_n^T\in \mr{Hom}(W^{\otimes n},W)^{\mr{Aut}(T)}$ are some multilinear operations invariant w.r.t. automorphisms of $T$ acting by permutations of factors in $W^{\otimes n}$ (with appropriate signs); $w_1,\ldots,w_n\in W$ are arbitrary vectors; $\pi^{-1}(\cdots)$ is understood as a permutation of factors in $W^{\otimes n}$ corresponding to going from the planar representative $\pi$ to the ``standard'' representative of $T$. Then the tensor product $L_\infty$ algebra structure on $W\otimes \g$ is given by 
\begin{multline}\label{l_n from m_n^T and Jacobi}
l_n(w_1\otimes \alpha_1,\ldots, w_n\otimes\alpha_n)=\\
=\sum_{\sigma\in S_n}\sum_T \pm \frac{1}{|\mr{Aut}(T)|} m_n^T(w_{\sigma_1},\ldots,w_{\sigma_n}) \otimes \mr{Jacobi}_T(\alpha_{\sigma_1},\ldots \alpha_{\sigma_n})
\end{multline} 
with 
$\mr{Jacobi}_T(\cdots)$ the nested commutator determined by the tree $T$. 

Here the first point of view on the tensor product (\ref{l_n from m_n}) is more direct and does not 
require
splitting $m_n$ into pieces $m_n^T$ possessing different symmetries. However, we wanted to also present the second point of view (\ref{l_n from m_n^T and Jacobi}) since it compares directly to the tree part of (\ref{Sbar cell}) and explains how to construct the corresponding $C_\infty$ algebra (via (\ref{m_n from m_n^T})).


\end{document}